\documentclass[reqno,12pt]{amsart}

\usepackage{amsmath, amsthm, amssymb}
\usepackage{enumitem}
\usepackage{pdflscape}
\usepackage{caption}
\usepackage{bm}

\usepackage{ifpdf}
\ifpdf
\usepackage[pdftex]{graphicx}
\else
\usepackage[dvips]{graphicx}
\fi
\usepackage[all]{xy}

\usepackage{multicol}
\usepackage{mathtools}

\usepackage{tocvsec2}

\usepackage{bbm}

\input xy
\xyoption{all}

\usepackage[pdftex,plainpages=false,hypertexnames=false,pdfpagelabels]{hyperref}
\newcommand{\arxiv}[1]{\href{http://arxiv.org/abs/#1}{\tt arXiv:\nolinkurl{#1}}}
\newcommand{\arXiv}[1]{\href{http://arxiv.org/abs/#1}{\tt arXiv:\nolinkurl{#1}}}
\newcommand{\doi}[1]{\href{http://dx.doi.org/#1}{{\tt DOI:#1}}}

\newcommand{\googlebooks}[1]{(preview at \href{http://books.google.com/books?id=#1}{google books})}

\usepackage{xcolor}
\definecolor{dark-red}{rgb}{0.7,0.25,0.25}
\definecolor{dark-blue}{rgb}{0.15,0.15,0.55}
\definecolor{medium-blue}{rgb}{0,0,.8}
\definecolor{DarkGreen}{RGB}{0,150,0}
\definecolor{rho}{named}{red}
\hypersetup{
   colorlinks, linkcolor={purple},
   citecolor={medium-blue}, urlcolor={medium-blue}
}

\usepackage{longtable}
\usepackage{fullpage}

\theoremstyle{plain}
\newtheorem{thm}{Theorem}[section]
\newtheorem*{thm*}{Theorem}
\newtheorem{thmalpha}{Theorem}

\newtheorem{cor}[thm]{Corollary}
\newtheorem{coralpha}[thmalpha]{Corollary}
\newtheorem*{cor*}{Corollary}

\newtheorem*{conj*}{Conjecture}
\newtheorem{lem}[thm]{Lemma}
\newtheorem{lemalpha}[thmalpha]{Lemma}
\newtheorem{prop}[thm]{Proposition}
\newtheorem{quest}[thm]{Question}
\newtheorem*{quest*}{Question}
\newtheorem*{claim*}{Claim}

\theoremstyle{definition}
\newtheorem{defn}[thm]{Definition}

\newtheorem{facts}[thm]{Facts}
\newtheorem{construction}[thm]{Construction}

\newtheorem{assumption}[thm]{Assumption}
\newtheorem{nota}[thm]{Notation}

\newtheorem{ex}[thm]{Example}
\newtheorem{sub-ex}[thm]{Sub-Example}
\newtheorem{counter-ex}[thm]{Counter-Example}
\newtheorem{rem}[thm]{Remark}
\newtheorem*{rem*}{Remark}

\DeclareMathOperator{\coev}{coev}

\DeclareMathOperator{\End}{End}
\DeclareMathOperator{\ev}{ev}
\DeclareMathOperator{\Hom}{Hom}

\DeclareMathOperator{\op}{op}

\DeclareMathOperator{\spann}{span}

\DeclareMathOperator{\id}{id}

\DeclareMathOperator{\ind}{ind}
\DeclareMathOperator{\im}{im}
\DeclareMathOperator{\Irr}{Irr}
\DeclareMathOperator{\FPdim}{FPdim}
\DeclareMathOperator{\Spec}{Spec}

\DeclareMathOperator{\tr}{tr}

\newcommand{\comment}[1]{}
\newcommand{\set}[2]{\left\{#1 \middle| #2\right\}}
\newcommand{\rmB}{\mathrm{B}}

\newcommand\longmapsfrom{\mathrel{\reflectbox{\ensuremath{\longmapsto}}}}

\newcommand{\noshow}[1]{}

\def\semicolon{;}
\def\applytolist#1{
    \expandafter\def\csname multi#1\endcsname##1{
        \def\multiack{##1}\ifx\multiack\semicolon
            \def\next{\relax}
        \else
            \csname #1\endcsname{##1}
            \def\next{\csname multi#1\endcsname}
        \fi
        \next}
    \csname multi#1\endcsname}

\def\calc#1{\expandafter\def\csname c#1\endcsname{{\mathcal #1}}}
\applytolist{calc}QWERTYUIOPLKJHGFDSAZXCVBNM;
\def\bbc#1{\expandafter\def\csname bb#1\endcsname{{\mathbb #1}}}
\applytolist{bbc}QWERTYUIOPLKJHGFDSAZXCVBNM;
\def\bfc#1{\expandafter\def\csname bf#1\endcsname{{\mathbf #1}}}
\applytolist{bfc}QWERTYUIOPLKJHGFDSAZXCVBNM;
\def\sfc#1{\expandafter\def\csname s#1\endcsname{{\sf #1}}}
\applytolist{sfc}QWERTYUIOPLKJHGFDSAZXCVBNM;
\def\fc#1{\expandafter\def\csname f#1\endcsname{{\mathfrak #1}}}
\applytolist{fc}QWERTYUIOPLKJHGFDSAZXCVBNM;

\newcommand{\Mod}{{\sf Mod}}

\newcommand{\Bim}{{\sf Bim}}
\newcommand{\fgprCorr}{{\mathsf{C^*Alg_{fgp}}}}

\newcommand{\Hilb}{{\sf Hilb}}
\newcommand{\fdHilb}{{\sf Hilb_{fd}}}
\newcommand{\rCorr}{{\mathsf{C^*Alg}}}
\newcommand{\WStarRCorr}{{\mathsf{W^*Alg}}}
\newcommand{\QSys}{{\mathsf{QSys}}}
\newcommand{\vNA}{{\sf vNA}}

\newcommand{\xz}{\boxtimes}
\newcommand{\xzq}{\otimes}

\usepackage{tikz}
\usepackage{tikz-cd}
\usetikzlibrary{arrows,backgrounds,patterns.meta}
\usetikzlibrary{positioning,shadings,cd}
\usetikzlibrary{shapes}
\usetikzlibrary{backgrounds}
\usetikzlibrary{decorations,decorations.pathreplacing,decorations.markings}
\usetikzlibrary{fit,calc,through}
\usetikzlibrary{external}
\usetikzlibrary{arrows}
\tikzset{vertex/.style = {shape=circle,draw,fill=black,inner sep=0pt,minimum size=5pt}}
\tikzset{edge/.style = {->,> = latex', bend right}}
\tikzset{
	super thick/.style={line width=3pt}
}
\tikzset{
    quadruple/.style args={[#1] in [#2] in [#3] in [#4]}{
        #1,preaction={preaction={preaction={draw,#4},draw,#3}, draw,#2}
    }
}
\tikzstyle{shaded}=[fill=red!10!blue!20!gray!30!white]
\tikzstyle{unshaded}=[fill=white]
\tikzstyle{empty box}=[circle, draw, thick, fill=white, opaque, inner sep=2mm]
\tikzstyle{annular}=[scale=.7, inner sep=1mm, baseline]
\tikzstyle{rectangular}=[scale=.75, inner sep=1mm, baseline=-.1cm]
\tikzstyle{mid>}=[decoration={markings, mark=at position 0.5 with {\arrow{>}}}, postaction={decorate}]
\tikzstyle{mid<}=[decoration={markings, mark=at position 0.5 with {\arrow{<}}}, postaction={decorate}]
\tikzstyle{over}=[double, draw=white, super thick, double=]

\tikzdeclarepattern{
  name=primeddots,
  type=uncolored,
  bounding box={(-.6pt,-.6pt) and (.6pt,.6pt)},
  tile size={(5pt,5pt)},
  tile transformation={rotate=60},
  parameters={none}, 
  code={
    \fill(0pt,0pt) circle (.35pt);
  }
}

\tikzdeclarepattern{
  name=primeddots2,
  type=uncolored,
  bounding box={(-.3pt,-.3pt) and (.3pt,.3pt)},
  tile size={(2.5pt,2.5pt)},
  tile transformation={rotate=60},
  parameters={none}, 
  code={
    \fill(0pt,0pt) circle (.35pt);
  }
}

\tikzstyle{primedregion}[none]=[
	preaction={fill=#1},
	pattern=primeddots,
  draw=#1,
]

\tikzstyle{primedregion2}[none]=[
	preaction={fill=#1},
	pattern=primeddots2,
  draw=#1,
]

\newcommand{\roundNbox}[6]{
	\draw[rounded corners=5pt, very thick, #1] ($#2+(-#3,-#3)+(-#4,0)$) rectangle ($#2+(#3,#3)+(#5,0)$);
	\coordinate (ZZa) at ($#2+(-#4,0)$);
	\coordinate (ZZb) at ($#2+(#5,0)$);
	\node at ($1/2*(ZZa)+1/2*(ZZb)$) {#6};
}

\newcommand{\halfRoundBox}[5]{
    \fill[rounded corners=5pt, #1] ($#2+(-#3,0)$) -- ($#2+(-#3,-#3)-(0,#4)$) -- ($#2+(#3,-#3)-(0,#4)$) -- ($#2+(#3,0)$);
    \fill[#1] ($#2+(-#3,0)$) -- ($#2+(-#3,0)+.5*(0,-#3)$) -- ($#2+(#3,0)+.5*(0,-#3)$) -- ($#2+(#3,0)$) arc (0:180:{#3});
	\draw[rounded corners=5pt, very thick] ($#2+(-#3,0)$) -- ($#2+(-#3,-#3)-(0,#4)$) -- ($#2+(#3,-#3)-(0,#4)$) -- ($#2+(#3,0)$);
	\draw[very thick] ($#2+(#3,0)$) arc (0:180:{#3});
	\node at #2 {#5};
}

\newcommand{\halfRoundBoxDag}[5]{
    \fill[rounded corners=5pt, #1] ($#2+(-#3,0)$) -- ($#2+(-#3,#3)+(0,#4)$) -- ($#2+(#3,#3)+(0,#4)$) -- ($#2+(#3,0)$);
    \fill[#1] ($#2+(-#3,0)$) -- ($#2+(-#3,0)+.5*(0,#3)$) -- ($#2+(#3,0)+.5*(0,#3)$) -- ($#2+(#3,0)$) arc (0:-180:{#3});
	\draw[rounded corners=5pt, very thick] ($#2+(-#3,0)$) -- ($#2+(-#3,#3)+(0,#4)$) -- ($#2+(#3,#3)+(0,#4)$) -- ($#2+(#3,0)$);
	\draw[very thick] ($#2+(#3,0)$) arc (0:-180:{#3});
	\node at #2 {#5};
}

\newcommand{\tikzmath}[2][]
     {\vcenter{\hbox{\begin{tikzpicture}[#1]#2
                     \end{tikzpicture}}}
     }

\newcommand{\DoubleStrand}[4]{
\draw[ultra thick, dash pattern=on .7pt off 2pt] #1 -- #2;
\draw[thick, #3] ($ #1 + (-.04,0)$) -- ($ #2 + (-.04,0)$);
\draw[thick, #4] ($ #1 + (.04,0)$) -- ($ #2 + (.04,0)$);;
}

\newcommand{\TripleStrand}[5]{
\draw[ultra thick, dash pattern=on .7pt off 2pt] ($ #1 + (-.04,0)$) -- ($ #2 + (-.04,0)$);
\draw[ultra thick, dash pattern=on .7pt off 2pt] ($ #1 + (.04,0)$) -- ($ #2 + (.04,0)$);
\draw[thick, #3] ($ #1 + (-.08,0)$) -- ($ #2 + (-.08,0)$);
\draw[thick, #4] #1 -- #2;
\draw[thick, #5] ($ #1 + (.08,0)$) -- ($ #2 + (.08,0)$);;
}

\newcommand{\MColor}{red}

\newcommand{\XColor}{red}
\newcommand{\YColor}{orange}
\newcommand{\ZColor}{blue}
\newcommand{\PsColor}{brown!80} 
\newcommand{\QsColor}{DarkGreen}
\newcommand{\RsColor}{cyan!90}

\newcommand{\AColor}{gray!30}
\newcommand{\BColor}{gray!55}
\newcommand{\CColor}{gray!80}
\newcommand{\DColor}{gray!95}
\newcommand{\PrColor}{yellow!50} 
\newcommand{\QrColor}{green!30}
\newcommand{\RrColor}{cyan!30}

\tikzcdset{scale cd/.style={every label/.append style={scale=#1},
    cells={nodes={scale=#1}}}}

\begin{document}
\title{Q-system completion for \texorpdfstring{$\rm C^*$}{C*/W*} 2-categories}
\author{Quan Chen, Roberto Hern\'{a}ndez Palomares, Corey Jones, and David Penneys}
\date{\today}
\maketitle
\begin{abstract}
A Q-system in a $\rm C^*$ 2-category 
is 
a unitary version of a separable Frobenius algebra object
and can be viewed as a unitary version
of a higher idempotent.
We define a higher unitary idempotent completion for $\rm C^*$ 2-categories called Q-system completion and study its properties.
We show that the $\rm C^*$ 2-category of right correspondences of unital $\rm C^*$-algebras is Q-system complete
by constructing an inverse realization $\dag$-2-functor.
We use this result to construct induced actions of 
group theoretical unitary fusion categories on continuous trace $\rm C^*$-algebras with connected spectra.
%
\end{abstract}

\tableofcontents

\section{Introduction}

A \emph{Q-system} is a unitary version of a Frobenius algebra object in a $\rm C^*$ tensor category or $\rm C^*$ 2-category.
Q-systems were first introduced in \cite{MR1257245} to characterize the canonical endomorphism associated to a finite index subfactor of an infinite factor \cite{MR739630,MR1027496}. 
Following \cite{MR1966524}, a Q-system in a unitary tensor category (a semisimple rigid $\rm C^*$ tensor category with simple unit) is an alternative axiomatization of the standard invariant of a finite index subfactor \cite{MR996454,MR1334479,MR4374438}. 
This viewpoint has been fruitful for classification of small index subfactors \cite{MR3166042,MR3345186,1509.00038} and constructing new subfactors from existing examples \cite{MR2909758,MR3859276,1810.06076}.

Given a unitary tensor category $\cC$, an indecomposable Q-system $Q\in \cC$ ($\End_{Q-Q}(Q)=\bbC$), and a fully-faithful unitary tensor functor $H: \cC \to \Bim(N)$ for some $\rm II_1$ factor $N$, we can perform the \emph{realization} procedure to reconstruct a $\rm II_1$-factor $M$ containing $N$ as a generalized crossed product $M= N\rtimes_H Q$ \cite{MR3948170,MR4079745}. Furthermore, \textit{every} irreducible, finite index extension of $N$ is of this form. Therefore, this technique splits the problem of classifying finite index extensions of a $\rm{II}_1$ factor $N$ into two parts:
\begin{enumerate}[label=(P\arabic*)]
    \item 
    \label{problem:analytic}
    The \emph{analytic} problem of building and classifying actions $H$ of unitary tensor categories $\cC$ on $N$, and 
    \item
    \label{problem:algebraic}
    The \emph{algebraic} problem of classifying Q-systems $Q$ in unitary tensor categories $\cC$.
\end{enumerate}

The first part \ref{problem:analytic} is a generalization of the notoriously difficult problem of classifying group actions on $\rm{II}_1$ factors up to cocycle conjugacy \cite{MR394228,MR448101,MR587749,MR596082} \cite{MR2386109,MR2370283,MR2409162}, while the second part \ref{problem:algebraic} can be viewed as a non-abelian cohomology problem internal to $\cC$. 
In the case that $N=R$ is hyperfinite and $\cC$ is (strongly) amenable, a deep result of Popa shows that there exists a unique action of $\cC$ on $R$ \cite{MR1278111,MR1339767} (cf.~\cite{MR4236062}). 
Thus finite index extensions of $R$ with amenable standard invariants correspond to Q-systems in amenable unitary tensor categories.
For finite depth, this a fundamentally algebraic problem which has seen tremendous success in the aforementioned small index subfactor classification results.

Beyond the hyperfinite and amenable case, little is known about classification of unitary tensor category actions, with some exceptions \cite{MR2504433,MR2471930,MR2838524,MR3028581}.
However, the algebraic understanding obtained from studying the second part \ref{problem:algebraic} immediately carries over to new situations, since it is completely independent from $N$. 
Thus whenever a unitary tensor category $\cC$ acts on a $\rm{II}_{1}$ factor $N$, we can automatically \emph{realize} the algebraic theory internal to $\cC$ to build finite index extensions of $N$ and bimodules between them. 
In particular, we can construct new actions of categories Morita equivalent to $\cC$ on finite index extensions of $N$ \emph{for free}. 

Recently there has been significant interest in classifying actions of groups and unitary tensor categories on (unital) $\rm C^*$-algebras \cite{MR2053753,MR4015345,MR3572256,1810.05850,1906.03818} \cite{MR4328058,2105.05587,AranoRokhlin}.
One of the main goals of this article is to perform realization in the $\rm C^*$ setting to demonstrate how to apply the subfactor techniques discussed above directly to the $\rm C^*$ context.
That is,  whenever a unitary tensor category $\cC$ acts on a unital $\rm C^*$-algebra $A$, we can automatically \emph{realize} solutions to \ref{problem:algebraic} to build finite (Watatani \cite{MR996807}) index extensions of $A$ and bimodules between them.

Realization is not only valuable as a method to construct inclusions of operator algebras, but it is also inverse to a \emph{higher idempotent completion}. 
Idempotents in a category can replicate freely.
$$
\tikzmath{
\draw (0,0)node[below]{$\scriptstyle a$} -- (1,0)node[below]{$\scriptstyle a$} ;
\filldraw (.5,0) node[above]{$\scriptstyle e$} circle (.05cm);
}
=
\tikzmath{
\draw (0,0)node[below]{$\scriptstyle a$} --node[below]{$\scriptstyle a$} (2,0)node[below]{$\scriptstyle a$} ;
\filldraw (.5,0) node[above]{$\scriptstyle e$} circle (.05cm);
\filldraw (1.5,0) node[above]{$\scriptstyle e$} circle (.05cm);
}
=
\tikzmath{
\draw (0,0)node[below]{$\scriptstyle a$} -- (2,0)node[below]{$\scriptstyle a$} ;
\filldraw (.5,0) node[above]{$\scriptstyle e$} circle (.05cm);
\filldraw (.75,0) node[above]{$\scriptstyle e$} circle (.05cm);
\node at (1.03,.1) {$\scriptstyle \cdots$};
\filldraw (1.25,0) node[above]{$\scriptstyle e$} circle (.05cm);
\filldraw (1.5,0) node[above]{$\scriptstyle e$} circle (.05cm);
}
=
\tikzmath{
\draw (0,0)node[below]{$\scriptstyle a$} -- (2,0)node[below]{$\scriptstyle a$} ;
\draw[very thick] (.5,0) node[below]{$\scriptstyle r$} --node[above]{$\scriptstyle e$} (1.5,0) node[below]{$\scriptstyle s$};
}
$$
Densely packing the idempotent $e$, we see that $e$ behaves like the identity of some new object in our category.
This directly leads to the notion of idempotent splitting.

One level higher, a Q-system is a 1-morphism $Q\in \End_\cC(c)$ in a $\rm C^*/W^*$ 2-category $\cC$ equipped with a multiplication $m: Q\xz Q \to Q$ and unit $i: 1_c \to Q$ which are denoted pictorially by a trivalent and univalent vertex respectively.
$$
\tikzmath{
\fill[\AColor, rounded corners=5pt] (-.3,0) rectangle (.9,.6);
\draw (0,0) arc (180:0:.3cm);
\draw (.3,.3) -- (.3,.6);
\filldraw (.3,.3) circle (.05cm);
}=m
\qquad\qquad
\tikzmath{
\fill[\AColor, rounded corners=5pt] (-.3,0) rectangle (.9,-.6);
\draw (0,0) arc (-180:0:.3cm);
\draw (.3,-.3) -- (.3,-.6);
\filldraw (.3,-.3) circle (.05cm);
}=m^\dag
\qquad\qquad
\tikzmath{
\fill[\AColor, rounded corners=5pt] (0,0) rectangle (.6,.6);
\draw (.3,.3) -- (.3,.6);
\filldraw (.3,.3) circle (.05cm);
}=i
\qquad\qquad
\tikzmath{
\fill[\AColor, rounded corners=5pt] (0,0) rectangle (.6,-.6);
\draw (.3,-.3) -- (.3,-.6);
\filldraw (.3,-.3) circle (.05cm);
}=i^\dag
\qquad\qquad
\tikzmath{
\fill[\AColor, rounded corners=5pt] (0,0) rectangle (.5,.5); 
}
=c
$$
These 2-morphisms satisfy certain associativity, Frobenius, and separability relations (and more)
$$
\tikzmath{
\fill[\AColor, rounded corners=5pt] (-.3,-.3) rectangle (1.2,.6);
\draw (0,-.3) -- (0,0) arc (180:0:.3cm);
\draw (.3,-.3) arc (180:0:.3cm);
\draw (.3,.3) -- (.3,.6);
\filldraw (.3,.3) circle (.05cm);
\filldraw (.6,0) circle (.05cm);
}
=
\tikzmath{
\fill[\AColor, rounded corners=5pt] (-.6,-.3) rectangle (.9,.6);
\draw (0,0) arc (180:0:.3cm) -- (.6,-.3);
\draw (-.3,-.3) arc (180:0:.3cm);
\draw (.3,.3) -- (.3,.6);
\filldraw (.3,.3) circle (.05cm);
\filldraw (0,0) circle (.05cm);
}
\qquad\qquad
\tikzmath{
\fill[\AColor, rounded corners=5pt] (-.3,-.6) rectangle (1.5,.6);
\draw (0,-.6) -- (0,0) arc (180:0:.3cm) arc (-180:0:.3cm) -- (1.2,.6);
\draw (.3,.3) -- (.3,.6);
\draw (.9,-.3) -- (.9,-.6);
\filldraw (.3,.3) circle (.05cm);
\filldraw (.9,-.3) circle (.05cm);
}
=
\tikzmath{
\fill[\AColor, rounded corners=5pt] (-.3,0) rectangle (.9,1.2);
\draw (0,0) arc (180:0:.3cm);
\draw (0,1.2) arc (-180:0:.3cm);
\draw (.3,.3) -- (.3,.9);
\filldraw (.3,.3) circle (.05cm);
\filldraw (.3,.9) circle (.05cm);
}
=
\tikzmath{
\fill[\AColor, rounded corners=5pt] (-.3,.6) rectangle (1.5,-.6);
\draw (0,.6) -- (0,0) arc (-180:0:.3cm) arc (180:0:.3cm) -- (1.2,-.6);
\draw (.3,-.3) -- (.3,-.6);
\draw (.9,.3) -- (.9,.6);
\filldraw (.3,-.3) circle (.05cm);
\filldraw (.9,.3) circle (.05cm);
}
\qquad\qquad
\tikzmath{
\fill[\AColor, rounded corners=5pt] (-.3,0) rectangle (.9,1.2);
\draw (0,.6) arc (180:-180:.3cm);
\draw (.3,1.2) -- (.3,.9);
\draw (.3,0) -- (.3,.3);
\filldraw (.3,.3) circle (.05cm);
\filldraw (.3,.9) circle (.05cm);
}
=
\tikzmath{
\fill[\AColor, rounded corners=5pt ] (0,0) rectangle (.6,1.2);
\draw (.3,0) -- (.3,1.2);
}
$$
which allow us to change the connectivity of a trivalent graph whose edges are labelled by $Q$ and whose vertices are labelled by $m,m^\dag$.
The separability condition allows these 2D meshes to replicate in 2D, so that the trivalent graph depends only on the connectivity, and not on the genus.
$$
\tikzmath{
\filldraw[rounded corners=5pt, \AColor] (-.8,-.6) rectangle (.8,.6);
\draw (0,-.6) -- (0,0);
\draw (-.6,.6) -- (0,0) -- (.6,.6);
\filldraw (0,0) circle (.03cm);
}
=
\tikzmath{
\filldraw[rounded corners=5pt, \AColor] (-.8,-.6) rectangle (.8,.6);
\draw (0,-.6) -- (0,0);
\draw (-.6,.6) -- (0,0) -- (.6,.6);
\draw (-.2,.2) -- (.2,.2);
\filldraw (0,0) circle (.03cm);
\filldraw (-.2,.2) circle (.03cm);
\filldraw (.2,.2) circle (.03cm);
}
=
\tikzmath{
\filldraw[rounded corners=5pt, \AColor] (-.8,-.6) rectangle (.8,.6);
\draw (0,-.6) -- (0,0);
\draw (-.6,.6) -- (0,0) -- (.6,.6);
\draw (-.2,.2) -- (.2,.2);
\draw (-.4,.4)  .. controls ++(-120:.3cm) and ++(180:.3cm) .. (0,.-.4); 
\draw (0,-.2)  .. controls ++(0:.3cm) and ++(-30:.3cm) .. (.4,.4); 
\filldraw (0,-.2) circle (.03cm);
\filldraw (0,0) circle (.03cm);
\filldraw (-.2,.2) circle (.03cm);
\filldraw (0,-.4) circle (.03cm);
\filldraw (-.4,.4) circle (.03cm);
\filldraw (.2,.2) circle (.03cm);
\filldraw (.4,.4) circle (.03cm);
}
=
\cdots
=
\tikzmath{
\filldraw[rounded corners=5pt, \AColor] (-.8,-.6) rectangle (.8,.6);
\fill[\BColor] (-.4,.4)  .. controls ++(-120:.3cm) and ++(180:.3cm) .. (0,.-.4) -- (0,-.2)  .. controls ++(0:.3cm) and ++(-30:.3cm) .. (.4,.4) -- (.2,.2) -- (-.2,.2);
\draw (0,-.6) -- (0,0);
\draw (-.6,.6) -- (0,0) -- (.6,.6);
\draw (-.2,.2) -- (.2,.2);
\draw (-.4,.4)  .. controls ++(-120:.3cm) and ++(180:.3cm) .. (0,.-.4); 
\draw (0,-.2)  .. controls ++(0:.3cm) and ++(-30:.3cm) .. (.4,.4); 
\filldraw (0,-.2) circle (.03cm);
\filldraw (0,0) circle (.03cm);
\filldraw (-.2,.2) circle (.03cm);
\filldraw (0,-.4) circle (.03cm);
\filldraw (-.4,.4) circle (.03cm);
\filldraw (.2,.2) circle (.03cm);
\filldraw (.4,.4) circle (.03cm);
}
$$
Densely packing these strings, we see that $Q$ behaves like an identity 1-morphism and $m$ behaves like a unitor 2-morphism for a new object in our $\rm C^*/W^*$ 2-category $\cC$.
This directly leads to the notion of split higher idempotent.
This story is a unitary version of that for separable monads in \cite{1812.11933} and unital condensation monads in \cite{1905.09566}.

Given a $\rm C^*/W^*$ 2-category $\cC$, its \emph{Q-system completion} is the 2-category $\QSys(\cC)$ of Q-systems, bimodules, and intertwiners in $\cC$, which has been studied previously in \cite{MR3509018,MR3308880,2010.01072}, building on algebraic versions in \cite{MR2075605,MR3459961,1812.11933,1905.09566}.
This is one version of a higher idempotent completion for $\rm C^*/W^*$ 2-categories in comparison with 2-categories of separable monads in \cite{1812.11933} and condensation monads in \cite{1905.09566}.
We have a canonical inclusion $\dag$ 2-functor $\iota_\cC: \cC \hookrightarrow \QSys(\cC)$ which is always an equivalence on all hom categories.
We call $\cC$ \emph{Q-system complete} if $\iota_\cC$ is a $\dag$-equivalence of $\dag$ 2-categories.
In Theorem \ref{Thm:QSystemSplit} below, we show $\cC$ is Q-system complete if and only if all Q-systems split (cf.~\cite[Prop.~A.4.2]{1812.11933}).

After establishing these basic general results for Q-system completion for $\rm C^*/W^*$ 2-categories, our first main theorem analyzes $\rCorr$, the $\rm C^*$ 2-category of unital $\rm C^*$-algebras, right correspondences, and intertwiners.

\begin{thmalpha}
\label{thm:QSysComplete}
The $\rm C^*$ 2-category $\rCorr$ is Q-system complete.
\end{thmalpha}

Essentially, $\rCorr$ being Q-system complete allows for the straightforward adaptation of subfactor results to the $\rm C^*$ setting.
This theorem also holds for the 2-categories $\WStarRCorr$ of $\rm W^*$-correspondences and $\vNA$ of von Neumann algebras, but we focus on $\rCorr$ to streamline the exposition and to make subfactor results more easily adapted to $\rm C^*$-algebraists.
These ideas have been implicit in subfactor theory since Q-systems were introduced, and one purpose of this paper is to make these implicit techniques explicit and easily accessible for applications.

To prove that $\rCorr$ is Q-system complete, we extend realization to a 2-functor taking values in $\rCorr$, and prove it is inverse to Q-system completion.


\begin{lemalpha}
Realization extends to a $\dag$ 2-functor $|\cdot|:\QSys(\rCorr) \to \rCorr$ which is inverse to the inclusion $\dag$ 2-functor $\iota_{\rCorr} : \rCorr \to \QSys(\rCorr)$.
\end{lemalpha}

For a unital $\rm C^*$-algebra $B$ and a Q-system $Q\in \rCorr(B\to B)$, the realization is fairly straightforward to describe.
Our method follows \cite{MR2097363}, which is expressed diagrammatically in \cite[\S4.1]{MR3221289}.
We define
$|Q|:=\Hom_{\bbC-B}(B \to {}_\bbC B \boxtimes_B Q_B)$
with multiplication, unit, and $*$-structure given by 
$$
q_1\cdot q_2
:=
\tikzmath{
\begin{scope}
\clip[rounded corners=5pt] (-.4,-2) rectangle (1.2,1);
\fill[\BColor] (.6,-2) -- (.6,-1.3) -- (.5,-1) .. controls ++(90:.3cm) and ++(270:.3cm) .. (0,-.3) -- (-.1,.3) -- (-.1,1) -- (1.2,1) -- (1.2,-2);
\end{scope}
\draw (.4,.6) -- (.4,1);
\filldraw (.4,.6) circle (.05cm);
\draw (.1,.3) arc (180:0:.3cm) -- (.7,-1);
\draw[dashed] (.6,-2) -- (.6,-1.3) -- (.5,-1) .. controls ++(90:.3cm) and ++(270:.3cm) .. (0,-.3) -- (-.1,.3) -- (-.1,1);
\roundNbox{unshaded}{(0,0)}{.3}{0}{0}{$q_1$};
\roundNbox{unshaded}{(.6,-1.3)}{.3}{0}{0}{$q_2$};
}\,,
\qquad
1_{|Q|}
:=
\tikzmath{
\begin{scope}
\clip[rounded corners=5pt] (-.2,0) rectangle (.6,1);
\fill[\BColor] (0,0) rectangle (.6,1);
\end{scope}
\draw[dashed] (0,0) -- (0,1);
\filldraw (.3,.5) circle (.05cm);
\draw (.3,.5) -- (.3,1);
}\,,
\qquad\text{and}\qquad
q^*:=\ 
\tikzmath{
\begin{scope}
\clip[rounded corners=5pt] (-.4,-1) rectangle (1,.7);
\fill[\BColor] (-.1,-1) -- (-.1,0) -- (0,0) -- (0,.7) -- (1,.7) -- (1,-1);
\end{scope}
\draw[dashed] (-.1,-1) -- (-.1,0) -- (0,0) -- (0,.7);
\draw (.1,-.3) arc (-180:0:.3cm) -- (.7,.7);
\draw (.4,-.6) -- (.4,-.8);
\filldraw (.4,-.6) circle (.05cm);
\filldraw (.4,-.8) circle (.05cm);
\roundNbox{unshaded}{(0,0)}{.3}{0}{0}{$q^\dag$};
}\,.
\qquad\qquad
\begin{aligned}
\tikzmath{
\filldraw[dotted,thin,fill=white, rounded corners=5pt] (0,0) rectangle (.5,.5); 
}
&=
\bbC
\\
\tikzmath{
\fill[\BColor, rounded corners=5pt] (0,0) rectangle (.5,.5); 
}
&=B
\end{aligned}
$$
Similarly, we define the realization of a $P-Q$ bimodule using the hom spaces in $\rCorr$, along with the realization of a $P-Q$ bimodule map.
We refer the reader to \S\ref{sec:C*AlgQSysComplete} for more details.

In \cite{MR4369356}, we will prove that Q-system completion is a 3-endofunctor on the 3-category (algebraic tricategory \cite{MR3076451}) of $\rm C^*/W^* $ 2-categories.
For the purposes of this article, we show that any $\dag$ 2-functor $F: \cC \to \cD$ between $\rm C^*/W^* $ 2-categories induces a $\dag$ 2-functor $\QSys(F): \QSys(\cC) \to \QSys(\cD)$.
Using Theorem \ref{thm:QSysComplete}, we get the following immediate corollary.

\begin{coralpha}\label{cor: induced action}
Suppose $\cC$ is a unitary tensor category, $A$ is a unital $\rm C^*$-algebra, and $F: \cC \to \rCorr(A\to A)$ is a $\dag$-tensor functor.
We have a composite $\dag$ 2-functor
$$
\QSys(\cC) 
\xrightarrow{\QSys(F)} 
\QSys(\rCorr(A\to A)) 
\xrightarrow{|\,\cdot\,|} 
\rCorr.
$$
Moreover, when $F$ is fully faithful, then so is our composite 2-functor on 2-morphisms (see Remark \ref{rem:2FullyFaithful}).
\end{coralpha}

We remark that the same strategy as Corollary \ref{cor: induced action} was recently used to induce actions of unitary multitensor categories on $\rm II_1$ multifactors in \cite{2010.01072}. (The multifactors there are hyperfinite when the categories are multifusion.)

There are natural K-theory obstructions to the existence of fusion category actions on $\rm C^*$-algebras. 
In particular, 
in Corollary \ref{non-integral}, we show that if a fusion category $\cC$ admits an action on a continuous trace $\rm C^*$-algebra with compact connected spectrum $X$, then $\cC$ must be integral, i.e., all objects of $\cC$ have integral Frobenius-Perron dimension.
This naturally leads to the question: which integral fusion categories admit such actions? 
It is \emph{a priori} plausible (but not the case) that such a fusion category might necessarily admit a fiber functor, where the space $X$ is a point. 

We use Corollary \ref{cor: induced action} to address the question in the previous paragraph. 
A large class integral fusion categories are the \textit{group theoretical} fusion categories, which are unitary fusion categories Morita equivalent to $\fdHilb(G,\omega)$ for a finite group $G$ and $\omega\in Z^{3}(G, \omega)$.
In general, these fusion categories do not admit fiber functors. 
In \cite{MR4328058} it was shown that the unitary fusion categories $\fdHilb(G, \omega)$ always admit an action on some $C(X)$ for a closed connected manifold $X$. 
Combining this fact with the above Corollary \ref{cor: induced action}, we obtain the following result.




\begin{coralpha}
\label{cor:MainC}
Let $\mathcal{C}$ be a group theoretical fusion category and $n\ge 2$. 
There exists a closed, connected manifold $X$ of dimension $n$ and an action of $\mathcal{C}$ on a unital continuous trace $\rm C^*$-algebra with spectrum $X$. 
\end{coralpha}

\subsection*{Acknowledgements.}
We would like to thank
Marcel Bischoff,
Luca Giorgetti,
Andr\'e Henriques,
David Reutter,
Jan Steinebrunner, and
Christoph Weis
for helpful discussions.
Quan Chen, Roberto Hern\'{a}ndez Palomares, and David Penneys were supported by NSF grants DMS 1654159, 1927098, and 2051170.
Quan Chen was partially supported by NSF DMS grant 1936283.
Roberto Hern\'{a}ndez Palomares was partially financially supported by 
CONACyT grant \emph{Becas al Extranjero Demanda Libre 2018.}
Corey Jones was supported by NSF grant DMS 1901082/2100531.

\section{Background}

In this section, we recall the notions of $\rm C^*/W^*$ 2-category, and we define the particular examples most relevant to this article, i.e., $\rCorr$, $\WStarRCorr$, and $\vNA$.

\subsection{\texorpdfstring{$\rm C^*$}{C*}-2-categories and unitary tensor categories}

\begin{nota}
In this article, \emph{2-category} means a weak 2-category, also known as a bicategory.
We refer the reader to \cite{MR4261588} for background on 2-categories.

Given a 2-category $\cC$, the objects are denoted by lower case letters $a,b,c,\dots$,
1-morphisms $a\to b$ are denoted using notation ${}_aX_b$ similar to bimodules, and 2-morphisms are denoted by $f,g,h,\dots$.
We write 1-composition $\xz$ from \emph{left to right}, e.g.~${}_aX\xz_b Y_c$,
and 2-composition $\circ$ from \emph{right to left}.
This notation is consistent with the tensor product of bimodules and composition of intertwiners in the examples in this article.
\end{nota}

To the best of our knowledge, 
the notion of a $\rm C^*$ 2-category first appeared in \cite{MR1444286} building on the notion of $\rm C^*$ tensor category,
and
the notions of $\rm W^*$ 2-category and $\rm W^*$ tensor category first appeared in \cite{MR2325696}.
We refer the reader to \cite{MR2298822}, \cite{MR3994584}, and \cite{1705.05600} for further discussion and applications.
The notion of $\rm W^*$ category was studied in detail in \cite{MR808930}.

\begin{defn}
A dagger structure on a 2-category $\cC$ consists of an anti-linear map $\dag: \cC({}_aX_b \Rightarrow {}_aY_b) \to \cC({}_aY_b \Rightarrow {}_aX_b)$ for all 1-morphisms ${}_aX_b,{}_aY_b\in \cC(a\to b)$ for all objects $a,b\in\cC$ satisfying the following conditions:
\begin{enumerate}[label=($\dag$\arabic*)]
\item 
For all $f\in \cC({}_aX_b \Rightarrow {}_aY_b)$, $f^{\dag\dag}=f$.
\item
For all $f\in \cC({}_aX_b \Rightarrow {}_aY_b)$ and $g\in \cC({}_aY_b \Rightarrow {}_aZ_b)$, $(g\circ f)^\dag = f^\dag \circ g^\dag$.
\item
For all 
$f\in \cC({}_aW_b \Rightarrow {}_aX_b)$ 
and 
$g\in \cC({}_bY_c \Rightarrow {}_bZ_c)$,
$(f\xz_b g)^\dag = f^\dag \xz_b g^\dag$.
\item
All unitors and associators in $\cC$ are \emph{unitary} ($u^\dag=u^{-1}$).
\end{enumerate}
A $\dag$ 2-category is a 2-category equipped with a dagger structure.

We call a $\dag$ 2-category $\cC$ a \emph{$\rm C^*$ 2-category} if every hom 1-category is a 
$\rm C^*$ category.
This means:
\begin{enumerate}[label=\textup{($\rm C^*$\arabic*)}]
\item 
\label{C*:Positive}
For all $f\in \cC({}_aX_b \Rightarrow {}_aY_b)$, there is a $g\in \cC({}_aX_b \Rightarrow {}_aX_b)$ such that $f^\dag \circ f = g^\dag \circ g$.
\item
\label{C*:Norm}
For each ${}_aX_b,{}_aY_b\in\cC(a\to b)$, the function $\|\cdot\|: \cC({}_aX_b \Rightarrow {}_aY_b)\to [0,\infty]$ given by
$$
\|f\|^2 := \sup\set{|\lambda|\geq 0\,}{\,f^\dag\circ f - \lambda\id_{X} \text{ is not invertible}}
$$
gives a complete norm on $\cC({}_aX_b \Rightarrow {}_aY_b)$ which is:
\begin{itemize}
\item 
(sub-multiplicative) 
For all $f\in \cC({}_aX_b \Rightarrow {}_aY_b)$ and $g\in \cC({}_aY_b \Rightarrow {}_aZ_b)$,
$\|g\circ f\|\leq \|g\|\cdot\|f\|$, and
\item
($\rm C^*$)
For all $f\in \cC({}_aX_b \Rightarrow {}_aY_b)$,
$\|f^\dag \circ f\| = \|f\|^2$.
\end{itemize}
\end{enumerate}
If $\cC$ admits direct sums of 1-morphisms, then by Roberts' $2\times 2$ trick \cite[Lem.~2.6]{MR808930},
\ref{C*:Positive} and \ref{C*:Norm} are equivalent to:
\begin{enumerate}[label=\textup{($\rm C^*$)}]
\item
\label{C*:2x2}
For all ${}_aX_b, {}_aY_b$, $\End_\cC({}_aX_b \oplus {}_aY_b)$ is a $\rm C^*$ algebra.
\end{enumerate}

We call a $\rm C^*$ 2-category a \emph{$\rm W^*$ 2-category} if 
\begin{enumerate}[label=\textup{($\rm W^*$\arabic*)}]
\item
\label{W*:Predual}
each hom 1-category is a $\rm W^*$ category, i.e., for each ${}_aX_b,{}_aY_b\in\cC(a\to b)$, the Banach space $\cC({}_aX_b \Rightarrow {}_aY_b)$ has a predual, and
\item
\label{W*:TensorIsSeparatelyNormal}
1-composition $\xz$ is separately normal (weak* continuous) in each variable.
\end{enumerate}
\end{defn}

\begin{lem}
\label{lem:CompositionNormContinuous}
Suppose $\cC$ is a $\rm C^*$ 2-category,
${}_aX_b, {}_aY{}_b, {}_aZ_b$ are 1-morphisms in $\cC$, and $f\in \cC({}_aY{}_b\Rightarrow {}_aZ_b)$.
The map 
$$
f\circ - :
\cC({}_aX{}_b\Rightarrow {}_aY_b)
\longrightarrow
\cC({}_aX{}_b\Rightarrow {}_aZ_b)
$$
is norm continuous.
If $\cC$ is $\rm W^*$, then $f\circ -$ is normal (weak*-continuous).
\end{lem}
\begin{proof}
Norm continuity follows from submultiplicativity of the norm.
When $\cC$ is $\rm W^*$, composition is multiplication in the $3\times 3$ $\rm W^*$ linking algebra $\End_\cC({}_a(X\oplus Y\oplus Z)_b)$, which is separately weak*-continuous.
\end{proof}

\begin{prop}
Suppose $\cC$ is a $\rm C^*$ 2-category that satisfies \ref{W*:Predual}.
Condition \ref{W*:TensorIsSeparatelyNormal}
is equivalent to:
\begin{enumerate}[label=\textup{($\rm W^*2'$)}]
\item 
\label{W*:TensorIsNormal}
For every 
${}_aW_b,{}_aX_b, {}_bY_c, {}_bZ_c$, the following maps are normal (weak* continuous):
\begin{itemize}
\item 
$
\id_X \xz -:
\Hom({}_bY_c \Rightarrow {}_bZ_c)
\to
\Hom({}_aX\xz_b Y_c \Rightarrow {}_a X\xz_b Z_c)
$
given 
by
$
f\mapsto \id_X\xz f
$
\item
$
- \xz \id_Z:
\Hom({}_aW_b \Rightarrow {}_aX_b)
\to
\Hom({}_aW\xz_b Z_c \Rightarrow {}_a X\xz_b Z_c)
$
given by
$f\mapsto f\xz \id_Z$
\end{itemize}
\end{enumerate}
\end{prop}
\begin{proof}
It is immediate that \ref{W*:TensorIsSeparatelyNormal} implies \ref{W*:TensorIsNormal}.
The converse follows immediately from Lemma \ref{lem:CompositionNormContinuous}.
\end{proof}

\begin{rem}
\label{rem:TensorIdNormContinuous}
Observe that for a $\rm C^*$ 2-category $\cC$, the analogous norm continuous version of \ref{W*:TensorIsNormal} follows automatically from \ref{C*:2x2}.
For example, the map
$$
\id_X \xz -:
\Hom({}_bY_c \Rightarrow {}_bZ_c)
\longrightarrow
\Hom({}_aX\xz_b Y_c \Rightarrow {}_a X\xz_b Z_c)
\qquad\qquad
f\mapsto \id_X \xz f
$$
is a unital $*$-homomorphism on the $2\times 2$ $\rm C^*$ linking algebra $\End_\cC({}_b(Y\oplus Z)_c)$, and thus norm contractive.
\end{rem}

The following counter-example shows that \ref{C*:2x2} and \ref{W*:Predual} do not imply \ref{W*:TensorIsNormal}.


\begin{counter-ex}
Let $A=L^\infty[0,1]$ with Lebesgue measure, and consider the dagger tensor category $\End(A)$ whose objects are (not necessarily normal) unital $*$-endomorphisms of $A$ and whose morphisms
$$
\Hom(\sigma \to \tau)
:=
\set{a\in A}{a\sigma(b) = \tau(b)a \text{ for all }b\in A}.
$$
For $\sigma, \tau\in \End(A)$, we define $\sigma \xz \tau := \sigma \circ \tau$
and for $a\in \Hom(\sigma_1 \to \sigma_2)$ and $b\in\Hom(\tau_1\to \tau_2)$, we set
\begin{equation}
\label{eq:TensorProductInEnd(A)}
a\xz b := a\sigma_1(b) = \sigma_2(b)a.
\end{equation}
Now $\End(A)$ is a dagger category with $a^\dag := a^*$, and the underlying dagger category is $\rm W^*$ as every $2\times 2$ linking algebra is a von Neumann algebra:
$$
\begin{pmatrix}
\Hom(\sigma \to \sigma) 
&
\Hom(\tau \to \sigma)
\\
\Hom(\sigma \to \tau) 
&
\Hom(\tau \to \tau)
\end{pmatrix}
=
\set{
\begin{pmatrix}
\sigma(a)&0
\\
0& \tau(a)
\end{pmatrix}
}{a\in A}'\cap M_2(A).
$$
By \cite[Thm.~1.18]{MR1873025}, $A\cong C(X)$ for some hyperstonean space $X$, and point evaluation $\ev_x: A \to \bbC \subset A$ is a non-normal unital $*$-endomorphism.
(Indeed, the kernel projection $p$ of a normal unital $*$-homomorphism $A\to \bbC$ must satisfy $1-p$ is minimal, and $A$ has no minimal projections.)
Observe that by \eqref{eq:TensorProductInEnd(A)}, the map
$$
\ev_x \xz - : \Hom(\id\to \id)=A \to \Hom(\ev_x \to \ev_x)=A
$$
is exactly $\ev_x$, which is not normal.
\end{counter-ex}

\begin{assumption}
We will further assume all hom 1-categories in our $\rm C^*/W^*$ 2-categories are \emph{unitarily Cauchy complete}, i.e., 
for each $a,b\in \cC$, the $\rm C^*/W^*$ category $\cC(a\to b)$ admits all orthogonal direct sums, and all orthogonal projections split via an isometry.
We provide the definitions of the relevant terms below.
\begin{itemize}
\item 
An object $\bigoplus_{i=1}^n X_i$ is the orthogonal direct sum of $X_1,\dots, X_n$ if there are isometries $v_j: X_j \Rightarrow \bigoplus X_i$ such that $\sum_j v_jv_j^\dag = \id$.
\item
An orthogonal splitting of an orthogonal projection $p\in \End({}_aX_b)$ is a 1-morphism ${}_aY_b$ and an isometry $v: Y \Rightarrow X$ such that $vv^\dag = p$.
\end{itemize}
\end{assumption}

\begin{rem}
\label{rem:IsoIffUnitarilyIso}
In a $\rm C^*$ 1-category, two objects $X,Y$ are isomorphic if and only if they are unitarily isomorphic.
Indeed, if $f: X\to Y$ is an isomorphism, then $f^\dag f: X\to X$ is invertible, and thus so is $|f|:=\sqrt{f^\dag f}$.
Setting $u:=f |f|^{-1}: X\to Y$, we calculate that $u$ is a unitary isomorphism:
\begin{align*}
u^\dag u 
&= 
|f|^{-1} f^\dag f |f|^{-1} = |f|^{-1} |f|^2|f|^{-1} = \id_X
\\
u u^\dag 
&= 
f |f|^{-2} f^\dag = f(f^\dag f)^{-1} f^\dag = f f^{-1}(f^\dag)^{-1}f^\dag = \id_Y.
\end{align*}
\end{rem}

In this article we will make heavy use of the graphical calculus of string diagrams for 2-categories \cite[\S8.1.2]{MR3971584}, which are dual to pasting diagrams.
In a pasting diagram, one represents objects as vertices, 1-morphisms as arrows, and 2-morphisms as 2-cells.
In the string diagram calculus, we represent objects by shaded regions, 1-morphisms by (oriented) strands between these regions, and 2-morphisms by coupons.
$$
f: {}_aX\xz_b Y_c \Rightarrow {}_aZ_c
\qquad\rightsquigarrow\qquad
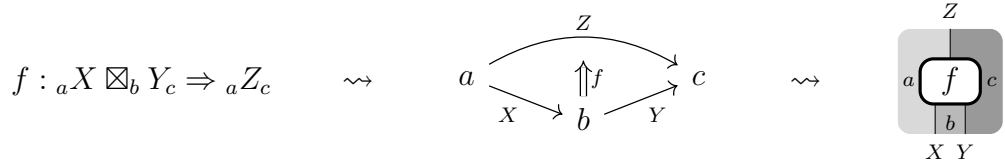
\begin{tikzcd}[row sep=0]
&\mbox{}
\\
a
\arrow[rr, bend left = 30, "Z"]
\arrow[dr, swap, "X"]
&
\mbox{}
&
c
\\
&b
\arrow[uu, Rightarrow, swap, "f"]
\arrow[ur, swap, "Y"]
\end{tikzcd}
\qquad\rightsquigarrow\qquad
\tikzmath{
\begin{scope}
\clip[rounded corners=5pt] (-.7,-.7) rectangle (.7,.7);
\fill[\AColor] (-.7,-.7) -- (-.2,-.7) -- (-.2,0) -- (0,0) -- (0,.7) -- (-.7,.7);
\fill[\CColor] (.7,-.7) -- (.2,-.7) -- (.2,0) -- (0,0) -- (0,.7) -- (.7,.7);
\fill[\BColor] (-.2,-.7) rectangle (.2,0);
\end{scope}
\draw (0,0) -- (0,.7) node[above]{$\scriptstyle Z$};
\draw (-.2,-.7) node[below]{$\scriptstyle X$} -- (-.2,0);
\draw (.2,-.7) node[below]{$\scriptstyle Y$} -- (.2,0);
\roundNbox{fill=white}{(0,0)}{.3}{.1}{.1}{$f$}
\node at (-.55,0) {$\scriptstyle a$};
\node at (.55,0) {$\scriptstyle c$};
\node at (0,-.55) {$\scriptstyle b$};
}
$$
Horizontal 1-composition is read from \emph{left to right}, and vertical 2-composition is read from \emph{bottom to top}.

\begin{defn}
A 2-category $\cC$ is called \emph{rigid} if for every 1-morphism ${}_aX_b\in \cC(a\to b)$, there is a ${}_bX_a^\vee \in \cC(b\to a)$ together with maps $\ev_X \in \cC({}_bX^\vee \xz_a X_b \Rightarrow 1_b)$ and $\coev_X\in \cC(1_a \Rightarrow {}_aX\xz_b X^\vee_a)$ satisfying the zig-zag/snake equations, which are best depicted in the graphical calculus for 2-categories:
$$
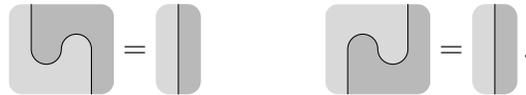
\tikzmath{
\begin{scope}
\clip[rounded corners=5pt] (-.7,-.6) rectangle (.7,.6);
\fill[\AColor] (-.7,.6) -- (-.4,.6) -- (-.4,0) arc (-180:0:.2cm) arc (180:0:.2cm) -- (.4,-.6) -- (-.7,-.6);
\fill[\BColor] (.7,.6) -- (-.4,.6) -- (-.4,0) arc (-180:0:.2cm) arc (180:0:.2cm) -- (.4,-.6) -- (.7,-.6);
\end{scope}
\draw (-.4,.6) -- (-.4,0) arc (-180:0:.2cm) arc (180:0:.2cm) -- (.4,-.6);
}
=
\tikzmath{
\begin{scope}
\clip[rounded corners=5pt] (-.3,-.6) rectangle (.3,.6);
\fill[\AColor] (-.3,-.6) rectangle (0,.6);
\fill[\BColor] (.3,-.6) rectangle (0,.6);
\end{scope}
\draw (0,-.6) -- (0,.6);
}
\qquad\qquad
\tikzmath{
\begin{scope}
\clip[rounded corners=5pt] (-.7,-.6) rectangle (.7,.6);
\fill[\AColor] (-.7,-.6) -- (-.4,-.6) -- (-.4,0) arc (180:0:.2cm) arc (-180:0:.2cm) -- (.4,.6) -- (-.7,.6);
\fill[\BColor] (.7,-.6) -- (-.4,-.6) -- (-.4,0) arc (180:0:.2cm) arc (-180:0:.2cm) -- (.4,.6) -- (.7,.6);
\end{scope}
\draw (-.4,-.6) -- (-.4,0) arc (180:0:.2cm) arc (-180:0:.2cm) -- (.4,.6);
}
=
\tikzmath{
\begin{scope}
\clip[rounded corners=5pt] (-.3,-.6) rectangle (.3,.6);
\fill[\AColor] (-.3,-.6) rectangle (0,.6);
\fill[\BColor] (.3,-.6) rectangle (0,.6);
\end{scope}
\draw (0,-.6) -- (0,.6);
}\,.
$$
Moreover, we assume for each 1-morphism ${}_aX_b\in \cC(a\to b)$, there is a \emph{predual} object ${}_b(X_\vee)_a \in \cC(b\to a)$ such that $(X_\vee)^\vee \cong X$ in $\cC(a\to b)$.
Observe that being rigid is a property of $\cC$, and not extra 
structure.

Now suppose $\cC$ is $\rm C^*/W^*$.
A \emph{unitary dual functor} \cite{MR4133163} on $\cC$ consists of a choice of dual $({}_bX^\vee_a, \ev_X, \coev_X)$ for each 1-morphism ${}_aX_b\in \cC(a\to b)$ such that
\begin{enumerate}[label=($\vee$\arabic*)]
\item 
For each ${}_aX_b, {}_bY_c$, the canonical tensorator
$$
\nu_{X,Y}:=
\tikzmath{
\begin{scope}
\clip[rounded corners=5pt] (-1.2,-.6) rectangle (1.1,.7);
\fill[\AColor] (-1.2,.7) -- (.7,.7) -- (.7,0) arc (0:-180:.3cm) arc (0:180:.5cm) -- (-.9,-.6) -- (-1.2,-.6);
\fill[\BColor] (.8,.7) -- (.8,0) arc (0:-180:.4cm) arc (0:180:.2cm) -- (-.4,-.6) -- (-.9,-.6) -- (-.9,0) arc (180:0:.5cm) arc (-180:0:.3cm) -- (.7,.7);
\fill[\CColor] (1.1,.7) -- (.8,.7) -- (.8,0) arc (0:-180:.4cm) arc (0:180:.2cm) -- (-.4,-.6) -- (1.1,-.6);
\end{scope}
\draw (.8,.7) -- (.8,0) arc (0:-180:.4cm) arc (0:180:.2cm) -- (-.4,-.6) node[below]{$\scriptstyle Y^\vee$};
\draw (.7,.7) node[above]{$\scriptstyle (Y\xz X)^\vee$} -- (.7,0) arc (0:-180:.3cm) arc (0:180:.5cm) -- (-.9,-.6) node[below]{$\scriptstyle X^\vee$};
}
=
\begin{aligned}
(\ev_X\xz \id_{(Y\xz X)^\vee})
&\circ 
(\id_{X^\vee}\xz \ev_Y \xz \id_X \xz \id_{(Y\xz X)^\vee})
\\&\circ
(\id_{X^\vee\xz Y^\vee} \xz \coev_{Y\xz X})
\end{aligned}
$$
is a unitary 2-morphism, and
\item
for each $f\in \cC({}_aX_b\Rightarrow {}_aY_b)$, we have $f^{\vee\dag}=f^{\dag\vee}$ where 
$
f^\vee:=
\tikzmath{
\begin{scope}
\clip[rounded corners=5pt] (-.9,-.8) rectangle (.9,.8);
\fill[\AColor] (-.6,-.8) -- (-.6,.3) arc (180:0:.3cm) -- (0,-.3) arc (-180:0:.3cm) -- (.6,.8) -- (.9,.8) -- (.9,-.8);
\fill[\BColor] (-.6,-.8) -- (-.6,.3) arc (180:0:.3cm) -- (0,-.3) arc (-180:0:.3cm) -- (.6,.8) -- (-.9,.8) -- (-.9,-.8);
\end{scope}
\draw (0,.3) arc (0:180:.3cm) -- (-.6,-.8);
\draw (0,-.3) arc (-180:0:.3cm) -- (.6,.8);
\roundNbox{fill=white}{(0,0)}{.3}{0}{0}{$f$}
}
$.
\end{enumerate}
In contrast to rigidity being a property, a unitary dual functor is extra structure.
\end{defn}

\begin{rem}\label{rem:delooping}
It is well known that 
($\rm C^*$) tensor categories
are equivalent to ($\rm C^*$) 2-categories with exactly one object \cite[Delooping Hypothesis 22]{MR2664619}; 
in the graphical calculus, this object is represented by the empty (lack of) shading.
Starting with a tensor category $\cT$, we take its \emph{delooping} $\rmB\cT$ to get a 2-category with one object $*$ whose endomorphism tensor category is $\cT$.
Conversely, given a 2-category $\cC$ with one object we take the so-called \emph{loop space} $\Omega\cC$ which ignores the object and considers 1-morphisms as objects and 2-morphisms as 1-morphisms.
(In fact, more is true; while 2-categories form a 3-category and tensor categories form a 2-category, 2-categories with one object considered as pointed 2-categories actually form a 2-category which is equivalent to the 2-category of tensor categories. 
See \cite[\S5.6]{MR2664619} or \cite[\S1.2]{2009.00405} for more details.)
\end{rem}

\begin{defn}
A \emph{unitary multitensor category} is a 
semisimple rigid $\rm C^*$-tensor category.
A \emph{unitary tensor category} is a unitary multitensor category with simple unit object.
\end{defn}

\subsection{Receptacles for actions}

In this section, we define several $\rm C^*$ and $\rm W^*$ 2-categories which admit tensor functors from unitary tensor categories; we call such 2-categories \emph{receptacles} for actions of unitary tensor categories.
One of the main results of this article is that all these receptacles are \emph{Q-system complete}, a notion we define in \S\ref{sec:QSystemCompletion} below.
We will only prove this result in detail for the first $\rm C^*$ 2-category $\rCorr$ defined below, but we will include many remarks along the way on how to adapt our proof to the other 2-categories defined here.

Our treatment below of right $\rm C^*$/$\rm W^*$-correspondences follows \cite[\S8]{MR2111973}.
Other earlier references include \cite{MR355613,MR0367670}.

\begin{defn}[{Right $\rm C^*$-correspondences}]
The $\rm C^*$-2-category $\rCorr$ of right correspondences has objects unital $\rm C^*$-algebras, 1-morphisms ${}_AX_B$ are \emph{right correspondences}, and 2-morphisms ${}_AX_B \Rightarrow {}_AY_B$ are adjointable $A-B$ bimodular maps.

For unital $C^*$-algebras $A,B$, a right correspondence ${}_AX_B\in\rCorr(A\to B)$
is a $\bbC$-vector space $X$ equipped with commuting left $A$- and right $B$-actions, and a right $B$-valued inner product $\langle \,\cdot\,|\,\cdot\,\rangle_B : \overline{X}\times X \to B$
which is:
\begin{itemize}
\item 
right $B$-linear: $\langle \xi_1| \xi_2 \lhd b + \xi_3\rangle_B = \langle \xi_1| \xi_2 \rangle_B b + \langle \xi_1|\xi_3\rangle_B$,
\item
left conjugate $B$-linear: 
$\langle \xi_1\lhd b + \xi_2 | \xi_3\rangle_B = b^*\langle \xi_1| \xi_2 \rangle_B + \langle \xi_2|\xi_3\rangle_B$
\item
positive: $\langle \xi|\xi\rangle_B \geq 0$ in $B$, and
\item
definite: $\langle \xi|\xi\rangle_B = 0$ implies $\xi=0$.
\end{itemize}
Observe that $\langle \eta|\xi\rangle_B^* = \langle \xi|\eta\rangle_B$ by the polarization identity, and we have the \emph{Cauchy-Schwarz inequality}:
$$
\langle \eta|\xi\rangle_B \langle\xi|\eta \rangle_B
\leq
\|\langle \xi|\xi\rangle_B\|_B \cdot \langle \eta|\eta\rangle_B
\qquad\qquad
\forall\, \eta,\xi\in X.
$$
This identity implies that
$$
\|\xi\|_X^2:=\|\langle \xi|\xi\rangle_B\|_B
$$
gives a well-defined norm on $X$. 
We require that $X$ is complete with respect to the metric induced by this norm. 

Moreover, we require the left $A$-action on $X$ is by \emph{adjointable operators}. 
A right $B$-linear map $T:X_B \to Y_B$ between right $B$-modules is \emph{adjointable} if there is a right $B$-linear map $T^\dag:Y_B \to X_B$ such that 
$$
\langle \eta| T\xi\rangle_B = \langle T^\dag \eta|\xi\rangle_B
\qquad\qquad
\forall\, \xi\in X,\,\,
\forall\, \eta\in Y.
$$
Observe that adjointable maps are necessarily bounded by the Closed Graph Theorem.

Suppose 
$A,B,C$ are unital $\rm C^*$-algebras and
${}_AX_B$ and ${}_BY_C$
are right correspondences.
The composition of 1-morphisms is given by the \emph{relative tensor product} $X\boxtimes_B Y$, which is defined as follows.
First, we take the algebraic tensor product
$X\otimes_\bbC Y$
and consider the right $C$-valued sesquilinear form
$$
\langle \xi_1 \otimes \eta_1 | \xi_2\otimes \eta_2\rangle_C
:=
\langle  \eta_1 | \langle \xi_1|\xi_2\rangle^X_B \rhd \eta_2\rangle^Y_C.
$$
We let $N$ denote the left kernel of this form, i.e.,
$$
N=\set{\zeta \in X\otimes_\bbC Y}{\langle \zeta|\zeta\rangle_C = 0},
$$
which is an algebraic $A-C$ sub-bimodule of $X\otimes_\bbC Y$,
and note
$$
\operatorname{span}_{\bbC}\{(\xi\lhd b)\otimes_\bbC \eta - \xi\otimes_\bbC(b\rhd \eta)\}
\subset N.
$$
Observe that $\langle \,\cdot\,|\,\cdot\,\rangle_C$ descends to a well-defined $C$-valued inner product on $(X\otimes_\bbC Y)/N$.
We define $X\boxtimes_B Y$ to be the completion of $(X\otimes_\bbC Y) / N$
under the norm
$\|\xi\|_{X\boxtimes_B Y}^2:= \|\langle \xi|\xi\rangle_X\|_C$.

The unit $A-A$ correspondence is given by ${}_AA_A$ with $\langle a_1|a_2\rangle_A :=a_1^*a_2$, and the unitors are given by the obvious unitary maps 
$$
\begin{aligned}
\lambda^A_X:A\boxtimes_A X&\to X
\\
a\boxtimes \xi &\mapsto a\xi
\end{aligned}
\qquad\qquad
\begin{aligned}
\rho^B_X:X\boxtimes_B B &\to X
\\
\xi\boxtimes b &\mapsto \xi b
\end{aligned}
$$
To construct the associator
$\alpha_{X,Y,Z}:
{}_A(X\boxtimes_B Y)\boxtimes_C Z_D 
\to
{}_AX\boxtimes_B (Y\boxtimes_C Z)_D 
$, 
we observe the map 
$(\xi\otimes_\bbC \eta) \otimes_\bbC \zeta \mapsto \xi\otimes_\bbC (\eta \otimes_\bbC \zeta)$
on the algebraic tensor product preserves $D$-valued inner products, and thus descends to a unitary isomorphism.
\end{defn}

\begin{ex}[Creation/Annihilation operators]
Suppose ${}_AX_B, {}_BY_C$ are right correspondences.
Every $\xi \in X$ gives rise to a right $C$-linear adjointable operator $L_\xi:Y \to X\boxtimes_B Y$ given by $\eta\mapsto \xi\boxtimes \eta$ with adjoint $L_\xi^\dag(\xi'\boxtimes \eta) = \langle \xi|\xi'\rangle_B \rhd\eta$.
\end{ex}

\begin{rem}
\label{rem:SeparatePoints}
For any right Hilbert $\rm C^*$ $B$-module $X_B$,
the maps $T\mapsto \langle \eta|T\xi\rangle_B$
for $\eta, \xi \in X$ separate points of $\End(X_B)$.
Indeed, observe that if $S,T\in \End(X_B)$, then
$\langle \eta| (S-T)\xi\rangle_B = 0$ for all $\eta, \xi \in X$ implies $\langle (S-T)\xi | (S-T)\xi\rangle_B = 0$ for all $\xi \in X$.
Thus $(S-T)\xi = 0$ for all $\xi \in X$ and $S=T$.
\end{rem}

\begin{defn}[Right $\rm W^*$-correspondences]
The $\rm W^*$ 2-category of right $\rm W^*$-correspondences $\WStarRCorr$
has objects von Neumann algebras and 1-morphisms ${}_AX_B$ are right $\rm C^*$-correspondences 
${}_AX_B$
such that 
\begin{enumerate}[label=(\arabic*)]
\item
\label{W*Corr:SelfDual}
one of the following two equivalent conditions holds \cite[Lem.~8.5.4]{MR2111973}:
\begin{itemize}
\item 
(self-dual)
the map $\overline{X}\to \Hom(X_B \to B_B)$ given by $\overline{\xi}\mapsto L_\xi^\dag$ is an isometric isomorphism, or
\item
(predual)
${}_AX_B$ has a predual, and the $B$-valued inner product $\langle \,\cdot\,|\,\cdot\,\rangle_B$ is separately weak*-continuous.
\end{itemize}
Under these equivalent conditions, 
all operators in $\End(X_B)$ are adjointable, and
$\End(X_B)$ is a von Neumann algebra.
\item
The left action map $A\to \End(X_B)$ is normal.
\end{enumerate}
The 2-morphisms ${}_AX_B \Rightarrow {}_AY_B$ are normal bounded $A-B$ bimodular maps.
Composition of 1-morphisms and coheretors are defined as in $\rCorr$.

It is easy to see that $\WStarRCorr$ admits direct sums of 1-morphisms, and every endomorphism algebra of a 1-morphism is a $\rm W^*$-algebra.
Thus $\WStarRCorr$ is 
$\rm C^*$ by \ref{C*:2x2}
and \ref{W*:Predual} holds.
To prove $\WStarRCorr$ is $\rm W^*$, we use the following lemma.
\end{defn}

\begin{lem}
\label{lem:RieffelsNormalityCondition}
For all $\eta, \xi\in X$, the map $A\to B$ given by $a\mapsto \langle \eta|a\xi\rangle_B$ is normal (weak* continuous).\footnote{This condition appears in Rieffel's definition of a normal $B$-rigged $A$-module in \cite[Def.~5.1]{MR0367670}.}
\end{lem}
\begin{proof}
Let $\varphi : A \to \End(X_B)$ denote the normal left action unital $*$-homomorphism.
By \cite[Cor.~8.5.5]{MR2111973} a bounded net $(T_i)\subset \End(X_B)$ converges to $T$ weak* if and only if $T_i\xi \to T\xi$ weak* for all $\xi\in X$.
Hence if $(a_i)\subset A^+$ is an increasing net with $a_i \nearrow a$, then $(a_i)$ is bounded, and thus so is $(\varphi(a_i))$.
Then $\varphi(a_i) \xi \to \varphi(a)\xi$ weak*, and thus $\langle \xi | \varphi(a_i) \xi\rangle_B \nearrow \langle \xi|\varphi(a) \xi\rangle_B$ by \cite[Lem.~8.5.4]{MR2111973}, so $a\mapsto \langle \xi|\varphi(a)\xi\rangle_B$ is normal.
Finally, $a\mapsto \langle \eta | \varphi(a) \xi\rangle_B$ is normal by polarization.
\end{proof}

\begin{prop}
\label{prop:vNAW*2Cat}
The $\rm C^*$ 2-category $\WStarRCorr$ is $\rm W^*$.
\end{prop}
\begin{proof}
It remains to show \ref{W*:TensorIsNormal}.
It suffices to consider the case ${}_BY_C={}_BZ_C$ by Roberts' $2\times 2$ trick \cite[Lem.~2.6]{MR808930}, i.e., we need to prove that the unital $*$-algebra map
$\id_X \xz - :\End_{B-C}(Y) \to \End_{A-C}(X\boxtimes_B Y)$
between von Neumann algebras is normal.
Suppose $f_i \nearrow f$ in $\End_{B-C}(Y)^+$.
We can conclude 
$\id_X\xz f_i \nearrow \id_X\xz f$ in $\End_{A-C}(X\boxtimes_B Y)$
if we can show
$\varphi(\id_X\xz f_i ) \to \varphi(\id_X\xz f)$
for all $\varphi$ in a dense subspace $D_*\subset \End_{A-C}(X\boxtimes_B Y)_*$
(for example, see \cite[Lem.~A.2]{MR3040370}).
Define
$$
D_*:=
\spann
\set{
g\mapsto\phi(\langle \eta\xz \xi | g (\eta\xz \xi)\rangle_C) 
\,}{\, 
\eta\in X, \xi \in Y, \phi\in C_*^+
}.
$$ 
\begin{claim*}
$D_*$ is dense in $\End_{A-C}(X\boxtimes_B Y)_*$.
\end{claim*}
\begin{proof}[Proof of Claim]
By Lemma \ref{lem:RieffelsNormalityCondition}, for all $\eta\in {}_AX_B$, $\xi\in {}_BY_C$, and $\varphi \in C_*^+$, the map
$g\mapsto\varphi(\langle \eta\xz \xi | g (\eta\xz \xi)\rangle_C)$
is normal, i.e., $D_*\subset \End_{A-C}(X\boxtimes_B Y)_*$.
By polarization and taking linear combinations, every map of the form
$$
g
\mapsto 
\sum_{i=1}^n 
\phi_i
\left(\left\langle 
\sum_{j=1}^p\eta_j\xz \xi_j 
\middle| 
g
\sum_{k=1}^q(\eta_k\xz \xi_k)
\right\rangle_C\right)
\qquad\qquad
\phi_i \in C_*, \,
\eta_j, \eta_k \in X,\,
\xi_j,\xi_k\in Y
$$
is in $D_*$.
It remains to show that $D_*$ separates points of $\End_{A-C}(X\boxtimes_B Y)$.
This follows by Remark \ref{rem:SeparatePoints}, approximating arbitrary vectors in $X\boxtimes_B Y$ by finite sums of elementary tensors, and the fact that $C_*$ separates points of $C$.
\end{proof}

Now observe that for all $b\in B$, $b^*f_ib \nearrow bfb$ in $\End(Y_C)$.
Hence for every $\eta\in X$ and $\xi \in Y$,
\begin{align*}
\langle \eta \xz \xi
|
(\id\xz f_i) \eta \xz \xi\rangle_C
&=
\langle \eta \xz \xi
|
\eta \xz f_i\xi\rangle_C
=
\langle \xi
|
\langle \eta|\eta\rangle_B f_i\xi\rangle_C
=
\langle \xi
|
\langle \eta|\eta\rangle_B^{1/2} f_i\langle\eta|\eta\rangle_B^{1/2}\xi\rangle_C
\\&\nearrow
\langle \xi
|
\langle \eta|\eta\rangle_B^{1/2} f\langle\eta|\eta\rangle_B^{1/2}\xi\rangle_C
=
\cdots
=
\langle \eta \xz \xi
|
(\id\xz f) \eta \xz \xi\rangle_C.
\end{align*}
Hence for all $\phi \in C_*^+$,
$$
\phi(
\langle \eta \xz \xi
|
(\id\xz f_i) \eta \xz \xi\rangle_C
)
\nearrow
\phi(
\langle \eta \xz \xi
|
(\id\xz f) \eta \xz \xi\rangle_C
).
$$
This completes the proof.
\end{proof}

\begin{rem}
There are variants of right $\rm C^*$/$\rm W^*$ correspondences which equip ${}_AX{}_B$ simultaneously with a left $A$-valued inner product.
One can perform the same analysis in this article with these 2-categories, but we will not comment further. 
\end{rem}

\begin{defn}[von Neumann algebras]
The $\rm W^*$ 2-category $\vNA$ of von Neumann algebras has objects von Neumann algebras, 1-morphisms ${}_AH_B$ are Hilbert spaces with commuting normal actions of $A$ and $B^{\op}$, and 2-morphisms bounded intertwiners.
Composition of 1-morphisms is given by Connes' fusion relative tensor product.
One proves $\vNA$ is a $\rm W^*$ 2-category similar to Proposition \ref{prop:vNAW*2Cat}, but the result is easier as 1-morphisms are honest Hilbert spaces.
We refer the reader to \cite{MR2325696} for an alternative proof.
Yet another proof follows from the fact that $\vNA$ is $\dag$ 2-equivalent to  $\WStarRCorr$; we refer the reader to \cite[8.5.39]{MR2111973} for more details.
\end{defn}

\begin{rem}
For the $\rm C^*$/$\rm W^*$ 2-categories defined above,
the sub 2-category of dualizble 1-morphisms is equivalent to the 2-category of 
finitely generated projective bimodules (equipped with the appropriate number of inner products) with the algebraic relative tensor product; references are as follows:
\begin{itemize}
\item 
right $\rm C^*$-correspondences: \cite[Lem 1.11, Lem 1.12]{MR1624182}
\item
right $\rm W^*$-correspondences and von Neumann algebras: \cite{MR2771095}
(see also \cite{MR2501843})
\end{itemize}
This offers the advantage that while the analytic relative tensor product $\boxtimes$ does not in general satisfy a universal property, the algebraic relative tensor product does.

The most important step in the proof of this fact is to show that the left kernel $N$ of the right $C$-valued inner product on $X\otimes_\bbC Y$ is equal to 
$$
\operatorname{span}_{\bbC}\{(\xi\lhd b)\otimes_\bbC \eta - \xi\otimes_\bbC(b\rhd \eta)\}.
$$
(For $\vNA$, we must replace $X$ with the \emph{bounded vectors} $X^0:=\set{\xi \in X}{L_\xi\in \Hom(B_B \to X_B)}$.)
Here is a sketch which works simultaneously for all the examples above.
In what follows, $\odot$ denotes the algebraic relative tensor product over $B$.
\begin{enumerate}[label=\underline{Step \arabic*:}]
\item
First, every element of $X\boxtimes_B Y$ can be written as a finite sum $\sum_{i=1}^n \xi_i \boxtimes_B \eta_i$ 
using finite bases for $X_B$ and $Y_C$.
(For $\vNA$, we may assume each $\xi_i\in X^0$, i.e., is right $B$-bounded.)

\item
The space $X\odot_B \overline{X}$ is a unital algebra which acts unitally on $X\odot_B Y$:
$$
(\xi_1\odot \overline{\xi_2})\cdot (\xi_3\odot\overline{\xi_4})
:= \xi_1 \odot \langle \xi_2|\xi_3\rangle_B \rhd\overline{\xi_4}
\qquad\qquad
(\xi_1\odot \overline{\xi_2})\rhd (\xi_3\odot \eta)
:= \xi_1 \odot \langle \xi_2|\xi_3\rangle_B \rhd\eta
$$
The unit $1_{X\odot_B \overline{X}}=\sum_\beta \beta \odot \overline{\beta}$
is given by summing over a finite $X_B$-basis $\{\beta\}$.
(For $\vNA$, we replace $X$ with $X^0$, and observe the $X_B$-basis lies in $X^0_B$.)

\item
If $\sum_{i=1}^n \xi_i \odot_B \eta_i=0$ in $X\odot_B Y$ (again, for $\vNA$, replace $X$ with $X^0$), then
applying the adjointable map $L_\beta^\dag : X\boxtimes_B Y \to Y$ given by $\xi\boxtimes_B \eta \mapsto \langle \beta|\xi\rangle_B \eta$ 
yields
$$
\sum_{i=1}^n \langle \beta | \xi_i\rangle_B \rhd \eta_i = 0
\qquad\qquad
\forall\, \beta \in X_B.
$$
We conclude that
$$
\sum_{i=1}^n \xi_i \odot_B\eta_i
=
\left(\sum_\beta \beta \odot_B \overline{\beta}\right)
\rhd
\left(\sum_{i=1}^n \xi_i \odot_B\eta_i\right)
=
\sum_\beta \sum_{i=1}^n \beta \odot \langle \beta|\xi_i\rangle_B \rhd \eta_i
=
0.
$$
\end{enumerate}
\end{rem}

\begin{rem}
While we do not need it for our purposes here, observe all the $\rm C^*$/$\rm W^*$ 2-categories defined above are actually symmetric monoidal 2-categories.
\end{rem}

\begin{rem}
The $\rm W^*$ 2-category $\vNA$ has an extra involution given by complex conjugation, which is a 2-functor $\cC\to \cC$
which is the identity on objects, 
contravariant on 1-morphisms,
and anti-linear and involtutive on 2-morphisms.
In more detail, for ${}_AH_B \in \vNA(A\to B)$,
$\overline{H}$ is the complex conjugate Hilbert space with left $B$-action and right $A$-action given by $b\rhd \overline{\eta}\lhd a := \overline{a^*\rhd \eta\lhd b^*}$.
For a map $T: {}_AH_B \to {}_AK_B$, we get the complex conjugate map $\overline{T}: {}_B\overline{H}_A \to {}_B\overline{K}_A$ by $\overline{T}(\overline{\eta}):= \overline{T\eta}$.
\end{rem}

\section{Q-system completion}

Q-systems were first introduced in \cite{MR1257245} to characterize the canonical endomorphism associated to a subfactor of an infinite factor \cite{MR739630,MR1027496}.
Since, they have been studied in many articles, including \cite{MR1444286,MR1966524,MR2298822,MR3308880,2010.01072}.
In this section, we define the notion of \emph{Q-system completion}, which is a unitary version of the (co)unital higher idempotent completion for 2-categories discussed in \cite{MR3459961} and \cite{1812.11933}.
To do so, we first introduce the $\rm C^*/W^*$ 2-category structure on Q-systems following  \cite{MR2075605,MR3459961,MR3509018,1812.11933,1905.09566}.
We then introduce a graphical calculus for working with Q-systems.

\subsection{Q-systems}
\label{sec:QSystems}

For this section, we fix a $\rm C^*/W^*$ 2-category $\cC$.

\begin{defn}
A \emph{Q-system} in $\cC$ is a 1-endomorphism $Q\in \cC(b\to b)$
with a multiplication and unit that satisfy certain properties best represented graphically. 
We typically represent the object $b\in \cC$ by a shaded region and $Q\in \cC(b\to b)$ by a black strand which is $b$-shaded on either side
$$
\tikzmath{\filldraw[\BColor, rounded corners=5, very thin, baseline=1cm] (0,0) rectangle (.6,.6);}=b
\qquad\qquad
\tikzmath{
\fill[\BColor, rounded corners=5pt ] (0,0) rectangle (.6,.6);
\draw (.3,0) -- (.3,.6);
}={}_bQ_b.
$$
We denote the multiplication $m$ and the unit $i$ by trivalent and univalent vertices respectively, and $m^\dag$ and $i^\dag$ are given by the vertical reflections
$$
\tikzmath{
\fill[\BColor, rounded corners=5pt] (-.3,0) rectangle (.9,.6);
\draw (0,0) arc (180:0:.3cm);
\draw (.3,.3) -- (.3,.6);
\filldraw (.3,.3) circle (.05cm);
}=m
\qquad\qquad
\tikzmath{
\fill[\BColor, rounded corners=5pt] (-.3,0) rectangle (.9,-.6);
\draw (0,0) arc (-180:0:.3cm);
\draw (.3,-.3) -- (.3,-.6);
\filldraw (.3,-.3) circle (.05cm);
}=m^\dag
\qquad\qquad
\tikzmath{
\fill[\BColor, rounded corners=5pt] (0,0) rectangle (.6,.6);
\draw (.3,.3) -- (.3,.6);
\filldraw (.3,.3) circle (.05cm);
}=i
\qquad\qquad
\tikzmath{
\fill[\BColor, rounded corners=5pt] (0,0) rectangle (.6,-.6);
\draw (.3,-.3) -- (.3,-.6);
\filldraw (.3,-.3) circle (.05cm);
}=i^\dag.
$$
We call $Q$ a \emph{Q-system} if $(Q,m,i)$ satisfies:
\begin{enumerate}[label=(Q\arabic*)]
\item 
\label{Q:associativity}
(associativity)
$\tikzmath{
\fill[\BColor, rounded corners=5pt] (-.3,-.3) rectangle (1.2,.6);
\draw (0,-.3) -- (0,0) arc (180:0:.3cm);
\draw (.3,-.3) arc (180:0:.3cm);
\draw (.3,.3) -- (.3,.6);
\filldraw (.3,.3) circle (.05cm);
\filldraw (.6,0) circle (.05cm);
}
=
\tikzmath{
\fill[\BColor, rounded corners=5pt] (-.6,-.3) rectangle (.9,.6);
\draw (0,0) arc (180:0:.3cm) -- (.6,-.3);
\draw (-.3,-.3) arc (180:0:.3cm);
\draw (.3,.3) -- (.3,.6);
\filldraw (.3,.3) circle (.05cm);
\filldraw (0,0) circle (.05cm);
}$
\item
\label{Q:unitality}
(unitality)
$\tikzmath{
\fill[\BColor, rounded corners=5pt] (-.3,-.3) rectangle (.9,.6);
\draw (0,-.1) -- (0,0) arc (180:0:.3cm) -- (.6,-.3);
\draw (.3,.3) -- (.3,.6);
\filldraw (.3,.3) circle (.05cm);
\filldraw (0,-.1) circle (.05cm);
}
=
\tikzmath{
\fill[\BColor, rounded corners=5pt ] (0,-.3) rectangle (.6,.6);
\draw (.3,-.3) -- (.3,.6);
}
=
\tikzmath{
\fill[\BColor, rounded corners=5pt] (-.3,-.3) rectangle (.9,.6);
\draw (0,-.3) -- (0,0) arc (180:0:.3cm) -- (.6,-.1);
\draw (.3,.3) -- (.3,.6);
\filldraw (.3,.3) circle (.05cm);
\filldraw (.6,-.1) circle (.05cm);
}$
\item
\label{Q:Frobenius}
(Frobenius)
$
\tikzmath{
\fill[\BColor, rounded corners=5pt] (-.3,-.6) rectangle (1.5,.6);
\draw (0,-.6) -- (0,0) arc (180:0:.3cm) arc (-180:0:.3cm) -- (1.2,.6);
\draw (.3,.3) -- (.3,.6);
\draw (.9,-.3) -- (.9,-.6);
\filldraw (.3,.3) circle (.05cm);
\filldraw (.9,-.3) circle (.05cm);
}
=
\tikzmath{
\fill[\BColor, rounded corners=5pt] (-.3,0) rectangle (.9,1.2);
\draw (0,0) arc (180:0:.3cm);
\draw (0,1.2) arc (-180:0:.3cm);
\draw (.3,.3) -- (.3,.9);
\filldraw (.3,.3) circle (.05cm);
\filldraw (.3,.9) circle (.05cm);
}
=
\tikzmath{
\fill[\BColor, rounded corners=5pt] (-.3,.6) rectangle (1.5,-.6);
\draw (0,.6) -- (0,0) arc (-180:0:.3cm) arc (180:0:.3cm) -- (1.2,-.6);
\draw (.3,-.3) -- (.3,-.6);
\draw (.9,.3) -- (.9,.6);
\filldraw (.3,-.3) circle (.05cm);
\filldraw (.9,.3) circle (.05cm);
}
$
\item
\label{Q:separable}
(separable)
$
\tikzmath{
\fill[\BColor, rounded corners=5pt] (-.3,0) rectangle (.9,1.2);
\draw (0,.6) arc (180:-180:.3cm);
\draw (.3,1.2) -- (.3,.9);
\draw (.3,0) -- (.3,.3);
\filldraw (.3,.3) circle (.05cm);
\filldraw (.3,.9) circle (.05cm);
}
=
\tikzmath{
\fill[\BColor, rounded corners=5pt ] (0,0) rectangle (.6,1.2);
\draw (.3,0) -- (.3,1.2);
}$
\end{enumerate}
If $(Q,m,i)$ satisfies \ref{Q:associativity},
\ref{Q:unitality},
and
\ref{Q:Frobenius},
we call $Q$ a \emph{$\rm C^*$ Frobenius algebra}.
If $(Q,m)$ satisfies \ref{Q:associativity},
\ref{Q:Frobenius}, and
\ref{Q:separable},
we call $Q$ a \emph{unitary condensation algebra}. 
\end{defn}

\begin{rem}
One can view a Q-system in a $\rm C^*$ tensor category as a $\rm C^*$-algebra object internal to $\mathcal{C}$, with a choice of state.
We refer the reader to \cite{MR3687214, MR4079745} for a comparison between $\rm C^*$-algebra objects and Q-systems in the connected/irreducible setting, where there is a unique choice of state.
\end{rem}

\begin{defn}
Given a Q-system or $\rm C^*$-Frobenius algebra $(Q,m,i)$,
we define
\begin{equation}
\label{eq:dQ}
d_Q:=
\tikzmath{
\fill[\BColor, rounded corners=5pt ] (0,0) rectangle (.6,1);
\draw (.3,.3) -- (.3,.7);
\filldraw (.3,.3) circle (.05cm);
\filldraw (.3,.7) circle (.05cm);
}
\in \End_\cC(1_b)^+.
\end{equation}
\begin{itemize}
\item
If $d_Q$ is invertible, we call $Q$ \emph{non-degenerate} or an \emph{extension} of $1_b$.
\item
If $d_Q$ is an idempotent, we call $Q$ a \emph{summand} of $1_b$.
\end{itemize}
\end{defn}


\begin{facts}
We have the following dependencies amongst the Q-system axioms:
\begin{itemize}
\item 
By 
\cite[Eq.~(d)~p.148]{MR1444286},
\cite[Prop.~5.17]{MR2298822},
or
\cite[Lem.~3.7]{MR3308880}
\ref{Q:associativity}+\ref{Q:unitality}+\ref{Q:separable}
$\Rightarrow$
\ref{Q:Frobenius}.
\item
By \cite[Prop.~5.17]{MR2298822},
\ref{Q:Frobenius}+\ref{Q:separable}
$\Rightarrow$
\ref{Q:associativity}
by simplifying
$
\tikzmath{
\fill[\BColor, rounded corners=5pt] (-1.2,-.9) rectangle (1.2,.7);
\draw (-.9,-.9) -- (-.9,-.3) arc (180:0:.3cm) arc (-180:0:.3cm) arc (180:0:.3cm) -- (.9,-.9);
\draw (0,-.6) -- (0,-.9);
\draw (-.6,0) arc (180:0:.6 and .4);
\draw (0,.4) -- (0,.7);
\filldraw (-.6,0) circle (.05cm);
\filldraw (.6,0) circle (.05cm);
\filldraw (0,-.6) circle (.05cm);
\filldraw (0,.4) circle (.05cm);
}
$
in two ways.
\end{itemize}
\end{facts}


\begin{facts}
We have the following facts about a $\rm C^*$ Frobenius algebra ${}_bQ{}_b\in \cC$.
\begin{enumerate}[label=(Z\arabic*)]
\item
\label{Z:Dualizable}
$Q$ is dualizble with 
$\ev_Q
:=
\tikzmath{
\fill[\BColor, rounded corners=5pt] (-.5,-.1) rectangle (.5,.5);
\draw (-.2,.5) arc (-180:0:.2cm);
\draw (0,.3) -- (0,.1);
\filldraw (0,.3) circle (.05cm);
\filldraw (0,.1) circle (.05cm);
}
$ 
and 
$\coev_Q
:=
\tikzmath{
\fill[\BColor, rounded corners=5pt] (-.5,-.5) rectangle (.5,.1);
\draw (-.2,-.5) arc (180:0:.2cm);
\draw (0,-.3) -- (0,-.1);
\filldraw (0,-.1) circle (.05cm);
\filldraw (0,-.3) circle (.05cm);
}$\,.
\item
\label{Z:JonesProjBound}
By \ref{Z:Dualizable} and \cite[Lem.~1.16]{MR2298822},
$
\tikzmath{
\fill[\BColor, rounded corners=5pt] (-.5,-.5) rectangle (.5,.5);
\draw (-.2,-.5) arc (180:0:.2cm);
\draw (-.2,.5) arc (-180:0:.2cm);
\draw (0,-.3) -- (0,-.1);
\draw (0,.3) -- (0,.1);
\filldraw (0,.3) circle (.05cm);
\filldraw (0,.1) circle (.05cm);
\filldraw (0,-.1) circle (.05cm);
\filldraw (0,-.3) circle (.05cm);
}
\leq 
\tikzmath{
\fill[\BColor, rounded corners=5pt] (-.9,-.5) rectangle (.5,.5);
\draw (-.2,-.5) -- (-.2,.5);
\draw (.2,-.5) -- (.2,.5);
\draw (-.6,-.2) -- (-.6,.2);
\filldraw (-.6,.2) circle (.05cm);
\filldraw (-.6,-.2) circle (.05cm);
}
\leq
\|d_Q\|\cdot
\tikzmath{
\fill[\BColor, rounded corners=5pt] (-.5,-.5) rectangle (.5,.5);
\draw (-.2,-.5) -- (-.2,.5);
\draw (.2,-.5) -- (.2,.5);
}
$.
\item
\label{Z:Support}
By \cite[Cor.~1.19]{MR2298822},
either $d_Q$ is invertible, or zero is an isolated point in $\Spec(d_Q)$.
By a slight abuse of notation, set $d_Q^{-1}:=f(d_Q)$ via the continuous functional calculus where $f: \Spec(d_Q) \to \bbC$ is given by
$$
f(x) :=
\begin{cases}
0 & \text{if }x=0
\\
x^{-1} & \text{if }x\neq0,
\end{cases}
$$
and set $s_Q:= d_Qd_Q^{-1}$.
Then
$\tikzmath{
\fill[\BColor, rounded corners=5pt] (-1,-.5) rectangle (1,.5);
\draw (.7,-.5) -- (.7,.5);
\draw (-.7,-.2) -- (-.7,.2);
\filldraw (-.7,.2) circle (.05cm);
\filldraw (-.7,-.2) circle (.05cm);
\roundNbox{fill=white}{(0,0)}{.3}{.1}{.1}{$d_Q^{-1}$}
}
=
\tikzmath{
\fill[\BColor, rounded corners=5pt] (-.6,-.5) rectangle (.9,.5);
\draw (.6,-.5) -- (.6,.5);
\roundNbox{fill=white}{(0,0)}{.3}{0}{0}{$s_Q$}
}
=
\tikzmath{
\fill[\BColor, rounded corners=5pt] (-.3,-.5) rectangle (.3,.5);
\draw (0,-.5) -- (0,.5);
}$.
\end{enumerate}
\end{facts}

\begin{rem}
While \cite[Cor.~5.18]{MR2298822} states that each Frobenius algebra in a tensor $\rm C^*$-category is equivalent to a Q-system satisfying \ref{Q:associativity}-\ref{Q:separable}, the proof (which is omitted and inferred from \cite[Pf.~of (5.8)+(5.9)$\Rightarrow$(5.7) in Prop.~5.17]{MR2298822}) only applies to $\rm C^*$ Frobenius algebras, and not arbitrary Frobenius algebras.
\end{rem}

\begin{quest}
Is every arbitrary Frobenius algebra in a unitary tensor category equivalent to a Q-system?
\end{quest}

\begin{rem}
When $\End(1_b)$ is not finite dimensional, there may not be a good notion of quantum dimension in the monoidal category $\End(b)$, which may even fail to be rigid.
We refer the reader to \cite{MR2298822} for a discussion of rigid $\rm C^*$ tensor categories $\cT$ where $\End(1_\cT)$ is not finite dimensional, where it is not clear how to define a quantum dimension.
For example, it is an open question whether $\rCorr(A\to A)$ for $A$ abelian has a unitary spherical structure.
For the purposes of this article, $d_Q$ for a $Q$-system defined in \eqref{eq:dQ} is a suitable replacement.

When $\End(1_b)$ is finite dimensional and all 1-morphisms in $\End(b)$ are dualizable, then there are various notions of quantum dimension in $\End(b)$, depending on whether $\End(1_b)$ is one dimensional \cite{MR1444286,MR2091457} or finite dimensional \cite{MR3342166,MR3994584,MR4133163}. 
One may then consider \emph{normalized} non-degenerate Q-systems for which $d_Q$ equals this quantum dimension.
\end{rem}

\begin{ex}
\label{ex:XXvQSystem}
Suppose $\cC$ is a $\rm C^*/W^*$ 2-category and ${}_aX_b\in \cC(a\to b)$.
A \emph{unitarily separable left dual} for ${}_aX_b$ is a dual $({}_bX^{\vee}_a , \ev_X, \coev_X)$ such that $\ev_X \circ \ev_X^{\dag} = \id_{1_b}$.\footnote{
For a 1-morphism ${}_aX_b$ in a 2-category $\cC$, $X$ dualizable does not necessarily imply $X$ has a separable left dual cf.~\cite[\S A.3]{1812.11933}.}
There is a similar notion of a unitarily separable right dual.


Given a unitarily separable left dual for ${}_aX_b$, ${}_aX\xz_b X^\vee_a \in \cC(a\to a)$ is a Q-system with multiplication and unit given by
$$
m:=
\tikzmath{
\begin{scope}
\clip[rounded corners=5pt] (-.7,0) rectangle (.7,.9);
\fill[\AColor] (-.7,0) rectangle (.7,.9);
\fill[\BColor] (-.4,0) -- (-.4,.2) .. controls ++(90:.2cm) and ++(270:.2cm) .. (-.1,.7) -- (-.1,.9) -- (.1,.9) -- (.1,.7)  .. controls ++(270:.2cm) and ++(90:.2cm) .. (.4,.2) -- (.4,0);
\fill[\AColor] (-.2,0) -- (-.2,.2) arc (180:0:.2cm) -- (.2,0);
\end{scope}
\draw (-.2,0) -- (-.2,.2) arc (180:0:.2cm) -- (.2,0);
\draw (-.4,0) -- (-.4,.2) .. controls ++(90:.2cm) and ++(270:.2cm) .. (-.1,.7) -- (-.1,.9);
\draw (.4,0) -- (.4,.2) .. controls ++(90:.2cm) and ++(270:.2cm) .. (.1,.7) -- (.1,.9);
}
=
\id_X\xz \ev_X\xz \id_{X^\vee}
\qquad\qquad
i:=
\tikzmath{
\begin{scope}
\clip[rounded corners=5pt] (-.4,-.4) rectangle (.4,.5);
\fill[\AColor] (-.7,-.4) rectangle (.7,.5);
\fill[\BColor] (-.1,.5) -- (-.1,0) arc (-180:0:.1cm) -- (.1,.5);
\end{scope}
\draw (-.1,.5) -- (-.1,0) arc (-180:0:.1cm) -- (.1,.5);
}
=
\coev_X.
$$
\end{ex}

\begin{ex}
\label{ex:RestrictByExpectation}
Suppose now that $A$ is a unital $\rm C^*$-algebra and $A\subset B$ is a unital inclusion of $\rm C^*$-algebras equipped with a 
faithful
$A-A$ bimodular
unital 
completely positive (ucp) map
(a.k.a.~a faithful conditional expectation) $E_A: B \to A$.
We can endow $B$ with the structure of a right correspondence in $\rCorr(A\to A)$ by setting $\langle b_1|b_2\rangle_A := E_A(b_1^* b_2)$.
Observe that the inclusion $i_B:{}_AA_A\Rightarrow {}_AB_A$ is then an adjointable operator in $\rCorr({}_AA_A\Rightarrow {}_AB_A)$ with adjoint $E_A: B \Rightarrow A$.

The \emph{(Pimsner-Popa) index} of $E_A: B \to A$ 
is
$$
\ind(E_A)
:=
\inf\set{\lambda \in [1,\infty)}{\lambda \cdot E_A(b) \geq b  \text{ for all }b\in B_+},
$$
with the convention that $\inf\emptyset = \infty$.
If $\ind(E_A)<\infty$, then the $B-B$ bimodular multiplication map
$m:B\boxtimes_A B \to B$
given by
the extension of $b_1\boxtimes b_2 \mapsto b_1b_2$
is bounded
\begin{align*}
\langle b_1b_2 | b_1b_2\rangle_B
&=
b_2^*b_1^*b_1b_2
\\&\leq
\ind(E_A)\cdot b_2^* E_A(b_1^*b_1)b_2
\\&=
\ind(E_A)\cdot \langle b_2|\langle b_1|b_1\rangle_A \rhd b_2\rangle_B
\\&=
\ind(E_A)\cdot \langle b_1\boxtimes b_2|b_1\boxtimes b_2\rangle_{B\boxtimes_A B}.
\end{align*}
The following lemma characterizes exactly when $m$ is adjointable and $m^\dag$ is a $B-B$ bimodule map.

\begin{lem}
\label{lem:FGP-dualizable}
Suppose $A\subset B$ is a unital inclusion of $\rm C^*$-algebras
and $E_A: B \to A$ is a faithful conditional expectation.
The following are equivalent.
\begin{enumerate}[label=(\arabic*)]
\item 
$B_A$ is finitely generated projective,
\item
$m$ is adjointable, and $m^\dag$ is a $B-B$ bimodule map.
\item
$(B,m,i)$ is a $\rm C^*$ Frobenius algebra 
in $\rCorr(A\to A)$, and
\item
${}_AB_A \in \rCorr(A\to A)$ is dualizable.
\end{enumerate}
\end{lem}
\begin{proof}
\item[\underline{$(1)\Rightarrow (2)$:}]
Let $\{\beta\}\subset B$ be a finite $B_A$-basis, so $\sum_{\beta} \beta \langle \beta | b\rangle_A = b$ for all $b\in B$.
By \cite[\S4, p3462]{MR1624182} (see also \cite[Props.~1.2.6 and 1.2.8]{MR996807}), the element $\sum_\beta \beta\boxtimes \beta^* \in B\boxtimes_A B$ is independent of the choice of $B_A$-basis.
We claim that the map $m^\dag:b \mapsto \sum_\beta \beta\boxtimes \beta^* b$
is the adjoint of $m$.
One calculates for all $b_1,b_2,b_3\in B$,
\begin{align*}
\langle m^\dag(b_3)|b_1\boxtimes b_2\rangle_B
&=
\sum_{\beta}
\left\langle  \beta \boxtimes \beta^*b_3
| 
b_1\boxtimes b_2\right\rangle_B
\\&=
\sum_{\beta}
\left\langle \beta^*b_3
| 
\langle \beta |b_1\rangle_A b_2\right\rangle_B
\\&=
\sum_{\beta}
\left\langle b_3
| 
\beta \langle \beta |b_1\rangle_A b_2\right\rangle_B
\\&=
\langle b_3 | b_1b_2\rangle_B
\\&=
\langle b_3 | m(b_1\boxtimes b_2)\rangle_B.
\end{align*}
To see that $m^\dag$ is $B-B$ bilinear, we observe that given any unitary $u\in U(B)$, $\{u\beta\}$ is another finite $B_A$-basis, and thus 
$$
u\left(\sum_\beta \beta\boxtimes \beta^*\right)
=
\sum_\beta u\beta\boxtimes \beta^*
=
\sum_\beta u\beta\boxtimes (u\beta)^*u
=
\left(\sum_\beta \beta\boxtimes \beta^*\right)u.
$$
Since $B$ is linearly spanned by its unitaries, for all $b_1,b_2\in B$,
$$
b_1(m^\dag(b_2))
=
b_1\left(\sum_\beta \beta\boxtimes \beta^*\right)b_2
=
\left(\sum_\beta \beta\boxtimes \beta^*\right)b_1b_2
=
m^\dag(b_1b_2)
=
m^\dag(b_1)b_2.
$$

\item[\underline{$(2)\Rightarrow (3)$:}]
The straightforward verification of the algebra and Frobenius relations is left to the reader.

\item[\underline{$(3)\Rightarrow (4)$:}]
This is \ref{Z:Dualizable}.

\item[\underline{$(4)\Rightarrow (1)$:}]
By \cite[Thm 4.13]{MR2085108},
we have ${}_AB_A$ can be endowed with a finite index bi-Hilbertian structure which is dualizable in the 2 $\rm C^*$-category of bi-Hilbertian correspondences.
By \cite[Ex.~2.31]{MR2085108}, ${}_AB_A$ dualizable implies $B_A$ is finitely generated projective.
\end{proof}

Suppose now the equivalent conditions of Lemma \ref{lem:FGP-dualizable} hold.
The
\emph{Watatani index} $[B:A] := \sum_\beta \beta\beta^*\in Z(B)_+^\times$ is independent of the choice of basis \cite[Prop.~1.2.8]{MR996807}.
We may renormalize $m$ and $i$ so that $(B,m,i)$ is a Q-system.
Indeed, 
$[B:A]^{-1/2}$ is a $B-B$ bimodular map,
so we may replace $m$ by $[B:A]^{-1/2}m$ and $i$ by $[B:A]^{1/2} i$.
\end{ex}

\begin{ex}
\label{ex:NonunitalRestrictByExpecation}
Building on Example \ref{ex:RestrictByExpectation},
Suppose $A\subset B$ is a unital inclusion of $\rm C^*$-algebras equipped with a faithful surjective cp $A-A$ bimodular map $E: B\to A$ such that $B_A$ is finitely generated projective.\footnote{It follows that $E(1_B)$ is invertible.
Indeed, since $E$ is surjective, there is a $b\in B$ such that $E(b)=1_A$.
Replacing $b$ with $(b+b^*)/2$, we may may assume $b$ is self-adjoint.
Then $b \leq \|b\|1_B$, so $E(1_B) \geq \|b\|^{-1}1_A$.
}
For any auxiliary unital $\rm C^*$-algebra $A'$,
we may view $B$ as an $A\oplus A'-A\oplus A'$ bimodule, where $A'$ acts by zero.
Using the $A\oplus A'$-valued inner product $\langle b_1|b_2\rangle_{A\oplus A'} := E(b_1^*b_2)$, we may modify the multiplication and unit of $B$ to make it a Q-system in $\rCorr(A\oplus A' \to A\oplus A')$.
\end{ex}

\begin{rem}
When $\cC$ is a bi-involutive $\rm C^*$/$\rm W^*$ 2-category
and $\End(1_b)$ is finite dimensional, 
the tensor subcategory of dualizable endomorphisms $\End_{\mathsf d}(b)$
is a unitary multitensor category with canonical involutive structure $\overline{\,\cdot\,}$.
By \cite{MR3342166}, there is a unique \emph{balanced} unitary dual functor which gives the unique unitary \emph{spherical} structure, and there is a canonical unitary isomorphism $X^\vee \cong \overline{X}$.
(Such a canonical isomorphism exists for all unitary dual functors by \cite[Cor.~B]{MR4133163}.)

Suppose now that $b$ is simple so that $\End(1_b)=\bbC$, $\End_{\mathsf d}(b)$ is a unitary tensor category.
In this setting, we can ask for an additional axiom for a (non-degenerate) Q-system ${}_bQ_b \in \End_{\mathsf d}(b)$:
\begin{itemize}
\item 
(normalized)
$\ev_Q,\coev_Q$ from \ref{Z:Dualizable} are standard solutions to the conjugate equations, i.e., $(Q, \ev_Q,\coev_Q)$ is a balanced dual for $Q$.
\end{itemize}
This axiom was required for the definition of a Q-system in \cite[Def.~3.2 and 3.8]{MR3308880}.
In this setting, by 
\cite[Prop.~2.2.13]{MR3204665},
(see also \cite[\S6]{MR1444286},
\cite[\S3]{MR2794547},
\cite[Lem.~2.9]{MR4079745}),
$Q$ is a real (symmetrically self-dual) object in $\End_{\mathsf d}(b)$ via the unitary isomorphism
$$
\tikzmath{
\fill[\BColor,rounded corners=5pt] (-.5,-.6) rectangle (1.3,.6);
\draw (0,-.6) -- node[left]{$\scriptstyle Q$} (0,0) arc (180:0:.2cm) arc (-180:0:.2cm) --  node[right]{$\scriptstyle \overline{Q}$} (.8,.6);
\draw (.2,.2) -- (.2,.4);
\filldraw (.2,.2) circle (.05cm);
\filldraw (.2,.4) circle (.05cm);
}
=
(i^*\circ m)\xz \id_{\overline{Q}}) \circ (\id_Q\xz \coev_Q) 
\in\cC({}_bQ_b\Rightarrow {}_b\overline{Q}_b).
$$
\end{rem}

\subsection{Bimodules}
\label{sec:Bimodules}

\begin{defn}
Given two Q-systems $P\in \cC(a\to a)$ and $Q\in \cC(b\to b)$, a $P-Q$ bimodule consists of a triple $(X,\lambda, \rho)$
where 
$X\in \cC(a\to b)$
and
$\lambda \in \cC( P\xz X \Rightarrow X)$ and $\rho\in \cC(X\xz Q\Rightarrow X)$
satisfy certain axioms.
We denote ${}_aX_b$
by a \textcolor{\XColor}{red} colored strand which is $a$-shaded on the left and $b$-shaded on the right
$$
\tikzmath{\filldraw[\AColor, rounded corners=5, very thin, baseline=1cm] (0,0) rectangle (.6,.6);}=a
\qquad\qquad
\tikzmath{\filldraw[\BColor, rounded corners=5, very thin, baseline=1cm] (0,0) rectangle (.6,.6);}=b
\qquad\qquad
\tikzmath{
\begin{scope}
\clip[rounded corners=5pt] (-.3,0) rectangle (.3,.6);
\fill[\AColor] (0,0) rectangle (-.3,.6);
\fill[\BColor] (0,0) rectangle (.3,.6);
\end{scope}
\draw[thick, \XColor] (0,0) -- (0,.6);
}={}_aX_b.
$$
The maps $\lambda,\rho$
are denoted
trivalent vertices respectively, and $\lambda^\dag,\rho^\dag$ are denoted by the vertical reflections
$$
\tikzmath{
\begin{scope}
\clip[rounded corners = 5pt] (-.7,-.2) rectangle (.3,.5);
\filldraw[\AColor] (-.7,-.2) rectangle (0,.5);
\filldraw[\BColor] (0,-.2) rectangle (.3,.5);
\end{scope}
\draw[\XColor,thick] (0,-.2) -- (0,.5);
\draw (-.4,-.2) arc (180:90:.4cm);
\filldraw[\XColor] (0,.2) circle (.05cm);
}
=\lambda
\qquad\qquad
\tikzmath{
\begin{scope}
\clip[rounded corners = 5pt] (-.3,-.2) rectangle (.7,.5);
\filldraw[\AColor] (-.3,-.2) rectangle (0,.5);
\filldraw[\BColor] (0,-.2) rectangle (.7,.5);
\end{scope}
\draw[\XColor,thick] (0,-.2) -- (0,.5);
\draw (.4,-.2) arc (0:90:.4cm);
\filldraw[\XColor] (0,.2) circle (.05cm);
}
=\rho
\qquad\qquad
\tikzmath{
\begin{scope}
\clip[rounded corners = 5pt] (-.7,-.5) rectangle (.3,.2);
\filldraw[\AColor] (-.7,-.5) rectangle (0,.2);
\filldraw[\BColor] (0,-.5) rectangle (.3,.2);
\end{scope}
\draw[\XColor,thick] (0,-.5) -- (0,.2);
\draw (-.4,.2) arc (180:270:.4cm);
\filldraw[\XColor] (0,-.2) circle (.05cm);
}
=\lambda^\dag
\qquad\qquad
\tikzmath{
\begin{scope}
\clip[rounded corners = 5pt] (-.3,-.5) rectangle (.7,.2);
\filldraw[\AColor] (-.3,-.5) rectangle (0,.2);
\filldraw[\BColor] (0,-.5) rectangle (.7,.2);
\end{scope}
\draw[\XColor,thick] (0,-.5) -- (0,.2);
\draw (.4,.2) arc (0:-90:.4cm);
\filldraw[\XColor] (0,-.2) circle (.05cm);
}
=\rho^\dag.
$$
Here, $P,Q$ are both denoted by black strands as their type may be inferred from the side of the red strand on which they lie.

These maps must satisfy the following axioms:
\begin{enumerate}[label=(B\arabic*)]
\item 
\label{M:associativity}
(associativity)
$
\tikzmath{
\begin{scope}
\clip[rounded corners = 5pt] (-.9,-.6) rectangle (.3,.5);
\filldraw[\AColor] (-.9,-.6) rectangle (0,.5);
\filldraw[\BColor] (0,-.6) rectangle (.3,.5);
\end{scope}
\draw[\XColor,thick] (0,-.6) -- (0,.5);
\draw (-.6,-.6) -- (-.6,-.4) arc (180:90:.6cm);
\draw (-.3,-.6) -- (-.3,-.4) arc (180:90:.3cm);
\filldraw[\XColor] (0,.2) circle (.05cm);
\filldraw[\XColor] (0,-.1) circle (.05cm);
}
=
\tikzmath{
\begin{scope}
\clip[rounded corners = 5pt] (-.9,-.6) rectangle (.3,.5);
\filldraw[\AColor] (-.9,-.6) rectangle (0,.5);
\filldraw[\BColor] (0,-.6) rectangle (.3,.5);
\end{scope}
\draw[\XColor,thick] (0,-.6) -- (0,.5);
\draw (-.4,-.2) arc (180:90:.4cm);
\draw (-.6,-.6) -- (-.6,-.4)  arc (180:0:.2cm) -- (-.2,-.6);
\filldraw[\XColor] (0,.2) circle (.05cm);
\filldraw (-.4,-.2) circle (.05cm);
}
$,
$
\tikzmath{
\begin{scope}
\clip[rounded corners = 5pt] (-.3,-.6) rectangle (.9,.5);
\filldraw[\AColor] (-.3,-.6) rectangle (0,.5);
\filldraw[\BColor] (0,-.6) rectangle (.9,.5);
\end{scope}
\draw[\XColor,thick] (0,-.6) -- (0,.5);
\draw (.6,-.6) -- (.6,-.4) arc (0:90:.6cm);
\draw (.3,-.6) -- (.3,-.4) arc (0:90:.3cm);
\filldraw[\XColor] (0,.2) circle (.05cm);
\filldraw[\XColor] (0,-.1) circle (.05cm);
}
=
\tikzmath{
\begin{scope}
\clip[rounded corners = 5pt] (-.3,-.6) rectangle (.9,.5);
\filldraw[\AColor] (-.3,-.6) rectangle (0,.5);
\filldraw[\BColor] (0,-.6) rectangle (.9,.5);
\end{scope}
\draw[\XColor,thick] (0,-.6) -- (0,.5);
\draw (.4,-.2) arc (0:90:.4cm);
\draw (.6,-.6) -- (.6,-.4)  arc (0:180:.2cm) -- (.2,-.6);
\filldraw[\XColor] (0,.2) circle (.05cm);
\filldraw (.4,-.2) circle (.05cm);
}
$, and
$
\tikzmath{
\begin{scope}
\clip[rounded corners = 5pt] (-.7,-.5) rectangle (.7,.5);
\filldraw[\AColor] (-.7,-.5) rectangle (0,.5);
\filldraw[\BColor] (0,-.5) rectangle (.7,.5);
\end{scope}
\draw[\XColor,thick] (0,-.5) -- (0,.5);
\draw (-.4,-.5) -- (-.4,-.2) arc (180:90:.4cm);
\draw (.4,-.5) arc (0:90:.4cm);
\filldraw[\XColor] (0,.2) circle (.05cm);
\filldraw[\XColor] (0,-.1) circle (.05cm);
}
=
\tikzmath{
\begin{scope}
\clip[rounded corners = 5pt] (-.7,-.5) rectangle (.7,.5);
\filldraw[\AColor] (-.7,-.5) rectangle (0,.5);
\filldraw[\BColor] (0,-.5) rectangle (.7,.5);
\end{scope}
\draw[\XColor,thick] (0,-.5) -- (0,.5);
\draw (-.4,-.5) arc (180:90:.4cm);
\draw (.4,-.5) -- (.4,-.2) arc (0:90:.4cm);
\filldraw[\XColor] (0,.2) circle (.05cm);
\filldraw[\XColor] (0,-.1) circle (.05cm);
}
$
\item
\label{M:unitality}
(unitality)
$
\tikzmath{
\begin{scope}
\clip[rounded corners = 5pt] (-.7,-.5) rectangle (.3,.5);
\filldraw[\AColor] (-.7,-.5) rectangle (0,.5);
\filldraw[\BColor] (0,-.5) rectangle (.3,.5);
\end{scope}
\draw[\XColor,thick] (0,-.5) -- (0,.5);
\draw (-.4,-.2) arc (180:90:.4cm);
\filldraw[\XColor] (0,.2) circle (.05cm);
\filldraw (-.4,-.2) circle (.05cm);
}
=
\tikzmath{
\begin{scope}
\clip[rounded corners = 5pt] (-.3,-.5) rectangle (.3,.5);
\filldraw[\AColor] (-.3,-.5) rectangle (0,.5);
\filldraw[\BColor] (0,-.5) rectangle (.3,.5);
\end{scope}
\draw[\XColor,thick] (0,-.5) -- (0,.5);
}
$ and $
\tikzmath{
\begin{scope}
\clip[rounded corners = 5pt] (-.3,-.5) rectangle (.7,.5);
\filldraw[\AColor] (-.3,-.5) rectangle (0,.5);
\filldraw[\BColor] (0,-.5) rectangle (.7,.5);
\end{scope}
\draw[\XColor,thick] (0,-.5) -- (0,.5);
\draw (.4,-.2) arc (0:90:.4cm);
\filldraw[\XColor] (0,.2) circle (.05cm);
\filldraw (.4,-.2) circle (.05cm);
}
=
\tikzmath{
\begin{scope}
\clip[rounded corners = 5pt] (-.3,-.5) rectangle (.3,.5);
\filldraw[\AColor] (-.3,-.5) rectangle (0,.5);
\filldraw[\BColor] (0,-.5) rectangle (.3,.5);
\end{scope}
\draw[\XColor,thick] (0,-.5) -- (0,.5);
}
$
\item
\label{M:Frobenius}
(Frobenius)
$
\tikzmath{
\begin{scope}
\clip[rounded corners = 5pt] (-1.3,-.5) rectangle (.3,.8);
\filldraw[\AColor] (-1.3,-.5) rectangle (0,.8);
\filldraw[\BColor] (0,-.5) rectangle (.3,.8);
\end{scope}
\draw[\XColor,thick] (0,-.5) -- (0,.8);
\draw (-1,-.5) -- (-1,.2) arc (180:0:.3cm) arc (180:270:.4cm);
\draw (-.7,.5) -- (-.7,.8);
\filldraw[\XColor] (0,-.2) circle (.05cm);
\filldraw (-.7,.5) circle (.05cm);
}
=
\tikzmath{
\begin{scope}
\clip[rounded corners = 5pt] (-.7,-.2) rectangle (.3,1.1);
\filldraw[\AColor] (-.7,-.2) rectangle (0,1.1);
\filldraw[\BColor] (0,-.2) rectangle (.3,1.1);
\end{scope}
\draw[\XColor,thick] (0,-.2) -- (0,1.1);
\draw (-.4,-.2) arc (180:90:.4cm);
\draw (-.4,1.1) arc (180:270:.4cm);
\filldraw[\XColor] (0,.2) circle (.05cm);
\filldraw[\XColor] (0,.7) circle (.05cm);
}
=
\tikzmath{
\begin{scope}
\clip[rounded corners = 5pt] (-1.3,-.8) rectangle (.3,.5);
\filldraw[\AColor] (-1.3,-.8) rectangle (0,.5);
\filldraw[\BColor] (0,-.8) rectangle (.3,.5);
\end{scope}
\draw[\XColor,thick] (0,-.8) -- (0,.5);
\draw (-1,.5) -- (-1,-.2) arc (-180:0:.3cm) arc (180:90:.4cm);
\draw (-.7,-.5) -- (-.7,-.8);
\filldraw[\XColor] (0,.2) circle (.05cm);
\filldraw (-.7,-.5) circle (.05cm);
}
$ and $
\tikzmath{
\begin{scope}
\clip[rounded corners = 5pt] (-.3,-.5) rectangle (1.3,.8);
\filldraw[\AColor] (-.3,-.5) rectangle (0,.8);
\filldraw[\BColor] (0,-.5) rectangle (1.3,.8);
\end{scope}
\draw[\XColor,thick] (0,-.5) -- (0,.8);
\draw (1,-.5) -- (1,.2) arc (0:180:.3cm) arc (0:-90:.4cm);
\draw (.7,.5) -- (.7,.8);
\filldraw[\XColor] (0,-.2) circle (.05cm);
\filldraw (.7,.5) circle (.05cm);
}
=
\tikzmath{
\begin{scope}
\clip[rounded corners = 5pt] (-.3,-.2) rectangle (.7,1.1);
\filldraw[\AColor] (-.3,-.2) rectangle (0,1.1);
\filldraw[\BColor] (0,-.2) rectangle (.7,1.1);
\end{scope}
\draw[\XColor,thick] (0,-.2) -- (0,1.1);
\draw (.4,-.2) arc (0:90:.4cm);
\draw (.4,1.1) arc (0:-90:.4cm);
\filldraw[\XColor] (0,.2) circle (.05cm);
\filldraw[\XColor] (0,.7) circle (.05cm);
}
=
\tikzmath{
\begin{scope}
\clip[rounded corners = 5pt] (-.3,-.8) rectangle (1.3,.5);
\filldraw[\AColor] (-.3,-.8) rectangle (0,.5);
\filldraw[\BColor] (0,-.8) rectangle (1.3,.5);
\end{scope}
\draw[\XColor,thick] (0,-.8) -- (0,.5);
\draw (1,.5) -- (1,-.2) arc (0:-180:.3cm) arc (0:90:.4cm);
\draw (.7,-.5) -- (.7,-.8);
\filldraw[\XColor] (0,.2) circle (.05cm);
\filldraw (.7,-.5) circle (.05cm);
}
$
\item
\label{M:separable}
(separable)
$
\tikzmath{
\begin{scope}
\clip[rounded corners = 5pt] (-.5,-.5) rectangle (.3,.5);
\filldraw[\AColor] (-.6,-.5) rectangle (0,.5);
\filldraw[\BColor] (0,-.5) rectangle (.3,.5);
\end{scope}
\draw[\XColor,thick] (0,-.5) -- (0,.5);
\draw (0,-.3) arc (270:90:.3cm);
\filldraw[\XColor] (0,.3) circle (.05cm);
\filldraw[\XColor] (0,-.3) circle (.05cm);
}
=
\tikzmath{
\begin{scope}
\clip[rounded corners = 5pt] (-.3,-.5) rectangle (.3,.5);
\filldraw[\AColor] (-.3,-.5) rectangle (0,.5);
\filldraw[\BColor] (0,-.5) rectangle (.3,.5);
\end{scope}
\draw[\XColor,thick] (0,-.5) -- (0,.5);
}
=
\tikzmath{
\begin{scope}
\clip[rounded corners = 5pt] (-.3,-.5) rectangle (.5,.5);
\filldraw[\AColor] (-.3,-.5) rectangle (0,.5);
\filldraw[\BColor] (0,-.5) rectangle (.6,.5);
\end{scope}
\draw[\XColor,thick] (0,-.5) -- (0,.5);
\draw (0,-.3) arc (-90:90:.3cm);
\filldraw[\XColor] (0,.3) circle (.05cm);
\filldraw[\XColor] (0,-.3) circle (.05cm);
}
$
\end{enumerate}
We leave it to the reader to define the notions of left and right modules by removing one of these actions (or replacing it with an identity Q-system).
\end{defn}

\begin{facts}
We collect some additional well-known facts about modules for Q-systems below.
For these statements below, $P\in \cC(a\to a)$ is a Q-system and ${}_aX_b,{}_aY_b\in \cC(a\to b)$ are left $P$-modules.
(We ignore any right action, which can be taken to be $1_b$ if needed.)
We denote $a$ and $b$ by shaded regions and ${}_aX_b, {}_aY_b$ by colored strands.
$$
\tikzmath{\filldraw[\AColor, rounded corners=5, very thin, baseline=1cm] (0,0) rectangle (.6,.6);}=a
\qquad\qquad
\tikzmath{\filldraw[\BColor, rounded corners=5, very thin, baseline=1cm] (0,0) rectangle (.6,.6);}=b
\qquad\qquad
\tikzmath{
\begin{scope}
\clip[rounded corners=5pt] (-.3,0) rectangle (.3,.6);
\fill[\AColor] (0,0) rectangle (-.3,.6);
\fill[\BColor] (0,0) rectangle (.3,.6);
\end{scope}
\draw[thick, \XColor] (0,0) -- (0,.6);
}={}_aX_b
\qquad\qquad
\tikzmath{
\begin{scope}
\clip[rounded corners=5pt] (-.3,0) rectangle (.3,.6);
\fill[\AColor] (0,0) rectangle (-.3,.6);
\fill[\BColor] (0,0) rectangle (.3,.6);
\end{scope}
\draw[thick, \YColor] (0,0) -- (0,.6);
}={}_aY_b
$$

\begin{enumerate}[label=(MM\arabic*)]
\item
\label{QSys:AdjointAction}
The adjoint $\lambda^\dag$ of the left $P$-action $\lambda: P\xz X \to X$ is given by
$$
\tikzmath{
\begin{scope}
\clip[rounded corners = 5pt] (-.8,0) rectangle (.3,1);
\filldraw[\AColor] (-.8,0) rectangle (0,1);
\filldraw[\BColor] (0,0) rectangle (.3,1);
\end{scope}
\draw[\XColor,thick] (0,0) -- (0,1);
\draw (-.5,1) arc (180:270:.5cm);
\filldraw[\XColor] (0,.5) circle (.05cm);
}
=
\tikzmath{
\begin{scope}
\clip[rounded corners = 5pt] (-1.05,-.25) rectangle (.3,1);
\filldraw[\AColor] (-1.05,-.25) rectangle (0,1);
\filldraw[\BColor] (0,-.25) rectangle (.3,1);
\end{scope}
\draw[\XColor,thick] (0,-.25) -- (0,1);
\draw (-.75,1) -- (-.75,.5) arc (-180:0:.25cm) arc (180:90:.25cm);
\draw (-.5,0) -- (-.5,.25);
\filldraw (-.5,.25) circle (.05cm);
\filldraw (-.5,0) circle (.05cm);
\filldraw[\XColor] (0,.75) circle (.05cm);
}\,.
$$
\begin{proof}
This is immediate from \ref{M:unitality} and \ref{M:Frobenius}.
\end{proof}
\item
\label{QSys:StarClosed}
Given a left $P$-module map $f: {}_PX\to {}_PY$, $f^\dag: Y \to X$ is also a left $P$-module map.
\begin{proof}
Apply $\dag$ to
$
\tikzmath{
\begin{scope}
\clip[rounded corners=5pt] (-.8,-.5) rectangle (.5,1.5);
\fill[\AColor] (-.8,-.5) rectangle (0,1.5);
\fill[\BColor] (0,-.5) rectangle (.5,1.5);
\end{scope}
\draw[\XColor,thick] (0,-.5) -- (0,.25);
\draw[\YColor,thick] (0,.25) -- (0,1.5);
\draw (-.5,1.5) arc (180:270:.5cm);
\filldraw[\YColor] (0,1) circle (.05cm);
\roundNbox{unshaded}{(0,.25)}{.3}{0}{0}{$f$};
}
\underset{\text{\ref{QSys:AdjointAction}}}{=}
\tikzmath{
\begin{scope}
\clip[rounded corners=5pt] (-1.05,-1.2) rectangle (.5,1);
\fill[\AColor] (-1.05,-1.2) rectangle (0,1);
\fill[\BColor] (0,-1.2) rectangle (.5,1);
\end{scope}
\draw[\XColor,thick] (0,0) -- (0,-1.2);
\draw[\YColor,thick] (0,-.5) -- (0,1);
\draw (-.75,1) -- (-.75,.5) arc (-180:0:.25cm) arc (180:90:.25cm);
\draw (-.5,0) -- (-.5,.25);
\filldraw (-.5,.25) circle (.05cm);
\filldraw (-.5,0) circle (.05cm);
\filldraw[\YColor] (0,.75) circle (.05cm);
\roundNbox{unshaded}{(0,-.5)}{.3}{0}{0}{$f$};
}
=
\tikzmath{
\begin{scope}
\clip[rounded corners=5pt] (-1.3,-1) rectangle (.5,1);
\fill[\AColor] (-1.3,-1) rectangle (0,1);
\fill[\BColor] (0,-1) rectangle (.5,1);
\end{scope}
\draw[\XColor,thick] (0,0) -- (0,-1);
\draw[\YColor,thick] (0,0) -- (0,1);
\draw (-1,1) -- (-1,-.25) arc (-180:0:.25cm) -- (-.5,.25) arc (180:90:.5cm);
\draw (-.75,-.5) -- (-.75,-.75);
\filldraw (-.75,-.5) circle (.05cm);
\filldraw (-.75,-.75) circle (.05cm);
\filldraw[\YColor] (0,.75) circle (.05cm);
\roundNbox{unshaded}{(0,0)}{.3}{0}{0}{$f$};
}
=
\tikzmath{
\begin{scope}
\clip[rounded corners=5pt] (-1.05,-.75) rectangle (.5,1.3);
\fill[\AColor] (-1.05,-.75) rectangle (0,1.3);
\fill[\BColor] (0,-.75) rectangle (.5,1.3);
\end{scope}
\draw[\XColor,thick] (0,-.75) -- (0,.75);
\draw[\YColor,thick] (0,.75) -- (0,1.3);
\draw (-.75,1.3) -- (-.75,0) arc (-180:0:.25cm) arc (180:90:.25cm);
\draw (-.5,-.5) -- (-.5,-.25);
\filldraw (-.5,-.25) circle (.05cm);
\filldraw (-.5,-.5) circle (.05cm);
\filldraw[\XColor] (0,.25) circle (.05cm);
\roundNbox{unshaded}{(0,.75)}{.3}{0}{0}{$f$};
}
\underset{\text{\ref{QSys:AdjointAction}}}{=}
\tikzmath{
\begin{scope}
\clip[rounded corners=5pt] (-.8,.5) rectangle (.5,2.5);
\fill[\AColor] (-.8,.5) rectangle (0,2.5);
\fill[\BColor] (0,.5) rectangle (.5,2.5);
\end{scope}
\draw[\YColor,thick] (0,1.75) -- (0,2.5);
\draw[\XColor,thick] (0,0.5) -- (0,1.75);
\draw (-.5,2.5) -- (-.5,1.5) arc (180:270:.5cm);
\filldraw[\XColor] (0,1) circle (.05cm);
\roundNbox{unshaded}{(0,1.75)}{.3}{0}{0}{$f$};
}
$\,.
\end{proof}
\end{enumerate}
\end{facts}

\begin{facts}
We again have the following dependencies amongst the relations.
\begin{itemize}
\item 
\ref{M:associativity}+\ref{M:separable}$\Rightarrow$\ref{M:unitality}
\item
By \cite[Lem.~3.23]{MR3308880},
\ref{M:associativity}+\ref{M:separable}$\Rightarrow$\ref{M:Frobenius}
\item
Using \ref{QSys:AdjointAction}, we have
\ref{M:associativity}+\ref{M:unitality}+\ref{M:Frobenius}
$\Rightarrow$
\ref{M:separable}
\item
\ref{QSys:StarClosed} follows from \ref{M:associativity}, \ref{M:Frobenius}, and \ref{M:separable} without even assuming \ref{Q:unitality}.
This fact is useful when one defines the $\rm C^*/W^*$ 2-category $\mathsf{Kar}^\dag(\cC)$ of unitary condensation algebras, unitary condensation bimodules, and intertwiners in $\cC$ in the spirit of \cite{1905.09566}.
Indeed, take $\dag$ of
$$
\tikzmath{
\begin{scope}
\clip[rounded corners=5pt] (-.8,-.5) rectangle (.5,1.5);
\fill[\AColor] (-.8,-.5) rectangle (0,1.5);
\fill[\BColor] (0,-.5) rectangle (.5,1.5);
\end{scope}
\draw[\XColor,thick] (0,-.5) -- (0,.25);
\draw[\YColor,thick] (0,.25) -- (0,1.5);
\draw (-.5,1.5) arc (180:270:.5cm);
\filldraw[\YColor] (0,1) circle (.05cm);
\roundNbox{unshaded}{(0,.25)}{.3}{0}{0}{$f$};
}
\underset{\text{\ref{M:separable}}}{=}
\tikzmath{
\begin{scope}
\clip[rounded corners=5pt] (-.8,-1.2) rectangle (.5,1.5);
\fill[\AColor] (-.8,-1.2) rectangle (0,1.5);
\fill[\BColor] (0,-1.2) rectangle (.5,1.5);
\end{scope}
\draw[\XColor,thick] (0,-1.2) -- (0,.25);
\draw[\YColor,thick] (0,.25) -- (0,1.5);
\draw (-.5,1.5) arc (180:270:.5cm);
\draw (0,-.4) arc (90:270:.3cm);
\filldraw[\YColor] (0,1) circle (.05cm);
\filldraw[\XColor] (0,-.4) circle (.05cm);
\filldraw[\XColor] (0,-1) circle (.05cm);
\roundNbox{unshaded}{(0,.25)}{.3}{0}{0}{$f$};
}
=
\tikzmath{
\begin{scope}
\clip[rounded corners=5pt] (-.8,-1.2) rectangle (.5,1.5);
\fill[\AColor] (-.8,-1.2) rectangle (0,1.5);
\fill[\BColor] (0,-1.2) rectangle (.5,1.5);
\end{scope}
\draw[\XColor,thick] (0,-1.2) -- (0,-.3);
\draw[\YColor,thick] (0,-.3) -- (0,1.5);
\draw (-.5,1.5) arc (180:270:.5cm);
\draw (0,-1) arc (270:180:.5cm) -- (-.5,-.1) arc (180:90:.5cm);
\filldraw[\YColor] (0,1) circle (.05cm);
\filldraw[\YColor] (0,.4) circle (.05cm);
\filldraw[\XColor] (0,-1) circle (.05cm);
\roundNbox{unshaded}{(0,-.3)}{.3}{0}{0}{$f$};
}
\underset{\text{\ref{M:Frobenius}}}{=}
\tikzmath{
\begin{scope}
\clip[rounded corners=5pt] (-1.1,-1.2) rectangle (.5,.7);
\fill[\AColor] (-1.1,-1.2) rectangle (0,.7);
\fill[\BColor] (0,-1.2) rectangle (.5,.7);
\end{scope}
\draw[\XColor,thick] (0,-1.2) -- (0,-.3);
\draw[\YColor,thick] (0,-.3) -- (0,.7);
\draw (-.5,-.3) arc (270:180:.3cm) -- (-.8,.7);
\draw (0,-1) arc (270:180:.5cm) -- (-.5,-.1) arc (180:90:.5cm);
\filldraw (-.5,-.3) circle (.05cm);
\filldraw[\YColor] (0,.4) circle (.05cm);
\filldraw[\XColor] (0,-1) circle (.05cm);
\roundNbox{unshaded}{(0,-.3)}{.3}{0}{0}{$f$};
}
\underset{\text{\ref{M:associativity}}^\dag}{=}
\tikzmath{
\begin{scope}
\clip[rounded corners=5pt] (-1.1,-1.7) rectangle (.5,.7);
\fill[\AColor] (-1.1,-1.7) rectangle (0,.7);
\fill[\BColor] (0,-1.7) rectangle (.5,.7);
\end{scope}
\draw[\XColor,thick] (0,-1.7) -- (0,-.3);
\draw[\YColor,thick] (0,-.3) -- (0,.7);
\draw (0,-1.4) arc (270:180:.8cm) -- (-.8,.7);
\draw (0,-1) arc (270:180:.5cm) -- (-.5,-.1) arc (180:90:.5cm);
\filldraw[\XColor] (0,-1.4) circle (.05cm);
\filldraw[\YColor] (0,.4) circle (.05cm);
\filldraw[\XColor] (0,-1) circle (.05cm);
\roundNbox{unshaded}{(0,-.3)}{.3}{0}{0}{$f$};
}
=
\tikzmath{
\begin{scope}
\clip[rounded corners=5pt] (-.9,-1.7) rectangle (.5,.8);
\fill[\AColor] (-.9,-1.7) rectangle (0,.8);
\fill[\BColor] (0,-1.7) rectangle (.5,.8);
\end{scope}
\draw[\XColor,thick] (0,-1.7) -- (0,.25);
\draw[\YColor,thick] (0,.25) -- (0,.8);
\draw (-.6,.8) -- (-.6,-.8) arc (180:270:.6cm);
\draw (0,-.4) arc (90:270:.3cm);
\filldraw[\XColor] (0,-1.4) circle (.05cm);
\filldraw[\XColor] (0,-.4) circle (.05cm);
\filldraw[\XColor] (0,-1) circle (.05cm);
\roundNbox{unshaded}{(0,.25)}{.3}{0}{0}{$f$};
}
\underset{\text{\ref{M:separable}}}{=}
\tikzmath{
\begin{scope}
\clip[rounded corners=5pt] (-.8,.5) rectangle (.5,2.5);
\fill[\AColor] (-.8,.5) rectangle (0,2.5);
\fill[\BColor] (0,.5) rectangle (.5,2.5);
\end{scope}
\draw[\YColor,thick] (0,1.75) -- (0,2.5);
\draw[\XColor,thick] (0,0.5) -- (0,1.75);
\draw (-.5,2.5) -- (-.5,1.5) arc (180:270:.5cm);
\filldraw[\XColor] (0,1) circle (.05cm);
\roundNbox{unshaded}{(0,1.75)}{.3}{0}{0}{$f$};
}\,.
$$
\end{itemize}
\end{facts}

\begin{defn}
Let $\cC$ be a $\rm C^*/\rm W^*$ 2-category.
We define the $\rm C^*/\rm W^*$ 2-category 
$\QSys(\cC)$ as follows:
\begin{itemize}
\item 
objects are Q-systems $(Q,m,i)\in \cC(b\to b)$,

\item
1-morphisms between Q-systems $P\in \cC(a\to a)$ and $Q\in \cC(b\to b)$ are (unital Frobenius) bimodules $({}_aX_b,\lambda,\rho)\in \cC(a\to b)$, and

\item
2-morphisms are bimodule intertwiners.
That is, if $P\in \cC(a\to a)$ and $Q\in \cC(b\to b)$ are Q-systems and
${}_aX_b, {}_aY_b\in \cC(a\to b)$ are $P-Q$ bimodules,
we define $\QSys(\cC)({}_PX_Q\Rightarrow {}_PY_Q)$ as the set of $f\in \cC({}_aX_b\Rightarrow {}_aY_b)$ such that
\begin{equation}
\label{eq:BimoduleMaps}
\tikzmath{
\begin{scope}
\clip[rounded corners=5pt] (-.8,.5) rectangle (.5,-1.5);
\fill[\AColor] (-.8,.5) rectangle (0,-1.5);
\fill[\BColor] (0,.5) rectangle (.5,-1.5);
\end{scope}
\draw[\YColor,thick] (0,.5) -- (0,-.25);
\draw[\XColor,thick] (0,-.25) -- (0,-1.5);
\draw (-.5,-1.5) arc (180:90:.5cm);
\filldraw[\XColor] (0,-1) circle (.05cm);
\roundNbox{unshaded}{(0,-.25)}{.3}{0}{0}{$f$};
}
=
\tikzmath{
\begin{scope}
\clip[rounded corners=5pt] (-.8,-.5) rectangle (.5,-2.5);
\fill[\AColor] (-.8,-.5) rectangle (0,-2.5);
\fill[\BColor] (0,-.5) rectangle (.5,-2.5);
\end{scope}
\draw[\XColor,thick] (0,-1.75) -- (0,-2.5);
\draw[\YColor,thick] (0,-0.5) -- (0,-1.75);
\draw (-.5,-2.5) -- (-.5,-1.5) arc (180:90:.5cm);
\filldraw[\YColor] (0,-1) circle (.05cm);
\roundNbox{unshaded}{(0,-1.75)}{.3}{0}{0}{$f$};
}
\qquad\text{and}\qquad
\tikzmath{
\begin{scope}
\clip[rounded corners=5pt] (.8,.5) rectangle (-.5,-1.5);
\fill[\BColor] (.8,.5) rectangle (0,-1.5);
\fill[\AColor] (0,.5) rectangle (-.5,-1.5);
\end{scope}
\draw[\YColor,thick] (0,.5) -- (0,-.25);
\draw[\XColor,thick] (0,-.25) -- (0,-1.5);
\draw (.5,-1.5) arc (0:90:.5cm);
\filldraw[\XColor] (0,-1) circle (.05cm);
\roundNbox{unshaded}{(0,-.25)}{.3}{0}{0}{$f$};
}
=
\tikzmath{
\begin{scope}
\clip[rounded corners=5pt] (.8,-.5) rectangle (-.5,-2.5);
\fill[\BColor] (.8,-.5) rectangle (0,-2.5);
\fill[\AColor] (0,-.5) rectangle (-.5,-2.5);
\end{scope}
\draw[\XColor,thick] (0,-1.75) -- (0,-2.5);
\draw[\YColor,thick] (0,-0.5) -- (0,-1.75);
\draw (.5,-2.5) -- (.5,-1.5) arc (0:90:.5cm);
\filldraw[\YColor] (0,-1) circle (.05cm);
\roundNbox{unshaded}{(0,-1.75)}{.3}{0}{0}{$f$};
}\,.
\end{equation}
By \ref{QSys:StarClosed}, every hom 1-category in $\QSys(\cC)$ is a $\dag$ 2-category.
Combining Remark \ref{rem:TensorIdNormContinuous} and Lemma \ref{lem:CompositionNormContinuous}, we see \eqref{eq:BimoduleMaps} is a norm-closed condition.
When $\cC$ is $\rm W^*$, \eqref{eq:BimoduleMaps} is a weak*-closed (normal) condition by replacing Remark \ref{rem:TensorIdNormContinuous} with \ref{W*:TensorIsNormal}.
Thus every hom 1-category in $\QSys(\cC)$ is $\rm C^*/W^*$ respectively.
\end{itemize}
\end{defn}


We now define composition of 1-morphisms and the associator of $\QSys(\cC)$.

\begin{defn}[Composition of 1-morphisms]\label{defn:QSysTensorProduct}
\label{defn:QSysTensorProduct}
Let $P\in \cC(a\to a)$, $Q\in \cC(b\to b)$, and $R\in \cC(c\to c)$ be Q-systems.
Let 
${}_{a}X_{b} \in \QSys(\cC)(P \to Q)$ 
and
${}_{b}Y_{c} \in \QSys(\cC)(Q\to R)$.
Since $\cC$ is unitarily idempotent complete, 
and
$\QSys(\cC)(P \to R)$ is a $\rm C^*/W^*$ subcategory of $\cC$,
it is straightforward to show that $\QSys(\cC)(P \to R)$ is also unitarily idempotent complete.
We define the 
$P-R$ bimodule ${}_aX\xzq_{Q}Y_b\in\QSys(\cC)(P \to R)$ 
by unitarily splitting the canonical separability projector
\begin{equation}
\label{eq:SeparabilityProjector}
p_{X,Y}
:=
\tikzmath{
\begin{scope}
\clip[rounded corners = 5pt] (-.5,-.5) rectangle (.5,.5);
\filldraw[\AColor] (-.5,-.5) rectangle (-.2,.5);
\filldraw[\BColor] (-.2,-.5) rectangle (.2,.5);
\filldraw[\CColor] (.2,-.5) rectangle (.5,.5);
\end{scope}
\draw[thick, \XColor] (-.2,-.5) -- (-.2,.5);
\draw[thick] (-.2,0) -- (.2,0);
\draw[thick, \YColor] (.2,-.5) -- (.2,.5);
}
:=
\tikzmath{
\begin{scope}
\clip[rounded corners = 5pt] (-.7,-.7) rectangle (.7,.5);
\filldraw[\AColor] (-.7,-.7) rectangle (-.4,.5);
\filldraw[\BColor] (-.4,-.7) rectangle (.4,.5);
\filldraw[\CColor] (.4,-.7) rectangle (.7,.5);
\end{scope}
\draw (-.4,.2) arc (90:0:.2cm) arc (-180:0:.2cm) arc (180:90:.2cm);
\draw (0,-.4) -- (0,-.2);
\draw[thick, \XColor] (-.4,-.7) -- (-.4,.5);
\draw[thick, \YColor] (.4,-.7) -- (.4,.5);
\filldraw[\XColor] (-.4,.2) circle (.05cm);
\filldraw[\YColor] (.4,.2) circle (.05cm);
\filldraw (0,-.2) circle (.05cm);
\filldraw (0,-.4) circle (.05cm);
}
=
\tikzmath{
\begin{scope}
\clip[rounded corners = 5pt] (-.6,-.6) rectangle (.6,.6);
\filldraw[\AColor] (-.6,-.6) rectangle (-.3,.6);
\filldraw[\BColor] (-.3,-.6) rectangle (.3,.6);
\filldraw[\CColor] (.3,-.6) rectangle (.6,.6);
\end{scope}\draw (-.3,-.3) arc (-90:0:.3cm) arc (180:90:.3cm);
\draw[thick, \XColor] (-.3,-.6) -- (-.3,.6);
\draw[thick, \YColor] (.3,-.6) -- (.3,.6);
\filldraw[\XColor] (-.3,-.3) circle (.05cm);
\filldraw[\YColor] (.3,.3) circle (.05cm);
}
\end{equation}
in the $\rm C^*/W^*$ category $\QSys(\cC)(P \to R)$.
That is, there is an object ${}_aX\xzq_Q Y_b \in \QSys(\cC)(P \to R)$ and a $P-R$ bimodular coisometry $u_{X,Y}: X\xz_b Y \to X\xzq_Q Y$ such that $p_{X,Y}=u^\dag_{X,Y}\circ u_{X,Y}$.
Since $u_{X,Y}$ is a coisometry, we have $u_{X,Y}\circ p_{X,Y} = u_{X,Y}$.
Graphically, we denote
$$
\tikzmath{
\begin{scope}
\clip[rounded corners = 5pt] (-.3,-.5) rectangle (.3,.5);
\filldraw[\AColor] (-.3,-.5) rectangle (-.04,.5);
\filldraw[\BColor] (-.04,-.5) rectangle (.04,.5);
\filldraw[\CColor] (.04,-.5) rectangle (.3,.5);
\end{scope}
\DoubleStrand{(0,-.5)}{(0,.5)}{\XColor}{\YColor}
}
=
X\xzq_Q Y
\qquad\qquad
\tikzmath{
\begin{scope}
\clip[rounded corners = 5pt] (-.5,-.7) rectangle (.5,.7);
\filldraw[\AColor] (-.5,-.7) -- (-.15,-.7) -- (-.15,0) -- (-.04,0) -- (-.04,.7) -- (-.5,.7);
\filldraw[\BColor] (-.15,-.7) -- (-.15,0) -- (-.04,0) -- (-.04,.7) -- (.04,.7) -- (.04,0) -- (.15,0) -- (.15,-.7);
\filldraw[\CColor] (.5,-.7) -- (.15,-.7) -- (.15,0) -- (.04,0) -- (.04,.7) -- (.5,.7);
\end{scope}
\draw[thick, \XColor] (-.15,-.7) -- (-.15,-.3);
\draw[thick, \YColor] (.15,-.7) -- (.15,-.3);
\DoubleStrand{(0,.3)}{(0,.7)}{\XColor}{\YColor}
\halfRoundBox{unshaded}{(0,0)}{.3}{0}{$u$};
}=u_{X,Y}.
$$
The dots between the red and orange strands are meant to evoke the separability projector.
The pair $({}_aX\xzq_Q Y_b,u_{X,Y})$ is unique up to canonical unitary isomorphism.

We now define the horizontal composition $\xzq_Q$ of bimodule intertwiners.
It suffices to define horizontal composition with an identity 2-morphism on both sides and prove the exchange relation.
Suppose $P\in \cC(a\to a)$, $Q\in \cC(b\to b)$, and $R\in \cC(c\to c)$ are Q-systems and ${}_PW_Q\in \cC(a\to b),{}_PX_Q\in \cC(a\to b)$ and ${}_QY_R,{}_QZ_R\in \cC(b\to c)$ are bimodules.
For $f\in \QSys(\cC)({}_PW_Q \Rightarrow {}_PX_Q)$, we define
\begin{equation}
\label{eq:QSysHorizontalComposition2Morphisms1}
f\xzq_Q \id_Y
:=
u_{X,Y} \circ (f\xz_b \id_Y) \circ u^{\dag}_{W,Y}
\in
\QSys(\cC)({}_PW\xzq_Q Y_R\Rightarrow {}_PX\xzq_Q Y_R)
\end{equation}
For $g\in \QSys(\cC)({}_QY_R \Rightarrow {}_QZ_R)$, we define
\begin{equation}
\label{eq:QSysHorizontalComposition2Morphisms2}
\id_X \xzq_Q g 
:=
u_{X,Z}\circ (\id_X \xz_b g)\circ u_{X,Y}^\dag
\in \QSys(\cC)({}_PX\xzq_Q Y_R \Rightarrow {}_PX\xzq_Q Z_R).
\end{equation}
The interchange relation now follows immediately by the fact that 
$$
p_{X,Y}\circ (f\xz_b \id_Y) = (f\xz_b \id_Y)\circ p_{W,Y}
\qquad\text{and}\qquad
p_{X,Z} \circ(\id_X\xz_b g) = (\id_X\xz_b g)\circ p_{X,Y}
$$
as $f,g$ are bimodule maps.
\end{defn}

\begin{rem}
\label{rem:QSysCoequalizer}
Observe that $(X\xzq_Q Y, u_{X,Y})$ is the coequalizer 
\begin{equation}
\label{eq:QSysCoequalizer}
\begin{tikzcd}[column sep=4em]
X\xz Q\xz Y
\arrow[r,shift left=1.5, "\id_{X}\xz \lambda_Y"]
\arrow[r,shift right=1.5, swap,"\rho_X \xz \id_{Y}"]
&
X\xz Y 
\arrow[r, shift left=1, "u_{X,Y}"]
\arrow[dr, swap, "T"]
&
X\xzq_Q  Y.
\arrow[l, shift left=1, "u^\dag_{X,Y}"]
\arrow[d, dashed, "\exists\,!\,\widetilde{T}"]
\\
&&
Z
\end{tikzcd}
\end{equation}
Indeed, any map $T\in \Hom_{P-R}(X\boxtimes_Q Y \to Z)$ which coequalizes the two maps on the left of \eqref{eq:QSysCoequalizer} satisfies 
$$
T\circ p_{X,Y}
=
T
\circ (\id_X\boxtimes \lambda_Y) 
\circ (\rho_X^\dag \boxtimes \id_Y)
=
T
\circ (\rho_X\boxtimes \id_Y) 
\circ (\rho_X^\dag \boxtimes \id_Y)
=
T.
$$
Since $p_{X,Y}=u_{X,Y}^\dag u_{X,Y}$ and $u_{X,Y}$ is a coisometry,
the unique map $\widetilde{T}: X\xzq_Q Y \to Z$ such that $\widetilde{T}\circ p_{X,Y}= T$ must be given by $\widetilde{T}:= T\circ u_{X,Y}^\dag$.
\end{rem}

\begin{defn}[Unitor and associator]
For each Q-system $Q\in \cC(b\to b)$, $Q$ is also the unit 1-morphism in $\QSys(\cC)(Q\to Q)$.
Given a $Q-R$ bimodule ${}_bY_c\in \cC(b\to c)$, the unitor natural isomorphism
$\lambda_Y^Q : Q\xzq_Q Y \to Y$ is given by $\lambda_Y\circ u_{Q,Y}^\dag$.
Similarly, given a $P-Q$ bimodule ${}_aX_b\in \cC(a\to b)$, the unitor natural isomorphisms $\rho^Q_X: X\xzq_Q Q \to X$ is given by $\rho_X \circ u_{X,Q}^\dag$.

The associator for $\QSys$ is also built from the $u_{X,Y}$.
Suppose $a,b,c,d\in\cC$, denoted by the following shaded regions:
$$
\tikzmath{\filldraw[\AColor, rounded corners=5, very thin, baseline=1cm] (0,0) rectangle (.6,.6);}=a
\qquad\qquad
\tikzmath{\filldraw[\BColor, rounded corners=5, very thin, baseline=1cm] (0,0) rectangle (.6,.6);}=b
\qquad\qquad
\tikzmath{\filldraw[\CColor, rounded corners=5, very thin, baseline=1cm] (0,0) rectangle (.6,.6);}=c
\qquad\qquad
\tikzmath{\filldraw[\DColor, rounded corners=5, very thin, baseline=1cm] (0,0) rectangle (.6,.6);}=d
$$
Suppose we have 1-morphism ${}_aX_b$, ${}_bY_c$, ${}_cZ_d$
which we denote graphically by
$$
\tikzmath{
\begin{scope}
\clip[rounded corners=5pt] (-.3,0) rectangle (.3,.6);
\fill[\AColor] (0,0) rectangle (-.3,.6);
\fill[\BColor] (0,0) rectangle (.3,.6);
\end{scope}
\draw[thick, \XColor] (0,0) -- (0,.6);
}={}_aX_b
\qquad\qquad
\tikzmath{
\begin{scope}
\clip[rounded corners=5pt] (-.3,0) rectangle (.3,.6);
\fill[\BColor] (0,0) rectangle (-.3,.6);
\fill[\CColor] (0,0) rectangle (.3,.6);
\end{scope}
\draw[thick, \YColor] (0,0) -- (0,.6);
}={}_bY_c
\qquad\qquad
\tikzmath{
\begin{scope}
\clip[rounded corners=5pt] (-.3,0) rectangle (.3,.6);
\fill[\CColor] (0,0) rectangle (-.3,.6);
\fill[\DColor] (0,0) rectangle (.3,.6);
\end{scope}
\draw[thick, \ZColor] (0,0) -- (0,.6);
}={}_cZ_d
$$
Assume these are $P-Q$, $Q-R$, and $R-S$ bimodule objects for Q-systems
$P\in \cC(a\to a)$, $Q\in \cC(b\to b)$, $R\in \cC(c\to c)$, and $S\in \cC(d\to d)$ respectively.
We define
\begin{equation}
\label{eq:DefinitionOfAlphaQSys}
\alpha^{\QSys}_{X,Y,Z}
:=
\tikzmath{
\begin{scope}
\clip[rounded corners=5pt] (-.7,-2.7) rectangle (.7,2.7);
\fill[\AColor] (-.7,-2.7) -- (-.08,-2.7) -- (-.08, -2) -- (-.19,-2) -- (-.19,-1) -- (-.3,-1) -- (-.3,1) .. controls ++(90:.3cm) and ++(270:.3cm) .. (-.15,1.7) -- (-.08,2) -- (-.08,2.7) -- (-.7,2.7);
\fill[\BColor] (-.08,-2.7) -- (-.08, -2) -- (-.19,-2) -- (-.19,-1) -- (-.3,-1) -- (-.3,1) .. controls ++(90:.3cm) and ++(270:.3cm) .. (-.15,1.7) -- (-.08,2) -- (-.08,2.7) -- (0,2.7) -- (0,2) -- (.11,2) -- (.11,1) -- (0,1) -- (0,-1) -- (-.11,-1) -- (-.1,-2) -- (0,-2) -- (0,-2.7);
\fill[\CColor] (.08,2.7) -- (.08, 2) -- (.19,2) -- (.19,1) -- (.3,1) -- (.3,-1) .. controls ++(270:.3cm) and ++(90:.3cm) .. (.15,-1.7) -- (.08,-2) -- (.08,-2.7) -- (0,-2.7) -- (0,-2) -- (-.11,-2) -- (-.11,-1) -- (0,-1) -- (0,1) -- (.11,1) -- (.1,2) -- (0,2) -- (0,2.7);
\fill[\DColor] (.7,2.7) -- (.08,2.7) -- (.08, 2) -- (.19,2) -- (.19,1) -- (.3,1) -- (.3,-1) .. controls ++(270:.3cm) and ++(90:.3cm) .. (.15,-1.7) -- (.08,-2) -- (.08,-2.7) -- (.7,-2.7);
\end{scope}
\draw[thick, \ZColor] (.3,1) -- (.3,-1) .. controls ++(270:.3cm) and ++(90:.3cm) .. (.15,-1.7);
\draw[thick, \XColor] (-.3,-1) -- (-.3,1) .. controls ++(90:.3cm) and ++(270:.3cm) .. (-.15,1.7);
\draw[thick, \YColor] (0,-.7) -- (0,.7);
\DoubleStrand{(-.15,-1)}{(-.15,-2)}{\XColor}{\YColor}
\DoubleStrand{(.15,1)}{(.15,2)}{\YColor}{\ZColor}
\TripleStrand{(0,2)}{(0,2.7)}{\XColor}{\YColor}{\ZColor}
\TripleStrand{(0,-2)}{(0,-2.7)}{\XColor}{\YColor}{\ZColor}
\halfRoundBox{unshaded}{(0,2)}{.4}{-.1}{$u$};
\halfRoundBox{unshaded}{(.15,1)}{.3}{0}{$u$};
\roundNbox{unshaded}{(0,0)}{.3}{.2}{.2}{$\alpha^\cC$}
\halfRoundBoxDag{unshaded}{(-.15,-1)}{.3}{0}{$u^\dag$};
\halfRoundBoxDag{unshaded}{(0,-2)}{.4}{-.1}{$u^\dag$};
}
:
(X\xzq_Q Y) \xzq_R Z
\Rightarrow
X \xzq_Q (Y\xzq_R Z).
\end{equation}
Here, we write the associator in $\cC$ for clarity, but we usually suppress it whenever possible.
While we use the same string type for both the source and target, the label may be inferred from the strings from the nearest $u$ or $u^\dag$.
It is a straightforward and enjoyable exercise 
to prove the pentagon relation
using the relation $u_{X,Y}^\dag \circ u_{X,Y} = p_{X,Y}$ 
and the fact that the $u$ and $u^\dag$ are bimodule maps. 
\end{defn}

\begin{rem}
One can show using \eqref{eq:QSysCoequalizer} that $\alpha^{\QSys}$, $\alpha^{\cC}$ and $u$ satisfy the following associativity relation \cite[p.~27]{MR3509018}:
\begin{equation*}
\tikzmath{
\begin{scope}
\clip[rounded corners=5pt] (-.8,-.2) rectangle (.8,3.2);
\fill[\AColor] (-.8,-.2) -- (-.3,-.2) -- (-.3, .5) -- (-.19,.5) -- (-.19,1.5) -- (-.08,1.5) -- (-.08,3.2) -- (-.8,3.2);
\fill[\BColor] (-.8,-.2) -- (-.3,-.2) -- (-.3, .5) -- (-.19,.5) -- (-.19,1.5) -- (-.08,1.5) -- (-.08,3.2) -- (0,3.2) -- (0,1.5) -- (-.11,1.5) -- (-.11,.5) -- (0,.5) -- (0,-.2);
\fill[\CColor] (0,3.2) -- (0,1.5) -- (-.11,1.5) -- (-.11,.5) -- (0,.5) -- (0,-.2) -- (.3,-.2) -- (.3,.5) .. controls ++(90:.3cm) and ++(270:.3cm) .. (.15,1.2) -- (.08,1.5) -- (.08,3.2);
\fill[\DColor] (.3,-.2) -- (.3,.5) .. controls ++(90:.3cm) and ++(270:.3cm) .. (.15,1.2) -- (.08,1.5) -- (.08,3.2) -- (.8,3.2) -- (.8,-.2);
\end{scope}
\draw[thick, \XColor] (-.3,.5) -- (-.3,-.2);
\draw[thick, \ZColor] (.3,-.2) -- (.3,.5) .. controls ++(90:.3cm) and ++(270:.3cm) .. (.15,1.2);
\draw[thick, \YColor] (0,-.2) -- (0,.2);
\DoubleStrand{(-.15,.5)}{(-.15,1.5)}{\XColor}{\YColor}
\TripleStrand{(0,1.5)}{(0,3.2)}{\XColor}{\YColor}{\ZColor}
\halfRoundBox{unshaded}{(0,1.5)}{.4}{-.1}{$u$};
\halfRoundBox{unshaded}{(-.15,.5)}{.3}{0}{$u$};
\roundNbox{unshaded}{(0,2.5)}{.3}{.3}{.3}{$\alpha^{\QSys}$}
}
=
\tikzmath{
\begin{scope}
\clip[rounded corners=5pt] (-.8,-1.2) rectangle (.8,2.2);
\fill[\AColor] (-.8,-1.2) -- (-.3,-1.2) -- (-.3,.5) .. controls ++(90:.3cm) and ++(270:.3cm) .. (-.15,1.2) -- (-.08,1.5) -- (-.08,2.2) -- (-.8,2.2);
\fill[\BColor] (-.3, -1.2) -- (-.3,-.5) -- (-.3,.5) .. controls ++(90:.3cm) and ++(270:.3cm) .. (-.15,1.2) -- (-.08,1.5) -- (-.08,2.2) -- (0,2.2) -- (0,1.5) -- (.11,1.5) -- (.11,.5) -- (0,.5) -- (0,-1.2);
\fill[\CColor] (.08,2.2) -- (.08, 1.5) -- (.19,1.5) -- (.19,.5) -- (.3,.5) -- (.3,-1.2) -- (0,-1.2) -- (0,.5) -- (.11,.5) -- (.1,1.5) -- (0,1.5) -- (0,2.2);
\fill[\DColor] (.8,2.2) -- (.08,2.2) -- (.08, 1.5) -- (.19,1.5) -- (.19,.5) -- (.3,.5) -- (.3,-1.2) -- (.8,-1.2);
\end{scope}
\draw[thick, \ZColor] (.3,.5) -- (.3,-1.2);
\draw[thick, \XColor] (-.3,-1.2) -- (-.3,.5) .. controls ++(90:.3cm) and ++(270:.3cm) .. (-.15,1.2);
\draw[thick, \YColor] (0,-1.2) -- (0,.2);
\DoubleStrand{(.15,.5)}{(.15,1.5)}{\YColor}{\ZColor}
\TripleStrand{(0,1.5)}{(0,2.2)}{\XColor}{\YColor}{\ZColor}
\halfRoundBox{unshaded}{(0,1.5)}{.4}{-.1}{$u$};
\halfRoundBox{unshaded}{(.15,.5)}{.3}{0}{$u$};
\roundNbox{unshaded}{(0,-.5)}{.3}{.2}{.2}{$\alpha^{\cC}$}
}
:
(X\xz Y) \xz Z
\Rightarrow
X \xzq_Q (Y\xzq_R Z).
\end{equation*}
One can also prove the pentagon relation for $\alpha^{\QSys(\cC)}$ using this relation.
\end{rem}

\begin{prop}
If $\cC$ is a $\rm C^*/W^*$ 2-category,
then
$\QSys(\cC)$ is a $\rm C^*/W^*$-2-category respectively.
\end{prop}
\begin{proof}
It remains to check \ref{W*:TensorIsNormal} in the $\rm W^*$ case.
Recall that tensoring with an identity 2-morphism on the right and left are given respectively by \eqref{eq:QSysHorizontalComposition2Morphisms1} and \eqref{eq:QSysHorizontalComposition2Morphisms2}.
The result now follows as $\id_X \xz_b - $ as $- \xz_b \id_Y$ are normal by \ref{W*:TensorIsNormal} for $\cC$, as is pre-composition with $u^\dag$ and post-composition with $u$.
\end{proof}

\begin{nota}
Suppose $\cC,\cD$ are 2-categories.
We use the following coherence conventions for a 2-functor
$F=(F,F^2,F^1):\cC \to \cD$:
$$
F^2_{X,Y}\in \cD( F(X)\xz_{F(b)} F(Y) \Rightarrow F(X\xz_b Y))
\qquad\text{and}\qquad
F^1_a \in \cD(1_{F(a)} \Rightarrow F(1_a)).
$$
\end{nota}

\begin{construction}
Suppose $\cC$ is a ${\rm C^*/W^*}$-2-category.
We have a canonical inclusion $\dag$ 2-functor $\iota_{\cC}:\cC\hookrightarrow \QSys(\cC)$ given as follows.
\begin{itemize}
\item
For an object $c\in \cC$, we map $c$ to $1_c\in \cC(c\to c)$ with its obvious Q-system structure as the tensor unit of $\cC(c\to c)$. 
\item
For a 1-morphism ${}_aX_b\in \cC(a\to b)$, $X$ itself is a unital Frobenius $1_a-1_b$ bimodule object, so we map $X$ to itself.
\item
For a 2-morphism $f\in \cC(X\Rightarrow Y)$,
we see that $f$ is $1_a-1_b$ bimodular, so we map $f$ to itself.
\end{itemize}
Recall that the associator in $\QSys(\cC)$ is built from the associator in $\cC$ using a unitary splitting of the canonical separability projector as in \eqref{eq:DefinitionOfAlphaQSys}. 
But the canonical separability projector for $X\xzq_{1_b} Y$ is the identity, so the associator on the image of $\iota$ in $\QSys(\cC)$ is the same as the associator in $\cC$.
We may thus take the tensorator and unitor of $\iota$ to be identity morphisms.
That is, $\iota$ is a strict $\dag$ 2-functor.
\end{construction}

\begin{ex}
\label{ex:EquivalentRestrictedQSystem}
Continuing Example \ref{ex:RestrictByExpectation}, suppose $A\subset B$ is a unital inclusion of $\rm C^*$-algebras together with a finite index 
faithful ucp conditional expectation $E_A: B \to A$ such that $B_A$ is finitely generated projective.
As before, we can renormalize $m,i$ so that $({}_AB_A,m,i)$ is a Q-system.
Let $1_B$ denote the trivial Q-system $B\in \rCorr(B\to B)$. 

We claim that ${}_AB{}_A$ and $1_B$ are equivalent as Q-systems in $\QSys(\rCorr)$.
Set $X:={}_AB_B$ with right $B$-valued inner product $\langle b_1|b_2\rangle_B:=b_1^*b_2$, and set
$X^\vee:= {}_BB_A$ with right $A$-valued inner product $\langle b_1|b_2\rangle_A:= E_A(b_1^*b_2)$.
Let the
evaluation $\ev_X : X^\vee \boxtimes_A X \to B$ be the renormalized multiplication, and the coevaluation $\coev_X: A \to X\boxtimes_B X^\vee$ be the renormalized inclusion.
Then $(X^\vee,\ev_X,\coev_X)$ is a unitarily separable left dual for $X$,
and $(X,\ev_{X^\vee}:=i^\dag[B:A]^{-1/2}, \coev_{X^\vee}:=m^\dag[B:A]^{1/2}$) is a unitarily separable left dual for $X^\vee$.
The Q-system structures on
on $X \xzq_{1_B} X^\vee$ and     
$X^\vee \xzq_{{}_AB_A} X$ are given by 
Example \ref{ex:XXvQSystem}.


It is clear that $X\xzq_{1_B}X^\vee = X\boxtimes_B X^\vee \cong {}_AB_A$ as Q-systems.
We now calculate that $X^\vee\xzq_{{}_AB_A}X \cong 1_B$, and this isomorphism intertwines the canonical Q-system structure on $X^\vee\xzq_{{}_AB_A}X$ with the unitor on $1_B$.
Denote $A,B$ and $X$ by the following shaded regions and strand respectively:
$$
\tikzmath{\filldraw[\AColor, rounded corners=5, very thin, baseline=1cm] (0,0) rectangle (.6,.6);}=A
\qquad\qquad
\tikzmath{\filldraw[\BColor, rounded corners=5, very thin, baseline=1cm] (0,0) rectangle (.6,.6);}=B
\qquad\qquad
\tikzmath{
\begin{scope}
\clip[rounded corners=5pt] (-.3,0) rectangle (.3,.6);
\fill[\AColor] (0,0) rectangle (-.3,.6);
\fill[\BColor] (0,0) rectangle (.3,.6);
\end{scope}
\draw (0,0) -- (0,.6);
}={}_AX_B
$$
Using the previous calculation, it is straightforward to see that the separability idempotent in $\End_{B-B}(X^\vee\boxtimes_A X)$ onto $X^\vee \xzq_{{}_AB_A} X$ is given by
$$
p_{X^\vee,X}
=
\ev_X^\dag \circ \ev_X
=
\tikzmath{
\begin{scope}
\clip[rounded corners=5pt] (-.6,-.5) rectangle (.6,.5);
\fill[\BColor] (-.6,-.5) rectangle (.6,.5);
\fill[\AColor] (-.3,-.5) arc (180:0:.3cm);
\fill[\AColor] (-.3,.5) arc (-180:0:.3cm);
\end{scope}
\draw (-.3,-.5) arc (180:0:.3cm);
\draw (-.3,.5) arc (-180:0:.3cm);
}\,.
$$
We see now that the coisometry 
$u_{X^\vee,X}
:=
\ev_X : X^\vee \boxtimes_A X \to B
$ 
unitarily splits the separability idempotent $p_{X^\vee,X}$.
Moreover, $u_{X^\vee,X}^\dag: B=1_B \to X^\vee\boxtimes_A X$ is clearly an isomorphism of Q-systems:
$$
\tikzmath{
\begin{scope}
\clip[rounded corners=5pt] (-.7,-.4) rectangle (.7,.9);
\fill[\BColor] (-.7,-.4) rectangle (.7,.9);
\fill[\AColor] (-.4,0) -- (-.4,.2) .. controls ++(90:.2cm) and ++(270:.2cm) .. (-.1,.7) -- (-.1,.9) -- (.1,.9) -- (.1,.7)  .. controls ++(270:.2cm) and ++(90:.2cm) .. (.4,.2) -- (.4,0) arc (0:-180:.1cm) -- (.2,.2) arc (0:180:.2cm) -- (-.2,0) arc (0:-180:.1cm);
\end{scope}
\draw (-.4,0) arc (-180:0:.1cm);
\draw (.2,0) arc (-180:0:.1cm);
\draw (-.2,0) -- (-.2,.2) arc (180:0:.2cm) -- (.2,0);
\draw (-.4,0) -- (-.4,.2) .. controls ++(90:.2cm) and ++(270:.2cm) .. (-.1,.7) -- (-.1,.9);
\draw (.4,0) -- (.4,.2) .. controls ++(90:.2cm) and ++(270:.2cm) .. (.1,.7) -- (.1,.9);
\draw[dotted, thick] (-.3,-.4) -- (-.3,-.1);
\draw[dotted, thick] (.3,-.4) -- (.3,-.1);
}
=
\tikzmath{
\begin{scope}
\clip[rounded corners=5pt] (-.7,-.5) rectangle (.7,.5);
\fill[\BColor] (-.7,-.5) rectangle (.7,.5);
\fill[\AColor] (-.1,.5) -- (-.1,.2) arc (-180:0:.1cm) -- (.1,.5);
\end{scope}
\draw (-.1,.5) -- (-.1,.2) arc (-180:0:.1cm) -- (.1,.5);
\draw[dotted, thick] (-.3,-.5) arc (180:0:.3cm);
\draw[dotted, thick] (0,-.2) -- (0,.1);
}\,.
$$
The claim now follows from the following lemma, whose straightforward proof is left to the reader.
\end{ex}

\begin{lem}
Let $\cC$ be a $\rm C^*/W^*$ 2-category.
Suppose ${}_aX_b$ has unitarily separable dual $({}_bX^\vee_a, \ev_X, \coev_X)$
such that $X$ is also a unitarily separable dual for $X^\vee$ via $\ev_{X^\vee}, \coev_{X^\vee}$.
Then 
the canonical Q-systems
$X\xz_b X^\vee$ and $X^\vee \xz_a X$
are equivalent in $\QSys(\cC)$ via the 
$X\xz_b X^\vee- X^\vee \xz_a X$ bimodule $X$
whose left and right actions are given by
$$
\tikzmath{
\begin{scope}
\clip[rounded corners = 5pt] (-.9,0) rectangle (.6,1.3);
\fill[\AColor] (-.9,0) rectangle (.3,1.3);
\fill[\BColor] (-.6,0) .. controls ++(90:.4cm) and ++(270:.4cm) .. (.3,1) -- (.3,1.3) --(.6,1.3) -- (.6,0) -- (.3,0) arc (0:180:.3cm);
\end{scope}
\draw[thick, red] (-.3,0) node[below, xshift=.1cm]{$\scriptstyle X^\vee$} arc (180:0:.3cm) node[below]{$\scriptstyle X$};
\draw[thick, red] (-.6,0) node[below]{$\scriptstyle X$} .. controls ++(90:.4cm) and ++(270:.4cm) .. (.3,1) -- (.3,1.3);
}
=
\id_X \xz_b \ev_X
\qquad\qquad
\tikzmath{
\begin{scope}
\clip[rounded corners = 5pt] (.9,0) rectangle (-.6,1.3);
\fill[\BColor] (.9,0) rectangle (-.3,1.3);
\fill[\AColor] (.6,0) .. controls ++(90:.4cm) and ++(270:.4cm) .. (-.3,1) -- (-.3,1.3) --(-.6,1.3) -- (-.6,0) -- (-.3,0) arc (180:0:.3cm);
\end{scope}
\draw[thick, red] (-.3,0) node[below]{$\scriptstyle X$} arc (180:0:.3cm) node[below]{$\scriptstyle X^\vee$};
\draw[thick, red] (.6,0) node[below, xshift=.1cm]{$\scriptstyle X$} .. controls ++(90:.4cm) and ++(270:.4cm) .. (-.3,1) -- (-.3,1.3);
}
=
\ev_{X^\vee}\xz \id_{X}
$$
\end{lem}

\begin{ex}
\label{ex:EquivalenceOfRestrictByExpectationExamples}
Continuing Example \ref{ex:NonunitalRestrictByExpecation},
suppose $A\subset B$ is a unital inclusion of $\rm C^*$-algebras equipped with a faithful cp $A-A$ bimodular map (non-normalized conditional expectation) $E: B\to A$, and suppose $A'$ is any auxiliary unital $\rm C^*$-algebra.
We can renormalize the multiplication on $B$ so that ${}_AB_A$ is a Q-system, and so that ${}_{A\oplus A'}B_{A\oplus A'}$ is a Q-system.

We claim that ${}_AB_A$ and ${}_{A\oplus A'}B_{A\oplus A'}$ are equivalent as Q-systems.
Indeed, define ${}_AX_{A\oplus A'}:= {}_AB_{A\oplus A'}$
as a unitarily separable ${}_AB_A-{}_{A\oplus A'}B_{A\oplus A'}$ bimodule with the obvious left and right actions.
It is straightforward to adapt the previous Example \ref{ex:EquivalentRestrictedQSystem} to show that 
$X \xzq_{{}_{A\oplus A'}B_{A\oplus A'}} X^\vee \cong {}_AB_A$
and
$X^\vee \xzq_{{}_AB_A} X \cong {}_{A\oplus A'}B_{A\oplus A'}$
as Q-systems.
\end{ex}

\begin{rem}
In \cite{MR4369356}, we prove that $\QSys$ is a 3-endofunctor on the 3-category (algebraic tricategory \cite{MR3076451}) of 2 $\rm C^*/W^*$ categories.
For the purposes of this article, in order to induce actions of $\rm C^*$-algebras, 
we need only the construction below at the level of $\dag$ 2-functors.
The following construction is similar to \cite[\S A.6]{1812.11933}.
\end{rem}

\begin{construction}
\label{construction:Qsys(F)}
A $\dag$ 2-functor $F: \cC \to \cD$ between $\rm C^*/W^*$ 2-categories induces a $\dag$ 2-functor $\QSys(F): \QSys(\cC) \to \QSys(\cD)$.

\item[\underline{Objects:}]
Given a Q-system $(Q,m,i)\in \cC(b\to b)$, 
$\QSys(F)$ maps 
$(Q,m,i)$ 
to
$(F(Q),F(m)\circ F^2_{Q,Q},F(i)\circ F^1_b)\in \cD(F(b)\to F(b))$.

\item[\underline{1-morphisms:}]
Given Q-systems $P\in\cC(a\to a)$ and $Q\in \cC(b\to b)$
and a $P-Q$ bimodule $({}_aX_b,\lambda, \rho) \in \cC(a\to b)$, 
$\QSys(F)$ maps 
$({}_aX_b,\lambda, \rho)$
to
$(F(X),F(\lambda)\circ F^2_{P,X},F(\rho)\circ F^2_{X,Q})\in \cD(F(a)\to F(b))$. 

\item[\underline{2-morphisms:}]
Given Q-systems $P\in\cC(a\to a)$ and $Q\in \cC(b\to b)$,
$P-Q$ bimodules ${}_aX_b, {}_aY_b \in \cC(a\to b)$,
and an intertwiner $f\in \cC({}_{a}X_{b}\Rightarrow {}_{a}Y_{b})$,
$\QSys(F)$ maps 
$f$
to
$F(f)\in \cD({}_{F(a)}F(X)_{F(b)}\Rightarrow {}_{F(a)}F(Y)_{F(b)})$.

\item[\underline{Coheretors:}]
For ${}_P X_Q \in \QSys(\cC)(P \to Q)$ and ${}_Q Y_R \in \QSys(\cC)(Q\to R)$, 
we define $\QSys(F)_{X,Y}^{2}\in \QSys(\cD)(F(X)\xzq_{F(Q)}F(Y)\Rightarrow F(X\xzq_Q Y))$ 
as 
the unique map given by the universal property of the coequalizer:
\begin{equation*} 
\label{eq:QSysCoheretorF2Coequalizer}
\begin{tikzcd}[column sep=4.5em]
F(X)\xz F(Q)\xz F(Y)
\arrow[r,shift left=1.5, "\id_{F(X)}\xz -\rhd-"]
\arrow[r,shift right=1.5, swap,"-\lhd- \xz \id_{F(Y)}"]
&
F(X)\xz F(Y) 
\arrow[r, "u_{F(X),F(Y)}"]
\arrow[d, "F^2_{X,Y}"]
&
F(X)\xzq_{F(Q)} F(Y).
\arrow[d, dashed, "\exists\,!\,\QSys(F)_{X,Y}^{2}"]
\\
&
F(X\xz Y)
\arrow[r, "F(u_{X,Y})"]
&
F(X\xzq_Q Y)
\end{tikzcd}
\end{equation*}
Observe we have the following direct formula
for $\QSys(F)_{X,Y}^{2}$:
\begin{equation}
\label{eq:DefOfQSysF2}
\QSys(F)_{X,Y}^{2} :=F(u_{X,Y}) \circ F^2_{X,Y} \circ u^\dag_{F(X), F(Y)}.
\end{equation}
By definition of the separability projector \eqref{eq:SeparabilityProjector} for $F(X)\xzq_{F(Q)}F(Y)$, we have
\begin{align*}
p_{F(X), F(Y)}
&=
(\id_{F(X)} \xz_{F(b)} \lambda_{F(Y)})\circ (\rho_{F(X)}^\dag \xz_{F(b)} \id_{F(Y)})
\\&=
(F(\id_{X}) \xz_{F(b)} (F(\lambda_Y) \circ F^2_{Q,Y}))\circ ((F(\rho_X)\circ F^2_{X,Q})^\dag \xz_{F(b)}F(\id_{Y}))
\\&=
(F^2_{X,Y})^\dag 
\circ 
F(p_{X,Y})
\circ 
F^2_{X,Y}.
\end{align*}
This formula for $p_{F(X), F(Y)}$ immediately implies 
$\QSys(F)^2_{X,Y}$ is unitary:
\begin{align*}
(\QSys(F)_{X,Y}^{2})^\dag\circ \QSys(F)_{X,Y}^{2}
&=(F(u_{X,Y}) \circ F^2_{X,Y} \circ u^\dag_{F(X), F(Y)})^\dag\circ (F(u_{X,Y}) \circ F^2_{X,Y} \circ u^\dag_{F(X), F(Y)})
\\&= 
u_{F(X),F(Y)}\circ (F^{2}_{X,Y})^\dag \circ F(p_{X,Y})\circ F^2_{X,Y} \circ u^\dag_{F(X), F(Y)}
\\&=
u_{F(X),F(Y)}\circ (F^{2}_{X,Y})^\dag \circ F^2_{X,Y}\circ p_{F(X),F(Y)} \circ u^\dag_{F(X), F(Y)}
\\&=
\id_{F(X)\xzq_{F(Q)}F(Y)}
\\
\QSys(F)_{X,Y}^{2}
\circ
(\QSys(F)_{X,Y}^{2})^\dag
&=
(F(u_{X,Y}) \circ F^2_{X,Y} \circ u^\dag_{F(X), F(Y)})
\circ
(F(u_{X,Y}) \circ F^2_{X,Y} \circ u^\dag_{F(X), F(Y)})^\dag
\\&=
F(u_{X,Y})\circ F^2_{X,Y} \circ p_{F(X), F(Y)} \circ (F^2_{X,Y})^\dag \circ F(u_{X,Y}^\dag) 
\\&=
F(u_{X,Y})\circ F(p_{X,Y}) \circ F^2_{X,Y}  \circ (F^2_{X,Y})^\dag \circ F(u_{X,Y}^\dag) 
\\&=
F(\id_{X\xzq_Q Y})
=
\id_{F(X\xzq_Q Y)}.
\end{align*}
We get the following two relations from \eqref{eq:DefOfQSysF2} using unitarity of $\QSys(F)^2$ and that $u$ is a coisometry:
\begin{align}
\QSys(F)^2_{X,Y} \circ u_{F(X),F(Y)} 
&= 
F(u_{X,Y})\circ F^2_{X,Y} 
\label{eq:QSysF2uCommute}
\\
F(u_{X,Y}^\dag)\circ \QSys(F)^2_{X,Y}  
&= 
F^2_{X,Y} \circ u^\dag_{F(X),F(Y)} 
\label{eq:QSysF2uDagCommute}
\end{align}

In turn, these identities are used to prove the associativity of $\QSys(F)^2$:
\begin{equation*}
\begin{tikzpicture}[scale=.68, transform shape]
\node (P1) at (-4,0) {$F(X)\xzq_{F(Q)}(F(Y)\xzq_{F(R)} F(Z))$};
\node (P2) at (6.5,0) {$F(X)\xzq_{F(Q)} F(Y\xzq_R Z)$};
\node (P3) at (15,0) {$F(X\xzq_Q (Y\xzq_R Z))$};
\node (P4) at (0,1.5) {$F(X)\xz (F(Y)\xzq_{F(R)} F(Z))$};
\node (P5) at (6.5,1.5) {$F(X)\xz F(Y\xzq_R Z)$};
\node (P6) at (12,1.5) {$F(X\xz (Y\xzq_R Z))$};
\node (P7) at (0,3) {$F(X)\xz (F(Y)\xz F(Z))$};
\node (P8) at (6.5,3) {$F(X)\xz F(Y\xz Z)$};
\node (P9) at (12,3) {$F(X\xz (Y\xz Z))$};
\node (P10) at (0,4.5) {$(F(X)\xz F(Y))\xz F(Z)$};
\node (P11) at (6.5,4.5) {$F(X\xz Y)\xz F(Z)$};
\node (P12) at (12,4.5) {$F((X\xz Y)\xz Z)$};
\node (P13) at (0,6) {$(F(X)\xzq_{F(Q)} F(Y))\xz F(Z)$};
\node (P14) at (6.5,6) {$F(X\xzq_Q Y)\xz F(Z)$};
\node (P15) at (12,6) {$F((X\xzq_Q Y)\xz Z)$};
\node (P16) at (-4,7.5) {$(F(X)\xzq_{F(Q)}F(Y))\xzq_{F(R)} F(Z)$};
\node (P17) at (6.5,7.5) {$F(X\xzq_Q Y)\xzq_{F(R)} F(Z)$};
\node (P18) at (15,7.5) {$F((X\xzq_Q Y)\xzq_R Z)$};
\path (P7) --node{\scriptsize \eqref{eq:QSysF2uCommute}} (P5);
\path (P13) --node{\scriptsize \eqref{eq:QSysF2uDagCommute}} (P11);
\path (P17) --node{\scriptsize \eqref{eq:QSysF2uDagCommute}} (P15);
\path (P2) --node{\scriptsize \eqref{eq:QSysF2uCommute}} (P6);
\path (P1) --node[right, xshift=.5cm]{\scriptsize \eqref{eq:DefinitionOfAlphaQSys}} (P16);
\path (P3) --node[left, xshift=-.5cm]{\scriptsize $F$\eqref{eq:DefinitionOfAlphaQSys}} (P18);
\path[commutative diagrams/.cd, every arrow, every label]
(P1) edge node {$\id\xzq\QSys(F)^2$} (P2)
(P2) edge node {$\QSys(F)^2$} (P3)
(P4) edge node {$\id\xz\QSys(F)^2$} (P5)
(P5) edge node {$F^2$} (P6)
(P7) edge node {$\id\xz F^2$} (P8)
(P8) edge node {$F^2$} (P9)
(P10) edge node {$F^2\xz\id$} (P11)
(P11) edge node {$F^2$} (P12)
(P13) edge node {$\QSys(F)^2\xz\id$} (P14)
(P14) edge node {$F^2$} (P15)
(P16) edge node {$\QSys(F)^2\xzq \id$} (P17)
(P17) edge node {$\QSys(F)^2$} (P18)
(P16) edge node[swap] {$\alpha^{\QSys(\cD)}$} (P1)
(P16) edge node {$u^\dag$} (P13)
(P13) edge node {$u^\dag\xz\id$} (P10)
(P10) edge node {$\alpha^{\cD}$} (P7)
(P7) edge node {$\id\xz u$} (P4)
(P4) edge node {$u$} (P1)
(P17) edge node {$u^\dag$} (P14)
(P14) edge node {$F(u^\dag)\xz\id$} (P11)
(P8) edge node {$\id\xz F(u)$} (P5)
(P5) edge node {$u$} (P2)
(P18) edge node {$F(\alpha^{\QSys(\cC)})$} (P3)
(P18) edge node[swap] {$F(u^\dag)$} (P15)
(P15) edge node {$F(u^\dag\xz\id)$} (P12)
(P12) edge node {$F(\alpha^{\cC})$} (P9)
(P9) edge node {$F(\id\xz u)$} (P6)
(P6) edge node[swap] {$F(u)$} (P3);
\end{tikzpicture}
\end{equation*}


Finally, for a Q-system $Q\in \cC(b\to b)$, 
we define $\QSys(F)^1_{Q}\in \QSys(\cD)(1_{F(Q)}=F(Q)\Rightarrow F(Q))$ to be the identity.
The diagram on the left below commutes as
\begin{align*}
\QSys(F)(\rho_X^Q)\circ \QSys(F)^2_{X,Q} 
& = 
F(\rho_X^Q) \circ F(u_{X,Q})\circ F^2_{X,Q}\circ u_{F(X),F(Q)}^\dag 
\\
& = F(\rho_X)\circ F^2_{X,Q} \circ u_{F(X),F(Q)}^\dag = 
\rho_{F(X)}\circ u_{F(X),F(Q)}^\dag
=
\rho_{F(X)}^{F(Q)}. 
\end{align*}
The one on the right is similar.
\begin{equation*}
\begin{tikzcd}[column sep=3.4em]
{F(X)\xzq_{F(Q)}1_{F(Q)}}  
\arrow["\rho_{F(X)}^{F(Q)}", rightarrow]{r} 
\arrow["\id\xzq\QSys(F)^1_{Q}"', rightarrow]{d} 
&
{F(X)} 
\\
{F(X)\xzq_{F(Q)}F(Q)} 
\arrow["\QSys(F)^2_{X,Q}"', rightarrow]{r} 
&
F(X\xzq_Q Q)
\arrow["\QSys(F)(\rho_X^Q)", rightarrow]{u} 
\end{tikzcd}
\qquad
\begin{tikzcd}[column sep=3.4em]
{1_{F(Q)}\xzq_{F(Q)}F(Y)}  
\arrow["\lambda_{F(Y)}^{F(Q)}", rightarrow]{r} 
\arrow["\QSys(F)^1_{Q}\xzq\id"', rightarrow]{d} 
&
{F(Y)} 
\\
{F(Q)\xzq_{F(Q)}F(Y)} 
\arrow["\QSys(F)^2_{Q,Y}"', rightarrow]{r} 
&
F(Q\xzq_Q Y)
\arrow["\QSys(F)(\lambda_Y^Q)", rightarrow]{u} 
\end{tikzcd}
\end{equation*}
\end{construction}

\begin{rem}
\label{rem:2FullyFaithful}
Observe that if our $\dag$ 2-functor $F: \cC\to \cD$ is fully faithful on 2-morphisms, then $\QSys(F): \QSys(\cC) \to \QSys(\cD)$ is as well.
\end{rem}

\subsection{Graphical calculus for \texorpdfstring{$\QSys(\cC)$}{QSysC}}
\label{sec:GraphicalCalculusForQSysC}

We now expand the graphical calculus for $\cC$ to a graphical calculus for $\QSys(\cC)$ which is compatible with the canonical embedding $\iota_\cC:\cC\hookrightarrow \QSys(\cC)$.
As before, objects $c\in \cC$, which are now the Q-systems $1_c\in \QSys(\cC)$, are represented by shaded regions, which we still denote in gray-scale:
$$
\tikzmath{\filldraw[\AColor, rounded corners=5, very thin, baseline=1cm] (0,0) rectangle (.6,.6);}=1_a
\qquad\qquad
\tikzmath{\filldraw[\BColor, rounded corners=5, very thin, baseline=1cm] (0,0) rectangle (.6,.6);}=1_b
\qquad\qquad
\tikzmath{\filldraw[\CColor, rounded corners=5, very thin, baseline=1cm] (0,0) rectangle (.6,.6);}=1_c.
$$
Given Q-systems $P\in \cC(a\to a), Q\in \cC(b\to b), R\in \cC(c\to c)$, we represent these objects in $\QSys(\cC)$ by colored regions as follows: 
$$
\tikzmath{\filldraw[\PrColor, rounded corners=5, very thin, baseline=1cm] (0,0) rectangle (.6,.6);}=P
\qquad\qquad
\tikzmath{\filldraw[\QrColor, rounded corners=5, very thin, baseline=1cm] (0,0) rectangle (.6,.6);}=Q
\qquad\qquad
\tikzmath{\filldraw[\RrColor, rounded corners=5, very thin, baseline=1cm] (0,0) rectangle (.6,.6);}=R.
$$
Suppose ${}_PX_Q\in \cC(a\to b)$ is a unital Frobenius $P-Q$ bimodule.
Depending on whether we view $X$ as a 1-morphism in  $\cC(a\to b) = \QSys(\cC)(1_a\to 1_b)$, $\QSys(\cC)(1_a\to Q)$, $\QSys(\cC)(P\to 1_b)$, 
or $X\in \QSys(P\to Q)$, we represent it by a colored strand with the appropriate left and right shadings:
$$
\tikzmath{
\begin{scope}
\clip[rounded corners=5pt] (-.3,0) rectangle (.3,.6);
\fill[\AColor] (0,0) rectangle (-.3,.6);
\fill[\BColor] (0,0) rectangle (.3,.6);
\end{scope}
\draw[thick, \XColor] (0,0) -- (0,.6);
}
={}_{1_a}X_{1_b}
\qquad\qquad
\tikzmath{
\begin{scope}
\clip[rounded corners=5pt] (-.3,0) rectangle (.3,.6);
\fill[\AColor] (0,0) rectangle (-.3,.6);
\fill[\QrColor] (0,0) rectangle (.3,.6);
\end{scope}
\draw[thick, \XColor] (0,0) -- (0,.6);
}
={}_{1_a}X_Q
\qquad\qquad
\tikzmath{
\begin{scope}
\clip[rounded corners=5pt] (-.3,0) rectangle (.3,.6);
\fill[\PrColor] (0,0) rectangle (-.3,.6);
\fill[\BColor] (0,0) rectangle (.3,.6);
\end{scope}
\draw[thick, \XColor] (0,0) -- (0,.6);
}
={}_{P}X_{1_b}
\qquad\qquad
\tikzmath{
\begin{scope}
\clip[rounded corners=5pt] (-.3,0) rectangle (.3,.6);
\fill[\PrColor] (0,0) rectangle (-.3,.6);
\fill[\QrColor] (0,0) rectangle (.3,.6);
\end{scope}
\draw[thick, \XColor] (0,0) -- (0,.6);
}
={}_{P}X_Q.
$$
We denote the Q-system as a bimodule over itself by a string of similar, but slightly darker color.
Again, the shadings on the left and right designate over which Q-systems we are considering the bimodule.
For example:
\begin{equation}
\label{eq:ShadingQSystemRegions}
\tikzmath{
\begin{scope}
\clip[rounded corners=5pt] (-.3,0) rectangle (.3,.6);
\fill[\AColor] (0,0) rectangle (-.3,.6);
\fill[\AColor] (0,0) rectangle (.3,.6);
\end{scope}
\draw[thick, \PsColor] (0,0) -- (0,.6);
}
={}_{1_a}P_{1_a}
\qquad\qquad
\tikzmath{
\begin{scope}
\clip[rounded corners=5pt] (-.3,0) rectangle (.3,.6);
\fill[\BColor] (0,0) rectangle (-.3,.6);
\fill[\QrColor] (0,0) rectangle (.3,.6);
\end{scope}
\draw[thick, \QsColor] (0,0) -- (0,.6);
}={}_{1_b}Q_Q
\qquad\qquad
\tikzmath{
\begin{scope}
\clip[rounded corners=5pt] (-.3,0) rectangle (.3,.6);
\fill[\RrColor] (0,0) rectangle (-.3,.6);
\fill[\CColor] (0,0) rectangle (.3,.6);
\end{scope}
\draw[thick, \RsColor] (0,0) -- (0,.6);
}={}_{R}R_{1_c}
\qquad\qquad
\tikzmath{
\begin{scope}
\clip[rounded corners=5pt] (-.3,0) rectangle (.3,.6);
\fill[\PrColor] (0,0) rectangle (-.3,.6);
\fill[\PrColor] (0,0) rectangle (.3,.6);
\end{scope}
\draw[thick, \PsColor] (0,0) -- (0,.6);
}={}_{P}P_P
\end{equation}
As before, the 1-composite of two separable bimodules in $\QSys(\cC)$ (i.e., relative tensor product) is denoted by horizontal juxtaposition with the appropriate shading in between.
For example, given separable bimodules ${}_PX_Q \in \cC(a\to b)$ and ${}_QY_R\in \cC(b\to c)$, we denote their 1-composition by
$$
\tikzmath{
\begin{scope}
\clip[rounded corners=5pt] (-.3,0) rectangle (.6,.6);
\fill[\PrColor] (0,0) rectangle (-.3,.6);
\fill[\QrColor] (0,0) rectangle (.3,.6);
\fill[\RrColor] (.6,0) rectangle (.3,.6);
\end{scope}
\draw[thick, \XColor] (0,0) -- (0,.6);
\draw[thick, \YColor] (.3,0) -- (.3,.6);
}={}_P X\xzq_Q Y_R
\qquad\qquad
\text{where}
\qquad\qquad
\tikzmath{
\begin{scope}
\clip[rounded corners=5pt] (-.3,0) rectangle (.3,.6);
\fill[\PrColor] (0,0) rectangle (-.3,.6);
\fill[\QrColor] (0,0) rectangle (.3,.6);
\end{scope}
\draw[thick, \XColor] (0,0) -- (0,.6);
}={}_PX_Q
\qquad\qquad
\tikzmath{
\begin{scope}
\clip[rounded corners=5pt] (-.3,0) rectangle (.3,.6);
\fill[\QrColor] (0,0) rectangle (-.3,.6);
\fill[\RrColor] (0,0) rectangle (.3,.6);
\end{scope}
\draw[thick, \YColor] (0,0) -- (0,.6);
}={}_QY_R.
$$
The coisometry $u^Q_{X,Y}: X\xz_b Y \to X\xzq_Q Y$ from Definition \ref{defn:QSysTensorProduct} is $P-R$ bimodular and can be represented graphically by
\[
u^Q_{X,Y}
:=
\tikzmath{
\begin{scope}
\clip[rounded corners = 5pt] (-.5,-.5) rectangle (.5,.5);
\filldraw[\QrColor] (-.2,0) rectangle (.2,.5);
\filldraw[\BColor] (-.2,-.5) rectangle (.2,0);
\end{scope}
\draw[thick, \XColor] (-.2,-.5) -- (-.2,.5);
\draw[thick, \QsColor] (-.2,0) -- (.2,0);
\draw[thick, \YColor] (.2,-.5) -- (.2,.5);
}
: X\xz_b Y \to X\xzq_Q Y
\qquad\text{and}\qquad
(u^Q_{X,Y})^\dag
=
\tikzmath{
\begin{scope}
\clip[rounded corners = 5pt] (-.5,-.5) rectangle (.5,.5);
\filldraw[\QrColor] (-.2,0) rectangle (.2,-.5);
\filldraw[\BColor] (-.2,.5) rectangle (.2,0);
\end{scope}
\draw[thick, \XColor] (-.2,-.5) -- (-.2,.5);
\draw[thick, \QsColor] (-.2,0) -- (.2,0);
\draw[thick, \YColor] (.2,-.5) -- (.2,.5);
}\,.
\]
Here, we suppress the external shading, which can be taken to be any of the four combinations given above in \eqref{eq:ShadingQSystemRegions}.
We have the following relations for $u^Q_{X,Y}$ and its adjoint:
\[
u^Q_{X,Y}\circ (u^Q_{X,Y})^\dag 
=
\tikzmath{
\begin{scope}
\clip[rounded corners = 5pt] (-.5,-.5) rectangle (.5,1);
\filldraw[\QrColor] (-.2,0) rectangle (.2,-.5);
\filldraw[\BColor] (-.2,.5) rectangle (.2,0);
\filldraw[\QrColor] (-.2,1) rectangle (.2,.5);
\end{scope}
\draw[thick, \XColor] (-.2,-.5) -- (-.2,1);
\draw[thick, \QsColor] (-.2,0) -- (.2,0);
\draw[thick, \QsColor] (-.2,.5) -- (.2,.5);
\draw[thick, \YColor] (.2,-.5) -- (.2,1);
}
=
\tikzmath{
\begin{scope}
\clip[rounded corners = 5pt] (-.5,-.5) rectangle (.5,1);
\filldraw[\QrColor] (-.2,1) rectangle (.2,-.5);
\end{scope}
\draw[thick, \XColor] (-.2,-.5) -- (-.2,1);
\draw[thick, \YColor] (.2,-.5) -- (.2,1);
}
=
\id_{X\xzq_Q Y}
\qquad
(u^Q_{X,Y})^\dag \circ u^Q_{X,Y} 
=
\tikzmath{
\begin{scope}
\clip[rounded corners = 5pt] (-.5,-.5) rectangle (.5,1);
\filldraw[\BColor] (-.2,0) rectangle (.2,-.5);
\filldraw[\QrColor] (-.2,.5) rectangle (.2,0);
\filldraw[\BColor] (-.2,1) rectangle (.2,.5);
\end{scope}
\draw[thick, \XColor] (-.2,-.5) -- (-.2,1);
\draw[thick, \QsColor] (-.2,0) -- (.2,0);
\draw[thick, \QsColor] (-.2,.5) -- (.2,.5);
\draw[thick, \YColor] (.2,-.5) -- (.2,1);
}
=
\tikzmath{
\begin{scope}
\clip[rounded corners = 5pt] (-.5,-.5) rectangle (.5,1);
\filldraw[\BColor] (-.2,1) rectangle (.2,-.5);
\end{scope}
\draw[thick, \XColor] (-.2,-.5) -- (-.2,1);
\draw[thick, \QsColor] (-.2,.25) -- (.2,.25);
\draw[thick, \YColor] (.2,-.5) -- (.2,1);
}
=
p^Q_{X, Y}
\]
We define 
\[
\lambda^P_X=
\tikzmath{
\begin{scope}
\clip[rounded corners = 5pt] (-.7,-.2) rectangle (.3,.5);
\filldraw[\AColor] (-.7,-.2) rectangle (0,.5);
\filldraw[\PrColor,thick] (-.4,-.2) arc (180:90:.4cm) -- (0,-.2);
\filldraw[\BColor] (0,-.2) rectangle (.3,.5);
\end{scope}
\draw[\XColor,thick] (0,-.2) -- (0,.5);
\draw[\PsColor,thick] (-.4,-.2) arc (180:90:.4cm);
\filldraw[\XColor] (0,.2) circle (.05cm);
}
:=
\tikzmath{
\begin{scope}
\clip[rounded corners = 5pt] (-.7,-.6) rectangle (.3,.5);
\filldraw[\AColor] (-.7,-.6) rectangle (0,.5);
\filldraw[\PrColor,thick] (-.4,-.6) rectangle (0,-.2);
\filldraw[\BColor] (0,-.6) rectangle (.3,.5);
\end{scope}
\draw[\XColor,thick] (0,-.6) -- (0,.5);
\draw[\PsColor,thick] (-.4,-.6) -- (-.4,-.2) arc (180:90:.4cm);
\draw[\PsColor,thick] (-.4,-.2) -- (0,-.2);
\filldraw[\XColor] (0,.2) circle (.05cm);
}
=\lambda_X\circ (u^P_{P,X})^\dag
\qquad\text{and}\qquad
\rho_X^Q=
\tikzmath{
\begin{scope}
\clip[rounded corners = 5pt] (-.3,-.2) rectangle (.7,.5);
\filldraw[\AColor] (-.3,-.2) rectangle (0,.5);
\filldraw[\BColor] (0,-.2) rectangle (.7,.5);
\filldraw[\QrColor] (.4,-.2) arc (0:90:.4cm) -- (0,-.2);
\end{scope}
\draw[\XColor,thick] (0,-.2) -- (0,.5);
\draw[\QsColor,thick] (.4,-.2) arc (0:90:.4cm);
\filldraw[\XColor] (0,.2) circle (.05cm);
}
:=
\tikzmath{
\begin{scope}
\clip[rounded corners = 5pt] (-.3,-.6) rectangle (.7,.5);
\filldraw[\AColor] (-.3,-.6) rectangle (0,.5);
\filldraw[\BColor] (0,-.6) rectangle (.7,.5);
\filldraw[\QrColor] (.4,-.6) rectangle (0,-.2);
\end{scope}
\draw[\XColor,thick] (0,-.6) -- (0,.5);
\draw[\QsColor,thick] (.4,-.6) -- (.4,-.2) arc (0:90:.4cm);
\draw[\QsColor,thick] (0,-.2) -- (.4,-.2);
\filldraw[\XColor] (0,.2) circle (.05cm);
}
= \rho_X\circ (u_{X,Q}^Q)^\dag.
\]
Observe that 
$$
\tikzmath{
\begin{scope}
\clip[rounded corners = 5pt] (-.7,-.5) rectangle (.3,.9);
\filldraw[\AColor] (-.7,-.5) rectangle (0,.9);
\filldraw[\PrColor] (0,.6) arc (90:270:.4cm);
\filldraw[\BColor] (0,-.5) rectangle (.3,.9);
\end{scope}
\draw[\XColor,thick] (0,-.5) -- (0,.9);
\draw[\PsColor,thick] (0,.6) arc (90:270:.4cm);
\filldraw[\XColor] (0,-.2) circle (.05cm);
\filldraw[\XColor] (0,.6) circle (.05cm);
}
=
\tikzmath{
\begin{scope}
\clip[rounded corners = 5pt] (-.7,-.5) rectangle (.3,1.3);
\filldraw[\AColor] (-.7,-.5) rectangle (0,1.3);
\filldraw[\PrColor] (0,.6) rectangle (-.4,.2);
\filldraw[\BColor] (0,-.5) rectangle (.3,1.3);
\end{scope}
\draw[\XColor,thick] (0,-.5) -- (0,1.3);
\draw[\PsColor,thick] (0,1) arc (90:180:.4cm) -- (-.4,.2) arc (180:270:.4cm);
\draw[\PsColor,thick] (-.4,.2) -- (0,.2);
\draw[\PsColor,thick] (-.4,.6) -- (0,.6);
\filldraw[\XColor] (0,-.2) circle (.05cm);
\filldraw[\XColor] (0,1) circle (.05cm);
}
=
\tikzmath{
\begin{scope}
\clip[rounded corners = 5pt] (-.7,-.5) rectangle (.3,.9);
\filldraw[\AColor] (-.7,-.5) rectangle (0,.9);
\filldraw[\BColor] (0,-.5) rectangle (.3,.9);
\end{scope}
\draw[\XColor,thick] (0,-.5) -- (0,.9);
\draw[\PsColor,thick] (0,.6) arc (90:270:.4cm);
\draw[\PsColor,thick] (-.4,.2) -- (0,.2);
\filldraw[\XColor] (0,-.2) circle (.05cm);
\filldraw[\XColor] (0,.6) circle (.05cm);
}
=
\tikzmath{
\begin{scope}
\clip[rounded corners=5pt] (-.3,0) rectangle (.3,1.4);
\fill[\AColor] (0,0) rectangle (-.3,1.4);
\fill[\BColor] (0,0) rectangle (.3,1.4);
\end{scope}
\draw[thick, \XColor] (0,0) -- (0,1.4);
}
\qquad\text{and}\qquad
\tikzmath{
\begin{scope}
\clip[rounded corners = 5pt] (-.7,-.7) rectangle (.3,.7);
\filldraw[\AColor] (-.7,-.7) rectangle (0,.7);
\filldraw[\PrColor] (0,.3) arc (270:180:.4cm) -- (0,.7);
\filldraw[\PrColor] (0,-.3) arc (90:180:.4cm) -- (0,-.7);
\filldraw[\BColor] (0,-.7) rectangle (.3,.7);
\end{scope}
\draw[\XColor,thick] (0,-.7) -- (0,.7);
\draw[\PsColor,thick] (0,.3) arc (270:180:.4cm);
\draw[\PsColor,thick] (0,-.3) arc (90:180:.4cm);
\filldraw[\XColor] (0,-.3) circle (.05cm);
\filldraw[\XColor] (0,.3) circle (.05cm);
}
=
\tikzmath{
\begin{scope}
\clip[rounded corners = 5pt] (-.7,-.9) rectangle (.3,.9);
\filldraw[\AColor] (-.7,-.9) rectangle (0,.9);
\filldraw[\PrColor] (-.4,.6) rectangle (0,.9);
\filldraw[\PrColor] (-.4,-.6) rectangle (0,-.9);
\filldraw[\BColor] (0,-.9) rectangle (.3,.9);
\end{scope}
\draw[\XColor,thick] (0,-.9) -- (0,.9);
\draw[\PsColor,thick] (0,.2) arc (270:180:.4cm) -- (-.4,.9);
\draw[\PsColor,thick] (0,-.2) arc (90:180:.4cm) -- (-.4,-.9);
\draw[\PsColor,thick] (-.4,.6) -- (0,.6);
\draw[\PsColor,thick] (-.4,-.6) -- (0,-.6);
\filldraw[\XColor] (0,-.2) circle (.05cm);
\filldraw[\XColor] (0,.2) circle (.05cm);
}
=
\tikzmath{
\begin{scope}
\clip[rounded corners = 5pt] (-.7,-.7) rectangle (.3,.7);
\filldraw[\AColor] (-.7,-.7) rectangle (0,.7);
\filldraw[\PrColor] (0,-.7) rectangle (-.4,-.4);
\filldraw[\PrColor] (0,.7) rectangle (-.4,.4);
\filldraw[\BColor] (0,-.7) rectangle (.3,.7);
\end{scope}
\draw[\XColor,thick] (0,-.7) -- (0,.7);
\draw[\PsColor,thick] (-.4,-.7) -- (-.4,.7);
\draw[\PsColor,thick] (-.4,-.4) -- (0,-.4);
\draw[\PsColor,thick] (-.4,.4) -- (0,.4);
\draw[\PsColor,thick] (-.4,0) -- (0,0);
}
=
\tikzmath{
\begin{scope}
\clip[rounded corners = 5pt] (-.7,-.7) rectangle (.3,.7);
\filldraw[\AColor] (-.7,-.7) rectangle (-.4,.7);
\filldraw[\PrColor] (0,-.7) rectangle (-.4,.7);
\filldraw[\BColor] (0,-.7) rectangle (.3,.7);
\end{scope}
\draw[\XColor,thick] (0,-.7) -- (0,.7);
\draw[\PsColor,thick] (-.4,-.7) -- (-.4,.7);
}\,,
$$
and thus $\lambda_X^P$ is unitary.
Similarly, $\rho_X^Q$ is unitary.

The unitor $\lambda_X : P\xz_a X \to X$ may be represented as a $P-1_b$ bimodular map in $\QSys(\cC)$ by introducing a shading on the left.
Similarly, we may introduce a shading on the right to see $\rho_X$ as a $1_a-Q$ bimodular map.
$$
\lambda_X=\tikzmath{
\begin{scope}
\clip[rounded corners = 5pt] (-.7,-.2) rectangle (.3,.5);
\filldraw[\PrColor] (-.7,-.2) rectangle (0,.5);
\filldraw[\AColor,thick] (-.4,-.2) arc (180:90:.4cm) -- (0,-.2);
\filldraw[\BColor] (0,-.2) rectangle (.3,.5);
\end{scope}
\draw[\XColor,thick] (0,-.2) -- (0,.5);
\draw[\PsColor,thick] (-.4,-.2) arc (180:90:.4cm);
\filldraw[\XColor] (0,.2) circle (.05cm);
}
\qquad\qquad
\rho_X=\tikzmath{
\begin{scope}
\clip[rounded corners = 5pt] (-.3,-.2) rectangle (.7,.5);
\filldraw[\AColor] (-.3,-.2) rectangle (0,.5);
\filldraw[\QrColor] (0,-.2) rectangle (.7,.5);
\filldraw[\BColor] (.4,-.2) arc (0:90:.4cm) -- (0,-.2);
\end{scope}
\draw[\XColor,thick] (0,-.2) -- (0,.5);
\draw[\QsColor,thick] (.4,-.2) arc (0:90:.4cm);
\filldraw[\XColor] (0,.2) circle (.05cm);
}
$$
These maps are partial isometries as ${}_aX_b$ is unitarily separable:
\[
\tikzmath{
\begin{scope}
\clip[rounded corners = 5pt] (-.7,-.5) rectangle (.3,.9);
\filldraw[\PrColor] (-.7,-.5) rectangle (0,.9);
\filldraw[\AColor] (0,.6) arc (90:270:.4cm);
\filldraw[\BColor] (0,-.5) rectangle (.3,.9);
\end{scope}
\draw[\XColor,thick] (0,-.5) -- (0,.9);
\draw[\PsColor,thick] (0,.6) arc (90:270:.4cm);
\filldraw[\XColor] (0,-.2) circle (.05cm);
\filldraw[\XColor] (0,.6) circle (.05cm);
}
=
\tikzmath{
\begin{scope}
\clip[rounded corners=5pt] (-.3,0) rectangle (.3,1.4);
\fill[\PrColor] (0,0) rectangle (-.3,1.4);
\fill[\BColor] (0,0) rectangle (.3,1.4);
\end{scope}
\draw[thick, \XColor] (0,0) -- (0,1.4);
}
\qquad\text{and}\qquad
\tikzmath{
\begin{scope}
\clip[rounded corners=5pt] (-.3,0) rectangle (.3,1.4);
\fill[\AColor] (0,0) rectangle (-.3,1.4);
\fill[\QrColor] (0,0) rectangle (.3,1.4);
\end{scope}
\draw[thick, \XColor] (0,0) -- (0,1.4);
}
=
\tikzmath{
\begin{scope}
\clip[rounded corners = 5pt] (-.3,-.5) rectangle (.7,.9);
\filldraw[\AColor] (-.3,-.5) rectangle (0,.9);
\filldraw[\QrColor] (0,-.5) rectangle (.7,.9);
\filldraw[\BColor] (0,.6) arc (90:-90:.4cm);
\end{scope}
\draw[\XColor,thick] (0,-.5) -- (0,.9);
\draw[\QsColor,thick] (0,.6) arc (90:-90:.4cm);
\filldraw[\XColor] (0,-.2) circle (.05cm);
\filldraw[\XColor] (0,.6) circle (.05cm);
}
\]

\begin{facts} 
We have the following easily verified relations amongst the separability coisometries and trivalent vertices.
\label{Facts:QSysU}
\begin{enumerate}[label=(\arabic*)]
\item 
$
\tikzmath{
\begin{scope}
\clip[rounded corners = 5] (-.3,-.4) rectangle (.7,.8);
\filldraw[\AColor] (-.3,-.4) rectangle (0,.8);
\filldraw[\BColor] (0,-.4) rectangle (.7,.8);
\filldraw[\QrColor] (0,0) arc (-90:0:.4cm) -- (0,.4);
\end{scope}
\draw[thick,\XColor] (0,-.4) -- (0,.8);
\draw[thick,\QsColor] (0,0) arc (-90:0:.4cm) -- (.4,.8);
\draw[thick,\QsColor] (0,.4) -- (.4,.4);
\filldraw[\XColor] (0,0) circle (.05cm);
}
=
\tikzmath{
\begin{scope}
\clip[rounded corners = 5] (-.3,-.4) rectangle (.7,1.2);
\filldraw[\AColor] (-.3,-.4) rectangle (0,1.2);
\filldraw[\BColor] (0,-.4) rectangle (.7,1.2);
\filldraw[\QrColor] (0,.4) rectangle (.4,.8);
\end{scope}
\draw[thick,\XColor] (0,-.4) -- (0,1.2);
\draw[thick,\QsColor] (0,0) arc (-90:0:.4cm) -- (.4,1.2);
\draw[thick,\QsColor] (0,.4) -- (.4,.4);
\draw[thick,\QsColor] (0,.8) -- (.4,.8);
\filldraw[\XColor] (0,0) circle (.05cm);
}
=
\tikzmath{
\begin{scope}
\clip[rounded corners = 5] (-.3,-.4) rectangle (.7,.8);
\filldraw[\AColor] (-.3,-.4) rectangle (0,.8);
\filldraw[\BColor] (0,-.4) rectangle (.7,.8);
\end{scope}
\draw[thick,\XColor] (0,-.4) -- (0,.8);
\draw[thick,\QsColor] (0,0) arc (-90:0:.4cm) -- (.4,.8);
\draw[thick,\QsColor] (0,.4) -- (.4,.4);
\filldraw[\XColor] (0,0) circle (.05cm);
}
=
\tikzmath{
\begin{scope}
\clip[rounded corners = 5] (-.3,-.4) rectangle (.7,.4);
\filldraw[\AColor] (-.3,-.4) rectangle (0,.4);
\filldraw[\BColor] (0,-.4) rectangle (.7,.4);
\end{scope}
\draw[thick,\XColor] (0,-.4) -- (0,.4);
\draw[thick,\QsColor] (0,0) arc (-90:0:.4cm);
\filldraw[\XColor] (0,0) circle (.05cm);
}
\quad\text{and}\quad
\tikzmath{
\begin{scope}
\clip[rounded corners = 5] (-.7,-.4) rectangle (.3,.8);
\filldraw[\AColor] (-.7,-.4) rectangle (0,.8);
\filldraw[\BColor] (0,-.4) rectangle (.3,.8);
\filldraw[\PrColor] (0,0) arc (270:180:.4cm) -- (0,.4);
\end{scope}
\draw[thick,\XColor] (0,-.4) -- (0,.8);
\draw[thick,\PsColor] (0,0) arc (270:180:.4cm) -- (-.4,.8);
\draw[thick,\PsColor] (0,.4) -- (-.4,.4);
\filldraw[\XColor] (0,0) circle (.05cm);
}
=
\tikzmath{
\begin{scope}
\clip[rounded corners = 5] (-.7,-.4) rectangle (.3,.4);
\filldraw[\AColor] (-.7,-.4) rectangle (0,.4);
\filldraw[\BColor] (0,-.4) rectangle (.3,.4);
\end{scope}
\draw[thick,\XColor] (0,-.4) -- (0,.4);
\draw[thick,\PsColor] (0,0) arc (270:180:.4cm);
\filldraw[\XColor] (0,0) circle (.05cm);
}
$.
\item
Considering $Q$ as a bimodule over itself,
$
\tikzmath{
\fill[\BColor, rounded corners=5pt] (-.3,0) rectangle (.9,.6);
\filldraw[\QrColor] (0,0) arc (180:0:.3cm);
\draw[thick,\QsColor] (0,0) arc (180:0:.3cm);
\draw[thick,\QsColor] (.3,.3) -- (.3,.6);
\filldraw[\QsColor] (.3,.3) circle (.05cm);
}
$
is a unitary,
and
$
\tikzmath{
\begin{scope}
\clip[rounded corners = 5] (-.6,-.6) rectangle (.6,.3);
\filldraw[\BColor] (-.6,-.6) rectangle (.6,.3);
\filldraw[\QrColor] (-.3,0) arc (-180:0:.3cm);
\end{scope}
\draw[thick,\QsColor] (-.3,.3) -- (-.3,0) arc (-180:0:.3cm) -- (.3,.3);
\draw[thick,\QsColor] (0,-.6) -- (0,-.3);
\draw[thick,\QsColor] (-.3,0) -- (.3,0);
\filldraw[\QsColor] (0,-.3) circle (.05cm);
}
=
\tikzmath{
\fill[\BColor, rounded corners=5pt] (-.3,0) rectangle (.9,.6);
\draw[thick,\QsColor] (0,.6) arc (-180:0:.3cm);
\draw[thick,\QsColor] (.3,.3) -- (.3,0);
\filldraw[\QsColor] (.3,.3) circle (.05cm);
}
$\,.

\item 
$
u^Q_{X,Y}
=
\tikzmath{
\begin{scope}
\clip[rounded corners = 5pt] (-.5,-.5) rectangle (.5,.5);
\filldraw[\QrColor] (-.2,0) rectangle (.2,.5);
\filldraw[\BColor] (-.2,-.5) rectangle (.2,0);
\end{scope}
\draw[thick, \XColor] (-.2,-.5) -- (-.2,.5);
\draw[thick, \QsColor] (-.2,0) -- (.2,0);
\draw[thick, \YColor] (.2,-.5) -- (.2,.5);
}
=
\tikzmath{
\begin{scope}
\clip[rounded corners = 5pt] (-.7,-.7) rectangle (.7,.5);
\filldraw[\QrColor] (-.4,.2) arc (90:0:.2cm) arc (-180:0:.2cm) arc (180:90:.2cm) -- (.4,.5) -- (-.4,.5);
\filldraw[\BColor] (-.4,.2) arc (90:0:.2cm) arc (-180:0:.2cm) arc (180:90:.2cm) -- (.4,-.7) -- (-.4,-.7);
\end{scope}
\draw[thick,\QsColor] (-.4,.2) arc (90:0:.2cm) arc (-180:0:.2cm) arc (180:90:.2cm);
\draw[thick,\QsColor] (0,-.4) -- (0,-.2);
\draw[thick, \XColor] (-.4,-.7) -- (-.4,.5);
\draw[thick, \YColor] (.4,-.7) -- (.4,.5);
\filldraw[\XColor] (-.4,.2) circle (.05cm);
\filldraw[\YColor] (.4,.2) circle (.05cm);
\filldraw[thick,\QsColor] (0,-.2) circle (.05cm);
\filldraw[thick,\QsColor] (0,-.4) circle (.05cm);
}
=
\tikzmath{
\begin{scope}
\clip[rounded corners = 5pt] (-.6,-.6) rectangle (.6,.6);
\filldraw[\QrColor] (-.3,-.3) arc (-90:0:.3cm) arc (180:90:.3cm) -- (.3,.6) -- (-.3,.6);
\filldraw[\BColor] (-.3,-.3) arc (-90:0:.3cm) arc (180:90:.3cm) -- (.3,-.6) -- (-.3,-.6);
\end{scope}
\draw[thick,\QsColor] (-.3,-.3) arc (-90:0:.3cm) arc (180:90:.3cm);
\draw[thick, \XColor] (-.3,-.6) -- (-.3,.6);
\draw[thick, \YColor] (.3,-.6) -- (.3,.6);
\filldraw[\XColor] (-.3,-.3) circle (.05cm);
\filldraw[\YColor] (.3,.3) circle (.05cm);
}
$
\item 
For a bimodule map 
$f:{}_PX_Q\to {}_PZ_Q$,
$
\tikzmath{
\begin{scope}
\clip[rounded corners=5pt] (-.8,.5) rectangle (.5,-1.5);
\fill[\AColor] (-.8,.5) rectangle (0,-1.5);
\fill[\BColor] (0,.5) rectangle (.5,-1.5);
\filldraw[\PrColor] (-.5,-1.5) arc (180:90:.5cm) -- (0,-1.5);
\end{scope}
\draw[\ZColor,thick] (0,.5) -- (0,-.25);
\draw[\XColor,thick] (0,-.25) -- (0,-1.5);
\draw[\PsColor,thick] (-.5,-1.5) arc (180:90:.5cm);
\filldraw[\XColor] (0,-1) circle (.05cm);
\roundNbox{unshaded}{(0,-.25)}{.3}{0}{0}{$f$};
}
=
\tikzmath{
\begin{scope}
\clip[rounded corners=5pt] (-.8,-.5) rectangle (.5,-2.5);
\fill[\AColor] (-.8,-.5) rectangle (0,-2.5);
\fill[\BColor] (0,-.5) rectangle (.5,-2.5);
\filldraw[\PrColor] (-.5,-2.5) -- (-.5,-1.5) arc (180:90:.5cm) -- (0,-2.5);
\end{scope}
\draw[\XColor,thick] (0,-1.75) -- (0,-2.5);
\draw[\ZColor,thick] (0,-0.5) -- (0,-1.75);
\draw[\PsColor,thick] (-.5,-2.5) -- (-.5,-1.5) arc (180:90:.5cm);
\filldraw[\ZColor] (0,-1) circle (.05cm);
\roundNbox{unshaded}{(0,-1.75)}{.3}{0}{0}{$f$};
}$
and
$
\tikzmath{
\begin{scope}
\clip[rounded corners=5pt] (.8,.5) rectangle (-.5,-1.5);
\fill[\QrColor] (.8,.5) rectangle (0,-1.5);
\fill[\AColor] (0,.5) rectangle (-.5,-1.5);
\fill[\BColor] (.5,-1.5) arc (0:90:.5cm) -- (0,-1.5);
\end{scope}
\draw[\ZColor,thick] (0,.5) -- (0,-.25);
\draw[\XColor,thick] (0,-.25) -- (0,-1.5);
\draw[\QsColor,thick] (.5,-1.5) arc (0:90:.5cm);
\filldraw[\XColor] (0,-1) circle (.05cm);
\roundNbox{unshaded}{(0,-.25)}{.3}{0}{0}{$f$};
}
=
\tikzmath{
\begin{scope}
\clip[rounded corners=5pt] (.8,-.5) rectangle (-.5,-2.5);
\fill[\QrColor] (.8,-.5) rectangle (0,-2.5);
\fill[\AColor] (0,-.5) rectangle (-.5,-2.5);
\fill[\BColor] (.5,-2.5) -- (.5,-1.5) arc (0:90:.5cm) -- (0,-2.5);
\end{scope}
\draw[\XColor,thick] (0,-1.75) -- (0,-2.5);
\draw[\ZColor,thick] (0,-0.5) -- (0,-1.75);
\draw[\QsColor,thick] (.5,-2.5) -- (.5,-1.5) arc (0:90:.5cm);
\filldraw[\ZColor] (0,-1) circle (.05cm);
\roundNbox{unshaded}{(0,-1.75)}{.3}{0}{0}{$f$};
}\,.
$
\end{enumerate}
\end{facts}

\begin{lem}
\label{Lem:NaturalSplit}
\mbox{}
\begin{enumerate}[label=(\arabic*)]
    \item Suppose $Q\in \cC(b\to b)$ is a Q-system, then ${}_bQ_b\cong {}_bQ\xzq_QQ_b$ as Q-systems in $\QSys(\cC)$.
    \item 
    Suppose $(R,m^R_b,i^R_b)\in \cC(b\to b)$ is a Q-system in $\cC$ 
    and $(Q,m^Q_R,i^Q_R)\in \QSys(\cC)(R\to R)$ is a Q-system in $\QSys(\cC)$.
    Then $Q\in \cC(b\to b)$ is a Q-system in $\cC$ with multiplication and unit given by
\[
m^Q_b :=
\tikzmath{
\begin{scope}
\clip[rounded corners = 5] (-.6,-.4) rectangle (.6,1);
\filldraw[\BColor] (-.6,-.4) rectangle (.6,1);
\filldraw[\RrColor] (-.3,0) to [bend left=50] (0,.6) to [bend left=50] (.3,0);
\end{scope}
\draw[\RsColor,thick] (-.3,0) to [bend left=50] (0,.6) to [bend left=50] (.3,0) -- (-.3,0);
\draw[\QsColor,thick] (-.3,-.4) -- (-.3,0) arc (180:0:.3cm) -- (.3,-.4);
\draw[\QsColor,thick] (0,.3) -- (0,1);
\filldraw[\QsColor,thick] (0,.3) circle (.05cm);
\filldraw[\QsColor,thick] (0,.6) circle (.05cm);
\filldraw[\QsColor,thick] (-.3,0) circle (.05cm);
\filldraw[\QsColor,thick] (.3,0) circle (.05cm);
}
=
\tikzmath{
\begin{scope}
\clip[rounded corners = 5] (-.6,-.4) rectangle (.6,1);
\filldraw[\BColor] (-.6,-.4) rectangle (.6,1);
\filldraw[\RrColor] (-.3,0) to [bend left=50] (0,.6) to [bend left=50] (.3,0);
\end{scope}
\draw[\RsColor,thick] (-.3,0) to [bend left=50] (0,.6) to [bend left=50] (.3,0) -- (-.3,0);
\draw[\QsColor,thick] (-.3,-.4) -- (-.3,0) arc (180:0:.3cm) -- (.3,-.4);
\draw[\QsColor,thick] (0,.3) -- (0,1);
\filldraw[\QsColor,thick] (0,.3) circle (.05cm);
}
\qquad
i^Q_b := 
\tikzmath{
\begin{scope}
\clip[rounded corners = 5] (-.6,-.9) rectangle (.6,.7);
\filldraw[\BColor] (-.6,-.9) rectangle (.6,.7);
\end{scope}
\filldraw[\RrColor] (0,0) circle (.3cm);
\draw[\RsColor,thick] (0,0) circle (.3cm);
\draw[thick,\QsColor] (0,0) -- (0,.7);
\draw[thick,\RsColor] (0,-.3) -- (0,-.6);
\filldraw[\QsColor] (0,0) circle (.05cm);
\filldraw[\QsColor] (0,.3) circle (.05cm);
\filldraw[\RsColor] (0,-.6) circle (.05cm);
\filldraw[\RsColor] (0,-.3) circle (.05cm);
}
=
\tikzmath{
\begin{scope}
\clip[rounded corners = 5] (-.6,-.6) rectangle (.6,.7);
\filldraw[\BColor] (-.6,-.6) rectangle (.6,.7);
\end{scope}
\filldraw[\RrColor] (0,0) circle (.3cm);
\draw[\RsColor,thick] (0,0) circle (.3cm);
\draw[thick,\QsColor] (0,0) -- (0,.7);
\filldraw[\QsColor] (0,0) circle (.05cm);
}
\]
Here, we use the same notation for $b,Q,R$ as in the beginning of this section, and the 4-valent vertex is defined as
\[
\tikzmath{
\fill[\BColor, rounded corners=5pt] (-.3,0) rectangle (.9,.6);
\filldraw[\RrColor] (0,0) arc (180:0:.3cm);
\draw[thick,\RsColor] (0,0) arc (180:0:.3cm);
\draw[thick,\QsColor] (.3,0) -- (.3,.6);
}
:=
\tikzmath{
\fill[\BColor, rounded corners=5pt] (-.3,0) rectangle (.9,.6);
\filldraw[\RrColor] (0,0) arc (180:0:.3cm);
\draw[thick,\RsColor] (0,0) arc (180:0:.3cm);
\draw[thick,\QsColor] (.3,0) -- (.3,.6);
\filldraw[\QsColor] (.3,.3) circle (.05cm);
}
:=
\tikzmath{
\begin{scope}
\clip[rounded corners = 5pt] (-.7,-.5) rectangle (.7,.5);
\filldraw[\BColor] (-.7,-.5) rectangle (.7,.5);
\filldraw[\RrColor] (-.4,-.5) -- (-.4,-.2) arc (180:90:.4cm) -- (0,-.5);
\filldraw[\RrColor] (.4,-.5) arc (0:90:.4cm) -- (0,-.5);
\end{scope}
\draw[\QsColor,thick] (0,-.5) -- (0,.5);
\draw[thick,\RsColor] (-.4,-.5) -- (-.4,-.2) arc (180:90:.4cm);
\draw[thick,\RsColor] (.4,-.5) arc (0:90:.4cm);
\filldraw[\QsColor] (0,.2) circle (.05cm);
\filldraw[\QsColor] (0,-.1) circle (.05cm);
}
=
\tikzmath{
\begin{scope}
\clip[rounded corners = 5pt] (-.7,-.5) rectangle (.7,.5);
\filldraw[\BColor] (-.7,-.5) rectangle (.7,.5);
\filldraw[\RrColor] (-.4,-.5) arc (180:90:.4cm) -- (0,-.5);
\filldraw[\RrColor] (.4,-.5) -- (.4,-.2) arc (0:90:.4cm) -- (0,-.5);
\end{scope}
\draw[\QsColor,thick] (0,-.5) -- (0,.5);
\draw[thick,\RsColor] (-.4,-.5) arc (180:90:.4cm);
\draw[thick,\RsColor] (.4,-.5) -- (.4,-.2) arc (0:90:.4cm);
\filldraw[\QsColor] (0,.2) circle (.05cm);
\filldraw[\QsColor] (0,-.1) circle (.05cm);
}\,.
\]
\end{enumerate}

\end{lem}
\begin{proof}
\item[(1)]
Note that 
\[
\tikzmath{
\fill[\BColor, rounded corners=5pt] (-.3,0) rectangle (.9,.6);
\filldraw[\QrColor] (0,.6) arc (-180:0:.3cm);
\draw[thick,\QsColor] (0,.6) arc (-180:0:.3cm);
\draw[thick,\QsColor] (.3,.3) -- (.3,0);
\filldraw[\QsColor] (.3,.3) circle (.05cm);
}
=
\tikzmath{
\begin{scope}
\clip[rounded corners = 5] (-.6,-.6) rectangle (.6,.3);
\filldraw[\BColor] (-.6,-.6) rectangle (.6,.3);
\filldraw[\QrColor] (-.3,.3) rectangle (.3,0);
\end{scope}
\draw[thick,\QsColor] (-.3,.3) -- (-.3,0) arc (-180:0:.3cm) -- (.3,.3);
\draw[thick,\QsColor] (0,-.6) -- (0,-.3);
\draw[thick,\QsColor] (-.3,0) -- (.3,0);
\filldraw[\QsColor] (0,-.3) circle (.05cm);
}
: {}_{1_b}Q_{1_b} \to {}_{1_b}Q\xzq_Q Q_{1_b}
\]
is a unitary from Facts \ref{Facts:QSysU}(2).
Moreover,
this unitary intertwines the Q-system structures as
\[
\tikzmath{
\begin{scope}
\clip[rounded corners = 5] (-.3,-.4) rectangle (1.5,1.2);
\filldraw[\BColor] (-.3,-.4) rectangle (1.5,1.2);
\filldraw[\QrColor] (0,0) -- (.4,0) arc (180:0:.2cm) -- (1.2,0)  .. controls ++(90:.2cm) and ++(270:.2cm) .. (.8,.8) -- (.4,.8) .. controls ++(270:.2cm) and ++(90:.2cm) .. (0,0);
\end{scope}
\draw[thick,\QsColor] (0,0) arc (-180:0:.2cm) arc (180:0:.2cm) arc (-180:0:.2cm) .. controls ++(90:.2cm) and ++(270:.2cm) .. (.8,.8) arc (0:180:.2cm) .. controls ++(270:.2cm) and ++(90:.2cm) .. (0,0);
\draw[thick,\QsColor] (.2,-.4) -- (.2,-.2);
\draw[thick,\QsColor] (1,-.4) -- (1,-.2);
\draw[thick,\QsColor] (.6,1) -- (.6,1.2);
%
\draw[thick,\QsColor] (.4,.8) -- (.8,.8);
\draw[thick,\QsColor] (0,0) -- (.4,0);
\draw[thick,\QsColor] (.8,0) -- (1.2,0);
\filldraw[\QsColor] (.2,-.2) circle (.05cm);
\filldraw[\QsColor] (1,-.2) circle (.05cm);
\filldraw[\QsColor] (.6,1) circle (.05cm);
}
=
\tikzmath{
\fill[\BColor, rounded corners=5pt] (-.3,-.4) rectangle (1.5,1.2);
\draw[thick,\QsColor] (0,0) arc (-180:0:.2cm) arc (180:0:.2cm) arc (-180:0:.2cm) .. controls ++(90:.2cm) and ++(270:.2cm) .. (.8,.8) arc (0:180:.2cm) .. controls ++(270:.2cm) and ++(90:.2cm) .. (0,0);
\draw[thick,\QsColor] (.2,-.4) -- (.2,-.2);
\draw[thick,\QsColor] (1,-.4) -- (1,-.2);
\draw[thick,\QsColor] (.6,1) -- (.6,1.2);
\draw[thick,\QsColor] (.6,.2) -- (.6,.4);
\draw[thick,\QsColor] (.4,.8) -- (.8,.8);
\filldraw[\QsColor] (.2,-.2) circle (.05cm);
\filldraw[\QsColor] (1,-.2) circle (.05cm);
\filldraw[\QsColor] (.6,.2) circle (.05cm);
\filldraw[\QsColor] (.6,1) circle (.05cm);
\filldraw[\QsColor] (.6,.4) circle (.05cm);
}
=
\tikzmath{
\fill[\BColor, rounded corners=5pt] (-.3,0) rectangle (.9,.6);
\draw[thick,\QsColor] (0,0) arc (180:0:.3cm);
\draw[thick,\QsColor] (.3,.3) -- (.3,.6);
\filldraw[\QsColor] (.3,.3) circle (.05cm);
}
\qquad\text{and}\qquad
\tikzmath{
\fill[\BColor, rounded corners=5pt] (-.6,-.9) rectangle (.6,.6);
\filldraw[\QrColor] (-.3,0) arc (-180:0:.3cm);
\draw[thick,\QsColor] (0,0) circle (.3cm);
\draw[thick,\QsColor] (-.3,0) -- (.3,0);
\draw[thick,\QsColor] (0,-.6) -- (0,-.3);
\draw[thick,\QsColor] (0,.3) -- (0,.6);
\filldraw[\QsColor] (0,-.3) circle (.05cm);
\filldraw[\QsColor] (0,.3) circle (.05cm);
\filldraw[\QsColor] (0,-.6) circle (.05cm);
}
=
\tikzmath{
\fill[\BColor, rounded corners=5pt] (-.6,-.9) rectangle (.6,.6);
\draw[thick,\QsColor] (0,0) circle (.3cm);
\draw[thick,\QsColor] (0,-.6) -- (0,-.3);
\draw[thick,\QsColor] (0,.3) -- (0,.6);
\filldraw[\QsColor] (0,-.3) circle (.05cm);
\filldraw[\QsColor] (0,.3) circle (.05cm);
\filldraw[\QsColor] (0,-.6) circle (.05cm);
}
=
\tikzmath{
\fill[\BColor, rounded corners=5pt] (0,0) rectangle (.6,.6);
\draw[thick,\QsColor] (.3,.3) -- (.3,.6);
\filldraw[\QsColor] (.3,.3) circle (.05cm);
}\,.
\]
\item[(2)] 
We verify that $(Q,m^Q_b,i^Q_b)\in\cC(b\to b)$ is a Q-system by using Facts \ref{Facts:QSysU}. 
\item[\underline{\ref{Q:associativity}:}]
$
\tikzmath{
\begin{scope}
\clip[rounded corners = 5] (-.6,-.4) rectangle (.9,2);
\filldraw[\BColor] (-.6,-.4) rectangle (.9,2);
\filldraw[\RrColor] (-.3,0) to [bend left=50] (0,.6) to [bend left=50] (.3,0);
\filldraw[\RrColor] (0,1) to [bend left=50] (.3,1.6) to [bend left=50] (.6,1);
\end{scope}
\draw[\RsColor,thick] (-.3,0) to [bend left=50] (0,.6) to [bend left=50] (.3,0) -- (-.3,0);
\draw[\RsColor,thick] (0,1) to [bend left=50] (.3,1.6) to [bend left=50] (.6,1) -- (0,1);
\draw[\QsColor,thick] (-.3,-.4) -- (-.3,0) arc (180:0:.3cm) -- (.3,-.4);
\draw[\QsColor,thick] (0,.3) -- (0,1) arc (180:0:.3cm) -- (.6,-.4);
\draw[\QsColor,thick] (.3,1.3) -- (.3,2);
\filldraw[\QsColor,thick] (0,.3) circle (.05cm);
\filldraw[\QsColor,thick] (.3,1.3) circle (.05cm);
}
=
\tikzmath{
\begin{scope}
\clip[rounded corners = 5] (-.6,-.4) rectangle (1,2);
\filldraw[\BColor] (-.6,-.4) rectangle (1,2);
\filldraw[\RrColor] (-.3,0) to [bend left=50] (.3,1.6) to [bend left=50] (.6,0);
\end{scope}
\draw[\RsColor,thick] (-.3,0) to [bend left=50] (.3,1.6) to [bend left=50] (.6,0) -- (-.3,0);
\draw[\QsColor,thick] (-.3,-.4) -- (-.3,0) arc (180:0:.3cm) -- (.3,-.4);
\draw[\QsColor,thick] (0,.3) -- (0,1) arc (180:0:.3cm) -- (.6,-.4);
\draw[\QsColor,thick] (.3,1.3) -- (.3,2);
\filldraw[\QsColor,thick] (0,.3) circle (.05cm);
\filldraw[\QsColor,thick] (.3,1.3) circle (.05cm);
}
=
\tikzmath{
\begin{scope}
\clip[rounded corners = 5] (-.7,-.4) rectangle (.9,2);
\filldraw[\BColor] (-.7,-.4) rectangle (.9,2);
\filldraw[\RrColor] (-.3,0) to [bend left=50] (0,1.6) to [bend left=50] (.6,0);
\end{scope}
\draw[\RsColor,thick] (-.3,0) to [bend left=50] (0,1.6) to [bend left=50] (.6,0) -- (-.3,0);
\draw[\QsColor,thick] (0,-.4) -- (0,0) arc (180:0:.3cm) -- (.6,-.4);
\draw[\QsColor,thick] (.3,.3) -- (.3,1) arc (0:180:.3cm) -- (-.3,-.4);
\draw[\QsColor,thick] (0,1.3) -- (0,2);
\filldraw[\QsColor,thick] (.3,.3) circle (.05cm);
\filldraw[\QsColor,thick] (0,1.3) circle (.05cm);
}
=
\tikzmath{
\begin{scope}
\clip[rounded corners = 5] (-.6,-.4) rectangle (.9,2);
\filldraw[\BColor] (-.6,-.4) rectangle (.9,2);
\filldraw[\RrColor] (0,0) to [bend left=50] (.3,.6) to [bend left=50] (.6,0);
\filldraw[\RrColor] (-.3,1) to [bend left=50] (0,1.6) to [bend left=50] (.3,1);
\end{scope}
\draw[\RsColor,thick] (0,0) to [bend left=50] (.3,.6) to [bend left=50] (.6,0) -- (0,0);
\draw[\RsColor,thick] (-.3,1) to [bend left=50] (0,1.6) to [bend left=50] (.3,1) -- (-.3,1);
\draw[\QsColor,thick] (0,-.4) -- (0,0) arc (180:0:.3cm) -- (.6,-.4);
\draw[\QsColor,thick] (.3,.3) -- (.3,1) arc (0:180:.3cm) -- (-.3,-.4);
\draw[\QsColor,thick] (0,1.3) -- (0,2);
\filldraw[\QsColor,thick] (.3,.3) circle (.05cm);
\filldraw[\QsColor,thick] (0,1.3) circle (.05cm);
}
$
\item[\underline{\ref{Q:unitality}:}]
$\tikzmath{
\begin{scope}
\clip[rounded corners = 5] (-.5,-.4) rectangle (.9,2);
\filldraw[\BColor] (-.5,-.4) rectangle (.9,2);
\filldraw[\RrColor] (0,.3) circle (.3cm);
\filldraw[\RrColor] (0,1) to [bend left=50] (.3,1.6) to [bend left=50] (.6,1);
\end{scope}
\draw[\RsColor,thick] (0,.3) circle (.3cm);
\draw[\RsColor,thick] (0,1) to [bend left=50] (.3,1.6) to [bend left=50] (.6,1) -- (0,1);
\draw[\QsColor,thick] (0,.3) -- (0,1) arc (180:0:.3cm) -- (.6,-.4);
\draw[\QsColor,thick] (.3,1.3) -- (.3,2);
\filldraw[\QsColor,thick] (0,.3) circle (.05cm);
\filldraw[\QsColor,thick] (.3,1.3) circle (.05cm);
}
=
\tikzmath{
\begin{scope}
\clip[rounded corners = 5] (-.4,-.4) rectangle (1,2);
\filldraw[\BColor] (-.4,-.4) rectangle (1,2);
\filldraw[\RrColor] (0,0) to [bend left=50] (.3,1.6) to [bend left=50] (.6,0);
\end{scope}
\draw[\RsColor,thick] (0,0) to [bend left=50] (.3,1.6) to [bend left=50] (.6,0) -- (0,0);
\draw[\QsColor,thick] (0,.3) -- (0,1) arc (180:0:.3cm) -- (.6,-.4);
\draw[\QsColor,thick] (.3,1.3) -- (.3,2);
\filldraw[\QsColor,thick] (0,.3) circle (.05cm);
\filldraw[\QsColor,thick] (.3,1.3) circle (.05cm);
}
=
\tikzmath{
\begin{scope}
\clip[rounded corners = 5] (-.5,-.7) rectangle (.5,.7);
\filldraw[\BColor] (-.5,-.7) rectangle (.5,.7);
\filldraw[\RrColor] (0,0) circle (.3cm);
\end{scope}
\draw[\RsColor,thick] (0,0) circle (.3cm);
\draw[\QsColor,thick] (0,-.7) -- (0,.7);
}
=
\tikzmath{
\begin{scope}
\clip[rounded corners = 5] (-.3,-.7) rectangle (.3,.7);
\filldraw[\BColor] (-.3,-.7) rectangle (.3,.7);
\end{scope}
\draw[\QsColor,thick] (0,-.7) -- (0,.7);
}
=\cdots=
\tikzmath{
\begin{scope}
\clip[rounded corners = 5] (-.3,-.4) rectangle (1.1,2);
\filldraw[\BColor] (-.3,-.4) rectangle (1.1,2);
\filldraw[\RrColor] (.6,.3) circle (.3cm);
\filldraw[\RrColor] (0,1) to [bend left=50] (.3,1.6) to [bend left=50] (.6,1);
\end{scope}
\draw[\RsColor,thick] (.6,.3) circle (.3cm);
\draw[\RsColor,thick] (0,1) to [bend left=50] (.3,1.6) to [bend left=50] (.6,1) -- (0,1);
\draw[\QsColor,thick] (0,-.4) -- (0,1) arc (180:0:.3cm) -- (.6,.3);
\draw[\QsColor,thick] (.3,1.3) -- (.3,2);
\filldraw[\QsColor,thick] (.6,.3) circle (.05cm);
\filldraw[\QsColor,thick] (.3,1.3) circle (.05cm);
}
$
\item[\underline{\ref{Q:Frobenius}:}]
$
\tikzmath{
\begin{scope}
\clip[rounded corners = 5] (-.9,-1.2) rectangle (.9,1.2);
\filldraw[\BColor] (-.9,-1.2) rectangle (.9,1.2);
\filldraw[\RrColor] (-.6,.2) to [bend left=50] (-.3,.8) to [bend left=50] (0,.2);
\filldraw[\RrColor] (.6,-.2) to [bend left=50] (.3,-.8) to [bend left=50] (0,-.2);
\end{scope}
\draw[\RsColor,thick] (-.6,.2) to [bend left=50] (-.3,.8) to [bend left=50] (0,.2) -- (-.6,.2);
\draw[\RsColor,thick] (.6,-.2) to [bend left=50] (.3,-.8) to [bend left=50] (0,-.2) -- (.6,-.2);
\draw[\QsColor,thick] (-.6,-1.2) -- (-.6,.2) arc (180:0:.3cm) -- (0,-.2) arc (-180:0:.3cm) -- (.6,1.2);
\draw[\QsColor,thick] (.3,-1.2) -- (.3,-.5);
\draw[\QsColor,thick] (-.3,1.2) -- (-.3,.5);
\filldraw[\QsColor,thick] (.3,-.5) circle (.05cm);
\filldraw[\QsColor,thick] (-.3,.5) circle (.05cm);
}
=
\tikzmath{
\begin{scope}
\clip[rounded corners = 5] (-1,-1.2) rectangle (1,1.2);
\filldraw[\BColor] (-1,-1.2) rectangle (1,1.2);
\filldraw[\RrColor] (-.6,-.8) to [bend left=50] (-.3,.8) -- (.6,.8) to [bend left=50] (.3,-.8);
\end{scope}
\draw[\RsColor,thick] (-.6,-.8) to [bend left=50] (-.3,.8) -- (.6,.8) to [bend left=50] (.3,-.8) -- (-.6,-.8);
\draw[\QsColor,thick] (-.6,-1.2) -- (-.6,.2) arc (180:0:.3cm) -- (0,-.2) arc (-180:0:.3cm) -- (.6,1.2);
\draw[\QsColor,thick] (.3,-1.2) -- (.3,-.5);
\draw[\QsColor,thick] (-.3,1.2) -- (-.3,.5);
\filldraw[\QsColor,thick] (.3,-.5) circle (.05cm);
\filldraw[\QsColor,thick] (-.3,.5) circle (.05cm);
}
=
\tikzmath{
\begin{scope}
\clip[rounded corners = 5] (-.9,-1.2) rectangle (.9,1.2);
\filldraw[\BColor] (-.9,-1.2) rectangle (.9,1.2);
\filldraw[\RrColor] (-.3,-.8) to [bend left=50] (-.3,.8) -- (.3,.8) to [bend left=50] (.3,-.8);
\end{scope}
\draw[\RsColor,thick] (-.3,-.8) to [bend left=50] (-.3,.8) -- (.3,.8) to [bend left=50] (.3,-.8) -- (-.3,-.8);
\draw[\QsColor,thick] (-.3,-1.2) -- (-.3,-.8) arc (180:0:.3cm) -- (.3,-1.2);
\draw[\QsColor,thick] (-.3,1.2) -- (-.3,.8) arc (-180:0:.3cm) -- (.3,1.2);
\draw[\QsColor,thick] (0,-.5) -- (0,.5);
\filldraw[\QsColor,thick] (0,-.5) circle (.05cm);
\filldraw[\QsColor,thick] (0,.5) circle (.05cm);
}
=
\tikzmath{
\begin{scope}
\clip[rounded corners = 5] (-.6,-1.2) rectangle (.6,1.2);
\filldraw[\BColor] (-.6,-1.2) rectangle (.6,1.2);
\filldraw[\RrColor] (-.3,-.8) to [bend left=50] (0,-.2) to [bend left=50] (.3,-.8);
\filldraw[\RrColor] (.3,.8) to [bend left=50] (0,.2) to [bend left=50] (-.3,.8);
\end{scope}
\draw[\RsColor,thick] (-.3,-.8) to [bend left=50] (0,-.2) to [bend left=50] (.3,-.8) -- (-.3,-.8);
\draw[\RsColor,thick] (.3,.8) to [bend left=50] (0,.2) to [bend left=50] (-.3,.8) -- (.3,.8);
\draw[\QsColor,thick] (-.3,-1.2) -- (-.3,-.8) arc (180:0:.3cm) -- (.3,-1.2);
\draw[\QsColor,thick] (-.3,1.2) -- (-.3,.8) arc (-180:0:.3cm) -- (.3,1.2);
\draw[\QsColor,thick] (0,-.5) -- (0,.5);
\filldraw[\QsColor,thick] (0,-.5) circle (.05cm);
\filldraw[\QsColor,thick] (0,.5) circle (.05cm);
}
=\cdots=
\tikzmath{
\begin{scope}
\clip[rounded corners = 5] (-.9,-1.2) rectangle (.9,1.2);
\filldraw[\BColor] (-.9,-1.2) rectangle (.9,1.2);
\filldraw[\RrColor] (0,.2) to [bend left=50] (.3,.8) to [bend left=50] (.6,.2);
\filldraw[\RrColor] (0,-.2) to [bend left=50] (-.3,-.8) to [bend left=50] (-.6,-.2);
\end{scope}
\draw[\RsColor,thick] (0,.2) to [bend left=50] (.3,.8) to [bend left=50] (.6,.2) -- (0,.2);
\draw[\RsColor,thick] (0,-.2) to [bend left=50] (-.3,-.8) to [bend left=50] (-.6,-.2) -- (0,-.2);
\draw[\QsColor,thick] (-.6,1.2) -- (-.6,-.2) arc (-180:0:.3cm) -- (0,.2) arc (180:0:.3cm) -- (.6,-1.2);
\draw[\QsColor,thick] (.3,1.2) -- (.3,.5);
\draw[\QsColor,thick] (-.3,-1.2) -- (-.3,-.5);
\filldraw[\QsColor,thick] (.3,.5) circle (.05cm);
\filldraw[\QsColor,thick] (-.3,-.5) circle (.05cm);
}
$
\item[\underline{\ref{Q:separable}:}]
$
\tikzmath{
\begin{scope}
\clip[rounded corners = 5] (-.6,-1.2) rectangle (.6,1.2);
\filldraw[\BColor] (-.6,-1.2) rectangle (.6,1.2);
\filldraw[\RrColor] (-.3,.2) to [bend left=50] (0,.8) to [bend left=50] (.3,.2);
\filldraw[\RrColor] (.3,-.2) to [bend left=50] (0,-.8) to [bend left=50] (-.3,-.2);
\end{scope}
\draw[\RsColor,thick] (-.3,.2) to [bend left=50] (0,.8) to [bend left=50] (.3,.2) -- (-.3,.2);
\draw[\RsColor,thick] (.3,-.2) to [bend left=50] (0,-.8) to [bend left=50] (-.3,-.2) -- (.3,-.2);
\draw[\QsColor,thick] (-.3,.2) arc (180:0:.3cm) -- (.3,-.2) arc (0:-180:.3cm) -- (-.3,.2);
\draw[\QsColor,thick] (0,.5) -- (0,1.2);
\draw[\QsColor,thick] (0,-.5) -- (0,-1.2);
\filldraw[\QsColor,thick] (0,.5) circle (.05cm);
\filldraw[\QsColor,thick] (0,-.5) circle (.05cm);
}
=
\tikzmath{
\begin{scope}
\clip[rounded corners = 5] (-.7,-1.2) rectangle (.7,1.2);
\filldraw[\BColor] (-.7,-1.2) rectangle (.7,1.2);
\filldraw[\RrColor] (0,-.8) to [bend left=80] (0,.8) to [bend left=80] (0,-.8);
\end{scope}
\draw[\RsColor,thick] (0,-.8) to [bend left=80] (0,.8) to [bend left=80] (0,-.8);
\draw[\QsColor,thick] (-.3,.2) arc (180:0:.3cm) -- (.3,-.2) arc (0:-180:.3cm) -- (-.3,.2);
\draw[\QsColor,thick] (0,.5) -- (0,1.2);
\draw[\QsColor,thick] (0,-.5) -- (0,-1.2);
\filldraw[\QsColor,thick] (0,.5) circle (.05cm);
\filldraw[\QsColor,thick] (0,-.5) circle (.05cm);
}
=
\tikzmath{
\begin{scope}
\clip[rounded corners = 5] (-.5,-.7) rectangle (.5,.7);
\filldraw[\BColor] (-.5,-.7) rectangle (.5,.7);
\filldraw[\RrColor] (0,0) circle (.3cm);
\end{scope}
\draw[\RsColor,thick] (0,0) circle (.3cm);
\draw[\QsColor,thick] (0,-.7) -- (0,.7);
}
=
\tikzmath{
\begin{scope}
\clip[rounded corners = 5] (-.3,-.7) rectangle (.3,.7);
\filldraw[\BColor] (-.3,-.7) rectangle (.3,.7);
\end{scope}
\draw[\QsColor,thick] (0,-.7) -- (0,.7);
}
$
\end{proof}

\begin{rem}
In Lemma \ref{Lem:NaturalSplit}(2), if 
$(R,m^R_b,i^R_b)$ 
and
$(Q,m^Q_R,i^Q_R)$ are non-degenerate, 
then so is
$(Q,m_b^Q,i_b^Q)$.
Indeed, $(i_b^R)^\dag \circ i_b^R$ is positive and invertible and thus bounded below.
Hence there is some $c>0$ such that
\[
\tikzmath{
\begin{scope}
\clip[rounded corners = 5] (-.5,-.6) rectangle (.5,1.7);
\filldraw[\BColor] (-.5,-.6) rectangle (.5,1.7);
\filldraw[\RrColor] (0,0) circle (.3cm);
\filldraw[\RrColor] (0,1.1) circle (.3cm);
\end{scope}
\draw[thick,\RsColor] (0,0) circle (.3cm);
\draw[thick,\RsColor] (0,1.1) circle (.3cm);
\draw[thick,\QsColor] (0,0) -- (0,1.1);
\filldraw[\QsColor] (0,0) circle (.05cm);
\filldraw[\QsColor] (0,1.1) circle (.05cm);
}
=
\tikzmath{
\begin{scope}
\clip[rounded corners = 5] (-.5,-.6) rectangle (.5,1.7);
\filldraw[\BColor] (-.5,-.6) rectangle (.5,1.7);
\filldraw[\RrColor] (-.3,1.1) arc (180:0:.3cm) -- (.3,0) arc (0:-180:.3cm);
\end{scope}
\draw[\RsColor,thick] (-.3,1.1) arc (180:0:.3cm) -- (.3,0) arc (0:-180:.3cm) -- (-.3,1.1);
\draw[thick,\QsColor] (0,0) -- (0,1.1);
\filldraw[\QsColor] (0,0) circle (.05cm);
\filldraw[\QsColor] (0,1.1) circle (.05cm);
}
\geq
c\cdot 
\tikzmath{
\begin{scope}
\clip[rounded corners = 5] (-.5,-.6) rectangle (.5,1.7);
\filldraw[\BColor] (-.5,-.6) rectangle (.5,1.7);
\filldraw[\RrColor] (-.3,1.1) arc (180:0:.3cm) -- (.3,0) arc (0:-180:.3cm);
\end{scope}
\draw[\RsColor,thick] (-.3,1.1) arc (180:0:.3cm) -- (.3,0) arc (0:-180:.3cm) -- (-.3,1.1);
}
\in
\End(1_b)^\times.
\]
\end{rem}

\subsection{Q-system completion}
\label{sec:QSystemCompletion}

We now present the notion of Q-system completeness for a $\rm C^*/W^*$ 2-category.
The following presentation is 
motivated by \cite[\S A.5]{1812.11933} and \cite{1905.09566}.
Similar results to those in this section were proved for particular $\rm W^*$ 2-subcategories of $\vNA$ in \cite[\S3.2]{2010.01072} (see Example \ref{ex:GY} below).

\begin{defn}
We say that a $\rm C^*$-2-category $\cC$ is \emph{Q-system complete} if the inclusion $\iota_\cC:\cC\hookrightarrow \QSys(\cC)$ is a $\dag$ 2-equivalence.
\end{defn}

\begin{rem}[{cf.~\cite[Prop.~A.5.4]{1812.11933}}]
\label{rem:QSysCompleteIffESon0}
Observe that by construction,
the inclusion $\dag$ 2-functor $\iota_\cC : \cC \to \QSys(\cC)$ is
\begin{enumerate}[label=($\iota$\arabic*)]
\item
\label{iota:FFon2}
unitarily essentially surjective on 1-morphisms between trivial Q-systems
(those of the form $1_c$ for $c\in\cC$), 
since the unital Frobenius $1_a-1_b$ bimodule objects are exactly the 1-morphisms
${}_aX_b \in \cC(a\to b)$,\footnote{For a unital $1_a-1_b$ bimodule ${}_aX_b\in \cC(a\to b)$, the left $1_a$- and right $1_b$-actions must be unitors in $\cC$.}
and

\item
\label{iota:ESon1}
fully faithful on 2-morphisms,
since the $1_a-1_b$ bimodule maps ${}_aX_b\Rightarrow {}_aY_b$ are exactly the 2-morphisms $f\in \cC({}_aX_b\Rightarrow {}_aY_b)$.
\end{enumerate}
This means that by a dagger version of \cite[Thm.~7.4.1]{MR4261588}, 
$\cC$ is Q-system complete if and only if $\iota_\cC$ is unitarily essentially surjective on objects.
By Remark \ref{rem:IsoIffUnitarilyIso}, it suffices to prove $\iota_\cC$ is (algebraically) essentially surjective on objects.
\end{rem}

Various results related to the next theorem for Q-systems  have appeared in the subfactor literature; for example, see \cite[Thm.~3.11]{MR3308880} and \cite{MR1966524}.

\begin{thm}[{cf.~\cite[Prop.~A.4.2]{1812.11933}}]
\label{Thm:QSystemSplit}
A $\rm C^*/W^*$ 2-category $\cC$ is Q-system complete if and only if
every Q-system $Q\in \cC(b\to b)$ splits, i.e., there is an object $c\in \cC$ and
a dualizable 1-morphism ${}_bX_c\in \cC(b\to c)$ which admits a unitarily separable dual $(X^\vee,\ev_X,\coev_X)$ such that
$(Q,m,i)$ is isomorphic to $X\xz_c X^\vee$ with its usual multiplication and unit from Example \ref{ex:XXvQSystem}.
\end{thm}
\begin{proof}
\item[\underline{$\Rightarrow$:}]
Suppose $Q\in \cC(b\to b)$ is a Q-system.
Since $\cC \hookrightarrow \QSys(\cC)$ is an equivalence, it is unitarily essentially surjective on objects.
This means there is an object $c\in \cC$ and an invertible (which implies dualizable) separable bimodule 
${}_{Q}Y_{1_c}=
\tikzmath{
\begin{scope}
\clip[rounded corners=5pt] (-.3,0) rectangle (.3,.6);
\fill[\QrColor] (0,0) rectangle (-.3,.6);
\fill[\CColor] (0,0) rectangle (.3,.6);
\end{scope}
\draw[thick, \YColor] (0,0) -- (0,.6);
}\in \QSys(\cC)(Q\to 1_c)$
together with unitaries
$$
\varepsilon
:=
\tikzmath{
\begin{scope}
\clip[rounded corners=5pt] (-.3,0) rectangle (.7,.6);
\fill[\QrColor] (-.3,0) rectangle (.7,.6);
\fill[\CColor] (0,0) -- (0,.2) arc (180:0:.2cm) -- (.4,0);
\end{scope}
\draw[thick,\YColor] (0,0) -- (0,.2) arc (180:0:.2cm) -- (.4,0);
}
:{}_{Q}Y\xzq_{1_c} Y^\vee_Q \xrightarrow{\cong} {}_QQ_Q.
\qquad \text{and} \qquad
\delta:=
\tikzmath{
\begin{scope}
\clip[rounded corners=5pt] (-.3,0) rectangle (.7,.6);
\fill[\CColor] (-.3,0) rectangle (.7,.6);
\fill[\QrColor] (0,.6) -- (0,.4) arc (-180:0:.2cm) -- (.4,.6);
\end{scope}
\draw[thick,\YColor] (0,.6) -- (0,.4) arc (-180:0:.2cm) -- (.4,.6);
}
:{}_{1_c}{1_c}_{1_c} \xrightarrow{\cong} {}_{1_c}Y^\vee\xzq_Q Y_{1_c}, 
$$
which satisfy the zigzag conditions.
Since $\varepsilon$ and $\delta$ are unitaries, we see that $Y^\vee$ is a unitarily separable dual for $Y$, and $Y$ is a unitarily separable dual for $Y^\vee$.

Now consider
${}_{1_b}Q_Q=\tikzmath{
\begin{scope}
\clip[rounded corners=5pt] (-.3,0) rectangle (.3,.6);
\fill[\BColor] (0,0) rectangle (-.3,.6);
\fill[\QrColor] (0,0) rectangle (.3,.6);
\end{scope}
\draw[thick, \QsColor] (0,0) -- (0,.6);
}$ 
as a 
$1_b-Q$ bimodule,
and set
$X:= {}_{1_b}Q\xzq_Q Y_{1_c}
=
\tikzmath{
\begin{scope}
\clip[rounded corners=5pt] (-.3,0) rectangle (.6,.6);
\fill[\BColor] (0,0) rectangle (-.3,.6);
\fill[\QrColor] (0,0) rectangle (.3,.6);
\fill[\CColor] (.6,0) rectangle (.3,.6);
\end{scope}
\draw[thick, \QsColor] (0,0) -- (0,.6);
\draw[thick, \YColor] (.3,0) -- (.3,.6);
}$.
Then clearly ${}_{1_b}X\xzq_{1_c} X^\vee_{1_b}=\tikzmath{
\begin{scope}
\clip[rounded corners=5pt] (-.3,0) rectangle (1.3,.6);
\fill[\BColor] (0,0) rectangle (-.3,.6);
\fill[\QrColor] (0,0) rectangle (.3,.6);
\fill[\CColor] (.7,0) rectangle (.3,.6);
\fill[\QrColor] (.7,0) rectangle (1,.6);
\fill[\BColor] (1,0) rectangle (1.3,.6);
\end{scope}
\draw[thick, \QsColor] (0,0) -- (0,.6);
\draw[thick, \YColor] (.3,0) -- (.3,.6);
\draw[thick, \YColor] (.7,0) -- (.7,.6);
\draw[thick, \QsColor] (1,0) -- (1,.6);
}$ is isomorphic to ${}_bQ\xzq_Q Q_b$ as Q-systems via the isomorphism
\[
u:=
\tikzmath{
\begin{scope}
\clip[rounded corners=5pt] (-.3,0) rectangle (1.3,.6);
\fill[\BColor] (0,0) rectangle (-.3,.6);
\fill[\QrColor] (0,0) rectangle (1,.6);
\fill[\CColor] (.3,0) -- (.3,.2) arc (180:0:.2cm) -- (.7,0);
\fill[\BColor] (1,0) rectangle (1.3,.6);
\end{scope}
\draw[thick, \QsColor] (0,0) -- (0,.6);
\draw[thick, \YColor] (.3,0) -- (.3,.2) arc (180:0:.2cm) -- (.7,0);
\draw[thick, \QsColor] (1,0) -- (1,.6);
}
\qquad\Longrightarrow\qquad
\tikzmath{
\begin{scope}
\clip[rounded corners=5pt] (-.6,-.5) rectangle (2.4,1.9);
\fill[\QrColor] (-.6,-.5) rectangle (2.4,1.9);
\fill[\BColor] (.7,-.5) -- (.7,0) arc (180:0:.2cm) -- (1.1,-.5);
\fill[\BColor] (-.3,-.5) -- (-.3,.1) .. controls ++(90:.3cm) and ++(270:.3cm) .. (.4,1.3) -- (.4,1.9) -- (-.6,1.9) -- (-.6,-.5);
\fill[\BColor] (2.1,-.5) -- (2.1,.1) .. controls ++(90:.3cm) and ++(270:.3cm) .. (1.4,1.3) -- (1.4,1.9) -- (2.4,1.9) -- (2.4,-.5);
\fill[\CColor] (0,0) arc (-180:0:.2cm) arc (180:0:.5cm) arc (-180:0:.2cm) -- (1.8,.1) .. controls ++(90:.3cm) and ++(270:.3cm) .. (1.1,1.3) -- (1.1,1.4) arc (0:180:.2cm) -- (.7,1.3) .. controls ++(270:.3cm) and ++(90:.3cm) .. (0,.1) -- (0,0);
\end{scope}
\draw[thick,\YColor] (0,0) arc (-180:0:.2cm) arc (180:0:.5cm) arc (-180:0:.2cm) -- (1.8,.1) .. controls ++(90:.3cm) and ++(270:.3cm) .. (1.1,1.3) -- (1.1,1.4) arc (0:180:.2cm) -- (.7,1.3) .. controls ++(270:.3cm) and ++(90:.3cm) .. (0,.1) -- (0,0);
\draw[thick,\QsColor] (.7,-.5) -- (.7,0) arc (180:0:.2cm) -- (1.1,-.5);
\draw[thick,\QsColor] (-.3,-.5) -- (-.3,.1) .. controls ++(90:.3cm) and ++(270:.3cm) .. (.4,1.3) -- (.4,1.9);
\draw[thick,\QsColor] (2.1,-.5) -- (2.1,.1) .. controls ++(90:.3cm) and ++(270:.3cm) .. (1.4,1.3) -- (1.4,1.9);
\draw[dashed] (-.6,0) -- (2.4,0);
\draw[dashed] (-.6,1.4) -- (2.4,1.4);
}
=
\tikzmath{
\begin{scope}
\clip[rounded corners=5pt] (-.7,0) rectangle (.7,.9);
\fill[\BColor] (-.7,0) rectangle (.7,.9);
\fill[\QrColor] (-.4,0) -- (-.4,.2) .. controls ++(90:.2cm) and ++(270:.2cm) .. (-.1,.7) -- (-.1,.9) -- (.1,.9) -- (.1,.7)  .. controls ++(270:.2cm) and ++(90:.2cm) .. (.4,.2) -- (.4,0);
\fill[\BColor] (-.2,0) -- (-.2,.2) arc (180:0:.2cm) -- (.2,0);
\end{scope}
\draw[thick,\QsColor] (-.2,0) -- (-.2,.2) arc (180:0:.2cm) -- (.2,0);
\draw[thick,\QsColor] (-.4,0) -- (-.4,.2) .. controls ++(90:.2cm) and ++(270:.2cm) .. (-.1,.7) -- (-.1,.9);
\draw[thick,\QsColor] (.4,0) -- (.4,.2) .. controls ++(90:.2cm) and ++(270:.2cm) .. (.1,.7) -- (.1,.9);
}
\qquad\qquad
\tikzmath{
\begin{scope}
\clip[rounded corners=5pt] (-.8,-.8) rectangle (.8,.5);
\fill[\BColor] (-.8,-.8) rectangle (.8,.5);
\fill[\QrColor] (-.5,.5) -- (-.5,0) arc (-180:0:.5cm) -- (.5,.5);
\fill[\CColor] (0,0) circle (.2cm);
\end{scope}
\draw[thick,\YColor] (0,0) circle (.2cm);
\draw[thick,\QsColor] (-.5,.5) -- (-.5,0) arc (-180:0:.5cm) -- (.5,.5);
\draw[dashed] (-.8,0) -- (.8,0);
}
=
\tikzmath{
\begin{scope}
\clip[rounded corners=5pt] (-.4,-.4) rectangle (.4,.5);
\fill[\BColor] (-.7,-.4) rectangle (.7,.5);
\fill[\QrColor] (-.1,.5) -- (-.1,0) arc (-180:0:.1cm) -- (.1,.5);
\end{scope}
\draw[thick,\QsColor] (-.1,.5) -- (-.1,0) arc (-180:0:.1cm) -- (.1,.5);
}\,.
\]
By Lemma \ref{Lem:NaturalSplit}(1), ${}_{1_b}X\xzq_{1_a} X^\vee_{1_b}\cong {}_bQ\xzq_Q Q_b\cong {}_bQ_b$ as Q-systems.
Hence $Q$ splits as desired. 



\item[\underline{$\Leftarrow$:}]
Suppose every Q-system $Q\in\cC(b\to b)$ splits.
To show $\cC$ is Q-system complete, by Remark \ref{rem:QSysCompleteIffESon0}, it suffices to show that every $Q$-system ${}_bQ_b$ is equivalent to a trivial Q-system. 
Let ${}_{b}Y_c$ be a dualizable 1-morphism 
with unitarily separable dual $({}_cY^\vee_b, \ev_Y, \coev_Y)$
such that ${}_bQ_b\cong {}_bY\xz_cY^\vee_b$ as Q-systems.
This isomorphism intertwines the $Q-Q$ bimodule structure on ${}_bQ{}_b$ with the ${}_bY\xz_c Y^\vee_b-{}_bY\xz_c Y^\vee_b$ bimodule structure on ${}_bY\xz_c Y^\vee_b$.
The rest of the proof is now similar to Example \ref{ex:EquivalentRestrictedQSystem}.
Observe that under the above isomorphism, the canonical separability projector
$p_{Y^\vee,Y} \in \End({}_cY^\vee \xz_b Y_c)$
onto
${}_cY^\vee \xzq_Q Y{}_c \cong {}_cY^\vee \xzq_{Y\xz_b Y^\vee} Y_c$ 
is given as in \eqref{eq:SeparabilityProjector} by
$$
p_{Y,Y^\vee}
=
\tikzmath{
\begin{scope}
\clip[rounded corners = 5pt] (-.6,-.6) rectangle (.6,.6);
\filldraw[\CColor] (-.6,-.6) rectangle (-.3,.6);
\filldraw[\BColor] (-.3,-.6) rectangle (.3,.6);
\filldraw[\CColor] (.3,-.6) rectangle (.6,.6);
\end{scope}
\draw[thick, \QsColor] (-.3,-.3) arc (-90:0:.3cm) arc (180:90:.3cm);
\draw[thick, \YColor] (-.3,-.6) -- (-.3,.6);
\draw[thick, \YColor] (.3,-.6) -- (.3,.6);
\filldraw[\YColor] (-.3,-.3) circle (.05cm);
\filldraw[\YColor] (.3,.3) circle (.05cm);
}
=
\tikzmath{
\fill[\CColor, rounded corners=5pt] (-.8,-.8) rectangle (.8,.8);
\filldraw[fill=\BColor, thick, draw=\YColor] (-.4,-.8) -- (-.4,-.5) arc (180:90:.1cm) arc (-90:0:.4cm) arc (180:90:.2cm) arc (90:0:.1cm) -- (.4,-.8);
\filldraw[fill=\BColor, thick, draw=\YColor] (-.4,.8) -- (-.4,-.1) arc (180:270:.1cm) arc (-90:0:.2cm) arc (180:90:.4cm) arc (-90:0:.1cm) -- (.4,.8);
}
=
\tikzmath{
\begin{scope}
\clip[rounded corners=5pt] (-.6,-.5) rectangle (.6,.5);
\fill[\CColor] (-.6,-.5) rectangle (.6,.5);
\fill[\BColor] (-.3,-.5) arc (180:0:.3cm);
\fill[\BColor] (-.3,.5) arc (-180:0:.3cm);
\end{scope}
\draw[thick, \YColor] (-.3,-.5) arc (180:0:.3cm);
\draw[thick, \YColor] (-.3,.5) arc (-180:0:.3cm);
}
=
\ev_Y^\dag \circ \ev_Y.
$$
Thus $\ev_Y$ is a coisometry which splits $p_{Y^\vee,Y}$.
Since the image of $\ev_Y$ is $1_c$, we see that ${}_QY_c$ is an invertible $Q-1_c$ bimodule in $\QSys(\cC)(Q\to 1_c)$, as desired.
\end{proof}

The following example of Q-system splitting will be useful in our proof of Theorem \ref{thm:QSysComplete}, which says that $\rCorr$ is Q-system complete.

\begin{cor}[{cf.~\cite[Prop. 4.2]{MR3459961} and \cite[Prop.~A.5.3]{1812.11933}}]
$\QSys(\cC)$ is Q-system complete.
\end{cor}
\begin{proof}
By Theorem \ref{Thm:QSystemSplit}, it suffices to prove every Q-system $P\in \QSys(\cC)(Q\to Q)$ splits, where $Q\in \cC(b\to b)$ is a Q-system in $\cC$. 
According to Lemma \ref{Lem:NaturalSplit}(2), $P\in\cC(b\to b)$ is a Q-system in $\cC$, ${}_QP_Q\cong {}_Q P\xzq_P P_Q$ as Q-systems, where ${}_QP_P\in \QSys(\cC)(Q\to P)$ is a unitarily separable dualizable 1-morphism. 
\end{proof}



\begin{ex}[{\cite{2010.01072}}]
\label{ex:GY}
Let $\vNA_{\rm II_1}^{\mathsf{factor}}$ denote the $\rm W^*$ 2-category whose objects are type $\rm II_1$ factors.
Then $\QSys(\vNA_{\rm II_1}^{\mathsf{factor}})=\vNA_{\rm II_1}^{\mathsf{multifactor}}$, whose objects are $\rm II_1$ multifactors, i.e., finite direct sums of $\rm II_1$ factors.
For $R$ the hyperfinite $\rm II_1$ factor, the objects of $\QSys(\Bim(R))$ are exactly the hyperfinite $\rm II_1$ \emph{multifactors}, i.e., finite direct sums of $R$.
\end{ex}

\begin{ex}
Let $\cC$ be a unitary fusion category.
Recall from Remark \ref{rem:delooping} that $\rmB\cC$ is the delooping of $\cC$, which is $\cC$ considered as a $\rm C^*$ 2-category with one object.

The $\rm C^*$ 2-category $\Mod^\dag(\cC)$ has 
\begin{itemize}
\item
objects finitely semisimple left $\cC$-module $\rm C^*$-categories, 
\item
1-morphisms $\dag$ $\cC$-module functors, and 
\item
2-morphisms $\cC$-module natural transformations.
\end{itemize}

By the unitary version of the Barr-Beck/Ostrik theorem \cite[Thm.~3.1]{MR1976459}, \cite[\S4]{MR3847209} for unitary fusion categories,
$\Mod^\dag(\cC)$ is unitarily 2-equivalent to the Q-system completion $\QSys(\rmB\cC)$.
We sketch a proof below for the reader's convenience.

It is straightforward to prove that the map
\begin{equation}
\label{eq:Equivalence2Functor}
Q \mapsto \Mod^\dag_\cC(Q)
\qquad\qquad
{}_PX_Q \mapsto (-\xzq_P X : \Mod^\dag_\cC(P) \to \Mod^\dag_\cC(Q))
\end{equation}
defines a $\dag$ 2-functor $\QSys(\rmB\cC)\to \Mod^\dag(\cC)$.
Here, $\Mod^\dag_\cC(Q)$ is the category of unitarily separable right $Q$-module objects in $\cC$ with right $Q$-module maps, which is $\rm C^*$ by the analogous version of \ref{QSys:StarClosed}.

One checks that for all unital $P-Q$ bimodules ${}_PX_Q, {}_PY_Q$,
$$
\Hom_{P-Q}(X \Rightarrow Y)
\ni
\theta
\longmapsto
(-\xzq \theta 
: - \xzq_PX_Q \Rightarrow -\xzq_P Y_Q)
$$
is an isomorphism.
Indeed, every $\cC$-module natural transformation $\theta: - \xzq_P X_Q \Rightarrow -\xzq_P X_Q$ is completely determined by $\theta_P$ using that $\Mod^\dag_\cC(P)$ is the unitary Cauchy completion of $\mathsf{FreeMod}_\cC(P)$.

Thus to show our 2-functor \eqref{eq:Equivalence2Functor} is fully faithful, we need to prove the hom functors are essentially surjective.
Suppose $\cF: \Mod^\dag_\cC(P) \to \Mod^\dag_\cC(Q)$ is a $\cC$-module $\dag$-functor. 
Then $\cF(P)\in \Mod^\dag_\cC(Q)$ carries a unitarily separable right $Q$-action be definition, and a left $P$-action using the modulator $\lambda_P:= \cF(m_P)\circ \cF^2_{P,P}$.
Moreover, it is straightforward to see 
$$
\lambda_P\lambda_P^\dag
=
\cF(m_P)\circ \cF^2_{P,P} \circ (\cF^2_{P,P})^\dag \circ \cF(m_P)^\dag
=
\cF(m_P\circ m_P^\dag)
=
\id_{X},
$$
so the left $P$-action on $\cF(P)$ is unitarily separable.
One then checks that $\cF$ is unitarily isomorphic to  $ -\xzq_P \cF(P)$.

It remains to show the $\dag$ 2-functor \eqref{eq:Equivalence2Functor} is essentially surjective.
This follows by \cite[Thm.~A.1]{MR3933035}, which essentially shows that a semisimple left $\cC$-module $\rm C^*$-category $\cM$ which admits a pointing is unitarily equivalent to $\Mod^\dag_\cC(Q)$ for some Q-system $Q$.
\end{ex}

\begin{rem}
 Q-system completion satisfies a universal property similar to \cite[\S1.2]{MR4372801}.
We leave this to \cite{MR4369356} as it would take us too far afield.
\end{rem}

\section{\texorpdfstring{$\rCorr$}{C*Alg} is Q-system complete}
\label{sec:C*AlgQSysComplete}

In this section we prove the following theorem.

\begin{thm*}[Theorem \ref{thm:QSysComplete}]
$\rCorr$ is Q-system complete.
\end{thm*}

To do so, we construct a $\dag$ 2-functor
$|\cdot|:\QSys(\rCorr) \to \rCorr$
called the \emph{realization functor},
and show it is inverse to the canonical inclusion $\iota:\cC\hookrightarrow \QSys(\cC)$.
Our strategy will be to adapt the 
usual subfactor method \cite{MR2097363}, 
which is clearly expressed diagrammatically in \cite[\S4.1]{MR3221289}. 
(This strategy is also similar to that employed to describe the Tannaka duality for semisimple Hopf algebras in terms of Frobenius algebras \cite[\S7]{MR2075605}.)

We begin with the following important initial observation.

\begin{construction}
\label{construction:BoundedVectorModule}
Suppose ${}_AX_B\in \rCorr(A\to B)$.
Consider the right $B$-module
$$
\Hom_{\bbC-B}({}_{\bbC}B_B \to {}_\bbC A\boxtimes_A X_B)
$$
with 
left $A$-action $(a\rhd f)(x):= (a\boxtimes \id_X)\circ f)(x)$
(here, we identify $A=\End(A_A)$),
right $B$-action $(f\lhd b)(x):=f(bx)$,
and right $B$-valued inner product 
$\langle f|g\rangle^{\Hom}_B := f^\dag \circ g \in \End(B_B)=B$.
Identifying ${}_\bbC X_B = {}_\bbC A\boxtimes_A X_B$, the maps
\[
\begin{aligned}
{}_AX_B
&\longrightarrow 
\Hom_{\bbC-B}({}_{\bbC}B_B \to {}_{\bbC}X_B)
\\
\xi &\longmapsto (L_\xi : b\mapsto \xi \lhd b)
\end{aligned}
\qquad\text{and}\qquad
\begin{aligned}
\Hom_{\bbC-B}({}_{\bbC}B_B \to {}_{\bbC}X_B)
&\longrightarrow
X_B
\\
f& \longmapsto f(1_B)
\end{aligned}
\]
are mutually inverse linear maps which preserve these $B$-valued inner products.
Indeed, $L_\xi$ is adjointable with adjoint
$L_\xi^\dag(\eta) := \langle \xi|\eta\rangle^X_B$.
Thus
$$
\langle L_\xi|L_\eta\rangle^{\Hom}_B
=
L_\xi^\dag \circ L_\eta
=
\langle \xi|\eta\rangle_B^X.
$$
\end{construction}

\begin{rem}
Observe that while we always have
$
\Hom_{\bbC-B}(B_B \to X_B)\cong X_B,
$
in general,
$
\Hom_{\bbC-B}(X_B \to B_B) \ncong {}_B\overline{X}.
$
Each $\xi \in X$ still gives us an adjointable map $L_\xi^\dag : X_B \to B_B$ by $\eta\mapsto \langle \xi|\eta\rangle_B$, but the anti-linear map $\xi\mapsto L_\xi^\dag$ need not be onto.
This map is onto precisely when $X_B$ is a \emph{self-dual} $B$-module; see \cite[\S3]{MR355613} for more details.
\end{rem}

\subsection{Realization of Q-systems}

For this section, we fix a unital $\rm C^*$-algebra $B$.
In the graphical calculus for $\rCorr$, we denote the algebra $B$ by a shaded region and $\bbC$ by an unshaded region.
$$
\tikzmath{\filldraw[fill=white,dashed, rounded corners=5, very thin, baseline=1cm] (0,0) rectangle (.5,.5);}=\bbC
\qquad\qquad
\tikzmath{\filldraw[\BColor, rounded corners=5, very thin, baseline=1cm] (0,0) rectangle (.5,.5);}=B
$$
Suppose
$({}_BQ_B, \mu, i)$ is a Q-system in $\rCorr(B\to B)$.
We denote $Q$ by a black strand shaded on both sides, and ${}_\bbC B_B$ by a dashed strand which is shaded on the right side
$$
{}_BQ{}_{B}
=
\tikzmath{
\begin{scope}
\clip[rounded corners=5pt] (-.5,0) rectangle (.5,1);
\fill[\BColor] (-.5,0) rectangle (.5,1);
\end{scope}
\draw (0,0) -- (0,1);
}
\qquad\qquad
{}_{\bbC}B{}_{B}
=
\tikzmath{
\begin{scope}
\clip[rounded corners=5pt] (-.2,0) rectangle (.5,1);
\fill[\BColor] (0,0) rectangle (.5,1);
\end{scope}
\draw[dashed] (0,0) -- (0,1);
}
\qquad\qquad
{}_{\bbC}B\boxtimes_B Q_{B}
=
\tikzmath{
\begin{scope}
\clip[rounded corners=5pt] (-.2,0) rectangle (.5,1);
\fill[\BColor] (0,0) rectangle (.7,1);
\end{scope}
\draw[dashed] (0,0) -- (0,1);
\draw (.2,0) -- (.2,1);
}
$$
We denote the multiplication and unit of $Q$ by a trivalent and univalent vertex respectively.
Observe that the $\bbC-B$ bimodule ${}_\bbC B\boxtimes_B Q_B$ which forgets the left $B$-action is a right $Q$-module object in $\rCorr(\bbC \to B)$.
$$
\tikzmath{
\fill[\BColor, rounded corners=5pt] (-.4,0) rectangle (1,1);
\draw (0,0) -- (0,.2) arc (180:0:.3cm) -- (.6,0);
\draw (.3,.5) -- (.3,1);
\filldraw (.3,.5) circle (.05cm);
}
\qquad\text{and}\qquad
\tikzmath{
\fill[\BColor, rounded corners=5pt] (0,0) rectangle (.8,1);
\draw (.4,.5) -- (.4,1);
\filldraw (.4,.5) circle (.05cm);
}
\qquad\qquad
\rightsquigarrow
\qquad\qquad
\tikzmath{
\begin{scope}
\clip[rounded corners=5pt] (-.4,0) rectangle (1,1);
\fill[\BColor] (-.2,0) rectangle(1,1);
\end{scope}
\draw[dashed] (-.2,0) -- (-.2,1);
\draw (0,0) -- (0,.2) arc (180:0:.3cm) -- (.6,0);
\draw (.3,.5) -- (.3,1);
\filldraw (.3,.5) circle (.05cm);
}
\qquad\text{and}\qquad
\tikzmath{
\begin{scope}
\clip[rounded corners=5pt] (0,0) rectangle (.8,1);
\fill[\BColor] (.2,0) rectangle(.8,1);
\end{scope}
\draw[dashed] (.2,0) -- (.2,1);
\draw (.4,.5) -- (.4,1);
\filldraw (.4,.5) circle (.05cm);
}
\,.
$$

\begin{construction}
\label{construction:|Q|}
Following Construction \ref{construction:BoundedVectorModule}, we define $|Q|$ to be the unital $*$-algebra
whose underlying vector space is $\Hom_{\bbC - B}({}_{\bbC}B_B \to {}_{\bbC}B\boxtimes_B Q_B)$,
and we denote its elements by
$$
\tikzmath{
\begin{scope}
\clip[rounded corners=5pt] (-.4,-.7) rectangle (.7,.7);
\fill[\BColor] (0,-.7) -- (0,0) -- (-.1,0) -- (-.1,.7) -- (.8,.7) -- (.8,-.7);
\end{scope}
\draw[dashed] (0,-.7) -- (0,0) -- (-.1,0) -- (-.1,.7);
\draw (.1,.3) -- (.1,.7);
\roundNbox{unshaded}{(0,0)}{.3}{0}{0}{$q$};
}
\in |Q| 
:=
\Hom_{\bbC - B}({}_{\bbC}B_B \to {}_{\bbC}B\boxtimes_B Q_B).
$$
We define the multiplication, unit, and adjoint $*$ of $|Q|$ respectively by
\[
q_1\cdot q_2
:=
\tikzmath{
\begin{scope}
\clip[rounded corners=5pt] (-.4,-2) rectangle (1.2,1);
\fill[\BColor] (.6,-2) -- (.6,-1.3) -- (.5,-1) .. controls ++(90:.3cm) and ++(270:.3cm) .. (0,-.3) -- (-.1,.3) -- (-.1,1) -- (1.2,1) -- (1.2,-2);
\end{scope}
\draw (.4,.6) -- (.4,1);
\filldraw (.4,.6) circle (.05cm);
\draw (.1,.3) arc (180:0:.3cm) -- (.7,-1);
\draw[dashed] (.6,-2) -- (.6,-1.3) -- (.5,-1) .. controls ++(90:.3cm) and ++(270:.3cm) .. (0,-.3) -- (-.1,.3) -- (-.1,1);
\roundNbox{unshaded}{(0,0)}{.3}{0}{0}{$q_1$};
\roundNbox{unshaded}{(.6,-1.3)}{.3}{0}{0}{$q_2$};
}\,,
\qquad
1_{|Q|}
:=
\tikzmath{
\begin{scope}
\clip[rounded corners=5pt] (-.2,0) rectangle (.6,1);
\fill[\BColor] (0,0) rectangle (.6,1);
\end{scope}
\draw[dashed] (0,0) -- (0,1);
\filldraw (.3,.5) circle (.05cm);
\draw (.3,.5) -- (.3,1);
}\,,
\qquad\text{and}\qquad
q^*:=\ 
\tikzmath{
\begin{scope}
\clip[rounded corners=5pt] (-.4,-1) rectangle (1,.7);
\fill[\BColor] (-.1,-1) -- (-.1,0) -- (0,0) -- (0,.7) -- (1,.7) -- (1,-1);
\end{scope}
\draw[dashed] (-.1,-1) -- (-.1,0) -- (0,0) -- (0,.7);
\draw (.1,-.3) arc (-180:0:.3cm) -- (.7,.7);
\draw (.4,-.6) -- (.4,-.8);
\filldraw (.4,-.6) circle (.05cm);
\filldraw (.4,-.8) circle (.05cm);
\roundNbox{unshaded}{(0,0)}{.3}{0}{0}{$q^\dag$};
}\,.
\]
Associativity of the multiplication $m_{Q}$ gives associativity of multiplication for $|Q|$, and unitality of $Q$ gives unitality of $|Q|$.
The Frobenius property for $Q$ and unitality yields $q^{**} = q$.
That $(q_1\cdot q_2)^*=q_2^*\cdot q_1^*$ again follows from unitality and the Frobenius properties:
$$
(q_1 \cdot q_2)^*
=
\tikzmath{
\begin{scope}
\clip[rounded corners=5pt] (-.4,2) rectangle (1.5,-1.4);
\fill[\BColor] (.6,2) -- (.6,1.3) -- (.5,1) .. controls ++(270:.3cm) and ++(90:.3cm) .. (0,.3) -- (-.1,-.3) -- (-.1,-1.4) -- (1.5,-1.4) -- (1.5,2);
\end{scope}
\draw (.4,-.6) arc (-180:0:.4cm) -- (1.2,2);
\draw (.8,-1) -- (.8,-1.2);
\filldraw (.8,-1.2) circle (.05cm);
\filldraw (.8,-1) circle (.05cm);
\filldraw (.4,-.6) circle (.05cm);
\draw (.1,-.3) arc (-180:0:.3cm) -- (.7,1);
\draw[dashed] (.6,2) -- (.6,1.3) -- (.5,1) .. controls ++(270:.3cm) and ++(90:.3cm) .. (0,.3) -- (-.1,-.3) -- (-.1,-1.4);
\roundNbox{unshaded}{(0,0)}{.3}{0}{0}{$q_1^\dag$};
\roundNbox{unshaded}{(.6,1.3)}{.3}{0}{0}{$q_2^\dag$};
}
=
\tikzmath{
\begin{scope}
\clip[rounded corners=5pt] (-1,2.3) rectangle (1,-1.5);
\fill[\BColor] (-.6,2.3) -- (-.6,1.3) -- (-.7,1) .. controls ++(270:.3cm) and ++(90:.3cm) .. (0,-.2) -- (-.1,-.8) -- (-.1,-1.5) -- (1,-1.5) -- (1,2.3);
\end{scope}
\filldraw (-.2,.7) circle (.05cm);
\filldraw (-.2,.5) circle (.05cm);
\filldraw (.4,1.9) circle (.05cm);
\filldraw (.4,-1.1) circle (.05cm);
\filldraw (.4,-1.3) circle (.05cm);
\draw (-.5,1) arc (-180:0:.3cm) -- (.1,1.6) arc (180:0:.3cm) -- (.7,-.2);
\draw (.4,1.9) -- (.4,2.3);
\draw (-.2,.7) -- (-.2,.5);
\draw (.1,-.8) arc (-180:0:.3cm) -- (.7,-.2);
\draw (.4,-1.1) -- (.4,-1.3);
\draw[dashed] (-.6,2.3) -- (-.6,1.3) -- (-.7,1) .. controls ++(270:.3cm) and ++(90:.3cm) .. (0,-.2) -- (-.1,-.8) -- (-.1,-1.5);
\roundNbox{unshaded}{(0,-.5)}{.3}{0}{0}{$q_1^\dag$};
\roundNbox{unshaded}{(-.6,1.3)}{.3}{0}{0}{$q_2^\dag$};
}
=
q_2^*\cdot q_1^*.
$$
To see that $|Q|$ is a $\rm C^*$-algebra,
we show it is $*$-isomorphic to a closed $*$-subalgebra of the unital $\rm C^*$-algebra 
$\End_{\bbC - B}({}_{\bbC}B\boxtimes_B Q_B)$,
which has 
multiplication, unit, and adjoint $*$ given respectively by
\[
x\cdot y :=
\tikzmath{
\begin{scope}
\clip[rounded corners=5pt] (-.4,-.7) rectangle (.8,1.7);
\fill[\BColor] (-.1,-.7) rectangle (.8,1.7);
\end{scope}
\draw[dashed] (-.1,-.7) -- (-.1,1.7);
\draw (.1,-.7) -- (.1,1.7);
\roundNbox{unshaded}{(0,1)}{.3}{0}{0}{$x$};
\roundNbox{unshaded}{(0,0)}{.3}{0}{0}{$y$};
}\,,
\qquad
1:=
\tikzmath{
\begin{scope}
\clip[rounded corners=5pt] (-.4,-.7) rectangle (.7,.7);
\fill[\BColor] (-.1,-.7) rectangle (.8,.7);
\end{scope}
\draw[dashed] (-.1,-.7) -- (-.1,.7);
\draw (.1,-.7) -- (.1,.7);
}
\,,
\qquad\text{and}\qquad
x^* :=
\tikzmath{
\begin{scope}
\clip[rounded corners=5pt] (-.4,-.7) rectangle (.7,.7);
\fill[\BColor] (-.1,-.7) rectangle (.8,.7);
\end{scope}
\draw[dashed] (-.1,-.7) -- (-.1,.7);
\draw (.1,-.7) -- (.1,.7);
\roundNbox{unshaded}{(0,0)}{.3}{0}{0}{$x^\dagger$};
}\,.
\]
Indeed, consider the unital $*$-subalgebra
\begin{equation}
\label{eq:NormClosedCondition}
\End_{-Q}({}_{\bbC}B\boxtimes_B Q_B)
:=
\set{
\tikzmath{
\begin{scope}
\clip[rounded corners=5pt] (-.4,-.7) rectangle (.5,.7);
\fill[\BColor] (-.1,-.7) rectangle (.5,.7);
\end{scope}
\draw[dashed] (-.1,-.7) -- (-.1,.7);
\draw (.1,-.7) -- (.1,.7);
\roundNbox{unshaded}{(0,0)}{.3}{0}{0}{$x$};
}\,\,
}{
\,\,
\tikzmath{
\begin{scope}
\clip[rounded corners=5pt] (-.4,-.5) rectangle (.9,1.3);
\fill[\BColor] (-.1,-.5) rectangle (.9,1.3);
\end{scope}
\draw[dashed] (-.1,-.5) -- (-.1,1.3);
\draw (.1,-.5) -- (.1,1.3);
\draw (.6,-.5) -- (.6,.4)  arc (0:90:.5cm); 
\filldraw (.1,.9) circle (.05cm);
\roundNbox{unshaded}{(0,.1)}{.3}{0}{0}{$x$};
}
\,=\,
\tikzmath{
\begin{scope}
\clip[rounded corners=5pt] (-.4,-.5) rectangle (.9,1.3);
\fill[\BColor] (-.1,-.5) rectangle (.9,1.3);
\end{scope}
\draw[dashed] (-.1,-.5) -- (-.1,1.3);
\draw (.1,-.5) -- (.1,1.3);
\draw (.6,-.5)  arc (0:90:.5cm); 
\filldraw (.1,0) circle (.05cm);
\roundNbox{unshaded}{(0,.6)}{.3}{0}{0}{$x$};
}\,
},
\end{equation}
which is clearly a norm closed unital $*$-subalgebra.\footnote{
Indeed, the map $\End({}_{\bbC}B\boxtimes_B Q_B)\ni x\mapsto x\boxtimes \id_Q\in \End({}_\bbC B\boxtimes_B Q\boxtimes_B Q_B)$ is injective and thus norm-preserving.
Since pre- and post-composition with a morphism in $\rCorr$ is a norm-continuous operation by Lemma \ref{lem:CompositionNormContinuous}, we see that 
the defining condition of $\End_{-Q}({}_{\bbC}B\boxtimes_B Q_B)$
in \eqref{eq:NormClosedCondition} is a norm-closed condition.
}
It is a straightforward and enjoyable exercise using the graphical calculus that the maps
$\Phi: |Q| \to \End_{-Q}({}_{\bbC}B\boxtimes_B Q_B)$ and $\Psi:\End_{-Q}({}_{\bbC}B\boxtimes_B Q_B)\to |Q|$  given by
\begin{equation}
\label{eq:CStarIso}
\Phi: 
\tikzmath{
\begin{scope}
\clip[rounded corners=5pt] (-.4,-.7) rectangle (.7,.7);
\fill[\BColor] (0,-.7) -- (0,0) -- (-.1,0) -- (-.1,.7) -- (.8,.7) -- (.8,-.7);
\end{scope}
\draw[dashed] (0,-.7) -- (0,0) -- (-.1,0) -- (-.1,.7);
\draw (.1,.3) -- (.1,.7);
\roundNbox{unshaded}{(0,0)}{.3}{0}{0}{$q$};
}
\mapsto
\tikzmath{
\begin{scope}
\clip[rounded corners=5pt] (-.4,-.7) rectangle (.9,1);
\fill[\BColor] (0,-.7) -- (0,0) -- (-.1,0) -- (-.1,1) -- (.9,1) -- (.9,-.7);
\end{scope}
\draw[dashed] (0,-.7) -- (0,0) -- (-.1,0) -- (-.1,1);
\draw (.1,.3) arc (180:0:.3cm) -- (.7,-.7);
\draw (.4,.6) -- (.4,1);
\filldraw (.4,.6) circle (.05cm);
\roundNbox{unshaded}{(0,0)}{.3}{0}{0}{$q$};
}
\qquad\qquad
\Psi:
\tikzmath{
\begin{scope}
\clip[rounded corners=5pt] (-.4,-.7) rectangle (.7,.7);
\fill[\BColor] (-.1,-.7) rectangle (.8,.7);
\end{scope}
\draw[dashed] (-.1,-.7) -- (-.1,.7);
\draw (.1,-.7) -- (.1,.7);
\roundNbox{unshaded}{(0,0)}{.3}{0}{0}{$x$};
}
\mapsto
\tikzmath{
\begin{scope}
\clip[rounded corners=5pt] (-.4,-.7) rectangle (.7,.7);
\fill[\BColor] (-.1,-.7) rectangle (.8,.7);
\end{scope}
\draw[dashed] (-.1,-.7) -- (-.1,.7);
\draw (.1,-.5) -- (.1,.7);
\filldraw (.1,-.5) circle (.05cm);
\roundNbox{unshaded}{(0,0)}{.3}{0}{0}{$x$};
}
\end{equation}
are mutually inverse unital $*$-isomorphisms.
\end{construction}

\begin{rem}
If we instead work in the $\rm W^*$ 2-category of $\rm W^*$-correspondences (or $\vNA$), observe that the unital $\rm C^*$-algebra $|Q|=\Hom_{\bbC-B}({}_{\bbC}(L^2 B)_B \to {}_\bbC (L^2B)\boxtimes_B X_B)$ has a predual as it is a hom space in a $\rm W^*$ category \cite{MR808930}, and is thus a von Neumann algebra.
One can also use that right $Q$-linearity is a weak* closed condition by combining \ref{W*:TensorIsNormal} and Lemma \ref{lem:CompositionNormContinuous}.
\end{rem}

\begin{defn}
\label{defn:ConditionalExpectation}
Observe we have a
unital $*$-algebra homomorphism $\iota:B\to |Q|$ given by
$$
B=\End_{\bbC-B}(B)
\ni
\tikzmath{
\begin{scope}
\clip[rounded corners=5pt] (-.4,-.7) rectangle (.7,.7);
\fill[\BColor] (0,-.7) rectangle (.8,.7);
\end{scope}
\draw[dashed] (0,-.7) -- (0,.7);
\roundNbox{unshaded}{(0,0)}{.3}{0}{0}{$b$};
}
\longmapsto
\tikzmath{
\begin{scope}
\clip[rounded corners=5pt] (-.4,-.7) rectangle (.9,1);
\fill[\BColor] (0,-.7) -- (0,0) -- (-.1,0) -- (-.1,1) -- (.9,1) -- (.9,-.7);
\end{scope}
\draw[dashed] (0,-.7) -- (0,0) -- (-.1,0) -- (-.1,1);
\draw (.4,.6) -- (.4,1);
\filldraw (.4,.6) circle (.05cm);
\roundNbox{unshaded}{(0,0)}{.3}{0}{0}{$b$};
}
\in |Q|.
$$
This map is adjointable in $\rCorr_{\bbC-B}(B\Rightarrow |Q|)$ 
(equipped with the right $B$-valued inner products from Construction \ref{construction:BoundedVectorModule})
with adjoint
$\iota^\dag: |Q| \Rightarrow B$
given by
$$
|Q|
\ni
\tikzmath{
\begin{scope}
\clip[rounded corners=5pt] (-.4,-.7) rectangle (.7,.7);
\fill[\BColor] (0,-.7) -- (0,0) -- (-.1,0) -- (-.1,.7) -- (.8,.7) -- (.8,-.7);
\end{scope}
\draw[dashed] (0,-.7) -- (0,0) -- (-.1,0) -- (-.1,.7);
\draw (.1,.3) -- (.1,.7);
\roundNbox{unshaded}{(0,0)}{.3}{0}{0}{$q$};
}
\longmapsto
\tikzmath{
\begin{scope}
\clip[rounded corners=5pt] (-.4,-.7) rectangle (.7,.7);
\fill[\BColor] (0,-.7) -- (0,0) -- (-.1,0) -- (-.1,.7) -- (.8,.7) -- (.8,-.7);
\end{scope}
\draw[dashed] (0,-.7) -- (0,0) -- (-.1,0) -- (-.1,.7);
\draw (.1,.3) -- (.1,.5);
\filldraw (.1,.5) circle (.05cm);
\roundNbox{unshaded}{(0,0)}{.3}{0}{0}{$q$};
}
\in \End_{\bbC-B}(B)=B.
$$
We define $E_B : |Q| \to B$ by
$$
E_B(q):=
\tikzmath{
\begin{scope}
\clip[rounded corners=5pt] (-.4,-.7) rectangle (1.6,.7);
\fill[\BColor] (0,-.7) -- (0,0) -- (-.1,0) -- (-.1,.7) -- (1.6,.7) -- (1.6,-.7);
\end{scope}
\draw[dashed] (0,-.7) -- (0,0) -- (-.1,0) -- (-.1,.7);
\draw (.1,.3) -- (.1,.5);
\filldraw (.1,.5) circle (.05cm);
\roundNbox{unshaded}{(0,0)}{.3}{0}{0}{$q$};
\roundNbox{unshaded}{(.9,0)}{.3}{.1}{.1}{$d_Q^{-1}$};
}\,,
$$
which is a $B-B$ bimodular map onto $s_QBs_Q$ where $s_Q=d_Qd_Q^{-1}\in B$ is the support of $Q$ as in \ref{Z:Support}.
\end{defn}

\begin{prop}
\label{prop:EBCPandFiniteIndex}
The map $E_B:|Q| \to B$ is completely positive, faithful and has finite Pimsner-Popa index, i.e., there is a $c>0$ such that for all $x\in |Q|^+$,
$x\leq c\cdot \iota(E_B(x))$.
\end{prop}
\begin{proof}
Under the isomorphisms \eqref{eq:CStarIso} above, 
$\iota$ is mapped to
$-\boxtimes \id_Q : B\to \End_{-Q}({}_{\bbC}B\boxtimes_B Q_B)$
and the expectation $E_B:|Q|\to B$ 
is mapped to 
$$
\cE_B(x)
:=
\tikzmath{
\begin{scope}
\clip[rounded corners=5pt] (-.4,-.8) rectangle (1.6,.8);
\fill[\BColor] (-.1,-.8) rectangle (1.6,.8);
\end{scope}
\draw[dashed] (-.1,-.8) -- (-.1,.8);
\draw (.1,.3) -- (.1,.5);
\draw (.1,-.3) -- (.1,-.5);
\filldraw (.1,.5) circle (.05cm);
\filldraw (.1,-.5) circle (.05cm);
\roundNbox{unshaded}{(0,0)}{.3}{0}{0}{$x$};
\roundNbox{unshaded}{(.9,0)}{.3}{.1}{.1}{$d_Q^{-1}$};
}
:
\End_{-Q}({}_\bbC B\boxtimes_B Q_B)\to \End_{-Q}({}_\bbC B\boxtimes_B Q_B).
$$
Since $\cE_B$ is manifestly cp, so is $E_B$.
We compute that $\cE_B$ has finite Pimsner-Popa index:
$$
\tikzmath{
\begin{scope}
\clip[rounded corners=5pt] (-.4,-.7) rectangle (.7,.7);
\fill[\BColor] (-.1,-.7) rectangle (.8,.7);
\end{scope}
\draw[dashed] (-.1,-.7) -- (-.1,.7);
\draw (.1,-.7) -- (.1,.7);
\roundNbox{unshaded}{(0,0)}{.3}{0}{0}{$x$};
}
=
\tikzmath{
\begin{scope}
\clip[rounded corners=5pt] (-.4,-1.2) rectangle (1.5,1.2);
\fill[\BColor] (-.1,-1.2) rectangle (1.5,1.2);
\end{scope}
\draw[dashed] (-.1,-1.2) -- (-.1,1.2);
\draw (.1,.3) -- (.1,.5) arc (180:0:.3cm) arc (-180:0:.2cm) -- (1.1,1.2);
\draw (.1,-.3) -- (.1,-.5) arc (-180:0:.3cm) arc (180:0:.2cm) -- (1.1,-1.2);
\draw (.9,-.1) -- (.9,-.3);
\draw (.9,.1) -- (.9,.3);
\filldraw (.9,.3) circle (.05cm);
\filldraw (.9,.1) circle (.05cm);
\filldraw (.9,-.1) circle (.05cm);
\filldraw (.9,-.3) circle (.05cm);
\draw (.4,.8) -- (.4,1);
\draw (.4,-.8) -- (.4,-1);
\filldraw (.4,.8) circle (.05cm);
\filldraw (.4,1) circle (.05cm);
\filldraw (.4,-1) circle (.05cm);
\filldraw (.4,-.8) circle (.05cm);
\roundNbox{unshaded}{(0,0)}{.3}{0}{0}{$x$};
\draw[dashed, thick, rounded corners=5pt] (.5,-.5) rectangle (1.3,.5);
}
\underset{\text{\ref{Z:JonesProjBound}}}{\leq}
\|d_Q\|\cdot
\tikzmath{
\begin{scope}
\clip[rounded corners=5pt] (-.4,-1.2) rectangle (1.5,1.2);
\fill[\BColor] (-.1,-1.2) rectangle (1.5,1.2);
\end{scope}
\draw[dashed] (-.1,-1.2) -- (-.1,1.2);
\draw (.1,.3) -- (.1,.5) arc (180:0:.3cm) -- (.7,-.5) arc (0:-180:.3cm) -- (.1,-.3);
\draw (1.1,-1.2) -- (1.1,1.2);
\draw (.4,.8) -- (.4,1);
\draw (.4,-.8) -- (.4,-1);
\filldraw (.4,.8) circle (.05cm);
\filldraw (.4,1) circle (.05cm);
\filldraw (.4,-1) circle (.05cm);
\filldraw (.4,-.8) circle (.05cm);
\roundNbox{unshaded}{(0,0)}{.3}{0}{0}{$x$};
\draw[dashed, thick, rounded corners=5pt] (.5,-.5) rectangle (1.3,.5);
}
=
\|d_Q\|\cdot
\tikzmath{
\begin{scope}
\clip[rounded corners=5pt] (-.4,-.8) rectangle (.8,.8);
\fill[\BColor] (-.1,-.8) rectangle (1.5,.8);
\end{scope}
\draw[dashed] (-.1,-.8) -- (-.1,.8);
\draw (.1,.3) -- (.1,.5);
\draw (.1,-.3) -- (.1,-.5);
\draw (.5,-.8) -- (.5,.8);
\filldraw (.1,.5) circle (.05cm);
\filldraw (.1,-.5) circle (.05cm);
\roundNbox{unshaded}{(0,0)}{.3}{0}{0}{$x$};
}
\underset{\text{\ref{Z:Support}}}{\leq}
\|d_Q\|^2\cdot (\cE_B(x)\boxtimes \id_Q).
$$
We conclude that the expectation $E_B: |Q| \to B$ has finite Pimsner-Popa index.
Faithfulness follows immediately from the Pimnser-Popa inequality.
\end{proof}

\begin{rem}
The observant reader may notice that we get the Pimsner-Popa bound $\|d_Q\|^2$ instead of $\|d_Q\|$; one should expect the latter from subfactor theory.
The reason is that we have secretly embedded $|Q|$ into the Jones/Watatani \emph{basic construction} $\End_{\bbC-B}({}_\bbC B\boxtimes_B Q_B)$ (where the right $Q$-action is replaced by $B$), and noticed that the conditional expectation down to $B$ restricts to the correct expectation on the image of $|Q| \cong \End_{\bbC-Q}({}_\bbC B\boxtimes_B Q_B)$.
Thus the best bound we can obtain in this way is $\|d_Q\|^2$ rather than $\|d_Q\|$.
\end{rem}

\subsection{Realization of 1-morphisms between Q-systems}

Suppose $A,B$ are unital $\rm C^*$-algebras and ${}_AX_B\in \rCorr(A\to B)$.
We denote the algebras $A,B$ by shaded regions
and $X$ by a \textcolor{red}{red} string respectively:
$$
\tikzmath{\filldraw[\AColor, rounded corners=5, very thin, baseline=1cm] (0,0) rectangle (.6,.6);}=A
\qquad\qquad
\tikzmath{\filldraw[\BColor, rounded corners=5, very thin, baseline=1cm] (0,0) rectangle (.6,.6);}=B
\qquad\qquad
\tikzmath{
\begin{scope}
\clip[rounded corners=5pt] (-.3,0) rectangle (.3,.6);
\fill[\AColor] (0,0) rectangle (-.3,.6);
\fill[\BColor] (0,0) rectangle (.3,.6);
\end{scope}
\draw[thick, \XColor] (0,0) -- (0,.6);
}=X.
$$
Suppose now that $P\in \rCorr(A\to A)$ and $Q\in \rCorr(B\to B)$ are Q-systems, and $X:P\to Q$ is a 1-morphism in $\QSys(\rCorr)$.
We denote the left $P$- and right $Q$-action morphisms graphically by \textcolor{red}{red} trivalent vertices:
$$
\lambda_X
=
\tikzmath{
\begin{scope}
\clip[rounded corners=5pt] (-.7,0) rectangle (.4,1);
\fill[\AColor] (-.7,0) rectangle (0,1);
\fill[\BColor] (0,0) rectangle (.4,1);
\end{scope}
\draw[thick, \XColor] (0,0) -- (0,1);
\draw (-.5,0) arc (180:90:.5cm);
\filldraw[\XColor] (0,.5) circle (.05cm);
}
\qquad\qquad
\rho_X
=
\tikzmath{
\begin{scope}
\clip[rounded corners=5pt] (-.4,0) rectangle (.7,1);
\fill[\AColor] (-.4,0) rectangle (0,1);
\fill[\BColor] (0,0) rectangle (.7,1);
\end{scope}
\draw[thick, \XColor] (0,0) -- (0,1);
\draw (.5,0) arc (0:90:.5cm);
\filldraw[\XColor] (0,.5) circle (.05cm);
}\,.
$$
We suppress the labellings of the $P$ and $Q$ strings, which can be inferred from whether they act from the left or the right on $X$.

As above,
${}_{\bbC}A\boxtimes_A X_B$ 
is the right $B$-module obtained from $X$ by forgetting the left $A$-action. 
We denote ${}_{\bbC}A_A$ by a dashed string which is shaded by $A$ on the right.
$$
\tikzmath{
\begin{scope}
\clip[rounded corners=5pt] (-.2,0) rectangle (.5,1);
\fill[\AColor] (0,0) rectangle (.5,1);
\end{scope}
\draw[dashed] (0,0) -- (0,1);
}
={}_{\bbC}A{}_{A}
\qquad\qquad
\tikzmath{
\begin{scope}
\clip[rounded corners=5pt] (-.2,0) rectangle (.5,1);
\fill[\AColor] (0,0) rectangle (.2,1);
\fill[\BColor] (.2,0) rectangle (.7,1);
\end{scope}
\draw[dashed] (0,0) -- (0,1);
\draw[thick, \XColor] (.2,0) -- (.2,1);
}
={}_{\bbC}A\boxtimes_A X_{B}
\qquad\qquad
\tikzmath{
\begin{scope}
\clip[rounded corners=5pt] (-.2,0) rectangle (.5,1);
\fill[\BColor] (0,0) rectangle (.5,1);
\end{scope}
\draw[dashed] (0,0) -- (0,1);
}
={}_{\bbC}B{}_{B}.
$$
Observe that ${}_{\bbC}A\boxtimes_A X_B$
has an obvious right $Q$-action, but there is no longer an obvious left $P$-action.

\begin{construction}
\label{construction:|X|}
Following Construction \ref{construction:BoundedVectorModule},
we define $|X|:= \Hom_{\bbC-B}({}_{\bbC}B_B\to {}_{\bbC}A\boxtimes_A X_B)$,
whose elements are represented graphically by
\[
\tikzmath{
\begin{scope}
\clip[rounded corners=5pt] (-.4,-.7) rectangle (.7,.7);
\fill[\AColor] (-.1,0) rectangle (.1,.7);
\fill[\BColor] (0,-.7) -- (0,0) -- (.1,0) -- (.1,.7) -- (.8,.7) -- (.8,-.7);
\end{scope}
\draw[dashed] (0,-.7) -- (0,0) -- (-.1,0) -- (-.1,.7);
\draw[thick, \XColor] (.1,.3) -- (.1,.7);
\roundNbox{unshaded}{(0,0)}{.3}{0}{0}{$\xi$};
}
\in |X|.
\]
We endow $|X|$ with a $|P|$-$|Q|$ bimodule structure by
\begin{equation*}
p \rhd \xi
:=
\tikzmath{
\begin{scope}
\clip[rounded corners=5pt] (-.4,-2) rectangle (1.2,1);
\fill[\AColor] (.4,-1) .. controls ++(90:.3cm) and ++(270:.3cm) .. (0,-.5) -- (-.1,.1) -- (-.1,1) -- (.6,1) -- (.6,-1);
\fill[\BColor] (.5,-2) -- (.5,-1.3) -- (.6,-1) -- (.6,1) -- (1.2,1) -- (1.2,-2);
\end{scope}
\draw (.1,.1) arc (180:90:.5cm);
\draw[thick, \XColor] (.6,-1) -- (.6,1);
\draw[dashed] (.5,-2) -- (.5,-1.3) -- (.4,-1) .. controls ++(90:.3cm) and ++(270:.3cm) .. (0,-.5) -- (-.1,.1) -- (-.1,1);
\filldraw[\XColor] (.6,.6) circle (.05cm);
\roundNbox{unshaded}{(0,-.2)}{.3}{0}{0}{$p$};
\roundNbox{unshaded}{(.5,-1.3)}{.3}{0}{0}{$\xi$};
}
\qquad\qquad
\xi \lhd q
:=
\tikzmath{
\begin{scope}
\clip[rounded corners=5pt] (-.4,-2) rectangle (1.2,1);
\fill[\AColor] (-.1,.3) rectangle (.1,1);
\fill[\BColor] (.5,-2) -- (.5,-1.3) -- (.4,-1) .. controls ++(90:.3cm) and ++(270:.3cm) .. (0,-.3) -- (.1,.3) -- (.1,1) -- (1.2,1) -- (1.2,-2);
\end{scope}
\draw[thick, \XColor] (.1,.3) -- (.1,1);
\draw (.6,-1) -- (.6,.1) arc (0:90:.5cm);
\draw[dashed] (.5,-2) -- (.5,-1.3) -- (.4,-1) .. controls ++(90:.3cm) and ++(270:.3cm) .. (0,-.3) -- (-.1,.3) -- (-.1,1);
\filldraw[\XColor] (.1,.6) circle (.05cm);
\roundNbox{unshaded}{(0,0)}{.3}{0}{0}{$\xi$};
\roundNbox{unshaded}{(.5,-1.3)}{.3}{0}{0}{$q$};
}
\qquad\qquad
\begin{aligned}
&\forall\, f\in |P|,
\\&
\forall\, \eta\in |M|,\text{ and}
\\&
\forall\, g\in |Q|.
\end{aligned}
\end{equation*}
It is a straightforward exercise using the graphical calculus to prove associativity and unitality for the $|P|$ and $|Q|$ actions.
\end{construction}

\begin{lem}
The $|Q|$-valued sesquilinear form on $|X|$ given by
\[
\langle \eta|\xi\rangle_{|Q|}:= 
\tikzmath{
\begin{scope}
\clip[rounded corners=5pt] (-.4,-1.2) rectangle (1,1.2);
\fill[\AColor] (-.1,-.5) rectangle (.1,.5);
\fill[\BColor] (0,-1.2) -- (0,-.7) -- (.1,-.7) -- (.1,.7) -- (0,.7) -- (0,1.2) -- (1,1.2) -- (1,-1.2);
\end{scope}
\draw[dashed] (0,-1.2) -- (0,-.7) -- (-.1,-.7) -- (-.1,.7) -- (0,.7) -- (0,1.2);
\draw[thick, \XColor] (.1,-.5) -- (.1,.5);
\draw (.6,1.2) -- (.6,.5) arc (0:-90:.5cm);
\filldraw[\XColor] (.1,0) circle (.05cm);
\roundNbox{unshaded}{(0,.5)}{.3}{0}{0}{$\eta^\dag$};
\roundNbox{unshaded}{(0,-.5)}{.3}{0}{0}{$\xi$};
}
\in \Hom_{\bbC-B}({}_{\bbC}B_B \to {}_{\bbC} B\boxtimes_B Q_B) = |Q|
\]
is a $|Q|$-valued inner product.
\end{lem}
\begin{proof}
We need to prove right $|Q|$-linearity, positivity (which implies self-adjointness by polarization), and definiteness.
\item[\underline{Right $|Q|$-linear:}]
$
\langle \eta|\xi\lhd q\rangle_{|Q|} 
=
\tikzmath{
\begin{scope}
\clip[rounded corners=5pt] (-.4,-2) rectangle (1.2,1.9);
\fill[\AColor] (-.1,0) rectangle (.1,1.2);
\fill[\BColor] (.5,-2) -- (.5,-1.3) -- (.4,-1) .. controls ++(90:.3cm) and ++(270:.3cm) .. (0,-.3) -- (.1,.3) -- (.1,1) -- (0,1) -- (0,1.9) -- (1.2,1.9) -- (1.2,-2);
\end{scope}
\draw[thick, \XColor] (.1,.3) -- (.1,1);
\draw (.6,-1) -- (.6,0) arc (0:90:.5cm);
\draw (.6,1.9) -- (.6,1.2) arc (0:-90:.5cm);
\draw[dashed] (.5,-2) -- (.5,-1.3) -- (.4,-1) .. controls ++(90:.3cm) and ++(270:.3cm) .. (0,-.3) -- (-.1,.3) -- (-.1,1) -- (0,1) -- (0,1.9);
\filldraw[\XColor] (.1,.5) circle (.05cm);
\filldraw[\XColor] (.1,.7) circle (.05cm);
\roundNbox{unshaded}{(0,1.2)}{.3}{0}{0}{$\eta^\dag$};
\roundNbox{unshaded}{(0,0)}{.3}{0}{0}{$\xi$};
\roundNbox{unshaded}{(.5,-1.3)}{.3}{0}{0}{$q$};
}
\underset{\text{\ref{Q:Frobenius}}}{=}
\tikzmath{
\begin{scope}
\clip[rounded corners=5pt] (-.4,-2) rectangle (1.7,1.9);
\fill[\AColor] (-.1,0) rectangle (.1,1.2);
\fill[\BColor] (1,-2) -- (1,-1.3) -- (.9,-1) .. controls ++(90:.3cm) and ++(270:.3cm) .. (0,-.3) -- (.1,.3) -- (.1,1) -- (0,1) -- (0,1.9) -- (1.7,1.9) -- (1.7,-2);
\end{scope}
\draw[thick, \XColor] (.1,.3) -- (.1,1);
\draw (1.1,-1) -- (1.1,1) arc (0:90:.5cm);
\draw (.6,1.9) -- (.6,1) arc (0:-90:.5cm);
\draw[dashed] (1,-2) -- (1,-1.3) -- (.9,-1) .. controls ++(90:.3cm) and ++(270:.3cm) .. (0,-.3) -- (-.1,.3) -- (-.1,1) -- (0,1) -- (0,1.9);
\filldraw[\XColor] (.1,.5) circle (.05cm);
\filldraw (.6,1.5) circle (.05cm);
\roundNbox{unshaded}{(0,1)}{.3}{0}{0}{$\eta^\dag$};
\roundNbox{unshaded}{(0,0)}{.3}{0}{0}{$\xi$};
\roundNbox{unshaded}{(1,-1.3)}{.3}{0}{0}{$q$};
}
=
\langle \eta|\xi\rangle_{|Q|}\lhd q
$
\item[\underline{Positive:}]
To show that $\langle \xi|\xi\rangle_{|Q|}\geq 0$ for all $\xi\in |X|$,
we use the unital $*$-algebra isomorphism
$\Phi: |Q| \to \End_{-Q}({}_{\bbC}B\boxtimes_B Q_B)$ from
\eqref{eq:CStarIso}.
For all $\xi\in |X|$,
\[
\Phi(\langle \xi|\xi\rangle_{|Q|})
= 
\tikzmath{
\begin{scope}
\clip[rounded corners=5pt] (-.4,-1.2) rectangle (1.6,1.6);
\fill[\AColor] (-.1,-.7) rectangle (.1,.7);
\fill[\BColor] (0,-1.2) -- (0,-.7) -- (.1,-.7) -- (.1,.7) -- (0,.7) -- (0,1.6) -- (1.6,1.6) -- (1.6,-1.6);
\end{scope}
\draw[dashed] (0,-1.2) -- (0,-.7) -- (-.1,-.7) -- (-.1,.7) -- (0,.7) -- (0,1.6);
\draw[thick, \XColor] (.1,-.5) -- (.1,.5);
\draw (.1,0) arc (-90:0:.5cm) -- (.6,.8) arc (180:0:.3cm) -- (1.2,-1.2);
\draw (.9,1.1) -- (.9,1.6);
\filldraw[\XColor] (.1,0) circle (.05cm);
\filldraw (.9,1.1) circle (.05cm);
\roundNbox{unshaded}{(0,.5)}{.3}{0}{0}{$\xi^\dag$};
\roundNbox{unshaded}{(0,-.5)}{.3}{0}{0}{$\xi$};
}
=
\tikzmath{
\begin{scope}
\clip[rounded corners=5pt] (-.4,-1.4) rectangle (1,1.2);
\fill[\AColor] (-.1,-.7) rectangle (.1,.7);
\fill[\BColor] (0,-1.4) -- (0,-.7) -- (.1,-.7) -- (.1,.7) -- (0,.7) -- (0,1.2) -- (1,1.2) -- (1,-1.4);
\end{scope}
\draw[dashed] (0,-1.4) -- (0,-.7) -- (-.1,-.7) -- (-.1,.7) -- (0,.7) -- (0,1.2);
\draw[thick, \XColor] (.1,-.5) -- (.1,.5);
\draw (.6,1.2) -- (.6,.5) arc (0:-90:.5cm);
\draw (.6,-1.4) -- (.6,-.7) arc (0:90:.5cm);
\filldraw[\XColor] (.1,0) circle (.05cm);
\filldraw[\XColor] (.1,-.2) circle (.05cm);
\roundNbox{unshaded}{(0,.5)}{.3}{0}{0}{$\xi^\dag$};
\roundNbox{unshaded}{(0,-.7)}{.3}{0}{0}{$\xi$};
}
\ge 0.
\]
Hence $\langle \xi|\xi\rangle_{|Q|}\geq 0$ as $\Phi$ is a unital $*$-isomorphism.

\item[\underline{Definite:}]
Suppose $\langle \xi|\xi\rangle_{|Q|}= 0$.
Again applying the isomorphism $\Phi$ and using the fact that $\rCorr$ is $\rm C^*$, we have
\begin{equation*}
\tikzmath{
\begin{scope}
\clip[rounded corners=5pt] (-.4,-1.4) rectangle (1,1.2);
\fill[\AColor] (-.1,-.7) rectangle (.1,.7);
\fill[\BColor] (0,-1.4) -- (0,-.7) -- (.1,-.7) -- (.1,.7) -- (0,.7) -- (0,1.2) -- (1,1.2) -- (1,-1.4);
\end{scope}
\draw[dashed] (0,-1.4) -- (0,-.7) -- (-.1,-.7) -- (-.1,.7) -- (0,.7) -- (0,1.2);
\draw[thick, \XColor] (.1,-.5) -- (.1,.5);
\draw (.6,1.2) -- (.6,.5) arc (0:-90:.5cm);
\draw (.6,-1.4) -- (.6,-.7) arc (0:90:.5cm);
\filldraw[\XColor] (.1,0) circle (.05cm);
\filldraw[\XColor] (.1,-.2) circle (.05cm);
\roundNbox{unshaded}{(0,.5)}{.3}{0}{0}{$\xi^\dag$};
\roundNbox{unshaded}{(0,-.7)}{.3}{0}{0}{$\xi$};
}
=
0
\qquad
\Longrightarrow
\qquad
\tikzmath{
\begin{scope}
\clip[rounded corners=5pt] (-.4,-1.4) rectangle (1,.5);
\fill[\AColor] (-.1,-.7) rectangle (.1,.5);
\fill[\BColor] (0,-1.4) -- (0,-.7) -- (.1,-.7) -- (.1,.5) -- (1,.5) -- (1,-1.4);
\end{scope}
\draw[dashed] (0,-1.4) -- (0,-.7) -- (-.1,-.7) -- (-.1,.5);
\draw[thick, \XColor] (.1,-.5) -- (.1,.5);
\draw (.6,-1.4) -- (.6,-.6) arc (0:90:.5cm);
\filldraw[\XColor] (.1,-.1) circle (.05cm);
\roundNbox{unshaded}{(0,-.7)}{.3}{0}{0}{$\xi$};
}
=
0
\qquad
\Longrightarrow
\qquad
\tikzmath{
\begin{scope}
\clip[rounded corners=5pt] (-.4,-.7) rectangle (.7,.7);
\fill[\AColor] (-.1,0) rectangle (.1,.7);
\fill[\BColor] (0,-.7) -- (0,0) -- (.1,0) -- (.1,.7) -- (.8,.7) -- (.8,-.7);
\end{scope}
\draw[dashed] (0,-.7) -- (0,0) -- (-.1,0) -- (-.1,.7);
\draw[thick, \XColor] (.1,.3) -- (.1,.7);
\roundNbox{unshaded}{(0,0)}{.3}{0}{0}{$\xi$};
}
=
\tikzmath{
\begin{scope}
\clip[rounded corners=5pt] (-.4,-1.4) rectangle (1,.5);
\fill[\AColor] (-.1,-.7) rectangle (.1,.5);
\fill[\BColor] (0,-1.4) -- (0,-.7) -- (.1,-.7) -- (.1,.5) -- (1,.5) -- (1,-1.4);
\end{scope}
\draw[dashed] (0,-1.4) -- (0,-.7) -- (-.1,-.7) -- (-.1,.5);
\draw[thick, \MColor] (.1,-.5) -- (.1,.5);
\draw (.6,-1.2) -- (.6,-.6) arc (0:90:.5cm);
\filldraw[\MColor] (.1,-.1) circle (.05cm);
\filldraw (.6,-1.2) circle (.05cm);
\roundNbox{unshaded}{(0,-.7)}{.3}{0}{0}{$\xi$};
}
=
0.
\qedhere
\end{equation*}
\end{proof}

\begin{prop}
The space $|X|$ equipped with its right $|Q|$-valued inner product is a right $|P|-|Q|$ correspondence.
\end{prop}
\begin{proof}
We must check that that the left $|P|$-action is by adjointable operators and that $|X|$ is complete in the norm induced from $|Q|$.
\item[\underline{Adjointable $|P|$-action:}]
Observe that for all $\eta,\xi \in |X|$ and $p\in |P|$,
\[
\langle p\rhd \eta|\xi\rangle_{|Q|}
=
\tikzmath{
\begin{scope}
\clip[rounded corners=5pt] (-.4,-2) rectangle (1.7,2);
\fill[\AColor] (.4,-1) .. controls ++(90:.3cm) and ++(270:.3cm) .. (-.1,-.1) -- (0,.5) .. controls ++(90:.3cm) and ++(270:.3cm) .. (.4,1) -- (.6,1) -- (.6,-1);
\fill[\BColor] (.5,-2) -- (.5,-1) -- (.6,-1) -- (.6,1) -- (.5,1)-- (.5,2) -- (1.7,2) -- (1.7,-2);
\end{scope}
\draw (.1,-.1) arc (180:270:.5cm);
\draw (.6,-.8) arc (-90:0:.5cm) -- (1.1,2);
\draw[thick, \XColor] (.6,-1) -- (.6,1);
\draw[dashed] (.5,-2) -- (.5,-1) -- (.4,-1) .. controls ++(90:.3cm) and ++(270:.3cm) .. (-.1,-.1) -- (0,.5) .. controls ++(90:.3cm) and ++(270:.3cm) .. (.4,1) -- (.5,1)-- (.5,2);
\filldraw[\XColor] (.6,-.6) circle (.05cm);
\filldraw[\XColor] (.6,-.8) circle (.05cm);
\roundNbox{unshaded}{(0,.2)}{.3}{0}{0}{$p^\dag$};
\roundNbox{unshaded}{(.5,1.3)}{.3}{0}{0}{$\eta^\dag$};
\roundNbox{unshaded}{(.5,-1.3)}{.3}{0}{0}{$\xi$};
}
\underset{\text{\ref{QSys:StarClosed}}}{=}
\tikzmath{
\begin{scope}
\clip[rounded corners=5pt] (-.4,-2) rectangle (1.7,2);
\fill[\AColor] (.4,-1) .. controls ++(90:.3cm) and ++(270:.3cm) .. (-.1,-.1) -- (0,.5) .. controls ++(90:.3cm) and ++(270:.3cm) .. (.4,1) -- (.6,1) -- (.6,-1);
\fill[\BColor] (.5,-2) -- (.5,-1) -- (.6,-1) -- (.6,1) -- (.5,1)-- (.5,2) -- (1.7,2) -- (1.7,-2);
\end{scope}
\draw (.1,-.1) arc (180:270:.5cm);
\draw (.6,.8) arc (-90:0:.5cm) -- (1.1,2);
\draw[thick, \MColor] (.6,-1) -- (.6,1);
\draw[dashed] (.5,-2) -- (.5,-1) -- (.4,-1) .. controls ++(90:.3cm) and ++(270:.3cm) .. (-.1,-.1) -- (0,.5) .. controls ++(90:.3cm) and ++(270:.3cm) .. (.4,1) -- (.5,1)-- (.5,2);
\filldraw[\MColor] (.6,-.6) circle (.05cm);
\filldraw[\MColor] (.6,.8) circle (.05cm);
\roundNbox{unshaded}{(0,.2)}{.3}{0}{0}{$p^\dag$};
\roundNbox{unshaded}{(.5,1.3)}{.3}{0}{0}{$\eta^\dag$};
\roundNbox{unshaded}{(.5,-1.3)}{.3}{0}{0}{$\xi$};
}
\underset{\text{\ref{QSys:AdjointAction}}}{=}
\tikzmath{
\begin{scope}
\clip[rounded corners=5pt] (-.4,-2) rectangle (1.7,2);
\fill[\AColor] (.4,-1) .. controls ++(90:.3cm) and ++(270:.3cm) .. (-.38,-.5) -- (-.4,-.1) -- (-.3,.5) .. controls ++(90:.3cm) and ++(270:.3cm) .. (.4,1) -- (.6,1) -- (.6,-1);
\fill[\BColor] (.5,-2) -- (.5,-1) -- (.6,-1) -- (.6,1) -- (.5,1)-- (.5,2) -- (1.7,2) -- (1.7,-2);
\end{scope}
\draw (.6,.6) arc (90:180:.2cm) -- (.4,-.1) arc (0:-180:.3cm);
\draw (.6,.8) arc (-90:0:.5cm) -- (1.1,2);
\draw[thick, \MColor] (.6,-1) -- (.6,1);
\draw[dashed] (.5,-2) -- (.5,-1) -- (.4,-1) .. controls ++(90:.3cm) and ++(270:.3cm) .. (-.38 ,-.5) -- (-.4,-.1) -- (-.3,.5) .. controls ++(90:.3cm) and ++(270:.3cm) .. (.4,1) -- (.5,1)-- (.5,2);
\draw (.1,-.4) -- (.1,-.6);
\filldraw[\MColor] (.6,.6) circle (.05cm);
\filldraw[\MColor] (.6,.8) circle (.05cm);
\filldraw (.1,-.4) circle (.05cm);
\filldraw (.1,-.6) circle (.05cm);
\roundNbox{unshaded}{(-.3,.2)}{.3}{0}{0}{$p^\dag$};
\roundNbox{unshaded}{(.5,1.3)}{.3}{0}{0}{$\eta^\dag$};
\roundNbox{unshaded}{(.5,-1.3)}{.3}{0}{0}{$\xi$};
}
=
\langle \eta|p^*\rhd \xi\rangle_{|Q|}.
\]

\item[\underline{Complete:}]
By definition, $|X|$ is complete in the norm induced from the right $B$-valued inner product given by $\langle \eta|\xi\rangle_B^{|X|} = \eta^\dag \circ \xi \in \End_{\bbC-B}(B)=B$ from Construction \ref{construction:BoundedVectorModule}.
We prove this norm is equivalent to the norm induced by the right $|Q|$-valued inner product.
First, by \ref{Z:Support} and \ref{M:unitality}, we have the identities
$$
E_B(\langle\eta|\xi\rangle^{|X|}_{|Q|})
=
d_Q^{-1}\cdot \langle \eta|\xi\rangle^{|X|}_B
\qquad\text{and}\qquad
d_Q\cdot E_B(\langle\eta|\xi\rangle^{|X|}_{|Q|})
=
\langle \eta|\xi\rangle^{|X|}_B,
$$
where $E_B: |Q| \to B$ is the finite index conditional expectation from Definition \ref{defn:ConditionalExpectation}.
Since $E_B$ is finite index,
there is a $c>0$ such that
$q \leq c \cdot E_B(q)$ 
for all $q\in Q_+$.
Since 
$\|E_B(q)\|_B \leq \|q\|_{|Q|}$ 
for all $q\in |Q|$, for all $\xi\in |X|$,
\begin{align*}
\|\xi\|_{|Q|}^2
&=
\|\langle \xi|\xi\rangle^{|X|}_{|Q|}\|_{|Q|}
\leq
c\cdot
\|E_B(\langle\xi|\xi\rangle^{|X|}_{|Q|})\|_B
=
c\cdot
\|d_Q^{-1}\cdot \langle \xi|\xi\rangle^{|X|}_B\|_B
\\&\leq
c\cdot \|d_Q^{-1}\|\cdot 
\underbrace{
\|\langle \xi|\xi\rangle^{|X|}_B\|_B
}_{=\|\xi\|^2_B}
=
c\cdot \|d_Q^{-1}\|\cdot 
\|d_Q\cdot E_B(\langle\xi|\xi\rangle^{|X|}_{|Q|})\|_{B}
\\&\leq
c\cdot \|d_Q^{-1}\|\cdot \|d_Q\|\cdot \|\langle\xi|\xi\rangle^{|X|}_{|Q|}\|_{|Q|}
=
c\cdot \|d_Q^{-1}\|\cdot \|d_Q\|\cdot \|\xi\|^2_{|Q|}.
\qedhere
\end{align*}
\end{proof}

\subsection{Realization of 2-morphisms}

Suppose $A,B$ are unital $\rm C^*$-algebras
and  ${}_AX_B, {}_AY_B, {}_AZ_B\in \rCorr(A\to B)$.
We denote $A,B$ as above by shaded regions and ${}_AX_B, {}_AY_B, {}_AZ_B$ by colored strands:
$$
\tikzmath{\filldraw[\AColor, rounded corners=5, very thin, baseline=1cm] (0,0) rectangle (.6,.6);}=A
\qquad\qquad
\tikzmath{\filldraw[\BColor, rounded corners=5, very thin, baseline=1cm] (0,0) rectangle (.6,.6);}=B.
\qquad
\tikzmath{
\begin{scope}
\clip[rounded corners=5pt] (-.3,0) rectangle (.3,.6);
\fill[\AColor] (0,0) rectangle (-.3,.6);
\fill[\BColor] (0,0) rectangle (.3,.6);
\end{scope}
\draw[thick, \XColor] (0,0) -- (0,.6);
}=X
\qquad\qquad
\tikzmath{
\begin{scope}
\clip[rounded corners=5pt] (-.3,0) rectangle (.3,.6);
\fill[\AColor] (0,0) rectangle (-.3,.6);
\fill[\BColor] (0,0) rectangle (.3,.6);
\end{scope}
\draw[thick, \YColor] (0,0) -- (0,.6);
}=Y
\qquad\qquad
\tikzmath{
\begin{scope}
\clip[rounded corners=5pt] (-.3,0) rectangle (.3,.6);
\fill[\AColor] (0,0) rectangle (-.3,.6);
\fill[\BColor] (0,0) rectangle (.3,.6);
\end{scope}
\draw[thick, \ZColor] (0,0) -- (0,.6);
}=Z.
$$

\begin{construction}
\label{defn:F2Mor}
Suppose $f\in \rCorr({}_AX_B \Rightarrow {}_AY_B)$ is a $P-Q$ bimodule intertwiner.
We define $|f|:|X| \to |Y|$ by
$$
|f|
\left(
\tikzmath{
\begin{scope}
\clip[rounded corners=5pt] (-.4,-.7) rectangle (.7,.7);
\fill[\AColor] (-.1,0) rectangle (.1,.7);
\fill[\BColor] (0,-.7) -- (0,0) -- (.1,0) -- (.1,.7) -- (.8,.7) -- (.8,-.7);
\end{scope}
\draw[dashed] (0,-.7) -- (0,0) -- (-.1,0) -- (-.1,.7);
\draw[thick, \MColor] (.1,.3) -- (.1,.7);
\roundNbox{unshaded}{(0,0)}{.3}{0}{0}{$\xi$};
}
\right)
:=
\tikzmath{
\begin{scope}
\clip[rounded corners=5pt] (-.6,-.7) rectangle (1,1.7);
\fill[\AColor] (-.3,0) rectangle (.3,1.7);
\fill[\BColor] (0,-.7) -- (0,0) -- (.3,0) -- (.3,1.7) -- (1,1.7) -- (1,-.7);
\end{scope}
\draw[dashed] (0,-.7) -- (0,0) -- (-.3,0) -- (-.3,1.7);
\draw[thick, \XColor] (.3,.3) -- (.3,.7);
\draw[thick, \YColor] (.3,1.3) -- (.3,1.7);
\roundNbox{unshaded}{(0,0)}{.3}{.2}{.2}{$\xi$};
\roundNbox{unshaded}{(.3,1)}{.3}{0}{0}{$f$};
}
\in |Y|.
$$
Since $f$ is $P-Q$ bimodular, it follows that $|f|$ is $|P|-|Q|$ bimodular.
\end{construction}

\begin{lem}
For $f\in \rCorr({}_{A}X_B \Rightarrow {}_{A}Y_B)$
and
$g\in\rCorr({}_{A}Y_B \Rightarrow {}_{A}Z_B)$ 
both $P-Q$ bimodular,
$|g\circ f| = |g|\circ |f|$.
\end{lem}
\begin{proof}
We observe that for all $\xi \in |X|$,
$$
|g\circ f|
\left(
\tikzmath{
\begin{scope}
\clip[rounded corners=5pt] (-.4,-.7) rectangle (.7,.7);
\fill[\AColor] (-.1,0) rectangle (.1,.7);
\fill[\BColor] (0,-.7) -- (0,0) -- (.1,0) -- (.1,.7) -- (.8,.7) -- (.8,-.7);
\end{scope}
\draw[dashed] (0,-.7) -- (0,0) -- (-.1,0) -- (-.1,.7);
\draw[thick, \MColor] (.1,.3) -- (.1,.7);
\roundNbox{unshaded}{(0,0)}{.3}{0}{0}{$\xi$};
}\,
\right)
=
\tikzmath{
\begin{scope}
\clip[rounded corners=5pt] (-.7,-.7) rectangle (1.1,1.7);
\fill[\AColor] (-.4,0) rectangle (.4,1.7);
\fill[\BColor] (0,-.7) -- (0,0) -- (.4,0) -- (.4,1.7) -- (1.6,1.7) -- (1.6,-.7);
\end{scope}
\draw[dashed] (0,-.7) -- (0,0) -- (-.4,0) -- (-.4,1.7);
\draw[thick, \XColor] (.4,.3) -- (.4,.7);
\draw[thick, \ZColor] (.4,1.3) -- (.4,1.7);
\roundNbox{unshaded}{(0,0)}{.3}{.3}{.3}{$\xi$};
\roundNbox{unshaded}{(.4,1)}{.3}{.2}{.2}{$g\circ f$};
}
=
\tikzmath{
\begin{scope}
\clip[rounded corners=5pt] (-.6,-.7) rectangle (1,2.7);
\fill[\AColor] (-.3,0) rectangle (.3,2.7);
\fill[\BColor] (0,-.7) -- (0,0) -- (.3,0) -- (.3,2.7) -- (1,2.7) -- (1,-.7);
\end{scope}
\draw[dashed] (0,-.7) -- (0,0) -- (-.3,0) -- (-.3,2.7);
\draw[thick, \XColor] (.3,.3) -- (.3,.7);
\draw[thick, \YColor] (.3,1.3) -- (.3,1.7);
\draw[thick, \ZColor] (.3,2.3) -- (.3,2.7);
\roundNbox{unshaded}{(0,0)}{.3}{.2}{.2}{$\xi$};
\roundNbox{unshaded}{(.3,1)}{.3}{0}{0}{$f$};
\roundNbox{unshaded}{(.3,2)}{.3}{0}{0}{$g$};
}
=
|g|\left(
\tikzmath{
\begin{scope}
\clip[rounded corners=5pt] (-.6,-.7) rectangle (1,1.7);
\fill[\AColor] (-.3,0) rectangle (.3,1.7);
\fill[\BColor] (0,-.7) -- (0,0) -- (.3,0) -- (.3,1.7) -- (1,1.7) -- (1,-.7);
\end{scope}
\draw[dashed] (0,-.7) -- (0,0) -- (-.3,0) -- (-.3,1.7);
\draw[thick, \XColor] (.3,.3) -- (.3,.7);
\draw[thick, \YColor] (.3,1.3) -- (.3,1.7);
\roundNbox{unshaded}{(0,0)}{.3}{.2}{.2}{$\xi$};
\roundNbox{unshaded}{(.3,1)}{.3}{0}{0}{$f$};
}\,
\right)
=
(|g|\circ|f|)\left(
\tikzmath{
\begin{scope}
\clip[rounded corners=5pt] (-.4,-.7) rectangle (.7,.7);
\fill[\AColor] (-.1,0) rectangle (.1,.7);
\fill[\BColor] (0,-.7) -- (0,0) -- (.1,0) -- (.1,.7) -- (.8,.7) -- (.8,-.7);
\end{scope}
\draw[dashed] (0,-.7) -- (0,0) -- (-.1,0) -- (-.1,.7);
\draw[thick, \MColor] (.1,.3) -- (.1,.7);
\roundNbox{unshaded}{(0,0)}{.3}{0}{0}{$\xi$};
}\,
\right).
$$
The claim follows.
\end{proof}

\begin{lem}
For $f\in \rCorr( {}_{A}X_B \Rightarrow {}_{A}Y_B)$ $P-Q$ bimodular, 
$|f|: |X| \to |Y|$ 
is adjointable with respect to the right $|Q|$-valued inner products with $|f|^* = |f^\dag|$.
\end{lem}
\begin{proof}
For all $\xi \in |X|$ and $\eta\in |Y|$,
$
\langle \eta| |f|\xi \rangle^{|Y|}_{|Q|}
=
\tikzmath{
\begin{scope}
\clip[rounded corners=5pt] (-.6,-.7) rectangle (1.1,2.9);
\fill[\AColor] (-.3,0) rectangle (.3,2.2);
\fill[\BColor] (0,-.7) -- (0,0) -- (.3,0) -- (.3,2.2) -- (0,2.2) -- (0,2.9) -- (1.1,2.9) -- (1.1,-.7);
\end{scope}
\draw[dashed] (0,-.7) -- (0,0) -- (-.3,0) -- (-.3,2) -- (0,2.2) -- (0,2.9);
\draw[thick, \XColor] (.3,.3) -- (.3,.7);
\draw[thick, \YColor] (.3,1.3) -- (.3,2.2);
\draw (.3,1.6) arc (-90:0:.5cm) -- (.8,2.9);
\filldraw[\YColor] (.3,1.6) circle (.05cm);
\roundNbox{unshaded}{(0,0)}{.3}{.2}{.2}{$\xi$};
\roundNbox{unshaded}{(.3,1)}{.3}{0}{0}{$f$};
\roundNbox{unshaded}{(0,2.2)}{.3}{.2}{.2}{$\eta^\dag$};
}
\underset{\text{\ref{QSys:StarClosed}}}{=}
\tikzmath{
\begin{scope}
\clip[rounded corners=5pt] (-.6,-.7) rectangle (1.1,2.9);
\fill[\AColor] (-.3,0) rectangle (.3,2.2);
\fill[\BColor] (0,-.7) -- (0,0) -- (.3,0) -- (.3,2.2) -- (0,2.2) -- (0,2.9) -- (1.1,2.9) -- (1.1,-.7);
\end{scope}
\draw[dashed] (0,-.7) -- (0,0) -- (-.3,0) -- (-.3,2) -- (0,2.2) -- (0,2.9);
\draw[thick, \XColor] (.3,.3) -- (.3,1.2);
\draw[thick, \YColor] (.3,1.3) -- (.3,2.2);
\draw (.3,.6) arc (-90:0:.5cm) -- (.8,2.9);
\filldraw[\XColor] (.3,.6) circle (.05cm);
\roundNbox{unshaded}{(0,0)}{.3}{.2}{.2}{$\xi$};
\roundNbox{unshaded}{(.3,1.2)}{.3}{0}{0}{$f$};
\roundNbox{unshaded}{(0,2.2)}{.3}{.2}{.2}{$\eta^\dag$};
}
=
\langle |f^\dag|\eta| \xi \rangle^{|X|}_{|Q|}
$.
\end{proof}

\begin{rem}
Suppose ${}_AX_B,{}_AY_B \in\rCorr(A\to B)$.
Recall from Construction \ref{construction:BoundedVectorModule} that we have canonical unitary isomorphisms 
${}_AX_B \cong {}_A|X|_B$ and ${}_AY_B \cong {}_A|Y|_B$.
We claim that for all $f\in \rCorr({}_AX_B \Rightarrow {}_AY_B)$, the following diagram commutes:
\begin{equation}
\label{eq:BoundedVectorNaturalIso}
\begin{tikzcd}
{}_AX_B
\arrow[r, "f"]
\arrow[d, <->, "\cong"]
&
{}_AY_B
\arrow[d, <->, "\cong"]
\\
{}_A|X|_B
\arrow[r, "|f|"]
&
{}_A|Y|_B.
\end{tikzcd}
\end{equation}
This is easily seen from the following calculation.
Given $\xi \in X$,
going right then down yields
$L_{f\xi}: b \mapsto 1_A\boxtimes (f\xi)\lhd b$,
whereas going down then right yields
$|f|\circ L_\xi : b \mapsto (\id_A \boxtimes f)(1_A\boxtimes \xi\lhd b)$,
which are visibly equal maps in $|Y|=\Hom_{\bbC-B}({}_\bbC B_B \to {}_{\bbC}A\boxtimes_A Y_B)$.
\end{rem}

\subsection{Composition of 1-morphisms}

Suppose $A,B,C$ are unital $\rm C^*$-algebras,  and ${}_AX_B\in \rCorr(A\to B)$ and ${}_BY_C\in \rCorr(B\to C)$.
We denote the algebras $A,B,C$ by shaded regions
and $X,Y$ by \textcolor{red}{red} and \textcolor{orange}{orange} strings respectively:
$$
\tikzmath{\filldraw[\AColor, rounded corners=5, very thin, baseline=1cm] (0,0) rectangle (.6,.6);}=A
\qquad\qquad
\tikzmath{\filldraw[\BColor, rounded corners=5, very thin, baseline=1cm] (0,0) rectangle (.6,.6);}=B
\qquad\qquad
\tikzmath{\filldraw[\CColor, rounded corners=5, very thin, baseline=1cm] (0,0) rectangle (.6,.6);}=C
\qquad\qquad
\tikzmath{
\begin{scope}
\clip[rounded corners=5pt] (-.3,0) rectangle (.3,.6);
\fill[\AColor] (0,0) rectangle (-.3,.6);
\fill[\BColor] (0,0) rectangle (.3,.6);
\end{scope}
\draw[thick, \XColor] (0,0) -- (0,.6);
}=X
\qquad\qquad
\tikzmath{
\begin{scope}
\clip[rounded corners=5pt] (-.3,0) rectangle (.3,.6);
\fill[\BColor] (0,0) rectangle (-.3,.6);
\fill[\CColor] (0,0) rectangle (.3,.6);
\end{scope}
\draw[thick, \YColor] (0,0) -- (0,.6);
}=Y.
$$
Suppose now that $P\in \rCorr(A\to A)$, $Q\in \rCorr(B\to B)$, and $R\in \rCorr(C\to C)$ are Q-systems,
and $X$ is a $P-Q$ unital Frobenius bimodule object and $Y$ is a $Q-R$ unital Frobenius bimodule object.
As in Definition \ref{defn:QSysTensorProduct}, 
we denote the separability idempotent 
$p_{X,Y}\in \End_{A-C}(X\boxtimes_B Y)$ with respect to $Q$,
the object
$X\xzq_Q Y \in \rCorr(A\to C)$,
and the coisometry
$u_{X,Y}\in \Hom_{A-C}(X\boxtimes_B Y \to X\xzq_Q Y)$ by
$$
\tikzmath{
\begin{scope}
\clip[rounded corners = 5pt] (-.5,-.5) rectangle (.5,.5);
\filldraw[\AColor] (-.5,-.5) rectangle (-.2,.5);
\filldraw[\BColor] (-.2,-.5) rectangle (.2,.5);
\filldraw[\CColor] (.2,-.5) rectangle (.5,.5);
\end{scope}
\draw[thick, \XColor] (-.2,-.5) -- (-.2,.5);
\draw[thick] (-.2,0) -- (.2,0);
\draw[thick, \YColor] (.2,-.5) -- (.2,.5);
}
=
p_{X,Y}
\qquad\qquad
\tikzmath{
\begin{scope}
\clip[rounded corners = 5pt] (-.3,-.5) rectangle (.3,.5);
\filldraw[\AColor] (-.3,-.5) rectangle (-.04,.5);
\filldraw[\BColor] (-.04,-.5) rectangle (.04,.5);
\filldraw[\CColor] (.04,-.5) rectangle (.3,.5);
\end{scope}
\DoubleStrand{(0,-.5)}{(0,.5)}{\XColor}{\YColor}
}
=
X\xzq_Q Y
\qquad\qquad
\tikzmath{
\begin{scope}
\clip[rounded corners = 5pt] (-.5,-.7) rectangle (.5,.7);
\filldraw[\AColor] (-.5,-.7) -- (-.15,-.7) -- (-.15,0) -- (-.04,0) -- (-.04,.7) -- (-.5,.7);
\filldraw[\BColor] (-.15,-.7) -- (-.15,0) -- (-.04,0) -- (-.04,.7) -- (.04,.7) -- (.04,0) -- (.15,0) -- (.15,-.7);
\filldraw[\CColor] (.5,-.7) -- (.15,-.7) -- (.15,0) -- (.04,0) -- (.04,.7) -- (.5,.7);
\end{scope}
\draw[thick, \XColor] (-.15,-.7) -- (-.15,-.3);
\draw[thick, \YColor] (.15,-.7) -- (.15,-.3);
\DoubleStrand{(0,.3)}{(0,.7)}{\XColor}{\YColor}
\halfRoundBox{unshaded}{(0,0)}{.3}{0}{$u$};
}
=
u_{X,Y}.
$$

We now construct 
the tensorator $\mu$ for our 2-functor $\QSys(\rCorr) \to \rCorr$, i.e, an associative unitary natural isomorphism
$\mu= \{\mu_{X,Y}\in \Hom_{|P|-|R|}(|X|\boxtimes_{|Q|} |Y| \Rightarrow |X\xzq_Q Y|)\}_{X,Y}$.
Before we do so, we make the following remark.

\begin{rem}
Observe that the map
$$
\tikzmath{
\begin{scope}
\clip[rounded corners=5pt] (-.4,-.7) rectangle (.7,.7);
\fill[\AColor] (-.1,0) rectangle (.1,.7);
\fill[\BColor] (0,-.7) -- (0,0) -- (.1,0) -- (.1,.7) -- (.8,.7) -- (.8,-.7);
\end{scope}
\draw[dashed] (0,-.7) -- (0,0) -- (-.1,0) -- (-.1,.7);
\draw[thick, \MColor] (.1,.3) -- (.1,.7);
\roundNbox{unshaded}{(0,0)}{.3}{0}{0}{$\xi$};
}
\otimes_{\bbC}
\tikzmath{
\begin{scope}
\clip[rounded corners=5pt] (-.4,-.7) rectangle (.7,.7);
\fill[\BColor] (-.1,0) rectangle (.1,.7);
\fill[\CColor] (0,-.7) -- (0,0) -- (.1,0) -- (.1,.7) -- (.8,.7) -- (.8,-.7);
\end{scope}
\draw[dashed] (0,-.7) -- (0,0) -- (-.1,0) -- (-.1,.7);
\draw[thick, \YColor] (.1,.3) -- (.1,.7);
\roundNbox{unshaded}{(0,0)}{.3}{0}{0}{$\eta$};
}
\longmapsto
\tikzmath{
\begin{scope}
\clip[rounded corners=5pt] (-.4,-2) rectangle (1.2,.5);
\fill[\AColor] (-.1,.1) -- (-.1,.5) -- (.1,.5) -- (.1,.1);
\fill[\BColor] (.4,-1) .. controls ++(90:.3cm) and ++(270:.3cm) .. (0,-.5) -- (.1,-.5) -- (.1,.5) -- (.6,.5) -- (.6,-1);
\fill[\CColor] (.5,-2) -- (.5,-1.3) -- (.6,-1) -- (.6,.5) -- (1.2,.5) -- (1.2,-2);
\end{scope}
\draw[thick, \XColor] (.1,.1) -- (.1,.5);
\draw[thick, \YColor] (.6,-1) -- (.6,.5);
\draw[dashed] (.5,-2) -- (.5,-1.3) -- (.4,-1) .. controls ++(90:.3cm) and ++(270:.3cm) .. (0,-.5) -- (-.1,.1) -- (-.1,.5);
\roundNbox{unshaded}{(0,-.2)}{.3}{0}{0}{$\xi$};
\roundNbox{unshaded}{(.5,-1.3)}{.3}{0}{0}{$\eta$};
}\,.
$$
descends to a unitary isomorphism $T_{X,Y}:|X|\boxtimes_B |Y| \to |X\boxtimes_B Y|$.
Indeed, it is $B$-middle linear, it preserves the $B$-valued inner products
$$
\langle\xi_1\boxtimes \eta_1 | \xi_2\boxtimes \eta_2\rangle_C^{|X\boxtimes_BY|}
=
\tikzmath{
\begin{scope}
\clip[rounded corners=5pt] (-.4,-2) rectangle (1.2,3);
\fill[\AColor] (-.1,.1) -- (-.1,1) -- (.1,1) -- (.1,.1);
\fill[\BColor] (.4,-1) .. controls ++(90:.3cm) and ++(270:.3cm) .. (0,-.5) -- (.1,.1) -- (.1,1) -- (0,1) -- (0,1.5) .. controls ++(90:.3cm) and ++(270:.3cm) .. (.4,2) -- (.6,2) -- (.6,-1);
\fill[\CColor] (.5,-2) -- (.5,-1.3) -- (.6,-1) -- (.6,2) -- (.5,2) -- (.5,3) -- (1.2,3) -- (1.2,-2);
\end{scope}
\draw[thick, \XColor] (.1,.1) -- (.1,1);
\draw[thick, \YColor] (.6,-1) -- (.6,2.1);
\draw[dashed] (.5,-2) -- (.5,-1.3) -- (.4,-1) .. controls ++(90:.3cm) and ++(270:.3cm) .. (0,-.5) -- (-.1,.1) -- (-.1,1) -- (0,1) -- (0,1.5) .. controls ++(90:.3cm) and ++(270:.3cm) .. (.4,2) -- (.5,2) -- (.5,3);
\roundNbox{unshaded}{(0,-.2)}{.3}{0}{0}{$\xi_2$};
\roundNbox{unshaded}{(.5,-1.3)}{.3}{0}{0}{$\eta_2$};
\roundNbox{unshaded}{(0,1.2)}{.3}{0}{0}{$\xi_1^\dag$};
\roundNbox{unshaded}{(.5,2.3)}{.3}{0}{0}{$\eta_1^\dag$};
} 
=
\langle \eta_1 | \langle \xi_1|\xi_2\rangle_B^{|X|}\rhd \eta_2\rangle_{C}^{|Y|},
$$
and it clearly has dense range as $|X|_B \cong X_B$
and ${}_B|Y|_C\cong {}_BY_C$ 
as in Construction \ref{construction:BoundedVectorModule}.
Observe that the composite
$$
|X|\boxtimes_B |Y|
\xrightarrow{T_{X,Y}}
|X\boxtimes_B Y|
\xrightarrow{|u_{X,Y}|}
|X\xzq_Q Y|
$$
is a coisometry as 
$|u_{X,Y}| \in \Hom_{\bbC-C}(|X\boxtimes_B Y| \Rightarrow |X\xzq_Q Y|)$
is a coisometry.
\end{rem}

\begin{prop}
\label{prop:RealizationIsQSystem}
The realization $(|Q|, |m|\circ T_{Q,Q}, |i|)$ is a Q-system in $\rCorr(B\to B)$ which is equivalent to $(Q,m,i)$ in $\QSys(\rCorr)$.
\end{prop}
\begin{proof}
Under the identification $|B|=\End_{\bbC-B}(B)=B$,
the map $B \to |Q|$ is given by $|i|$,
and the multiplication on $|Q|$ is the composite
$$
|Q|\boxtimes_B |Q| 
\xrightarrow{T_{Q,Q}}
|Q\boxtimes_B Q| 
\xrightarrow{|m|}
|Q|,
$$
which are both obviously adjointable operators.
It follows that the adjoint of multiplication is given by $|m|^\dag \circ T_{Q,Q}^{-1}$,
and that $(|Q|, |m|\circ T_{Q,Q}, |i|)$ is a Q-system in $\rCorr(B\to B)$, as
$$
(|m| \circ T_{Q,Q})\circ (|m| \circ T_{Q,Q})^\dag
=
|m| \circ |m|^\dag 
= 
\id_{|Q|}.
$$
Letting $s_Q\in B$ be the range projection of $d_Q$ as in \ref{Z:Support},
the $B-B$ bimodular map 
$E_B: |Q|\to B$ 
is a faithful cp map onto $B_0:=s_QBs_Q$
by Proposition \ref{prop:EBCPandFiniteIndex}.
In particular, $|Q|_{B_0}$ is finitely generated projective by Lemma \ref{lem:FGP-dualizable}, and thus so is $|Q|_B$ as the orthogonal complement of $B_0$ acts by zero on $|Q|$.

Moreover, by Construction \ref{construction:BoundedVectorModule}, we have that ${}_BQ_B$ is isomorphic to ${}_B|Q|_B$ as 1-morphisms in $\rCorr(B\to B)$.
It is straightforward to show that this isomorphism intertwines the multiplication and unit, giving an isomorphism of Q-systems.
This is seen by analyzing the following diagrams which commute by \eqref{eq:BoundedVectorNaturalIso}:
$$
\begin{tikzcd}
Q\boxtimes_B Q
\arrow[d, <->, "\cong\boxtimes \cong"]
\arrow[rr, "m"]
\arrow[dr, <->, "\cong"]
&&
Q
\arrow[d, <->, "\cong"]
\\
{|Q|\boxtimes_B |Q|}
\arrow[r,"T_{Q,Q}","\cong"']
&
{|Q\boxtimes_B Q|}
\arrow[r,"|m|"]
&
{|Q|}
\end{tikzcd}
\qquad\qquad
\begin{tikzcd}
B
\arrow[d,equals]
\arrow[r,"i"]
&
Q
\arrow[d,<->,"\cong"]
\\
{B=|B|}
\arrow[r,"|i|"]
&
{|Q|}
\end{tikzcd}
$$
This concludes the proof.
\end{proof}

The following corollary is a rephrasing of Proposition \ref{prop:RealizationIsQSystem}.

\begin{cor}
Every Q-system $P\in \rCorr(A\to A)$ is equivalent in $\QSys(\rCorr)$ to a Q-system of the form of Example \ref{ex:NonunitalRestrictByExpecation}.
In more detail, there is a direct sum decomposition $A=A_0\oplus A_1$ and a surjective faithful cp $A_0-A_0$ bimodular map $E:|P| \to A_0$ such that the Q-system built from $(A\subset |P|, E)$ as in Example \ref{ex:NonunitalRestrictByExpecation} is equivalent to ${}_AP_A$ in $\QSys(\rCorr)$.
\end{cor}

\begin{construction}
We define $\mu_{X,Y}:|X|\boxtimes_{|Q|}|Y| \to |X\xzq_QY|$
as follows.
First, the map
$X\otimes_\bbC Y \to |X\xzq_Q Y|$
given by
$$
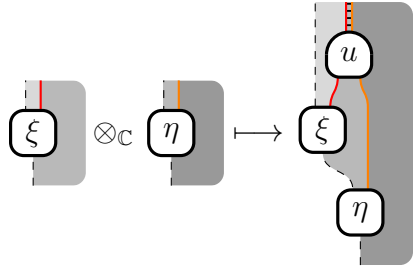
\tikzmath{
\begin{scope}
\clip[rounded corners=5pt] (-.4,-.7) rectangle (.7,.7);
\fill[\AColor] (-.1,0) rectangle (.1,.7);
\fill[\BColor] (0,-.7) -- (0,0) -- (.1,0) -- (.1,.7) -- (.8,.7) -- (.8,-.7);
\end{scope}
\draw[dashed] (0,-.7) -- (0,0) -- (-.1,0) -- (-.1,.7);
\draw[thick, \XColor] (.1,.3) -- (.1,.7);
\roundNbox{unshaded}{(0,0)}{.3}{0}{0}{$\xi$};
}
\otimes_\bbC
\tikzmath{
\begin{scope}
\clip[rounded corners=5pt] (-.4,-.7) rectangle (.7,.7);
\fill[\BColor] (-.1,0) rectangle (.1,.7);
\fill[\CColor] (0,-.7) -- (0,0) -- (.1,0) -- (.1,.7) -- (.8,.7) -- (.8,-.7);
\end{scope}
\draw[dashed] (0,-.7) -- (0,0) -- (-.1,0) -- (-.1,.7);
\draw[thick, \YColor] (.1,.3) -- (.1,.7);
\roundNbox{unshaded}{(0,0)}{.3}{0}{0}{$\eta$};
}
\longmapsto
\tikzmath{
\begin{scope}
\clip[rounded corners=5pt] (-.4,-2) rectangle (1.2,1.5);
\fill[\AColor] (.1,.1) .. controls ++(90:.2cm) and ++(270:.2cm) .. (.2,.5) -- (.31,.8) -- (.31,1.5) -- (-.1,1.5) -- (-.1,.1);
\fill[\BColor] (.4,-1) .. controls ++(90:.3cm) and ++(270:.3cm) .. (0,-.5) -- (.1,.1) .. controls ++(90:.2cm) and ++(270:.2cm) .. (.2,.5) -- (.31,.5) -- (.31,1.5) -- (.39,1.5) -- (.39,.5) -- (.5,.5) .. controls ++(270:.2cm) and ++(90:.2cm) .. (.6,.1) -- (.6,-1);
\fill[\CColor] (.5,-2) -- (.5,-1.3) -- (.6,-1) -- (.6,.1) .. controls ++(90:.2cm) and ++(270:.2cm) .. (.5,.5) -- (.39,.5) -- (.39,1.5) -- (1.2,1.5) -- (1.2,-2);
\end{scope}
\draw[thick, \XColor] (.1,.1) .. controls ++(90:.2cm) and ++(270:.2cm) .. (.2,.5);
\draw[thick, \YColor] (.6,-1) -- (.6,.1) .. controls ++(90:.2cm) and ++(270:.2cm) .. (.5,.5);
\draw[dashed] (.5,-2) -- (.5,-1.3) -- (.4,-1) .. controls ++(90:.3cm) and ++(270:.3cm) .. (0,-.5) -- (-.1,.1) -- (-.1,1.5);
\DoubleStrand{(.35,.8)}{(.35,1.5)}{\XColor}{\YColor}
\roundNbox{unshaded}{(0,-.2)}{.3}{0}{0}{$\xi$};
\roundNbox{unshaded}{(.5,-1.3)}{.3}{0}{0}{$\eta$};
\halfRoundBox{unshaded}{(.35,.8)}{.3}{0}{$u$}
}\,.
$$
is $|Q|$-middle linear as the left and right hand diagrams below agree:
$$
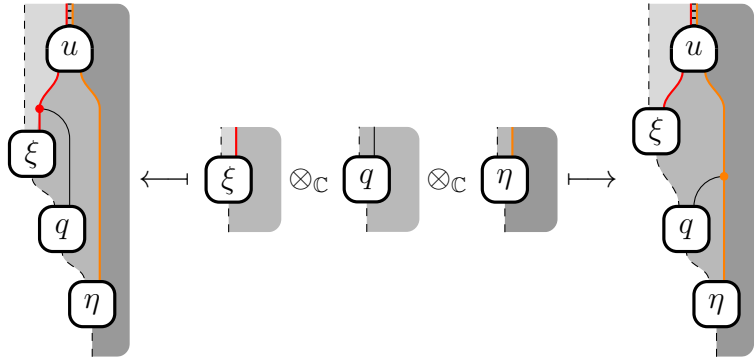
\tikzmath{
\begin{scope}
\clip[rounded corners=5pt] (-.4,-2.7) rectangle (1.3,2);
\fill[\AColor] (.1,.3) -- (.1,.6) .. controls ++(90:.2cm) and ++(270:.2cm) .. (.35,1.1) -- (.46,1.7) -- (.46,2) -- (-.1,2) -- (-.1,.3);
\fill[\BColor] (.7,-1.7) .. controls ++(90:.2cm) and ++(270:.2cm) .. (.4,-1.3) -- (.3,-.7) .. controls ++(90:.2cm) and ++(270:.2cm) .. (0,-.3) -- (.1,.6) .. controls ++(90:.2cm) and ++(270:.2cm) .. (.35,1.1) -- (.46,1.1) -- (.46,2) -- (.54,2) -- (.54,1.1) -- (.9,1.1) -- (.9,-1.7);
\fill[\CColor] (.8,-2.7) -- (.8,-2.3) -- (.9,-1.7) -- (.9,.6) .. controls ++(90:.2cm) and ++(270:.2cm) .. (.65,1.1) -- (.54,1.7) -- (.54,2) -- (1.3,2) -- (1.3,-2.7);
\end{scope}
\draw (.5,-1) -- (.5,.2) arc (0:90:.4cm);
\draw[dashed] (.8,-2.7) -- (.8,-2.3) -- (.7,-1.7) .. controls ++(90:.2cm) and ++(270:.2cm) .. (.4,-1.3) -- (.3,-.7) .. controls ++(90:.2cm) and ++(270:.2cm) .. (0,-.3) -- (-.1,.3) -- (-.1,2);
\draw[thick, \XColor] (.1,.3) -- (.1,.6) .. controls ++(90:.2cm) and ++(270:.2cm) .. (.35,1.1);
\draw[thick, \YColor] (.9,-1.7) -- (.9,.6) .. controls ++(90:.2cm) and ++(270:.2cm) .. (.65,1.1);
\DoubleStrand{(.5,1.7)}{(.5,2)}{\XColor}{\YColor}
\filldraw[\XColor] (.1,.6) circle (.05cm);
\roundNbox{unshaded}{(0,0)}{.3}{0}{0}{$\xi$};
\roundNbox{unshaded}{(.4,-1)}{.3}{0}{0}{$q$};
\roundNbox{unshaded}{(.8,-2)}{.3}{0}{0}{$\eta$};
\halfRoundBox{unshaded}{(.5,1.4)}{.3}{0}{$u$}
}
\longmapsfrom
\tikzmath{
\begin{scope}
\clip[rounded corners=5pt] (-.4,-.7) rectangle (.7,.7);
\fill[\AColor] (-.1,0) rectangle (.1,.7);
\fill[\BColor] (0,-.7) -- (0,0) -- (.1,0) -- (.1,.7) -- (.8,.7) -- (.8,-.7);
\end{scope}
\draw[dashed] (0,-.7) -- (0,0) -- (-.1,0) -- (-.1,.7);
\draw[thick, \XColor] (.1,.3) -- (.1,.7);
\roundNbox{unshaded}{(0,0)}{.3}{0}{0}{$\xi$};
}
\otimes_\bbC
\tikzmath{
\begin{scope}
\clip[rounded corners=5pt] (-.4,-.7) rectangle (.7,.7);
\fill[\BColor] (0,-.7) -- (0,0) -- (-.1,0) -- (-.1,.7) -- (.8,.7) -- (.8,-.7);
\end{scope}
\draw[dashed] (0,-.7) -- (0,0) -- (-.1,0) -- (-.1,.7);
\draw (.1,.3) -- (.1,.7);
\roundNbox{unshaded}{(0,0)}{.3}{0}{0}{$q$};
}
\otimes_\bbC
\tikzmath{
\begin{scope}
\clip[rounded corners=5pt] (-.4,-.7) rectangle (.7,.7);
\fill[\BColor] (-.1,0) rectangle (.1,.7);
\fill[\CColor] (0,-.7) -- (0,0) -- (.1,0) -- (.1,.7) -- (.8,.7) -- (.8,-.7);
\end{scope}
\draw[dashed] (0,-.7) -- (0,0) -- (-.1,0) -- (-.1,.7);
\draw[thick, \YColor] (.1,.3) -- (.1,.7);
\roundNbox{unshaded}{(0,0)}{.3}{0}{0}{$\eta$};
}
\longmapsto
\tikzmath{
\begin{scope}
\clip[rounded corners=5pt] (-.4,-2.7) rectangle (1.3,2);
\fill[\AColor] (.1,.3) -- (.1,.6) .. controls ++(90:.2cm) and ++(270:.2cm) .. (.35,1.1) -- (.46,1.7) -- (.46,2) -- (-.1,2) -- (-.1,.3);
\fill[\BColor] (.7,-1.7) .. controls ++(90:.2cm) and ++(270:.2cm) .. (.4,-1.3) -- (.3,-.7) .. controls ++(90:.3cm) and ++(270:.3cm) .. (0,0) -- (.1,.6) .. controls ++(90:.2cm) and ++(270:.2cm) .. (.35,1.1) -- (.46,1.1) -- (.46,2) -- (.54,2) -- (.54,1.1) -- (.9,1.1) -- (.9,-1.7);
\fill[\CColor] (.8,-2.7) -- (.8,-2.3) -- (.9,-1.7) -- (.9,.6) .. controls ++(90:.2cm) and ++(270:.2cm) .. (.65,1.1) -- (.54,1.7) -- (.54,2) -- (1.3,2) -- (1.3,-2.7);
\end{scope}
\draw (.5,-.7) arc (180:90:.4cm);
\draw[dashed] (.8,-2.7) -- (.8,-2.3) -- (.7,-1.7) .. controls ++(90:.2cm) and ++(270:.2cm) .. (.4,-1.3) -- (.3,-.7) .. controls ++(90:.3cm) and ++(270:.3cm) .. (0,0) -- (-.1,.3) -- (-.1,2);
\draw[thick, \XColor] (.1,.3) -- (.1,.6) .. controls ++(90:.2cm) and ++(270:.2cm) .. (.35,1.1);
\draw[thick, \YColor] (.9,-1.7) -- (.9,.6) .. controls ++(90:.2cm) and ++(270:.2cm) .. (.65,1.1);
\DoubleStrand{(.5,1.7)}{(.5,2)}{\XColor}{\YColor}
\filldraw[\YColor] (.9,-.3) circle (.05cm);
\roundNbox{unshaded}{(0,.3)}{.3}{0}{0}{$\xi$};
\roundNbox{unshaded}{(.4,-1)}{.3}{0}{0}{$q$};
\roundNbox{unshaded}{(.8,-2)}{.3}{0}{0}{$\eta$};
\halfRoundBox{unshaded}{(.5,1.4)}{.3}{0}{$u$}
}\,.
$$
Since $u_{X,Y} = u_{X,Y}\circ p_{X,Y}$, the left and right hand sides are visibly equal.
Next, observe that for all $\xi \in X$ and $\eta\in Y$,
\begin{align*}
\langle \mu_{X,Y}(\xi\otimes\eta)| \mu_{X,Y}(\xi\otimes\eta) \rangle_C^{|X\xzq_QY|} 
&= 
\tikzmath{
\begin{scope}
\clip[rounded corners=5pt] (-.4,-2) rectangle (1.2,3);
\fill[\AColor] (-.1,.1) -- (-.1,1) -- (.1,1) -- (.1,.1);
\fill[\BColor] (.4,-1) .. controls ++(90:.3cm) and ++(270:.3cm) .. (0,-.5) -- (.1,.1) -- (.1,1) -- (0,1) -- (0,1.5) .. controls ++(90:.3cm) and ++(270:.3cm) .. (.4,2) -- (.6,2) -- (.6,-1);
\fill[\CColor] (.5,-2) -- (.5,-1.3) -- (.6,-1) -- (.6,2) -- (.5,2) -- (.5,3) -- (1.2,3) -- (1.2,-2);
\end{scope}
\draw[thick, \XColor] (.1,.1) -- (.1,1);
\draw[thick, \YColor] (.6,-1) -- (.6,2.1);
\draw (.1,.5) -- (.6,.5);
\draw[dashed] (.5,-2) -- (.5,-1.3) -- (.4,-1) .. controls ++(90:.3cm) and ++(270:.3cm) .. (0,-.5) -- (-.1,.1) -- (-.1,1) -- (0,1) -- (0,1.5) .. controls ++(90:.3cm) and ++(270:.3cm) .. (.4,2) -- (.5,2) -- (.5,3);
\roundNbox{unshaded}{(0,-.2)}{.3}{0}{0}{$\xi_2$};
\roundNbox{unshaded}{(.5,-1.3)}{.3}{0}{0}{$\eta_2$};
\roundNbox{unshaded}{(0,1.2)}{.3}{0}{0}{$\xi_1^\dag$};
\roundNbox{unshaded}{(.5,2.3)}{.3}{0}{0}{$\eta_1^\dag$};
} 
=
\tikzmath{
\begin{scope}
\clip[rounded corners=5pt] (-.4,-2) rectangle (1.2,3);
\fill[\AColor] (-.2,.1) -- (-.2,1) -- (0,1) -- (0,.1);
\fill[\BColor] (.4,-1) .. controls ++(90:.3cm) and ++(270:.3cm) .. (-.1,-.5) -- (0,.1) -- (0,1) --  (-.1,1.3) .. controls ++(90:.3cm) and ++(270:.3cm) .. (.4,2) -- (.6,2) -- (.6,-1);
\fill[\CColor] (.5,-2) -- (.5,-1.3) -- (.6,-1) -- (.6,2) -- (.5,2) -- (.5,3) -- (1.2,3) -- (1.2,-2);
\end{scope}
\draw[thick, \XColor] (0,.1) -- (0,1);
\draw[thick, \YColor] (.6,-1) -- (.6,2.1);
\draw (0,.4) arc (-90:0:.3cm) -- (.3,1.3) arc (180:90:.3cm);
\draw[dashed] (.5,-2) -- (.5,-1.3) -- (.4,-1) .. controls ++(90:.3cm) and ++(270:.3cm) .. (-.1,-.5) -- (-.2,-.5) -- (-.2,1) -- (-.1,1) -- (-.1,1.3) .. controls ++(90:.3cm) and ++(270:.3cm) .. (.4,2) -- (.5,2) -- (.5,3);
\filldraw[\XColor] (0,.4) circle (.05cm);
\filldraw[\YColor] (.6,1.6) circle (.05cm);
\roundNbox{unshaded}{(-.1,-.2)}{.3}{0}{0}{$\xi_2$};
\roundNbox{unshaded}{(.5,-1.3)}{.3}{0}{0}{$\eta_2$};
\roundNbox{unshaded}{(-.1,1)}{.3}{0}{0}{$\xi_1^\dag$};
\roundNbox{unshaded}{(.5,2.3)}{.3}{0}{0}{$\eta_1^\dag$};
} 
\\&=
\left\langle \eta_1\left|\langle \xi_1|\xi_2\rangle_{|Q|}^{|X|} \rhd \eta_2\right.\right\rangle_C^{|Y|}
\\&=
\langle \xi_1\boxtimes \eta_1|\xi_2\boxtimes\eta_2\rangle_C^{|X|\boxtimes_{|Q|} |Y|}.
\end{align*}
Hence we get a well-defined isometric map $\mu_{X,Y}:|X|\boxtimes_{|Q|}|Y| \to |X\xzq_QY|$.
We immediately see that $\im(|u_{X,Y}|\circ T_{X,Y}) \subset \im(\mu_{X,Y})$,
so $\mu_{X,Y}$ is surjective.
Finally, it is straightforward to verify that $\mu_{X,Y}$ is $P-R$ bimodular.
\end{construction}

While the relative tensor product of right correspondences does \emph{not} satisfy a universal property, 
the following corollary establishes a universal property for changing the fusion by a finite index over algebra.

\begin{cor}
Let $B\subset Q$ be a unital inclusion of $\rm C^*$-algebras
equipped with a finite index faithful (completely positive) conditional expectation 
$E_B: Q\to B$ as in Example \ref{ex:RestrictByExpectation}, 
so that $i = E_B^\dag$ is the inclusion $B\hookrightarrow Q$.
Suppose that the multiplication $m: Q\boxtimes_B Q \to Q$ is adjointable such that $(Q,m,i)$ is a Q-system in $\rCorr(B\to B)$.
\begin{enumerate}[label=(\arabic*)]
\item 
If $X\in \rCorr(A\to Q)$ and $Y\in \rCorr(Q\to C)$ 
satisfy \ref{M:separable},
then restricting the $Q$-actions to $B$-actions and endowing $X_B$ with the right inner product $\langle \xi|\eta\rangle_B^X := i^\dag(\langle \eta|\xi\rangle_Q^X)$,
we may view
$X\in \rCorr(A\to B)$ and $Y\in \rCorr(B\to C)$.
\item
The unitary coequalizer 
(which exists by Definition \ref{defn:QSysTensorProduct} and Remark \ref{rem:QSysCoequalizer})
$$
\begin{tikzcd}[column sep=4em, row sep=4em]
{X\boxtimes_{B}Q \boxtimes_{B} Y} 
\arrow[shift left =1]{r} 
\arrow[shift left =-1]{r} 
&
{X \boxtimes_{B}Y}  
\arrow[twoheadrightarrow, "u_{X,Y}"]{r} 
&
{X \xzq_Q  Y} 
\end{tikzcd}
$$
is unitarily isomorphic to $X\boxtimes_Q Y$.
\end{enumerate}
\end{cor}
\begin{proof}
\item[(1)]
This is immediate.
\item[(2)]
By Construction \ref{construction:BoundedVectorModule}, ${}_AX_B\cong {}_A|X|_B$, ${}_BY_C \cong {}_B|Y|_C$, and ${}_BX\xzq_Q Y_C \cong {}_B|X\xzq_Q Y|_C$.
Moreover, observe that $Q\cong |Q|$ as unital $\rm C^*$-algebras.
This implies
\begin{equation*}
X\boxtimes_QY
\cong
|X|\boxtimes_{|Q|}|Y|
\xrightarrow{\mu_{X,Y}}
|X\xzq_Q Y|
\cong
X\xzq_QY.
\qedhere
\end{equation*}
\end{proof}

\begin{rem}
By the above corollary, there is a canonical coisometry $v_{X,Y}:|X|\boxtimes_B |Y| \to |X|\boxtimes_{|Q|}|Y|$
such that the following diagram commutes, where the first line is a coequalizer:
\begin{equation}
\label{eq:UniquenessOfTensoratorFromUniversalProperty}
\begin{tikzcd}[column sep=4em, row sep=4em]
{|X|\boxtimes_{B}|Q| \boxtimes_{B} |Y|} 
\arrow[shift left =1]{r} 
\arrow[shift left =-1]{r} 
&
{|X| \boxtimes_{B}|Y|}  
\arrow["v_{X,Y}", twoheadrightarrow]{r} 
\arrow["T_{X,Y}"', hookrightarrow]{d} 
&
{|X| \boxtimes_{|Q|} |Y|} 
\arrow["\exists\,!\,\mu_{X,Y}",dashrightarrow ]{d}
\\
&
{|X\boxtimes_B Y|} 
\arrow["|u_{X,Y}|"', twoheadrightarrow]{r} 
&
{|X\xzq_Q Y|}
\end{tikzcd}
\end{equation}
Since $|u_{X,Y}|\circ T_{X,Y}$ coequalizes the two maps on the left hand side of \eqref{eq:UniquenessOfTensoratorFromUniversalProperty}, we see that $\mu_{X,Y}$ is then the unique map from the universal property of the coequalizer. 
We record for later use the following identity for $\mu_{X,Y}$ which follows by precomposing with the isometry $v^{\dag}_{X,Y}$:
\begin{equation}
    \label{eq:FormulaForMu}
    \mu_{X,Y} = |u_{X,Y}|\circ T_{X,Y} \circ v_{X,Y}^\dag.
\end{equation}
\end{rem}

\begin{prop}
The collection $\{\mu_{X,Y}\}_{X,Y}$ assembles into a unitary natural isomorphism.
\end{prop}
\begin{proof}
It remains to prove naturality.
Consider bimodules ${}_{P}X'_{Q}$ and ${}_{Q}Y'_{R}$ and 
a $P-Q$ bimodular map $\phi\in\rCorr_{A-B}(X\Rightarrow X')$ and 
a $Q-R$ bimodular map $\psi\in \rCorr_{B-C}(Y\Rightarrow Y')$.
By naturality of the families of unitary isomorphisms $\{u_{X,Y}\}$ and $\{v_{X,Y}\}$ (as they are the canonical coequalizer maps), we have
\begin{align*}
    \mu_{X',Y'}\circ (|\phi|\boxtimes_{|Q|} |\psi|) \circ v_{X,Y}
    & = \mu_{X',Y'}\circ v_{|X'|, |Y'|} \circ (|\phi|\boxtimes_{B}|\psi|)\\
    & = |u_{X',Y'}| \circ T_{X',Y'} \circ (|\phi| \boxtimes_{B} |\psi|)\\
    & = |u_{X',Y'}| \circ |\phi\boxtimes_{B}\psi| \circ T_{X,Y} \\ 
    & = |u_{X',Y'} \circ (\phi\boxtimes_{B}\psi)| \circ T_{X,Y} \\ 
    & = |(\phi\xzq_Q \psi)\circ u_{X,Y}| \circ T_{X,Y} \\ 
    & = |\phi\xzq_Q \psi| \circ |u_{X,Y}| \circ T_{X,Y}\\
    & = |\phi\xzq_Q \psi| \circ \mu_{X,Y} \circ v_{X, Y}.
\end{align*}
We now precompose with the isometry $v^\dag_{X,Y}$ to obtain
\begin{equation*}
    |\phi\xzq_Q \psi| \circ \mu_{X,Y} 
    =  
    \mu_{X',Y'}\circ (|\phi|\boxtimes_{|Q|} |\psi|).
\qedhere
\end{equation*}
\end{proof}

Finally, we remark that associativity of $\mu=\{\mu_{X,Y}\}_{X,Y}$ follows similarly to that of $\QSys(F)^2$ from Construction \ref{construction:Qsys(F)}.
Indeed, the formula \eqref{eq:FormulaForMu} for $\mu$ is entirely similar to the definition \eqref{eq:DefOfQSysF2} of $\QSys(F)^2$.
We leave the straightforward details to the reader.

\subsection{Equivalence}

We now prove that the $\dag$ 2-functors 
$\iota:\rCorr \hookrightarrow \QSys(\rCorr)$
and
$|\cdot|: \QSys(\rCorr) \to \rCorr$
witness a 2-equivalence.
First, under the identifications
$$
A=
\End_{\bbC-A}({}_{\bbC}A_A)
\qquad\text{and}\qquad
{}_AX_B \underset{\text{Const.~\ref{construction:BoundedVectorModule}}}{\cong} \Hom_{\bbC-B}({}_{\bbC}B_B \to {}_{\bbC} A\boxtimes_A X_B)
$$
and naturality of the unitary isomorphism on the right hand side above from \eqref{eq:BoundedVectorNaturalIso},
we see
$|\cdot| \circ \iota$ is unitarily naturally isomorphic to the identity 2-functor on $\rCorr$.
Thus once we prove $\iota$ is an equivalence, it follows formally that $|\cdot|$ is an inverse $\dag$ 2-functor.
By Remark \ref{rem:QSysCompleteIffESon0}, $\iota$ is an equivalence if and only if it is essentially surjective on objects.


\begin{prop}
\label{prop:IotaEssSurj}
The $\dag$ 2-functor
$\iota$ is essentially surjective on objects.
\end{prop}
\begin{proof}
Suppose $(Q,m,i)$ in $\rCorr(B\to B)$ is a Q-system.
We must show $Q$ is equivalent to a trivial Q-system.
By Prop.~\ref{prop:RealizationIsQSystem}, $(Q,m,i)$ is equivalent to the realized Q-system $(|Q|, |m|\circ T_{Q,Q}, |i|)$.
Let $B_0=s_QBs_Q$, where $s_Q$ is the range projection of $d_Q$ as in \ref{Z:Support}.
By Ex.~\ref{ex:EquivalentRestrictedQSystem}, the realized Q-system $(|Q|, |m|\circ T_{Q,Q}, |i|)$ is equivalent to ${}_{B_0}|Q|_{B_0}$, the same Q-system, but considered over $B_0$ rather than $B$.
Finally, by Ex.~\ref{ex:EquivalenceOfRestrictByExpectationExamples}, ${}_{B_0}|Q|_{B_0}$ is equivalent to the trivial Q-system $1_{|Q|}\in \rCorr(|Q|\to |Q|)$.
The result follows by concatenating equivalences of Q-systems in $\QSys(\rCorr)$:
\begin{equation*}
\begin{tikzcd}[column sep = 4.5em]
{}_BQ_B
\arrow[r, <->,"\cong","\text{Prop.~\ref{prop:RealizationIsQSystem}}"']
&
{}_B|Q|_B
\arrow[r, <->,"\cong","\text{Ex.~\ref{ex:EquivalentRestrictedQSystem}}"']
&
{}_{B_0}|Q|_{B_0}
\arrow[r, <->,"\cong","\text{Ex.~\ref{ex:EquivalenceOfRestrictByExpectationExamples}}"']
&
{}_{|Q|}(1_{|Q|})_{|Q|}
\end{tikzcd}.
\qedhere    
\end{equation*}
\end{proof}


This concludes the proof of Theorem \ref{thm:QSysComplete}.
\qed

\section{Induced actions on \texorpdfstring{$\rm C^*$}{C*} algebras}

We now describe an application of our constructions above. 
Recall an \text{action} of a unitary fusion category (UFC) $\cC$ on a (unital) $\rm C^*$-algebra $A$ is a unitary tensor functor $F:\cC\rightarrow \rCorr(A \to A)$.
Finding actions of UFCs on $\rm C^*$-algebras is generally a difficult task. 
Most known examples use $\rm C^*$-analogues of subfactor constructions, e.g.~\cite{MR1604162,MR4139893}.
In \cite{MR4328058}, the author constructed actions of pointed fusion categories $\Hilb(G, \omega)$\footnote{Here, $G$ is a finite group and $\omega$ is some normalized representative for a class $[\omega]\in H^3(G, U(1)).$} on $C(X),$ where $X$ is a closed connected manifold. 
Using these types of constructions as input, we can apply the general theory we have developed to actions of dual fusion categories on Q-system extensions. We call such actions \textit{induced actions}.

We divide this section into two subsections. 
In \S\ref{sec:K_0obstructions}, we observe that an action of a UFC on a unital, stably finite $\rm C^*$-algebra $A$ equips the abelian group $K_{0}(A)$ with the structure of a right module over the fusion ring $K_{0}(\cC)$. 
Furthermore, this module structure plays nicely with the order structure on $K_{0}(A)$, which leads naturally to $K$-theoretical obstructions for the existence of actions.
Using these properties, we conclude in Corollary \ref{non-integral} that non-integral UFCs do not admit actions on any continuous trace $\rm C^*$ algebra with connected spectrum.

In \S\ref{sec:gptheoreticalactions} we address the existence question of which integral UFCs act on such $\rm C^*$ algebras. 
Given a unitary tensor functor $F:\Hilb(G,\omega)\to\rCorr(C(X)\to C(X)),$ 
Q-system completion yields induced actions of group theoretical UFCs, a.k.a.~UFCs unitarily Morita equivalent to $\Hilb(G, \omega)$ (see Definition \ref{dfn:UMoritaEq}), on Q-system extensions of $C(X)$, which are continuous trace $\rm C^*$ algebras whose spectrum is some closed connected $n$-manifold for $n\geq 2$ (Corollary \ref{cor:MainC}). 
This construction therefore produces larger families of actions of integral UFCs on continuous trace $\rm C^*$ algebras over connected spectra.


\subsection{\texorpdfstring{$K_0$}{K0}-obstructions to actions of unitary fusion categories}\label{sec:K_0obstructions}

Recall that if $A$ is \textit{stably finite} (the identities of $M_{n}(A)$ are not equivalent to proper sub-projections for all $n$), then $K_{0}(A)$ naturally acquires the structure of an ordered abelian group, with order unit $[1_{A}]$, whose positive cone $K^{+}_{0}(A)$ is given by the image of $[X]$, where $X\in \fgprCorr(\bbC\to A)$ is a finitely generated projective right Hilbert $A$-module (e.g. \cite[Chapter III.6]{MR1656031}).
A \textit{state} is a homomorphism of ordered groups $\phi:K_{0}(A)\rightarrow (\mathbb{R},+)$ such that $\phi([1_A])=1$.

Let $\cC$ be a unitary tensor category. 
Given a unitary tensor functor $F:\cC\rightarrow \fgprCorr(A\to A)$, by the universal property of the Grothendieck construction, we may canonically endow $K_{0}(A)$ with a the structure of a right $K_{0}(\cC)$ module. 
If $X\in \fgprCorr(\bbC\to A)$ and $a\in \cC$, the action is induced by  $[X]\triangleleft [a]:=X\boxtimes_{A} F(a)$. 
By definition, each $[a]$ acts by a \textit{positive} endomorphism, namely $ K^{+}_{0}(A)\triangleleft [a]\subseteq K^{+}_{0}(A)$.

Now if $\cC$ is a fusion category, we can take a state on $K_{0}$ and `average it' over the action of the fusion category. 
Using the Perron-Frobenius theorem, this allows us to produce an `eigenstate' for the action. 
As a consequence, we obtain no-go theorems for actions of fusion categories on large classes of $\rm C^*$-algebras.

\begin{lem} 
Let $A$ be a stably finite $\rm C^*$-algebra and $F: \cC\rightarrow \fgprCorr(A\to A)$ an action. 
For any object $a\in \cC$,  the endomorphism $-\triangleleft [a]\in \End(K^{+}_{0}(A))$ is not the zero operator.
\end{lem}

\begin{proof}
Suppose for contradiction that $[H]\triangleleft [a]=0$ for all $[H]\in  K^{+}_{0}(A)$. 
Then since $N^{1_\cC}_{\overline{a}, a}=1$,
\begin{align*}
0
&=
( [H]\triangleleft [\overline{a}])\triangleleft [a]
= 
[H]\triangleleft[\overline{a}\xzq a]=\sum_{b\in \Irr(\cC)} N^{b}_{\overline{a}, a}[H]\triangleleft [b]
=[H]+\sum_{b\ne 1_\cC} N^{b}_{\overline{a}, a}[H]\triangleleft [b].
\end{align*}
But $\sum_{b\ne 1_\cC} N^{b}_{\overline{a}, a}[H]\triangleleft [b] \ge 0$. 
Thus $[H]\le 0$, and since this group is ordered, this implies $[H]= 0$ for all $H$, a contradiction.
\end{proof}

The following proposition should be compared with \cite[Thm.~5.9]{MR1624182}.

\begin{prop}
\label{prop:ExistsFPState}
If $A$ is a unital stably finite $\rm C^*$-algebra, $\cC$ is a unitary fusion category, and $F:\cC\rightarrow \fgprCorr(A\to A) $ is a unitary tensor functor, there exists a state $\phi$ of $K_{0}(A)$ such that $\phi(- \triangleleft [a])=\FPdim(a)\phi(-)$ for all objects $a\in \cC$.
\end{prop}
\begin{proof}
By \cite[Cor.~3.3]{MR389965}, any ordered abelian group with an order unit admits a state $\psi$. Consider the linear functional $\varphi(-):=\sum_{a\in \Irr(\cC)} \FPdim(a)\cdot\psi(- \triangleleft [a])$. 
Then $\varphi([H])\ge 0$. We compute 
\begin{align*}
\varphi([H]\triangleleft [b])
&=
\sum_{a\in \Irr(\cC)} \FPdim(a)\psi([H]\triangleleft [b\xzq a])
\\&=
\sum_{a,c\in \Irr(\cC)}\FPdim(a)N^{c}_{ba}\psi([H]\triangleleft [c])
\\&=
\sum_{c\in \Irr(\cC)} \left(\sum_{a\in \Irr(\cC)}\FPdim(a)N^{\bar{a}}_{\bar{c}b}\right)\phi([H]\triangleleft [c])
\\&=
\sum_{c\in \Irr(\cC)}\FPdim(b)\FPdim(c)\psi( [H]\triangleleft [c]) 
\\&=
\FPdim(b) \varphi([H])
\end{align*}
We claim that $\varphi([1_{A}])\ne 0$. 
To see this, suppose for contradiction that
$$
0=\varphi( [1_{A}])=\sum_{a\in \Irr(\cC)} \FPdim(a)\psi( [1_{A}]\triangleleft [a]).
$$
Then by positivity, for all $a$, $\FPdim(a)\psi( [1_{A}] \triangleleft [a])=0$.
In particular, for $a=1_\cC$, $\psi([1_{A}])=0$, contradicting the assumption that $\psi$ is a state.
Hence the state
$\phi:=\varphi([1_A])^{-1}\cdot \varphi$ satisfies the desired property.
\end{proof}

Proposition \ref{prop:ExistsFPState} provides a particularly useful obstruction to the existence of actions of fusion categories on stably finite $\rm C^*$-algebras whose K-theory has a unique state $\phi$. 
In this case, the unique state is necessarily an `eigenstate' for the fusion category action in the sense of the above proposition. 
Thus for a given fusion category to act on $A$, it is necessarily the case that $\phi(K_{0}(A))\subseteq \bbR$ is closed under multiplication by the dimensions of $\cC$.

One class of $\rm C^*$-algebras that are closely related to classical topology are the \textit{continuous trace} $\rm C^*$-algebras. 
The unital continuous trace algebras are essentially locally trivial bundles of matrix algebras over their spectrum (which happens to be a compact Hausdorff space).
In particular, unital $\rm C^*$-algebras Morita equivalent to $C(X)$ are continuous trace, but there are many examples not of this form. 
However, unital continuous trace algebras are always stably finite, nuclear $\rm C^*$-algebras. 
It is natural to ask whether fusion categories admit actions on these algebras, in particular when the spectrum is a `nice' space, for example some sort of manifold. 
We have the following corollary, which greatly restricts the possibilities.

\begin{cor}\label{non-integral}
If a fusion category contains an object with a non-integral dimension, there is no action on any unital continuous trace $\rm C^*$-algebra with connected spectrum.
 \end{cor}

\begin{proof}
Let $A$ be a unital continuous trace $\rm C^*$-algebra with connected spectrum $\widehat{A}$. For a projection $p\in M_{\infty}(A)$, the function $x \mapsto \tr(x(p))\in \mathbbm{Z}$ for $x\in \widehat{A}$ is continuous, and since $\widehat{A}$ is connected, is some constant value we call $\widehat{p}$. Every state on $K_{0}(A)$ arises from a quasitrace on $A$ \cite{MR1190414}. Since $A$ is nuclear 
by \cite[B.44]{MR1634408},
all states on $K_{0}$ arise then arise from traces on $A$ \cite[Thm~5.11]{MR3241179}.
But since $A$ is continuous trace (hence postliminal), traces on $A$ correspond to probability measures on the spectrum $\widehat{A}$ \cite[\S8.8]{MR0458185}, and thus any trace $\tau$ on $A$ will satisfy 
$$
\tau(p)
=
\frac{1}{\widehat{1}_{A}}\int_{\widehat{A}}\tr(x(p)) d\mu
=
\int_{\widehat{A}}d\mu \frac{\widehat{p}}{\widehat{1}_{A}}
=
\frac{\widehat{p}}{\widehat{1}_{A}}.
$$ 

We conclude there is a unique state on $K_{0}(A)$, defined by $p\mapsto \frac{\widehat{p}}{\widehat{1}_{A}}\in \bbQ$. In particular, this implies that if a fusion category admits an action on $A$, then $\FPdim(a)\in \bbQ$ for all objects $a$. 
But these dimensions are algebraic integers, and rational algebraic integers are integers.
\end{proof}

\subsection{Group theoretical fusion categories act on continuous trace \texorpdfstring{$\rm C^*$}{C*}-algebras}\label{sec:gptheoreticalactions}

Corollary \ref{non-integral} in the previous section implies that only integral fusion categories admit actions on continuous trace $\rm C^*$-algebras. 
This motivates the question of which integral fusion categories admit such an action? 
A large class of integral fusion categories are given by \textit{group-theoretical} fusion categories. 
We recall the definition below.

\begin{defn}\label{dfn:UMoritaEq}
Two unitary fusion categories $\cC,\cD$ are called \emph{unitarily Morita equivalent} if there exists a connected Q-system $Q\in \cC$ such that $\cD\cong \QSys(\rmB\cC)(Q\to Q)$.\footnote{This definition is equivalent to the definitions given in \cite{MR1966524,2004.08271} using \cite[Thm.~A.1]{MR3933035}.}
A unitary fusion category $\cC$ is called \textit{group theoretical} if it is unitarily Morita equivalent to $\fdHilb(G, \omega)$ for some finite group $G$ and $[\omega]\in H^{3}(G, U(1))$. 
\end{defn}

Using our above results, we can build actions of group theoretical fusion categories on $\rm C^*$-algebras by finding actions of $\fdHilb(G, \omega)$ on algebras, and inducing actions on the corresponding Q-system extensions.
By \cite{MR4328058}, every $\fdHilb(G, \omega)$ admits an action on $C(X)$ where $X$ is a `nice'
compact Hausdorff space (e.g. closed connected $n$-manifold for $n\ge 2$). 
This implies we can build arbitrary group theoretical fusion category actions on Q-system extensions of these $C(X)$. 
We would like to characterize the $\rm C^*$-algebras that arise this way.

By \cite[Thm.~3.1]{MR1976233}, an irreducible Q-system in $\fdHilb(G, \omega)$ is given by a pair $(H, \mu)$ where $H\le G$ and $\mu: H\times H\rightarrow U(1)$ satisfies $d\mu=\omega|_{H}$. Equivalently, we have 
$$
\mu(gh,k)\mu(g,h)=\omega(g,h,k)\mu(g,hk)\mu(h,k)
\qquad\qquad
\forall\, g,h,k\in G.
$$

Following \cite[Example 5.5]{2105.05587}, for any (unital) $\rm C^*$-algebra $A$ and an automorphism $\alpha \in \text{Aut}(A)$, we can define a bimodule $A_{\alpha}\in \rCorr(A\rightarrow A)$ by setting 
$A_{\alpha}:=A$ as a vector space with $A$ actions and $A$-valued inner product defined by
$$
a\triangleright \eta \triangleleft b:=a\eta\alpha(b)
\qquad\quad
\langle \eta\ |\ \xi\rangle_{A}:=\alpha^{-1}(\eta^{*}\xi).
$$
For any two automorphisms $\alpha,\beta\in \text{Aut}(A)$, the intertwiner $J_{\alpha,\beta}:A_{\alpha}\boxtimes A_{\beta}\rightarrow A_{\alpha\circ \beta}$ given by the extension of $\xi\boxtimes \eta\mapsto \xi\alpha(\eta)$ is a unitary isomorphism in $\rCorr(A\rightarrow A)$. 
Furthermore, if 
$\alpha(\,\cdot\,)=u\beta(\,\cdot\,)u^{*}$, 
then the map $\eta\mapsto \eta u$ extends to a unitary isomorphism $A_{\alpha}\rightarrow A_{\beta}$ in $\rCorr(A\rightarrow A)$. 
Such unitaries $u$ form a torsor over the $U(Z(A))$, and comprise all unitary $A-A$ bimodule isomorphisms $A_\alpha \to A_\beta$. 

An action $\text{Hilb}(G, \omega)\rightarrow \rCorr(A\rightarrow A)$ is called \textit{automorphic} if for each $g\in G$, the image of the corresponding simple object is isomorphic to $A_{\alpha}$ for some $\alpha\in \text{Aut}(A)$. This yields an assignment $g\mapsto \alpha_{g}\in \text{Aut}(A)$. Combining the observations above, our tensorators must have the form 
$$
m_{g,h}(\eta\otimes \xi):=J_{\alpha_g, \alpha_h}(\eta\otimes \xi)u_{g,h}
$$ 
where the unitaries $u_{g,h}$ satisfy $\alpha_{g}\circ \alpha_h(\,\cdot\,)=u_{g,h}\alpha_{gh}(\,\cdot\,)u^{*}_{g,h}$. 
That this data defines a monoidal functor is equivalent to the conditions
$$
u_{g,h}\alpha_{g}\circ \alpha_{h}(\,\cdot\,)=\alpha_{gh}(\,\cdot\,) u_{gh}
\qquad\text{and}\qquad
u_{g,hk}\alpha_{g}(u_{h,k})=\omega(g,h,k)u_{gh,k}u_{g,h}.
$$
Taking the $*$ of both sides, this is precisely the same data as an $\overline{\omega}$-anomalous action of $G$, where $\overline{\omega}$ is the complex conjugate 3-cocycle (c.f.  \cite[Example 5.5, Proposition 5.6]{2105.05587}, \cite[Remark 2.5]{MR4328058}).

\begin{lem}
Let $A$ be a unital continuous trace $\rm C^*$-algebra and $F$ an automorphic action of $\fdHilb(G,\omega)$ on $A$ which induces a free action of $G$ on the spectrum $\widehat{A}$. Then for any Q system $(H,\mu)\in \fdHilb(G,\omega)$, the realization $|(H,\mu)|_F$ is continuous trace, with spectrum $X/H$. 
\end{lem}
\begin{proof}
First we claim the realization $|(H,\mu)|_F$ is isomorphic to a Busby-Smith twisted crossed product (also called a cocycle crossed product) $A\rtimes_{F,u} H$ (see \cite{MR1798596} for an overview). To see this, suppose we have an automorphic action of $\text{Hilb}(G, \omega)$ with $g\mapsto \alpha_{g}\in \text{Aut}(A)$ and unitaries $u_{g,h}$ as above.
Restricting to $H$, if we set $w_{g,h}=\mu_{g,h} u_{g,h}\in A$. We notice that $w_{g,h}$ satisfies the `non-anomalous' cocycle equation, namely
\begin{align*}
w_{g,hk}\alpha_{g}(w_{h,k})=(\mu_{g,hk}\mu_{h,k})(u_{g,hk}\alpha_{g}(u_{h,k}))
&=
\omega^{-1}(g,h,k)\omega(g,h,k)(\mu_{gh,k}\mu_{g,h})(u_{gh,k}u_{g,h})
\\&=
w_{gh,k}w_{g,h}.
\end{align*}
The realization of the $Q$-system $(H, \mu)$ as a bimodule can be written 
$$
|(H,\mu)|_F=\bigoplus_{g\in G} A_{\alpha_{g}}.
$$
Writing an element $a\in A_{\alpha_g}$ as $a\delta_{g}$, we then see the formula for the product on the realization is $(a \delta_{g})(b\delta_{h})=a\alpha_{g}(b)w_{g,h} \delta_{gh}$. This is exactly the definition for the product of Fourier coefficients for the cocycle crossed product with cocycle $\{w_{g,h}\}_{g,h}$. Since our groups are finite, this yields an isomorphism of $\rm C^*$-algebras.

Now, by the Packer-Raeburn stabilization trick \cite{MR1002543}, $|(H,\mu)|_F\otimes B(\ell^{2}(H))\cong (A\otimes B(\ell^{2}(H)))\rtimes_{\beta} H$, where $\beta$ is an ordinary (untwisted) action of $H$ on $A\otimes B(\ell^{2}(H))$ which induces the same action of $H$ on $Z(A\otimes B(\ell^{2}(H)))\cong Z(A)$ (see \cite[Formula 3.1]{MR1002543}), 
and thus also on $(A\otimes B(\ell^{2}(H)))^{\widehat{\,\,}}\cong \widehat{A}$. 
Since the action of $H$ is free on the spectrum, by \cite[Theorem 1.1]{MR920145}, $A\otimes B(\ell^{2}(H))$ is continuous trace with spectrum $\widehat{A}/H$. 
\end{proof}

\begin{cor*}[Corollary \ref{cor:MainC}]
Let $\mathcal{C}$ be a group theoretical fusion category, and $n\ge 2$. Then there exists a closed, connected manifold $X$ of dimension $n$ and an action of $\mathcal{C}$ on a unital continuous trace $\rm C^*$-algebra with spectrum $X$. 
\end{cor*}
\begin{proof}
Let $\mathcal{C}$ be a group theoretical unitary fusion category Morita equivalent to $\fdHilb(G, \omega)$ via a Q-system $(H,\mu)\in \fdHilb(G,\omega)$. 
Let $n\ge 2$. The proof of \cite[Thm.~1.3]{MR4328058} shows there exists an $\overline{\omega}$-anomalous action on a unital $\rm C^*$-algebra $A$ which is Morita equivalent to $C(M)$, where $M$ is a closed, connected $n$-manifold, which is thus continuous trace. 
By \cite[Proposition 5.6]{2105.05587}, this yields an automorphic action of $\fdHilb(G,\omega)$ on the continuous trace $\rm C^*$-algebra $A$ whose spectrum is $M$.
From the construction it is clear that the action of $G$ on $X$ is free. 
Applying the above lemma, $|(H,\mu)|_{F}$ is continuous trace with spectrum $X=M/H$, which is again a closed connected manifold of dimension $n$. 
By Theorem \ref{thm:QSysComplete}, $\mathcal{C}$ acts on $|(H,\mu)|_{F}$.
\end{proof}

We end this article by posing some questions about Q-systems, realization, and actions of UFCs on $\rm C^*$-algebras.

\begin{quest}
\label{quest:ReplaceCTwithC}
Can one replace `continuous trace' with `commutative' in Corollary \ref{cor:MainC}? 
\end{quest}

\begin{rem}
For pointed fusion categories, we have a positive answer to Question \ref{quest:ReplaceCTwithC} by \cite[Thm.~1.3]{MR4328058} (indeed we used this result as input to obtain the corollary). 
In general, the construction outlined above will produce actions on continuous trace $\rm C^*$-algebras which have non-trivial Dixmier-Douady invariant, valued in $H^{3}(X/H, \mathbbm{Z})$. 
These invariants are precisely obstructions to being Morita equivalent with $C(X/H)$. 
The value of this invariant for $|(H,\mu)|_{F}$ can be computed with the tools developed in \cite{MR1446378}. 
\end{rem}


\begin{quest}
What are the Q-systems in $\rCorr(C(X)\to C(X))$?
\end{quest}

Note that if we look at the full subcategory $\Mod(C(X))$, \cite{MR3828898} tells us the Q-systems are bundles of matrix algebras over $X$.

\begin{quest}
What are the connected components of $\QSys(\rCorr)$?
\end{quest}


\begin{quest}
Suppose $Q$ is a Q-system in a unitary tensor category $\cC$ acting on a unital $\rm C^*$-algebra $A$ via $F: \cC \to \fgprCorr(A\to A)$.
Can one compute the $K$-theory of $|F(Q)|$ in terms of the $K$-theory of $A$ and categorical data from $\cC$?
\end{quest}

\bibliographystyle{alpha}
{\footnotesize{
\bibliography{bibliography}

\newcommand{\etalchar}[1]{$^{#1}$}
\begin{thebibliography}{CKRW97}

\bibitem[AMP15]{1509.00038}
Narjess Afzaly, Scott Morrison, and David Penneys.
\newblock The classification of subfactors with index at most $5\frac{1}{4}$,
  2015.
\newblock \arXiv{1509.00038}, to appear {Mem. Amer. Math. Soc.}

\bibitem[Ara]{AranoRokhlin}
Yuki Arano.
\newblock Rokhlin actions of fusion categories.
\newblock In preparation.

\bibitem[BDH14]{MR3342166}
Arthur Bartels, Christopher~L. Douglas, and Andr{\'e} Henriques.
\newblock Dualizability and index of subfactors.
\newblock {\em Quantum Topol.}, 5(3):289--345, 2014.
\newblock \mathscinet{MR3342166} \doi{10.4171/QT/53} \arXiv{1110.5671}.

\bibitem[BKLR15]{MR3308880}
Marcel Bischoff, Yasuyuki Kawahigashi, Roberto Longo, and Karl-Henning Rehren.
\newblock {\em Tensor categories and endomorphisms of von {N}eumann
  algebras---with applications to quantum field theory}, volume~3 of {\em
  SpringerBriefs in Mathematical Physics}.
\newblock Springer, Cham, 2015.
\newblock \mathscinet{MR3308880} \doi{10.1007/978-3-319-14301-9}.

\bibitem[Bla98]{MR1656031}
Bruce Blackadar.
\newblock {\em {$K$}-theory for operator algebras}, volume~5 of {\em
  Mathematical Sciences Research Institute Publications}.
\newblock Cambridge University Press, Cambridge, second edition, 1998.
\newblock \mathscinet{MR1656031}.

\bibitem[BLM04]{MR2111973}
David~P. Blecher and Christian Le~Merdy.
\newblock {\em Operator algebras and their modules---an operator space
  approach}, volume~30 of {\em London Mathematical Society Monographs. New
  Series}.
\newblock The Clarendon Press, Oxford University Press, Oxford, 2004.
\newblock \mathscinet{MR2111973}
  \doi{10.1093/acprof:oso/9780198526599.001.0001}.

\bibitem[BR92]{MR1190414}
Bruce Blackadar and Mikael R{\o}rdam.
\newblock Extending states on preordered semigroups and the existence of
  quasitraces on {$C^*$}-algebras.
\newblock {\em J. Algebra}, 152(1):240--247, 1992.
\newblock \mathscinet{MR1190414} \doi{10.1016/0021-8693(92)90098-7}.

\bibitem[BS10]{MR2664619}
John~C. Baez and Michael Shulman.
\newblock Lectures on {$n$}-categories and cohomology.
\newblock In {\em Towards higher categories}, volume 152 of {\em IMA Vol. Math.
  Appl.}, pages 1--68. Springer, New York, 2010.
\newblock \mathscinet{MR2664619} \arxiv{math/0608420}.

\bibitem[BS17]{MR3572256}
Sel\c{c}uk Barlak and G\'{a}bor Szab\'{o}.
\newblock Rokhlin actions of finite groups on {UHF}-absorbing {${\rm
  C}^*$}-algebras.
\newblock {\em Trans. Amer. Math. Soc.}, 369(2):833--859, 2017.
\newblock \mathscinet{MR3572256} \doi{10.1090/tran6697} \arxiv{1403.7312}.

\bibitem[BZBJ18]{MR3847209}
David Ben-Zvi, Adrien Brochier, and David Jordan.
\newblock Integrating quantum groups over surfaces.
\newblock {\em J. Topol.}, 11(4):874--917, 2018.
\newblock \mathscinet{MR3847209} \doi{10.1112/topo.12072} \arxiv{1501.04652}.

\bibitem[CKRW97]{MR1446378}
David Crocker, Alexander Kumjian, Iain Raeburn, and Dana~P. Williams.
\newblock An equivariant {B}rauer group and actions of groups on
  {$C^*$}-algebras.
\newblock {\em J. Funct. Anal.}, 146(1):151--184, 1997.
\newblock \mathscinet{MR1446378} \doi{10.1006/jfan.1996.3010}.

\bibitem[Con75]{MR394228}
Alain Connes.
\newblock Outer conjugacy classes of automorphisms of factors.
\newblock {\em Ann. Sci. \'{E}cole Norm. Sup. (4)}, 8(3):383--419, 1975.
\newblock \mathscinet{MR394228}.

\bibitem[Con77]{MR448101}
A.~Connes.
\newblock Periodic automorphisms of the hyperfinite factor of type ${II}_1$.
\newblock {\em Acta Sci. Math. (Szeged)}, 39(1-2):39--66, 1977.
\newblock \mathscinet{MR448101}.

\bibitem[CP22]{MR4369356}
Quan Chen and David Penneys.
\newblock Q-system completion is a 3-functor.
\newblock {\em Theory Appl. Categ.}, 38:Paper No. 4, 101--134, 2022.
\newblock \mathscinet{MR4369356} \doi{10.1002/num.22828} \arxiv{2106.12437}.

\bibitem[CR16]{MR3459961}
Nils Carqueville and Ingo Runkel.
\newblock Orbifold completion of defect bicategories.
\newblock {\em Quantum Topol.}, 7(2):203--279, 2016.
\newblock \mathscinet{MR3459961} \doi{10.4171/QT/76} \arxiv{1210.6363}.

\bibitem[D\'22]{MR4372801}
Thibault~D. D\'{e}coppet.
\newblock Multifusion categories and finite semisimple 2-categories.
\newblock {\em J. Pure Appl. Algebra}, 226(8):Paper No. 107029, 16, 2022.
\newblock \mathscinet{MR4372801} \doi{10.1016/j.jpaa.2022.107029}
  \arxiv{2012.15774}.

\bibitem[Dix77]{MR0458185}
Jacques Dixmier.
\newblock {\em {$C\sp*$}-algebras}.
\newblock North-Holland Publishing Co., Amsterdam-New York-Oxford, 1977.
\newblock Translated from the French by Francis Jellett, North-Holland
  Mathematical Library, Vol. 15, \mathscinet{MR0458185}.

\bibitem[DR18]{1812.11933}
Christopher~L. Douglas and David~J. Reutter.
\newblock Fusion 2-categories and a state-sum invariant for 4-manifolds, 2018.
\newblock \arxiv{1812.11933}.

\bibitem[DV11]{MR2838524}
Steven Deprez and Stefaan Vaes.
\newblock A classification of all finite index subfactors for a class of
  group-measure space {${\rm II}\sb 1$} factors.
\newblock {\em J. Noncommut. Geom.}, 5(4):523--545, 2011.
\newblock \mathscinet{MR2838524} \doi{10.4171/JNCG/85}.

\bibitem[EGP21]{2105.05587}
Samuel Evington and Sergio Giron~Pacheco.
\newblock Anomalous symmetries of classifiable {C}*-algebras, 2021.
\newblock \arxiv{2105.05587}.

\bibitem[FR13]{MR3028581}
S{\'e}bastien Falgui{\`e}res and Sven Raum.
\newblock Tensor {${\rm C}\sp{*}$}-categories arising as bimodule categories of
  {${\rm II}\sb{1}$} factors.
\newblock {\em Adv. Math.}, 237:331--359, 2013.
\newblock \mathscinet{MR3028581} \doi{10.1016/j.aim.2012.12.020}
  \arxiv{1112.4088v2}.

\bibitem[FV08]{MR2409162}
S\'{e}bastien Falgui\`eres and Stefaan Vaes.
\newblock Every compact group arises as the outer automorphism group of a
  {${\rm II}_1$} factor.
\newblock {\em J. Funct. Anal.}, 254(9):2317--2328, 2008.
\newblock \mathscinet{MR2409162} \doi{10.1016/j.jfa.2008.02.002}.

\bibitem[GH76]{MR389965}
K.~R. Goodearl and D.~Handelman.
\newblock Rank functions and {$K_{O}$} of regular rings.
\newblock {\em J. Pure Appl. Algebra}, 7(2):195--216, 1976.
\newblock \mathscinet{MR389965} \doi{10.1016/0022-4049(76)90032-3}.

\bibitem[GIS18]{MR3859276}
Pinhas Grossman, Masaki Izumi, and Noah Snyder.
\newblock The {A}saeda--{H}aagerup fusion categories.
\newblock {\em J. Reine Angew. Math.}, 743:261--305, 2018.
\newblock \mathscinet{MR3859276} \doi{10.1515/crelle-2015-0078}
  \arxiv{1501.07324}.

\bibitem[GJF19]{1905.09566}
Davide Gaiotto and Theo Johnson-Freyd.
\newblock Condensations in higher categories, 2019.
\newblock \arxiv{1905.09566}.

\bibitem[GL19]{MR3994584}
Luca Giorgetti and Roberto Longo.
\newblock Minimal index and dimension for 2-{$C^*$}-categories with
  finite-dimensional centers.
\newblock {\em Comm. Math. Phys.}, 370(2):719--757, 2019.
\newblock \mathscinet{MR3994584} \doi{10.1007/s00220-018-3266-x}
  \arxiv{1805.09234}.

\bibitem[GLR85]{MR808930}
P.~Ghez, R.~Lima, and J.~E. Roberts.
\newblock {$W\sp \ast$}-categories.
\newblock {\em Pacific J. Math.}, 120(1):79--109, 1985.
\newblock \mathscinet{MR808930}.

\bibitem[GMP{\etalchar{+}}18]{1810.06076}
Pinhas Grossman, Scott Morrison, David Penneys, Emily Peters, and Noah Snyder.
\newblock The {E}xtended {H}aagerup fusion categories, 2018.
\newblock \arxiv{1810.06076}, to appear {Ann.~Sci.~\'Ec.~Norm.~Sup\'er.}

\bibitem[GS12]{MR2909758}
Pinhas Grossman and Noah Snyder.
\newblock Quantum subgroups of the {H}aagerup fusion categories.
\newblock {\em Comm. Math. Phys.}, 311(3):617--643, 2012.
\newblock \mathscinet{MR2909758}, \doi{10.1007/s00220-012-1427-x}.

\bibitem[Gur13]{MR3076451}
Nick Gurski.
\newblock {\em Coherence in three-dimensional category theory}, volume 201 of
  {\em Cambridge Tracts in Mathematics}.
\newblock Cambridge University Press, Cambridge, 2013.
\newblock \mathscinet{MR3076451} \doi{10.1017/CBO9781139542333}.

\bibitem[GY20]{2010.01072}
Luca Giorgetti and Wei Yuan.
\newblock Realization of rigid {$\rm C^*$}-bicategories as bimodules over type
  {$\rm II_1$} von {N}eumann algebras, 2020.
\newblock \arxiv{2010.01072}.

\bibitem[Haa14]{MR3241179}
Uffe Haagerup.
\newblock Quasitraces on exact {$C^*$}-algebras are traces.
\newblock {\em C. R. Math. Acad. Sci. Soc. R. Can.}, 36(2-3):67--92, 2014.
\newblock \mathscinet{MR3241179}.

\bibitem[Hen14]{MR3221289}
Andr\'e Henriques.
\newblock Three-tier {CFT}s from {F}robenius algebras.
\newblock In {\em Topology and field theories}, volume 613 of {\em Contemp.
  Math.}, pages 1--40. Amer. Math. Soc., Providence, RI, 2014.
\newblock \mathscinet{MR3221289} \doi{10.1090/conm/613/12233}
  \arxiv{1304.7328}.

\bibitem[HHP20]{MR4139893}
Michael Hartglass and Roberto Hern\'{a}ndez~Palomares.
\newblock Realizations of rigid {$\rm C^*$}-tensor categories as bimodules over
  {GJS} {$\rm C^*$}-algebras.
\newblock {\em J. Math. Phys.}, 61(8):081703, 32, 2020.
\newblock \mathscinet{MR4139893} \doi{10.1063/5.0015294}.

\bibitem[HP20]{2004.08271}
Andr\'e Henriques and David Penneys.
\newblock Representations of fusion categories and their commutants, 2020.
\newblock \arxiv{2004.08271}.

\bibitem[HR18]{MR3828898}
Chris Heunen and Manuel~L. Reyes.
\newblock Frobenius structures over {H}ilbert {${\rm C}^*$}-modules.
\newblock {\em Comm. Math. Phys.}, 361(2):787--824, 2018.
\newblock \mathscinet{MR3828898} \doi{10.1007/s00220-018-3166-0}
  \arxiv{1704.05725}.

\bibitem[HV19]{MR3971584}
Chris Heunen and Jamie Vicary.
\newblock {\em Categories for quantum theory}, volume~28 of {\em Oxford
  Graduate Texts in Mathematics}.
\newblock Oxford University Press, Oxford, 2019.
\newblock An introduction, \mathscinet{MR3971584}
  \doi{10.1093/oso/9780198739623.001.0001}.

\bibitem[IM19a]{1810.05850}
Masaki Izumi and Hiroki Matui.
\newblock {Poly-$\mathbb{Z}$ Group Actions on Kirchberg Algebras I}.
\newblock {\em International Mathematics Research Notices}, 08 2019.
\newblock \doi{10.1093/imrn/rnz140} \arxiv{1810.05850}.

\bibitem[IM19b]{1906.03818}
Masaki Izumi and Hiroki Matui.
\newblock Poly-$\mathbb{Z}$ group actions on {K}irchberg algebras {II}, 2019.
\newblock \arxiv{1906.03818}, to appear Invent.~Math.

\bibitem[IPP08]{MR2386109}
Adrian Ioana, Jesse Peterson, and Sorin Popa.
\newblock Amalgamated free products of weakly rigid factors and calculation of
  their symmetry groups.
\newblock {\em Acta Math.}, 200(1):85--153, 2008.
\newblock \mathscinet{MR2386109} \doi{10.1007/s11511-008-0024-5}.

\bibitem[Izu98]{MR1604162}
Masaki Izumi.
\newblock Subalgebras of infinite {$C^*$}-algebras with finite {W}atatani
  indices. {II}. {C}untz-{K}rieger algebras.
\newblock {\em Duke Math. J.}, 91(3):409--461, 1998.
\newblock \mathscinet{MR1604162}, \doi{10.1215/S0012-7094-98-09118-9}.

\bibitem[Izu04]{MR2053753}
Masaki Izumi.
\newblock Finite group actions on {$C^*$}-algebras with the {R}ohlin property.
  {I}.
\newblock {\em Duke Math. J.}, 122(2):233--280, 2004.
\newblock \mathscinet{MR2053753} \doi{10.1215/S0012-7094-04-12221-3}.

\bibitem[JMS14]{MR3166042}
Vaughan F.~R. Jones, Scott Morrison, and Noah Snyder.
\newblock The classification of subfactors of index at most 5.
\newblock {\em Bull. Amer. Math. Soc. (N.S.)}, 51(2):277--327, 2014.
\newblock \mathscinet{MR3166042}, \arxiv{1304.6141},
  \doi{10.1090/S0273-0979-2013-01442-3}.

\bibitem[Jon80]{MR587749}
Vaughan F.~R. Jones.
\newblock Actions of finite groups on the hyperfinite type {${\rm II}_{1}$}\
  factor.
\newblock {\em Mem. Amer. Math. Soc.}, 28(237):v+70, 1980.
\newblock \mathscinet{MR715556}.

\bibitem[Jon08]{MR2501843}
Vaughan F.~R. Jones.
\newblock Two subfactors and the algebraic decomposition of bimodules over
  {$\rm II\sb 1$} factors.
\newblock {\em Acta Math. Vietnam.}, 33(3):209--218, 2008.
\newblock \mathscinet{MR2501843}, available at
  \url{http://math.berkeley.edu/~vfr/algebraic.pdf}.

\bibitem[Jon21a]{MR4328058}
Corey Jones.
\newblock Remarks on anomalous symmetries of {C}*-algebras.
\newblock {\em Comm. Math. Phys.}, 388(1):385--417, 2021.
\newblock \mathscinet{MR4328058} \doi{10.1007/s00220-021-04234-4}
  \arxiv{2011.13898}.

\bibitem[Jon21b]{MR4374438}
V.~F.~R. Jones.
\newblock Planar algebras, {I}.
\newblock {\em New Zealand J. Math.}, 52:1--107, 2021.
\newblock \mathscinet{MR4374438} \doi{10.53733/172} \arxiv{math.QA/9909027}.

\bibitem[JP17]{MR3687214}
Corey Jones and David Penneys.
\newblock Operator algebras in rigid {$\rm C^*$}-tensor categories.
\newblock {\em Comm. Math. Phys.}, 355(3):1121--1188, 2017.
\newblock \mathscinet{MR3687214} \doi{10.1007/s00220-017-2964-0}
  \arxiv{1611.04620}.

\bibitem[JP19]{MR3948170}
Corey Jones and David Penneys.
\newblock Realizations of algebra objects and discrete subfactors.
\newblock {\em Adv. Math.}, 350:588--661, 2019.
\newblock \mathscinet{MR3948170} \doi{10.1016/j.aim.2019.04.039}
  \arxiv{1704.02035}.

\bibitem[JP20]{MR4079745}
Corey Jones and David Penneys.
\newblock Q-systems and compact {W}*-algebra objects.
\newblock In {\em Topological phases of matter and quantum computation}, volume
  747 of {\em Contemp. Math.}, pages 63--88. Amer. Math. Soc., Providence, RI,
  2020.
\newblock \mathscinet{MR4079745} \doi{10.1090/conm/747/15039}
  \arxiv{1707.02155}.

\bibitem[JPR20]{2009.00405}
Corey Jones, David Penneys, and David Reutter.
\newblock A 3-categorical perspective on {G}-crossed braided categories, 2020.
\newblock \arxiv{2009.00405}.

\bibitem[JY21]{MR4261588}
Niles Johnson and Donald Yau.
\newblock {\em 2-dimensional categories}.
\newblock Oxford University Press, Oxford, 2021.
\newblock \mathscinet{2002.06055} \doi{10.1093/oso/9780198871378.001.0001}
  \arxiv{MR4261588}.

\bibitem[KPW04]{MR2085108}
Tsuyoshi Kajiwara, Claudia Pinzari, and Yasuo Watatani.
\newblock Jones index theory for {H}ilbert {$C\sp *$}-bimodules and its
  equivalence with conjugation theory.
\newblock {\em J. Funct. Anal.}, 215(1):1--49, 2004.
\newblock \mathscinet{MR2085108} \doi{10.1016/j.jfa.2003.09.008}
  \arxiv{math/0301259}.

\bibitem[KW00]{MR1624182}
Tsuyoshi Kajiwara and Yasuo Watatani.
\newblock Jones index theory by {H}ilbert {$C^*$}-bimodules and {$K$}-theory.
\newblock {\em Trans. Amer. Math. Soc.}, 352(8):3429--3472, 2000.
\newblock \mathscinet{MR1624182}, \doi{10.1090/S0002-9947-00-02392-8}.

\bibitem[Liu15]{MR3345186}
Zhengwei Liu.
\newblock Composed inclusions of {$A\sb 3$} and {$A\sb 4$} subfactors.
\newblock {\em Adv. Math.}, 279:307--371, 2015.
\newblock \mathscinet{MR3345186} \doi{10.1016/j.aim.2015.03.017}
  \arxiv{1308.5691}.

\bibitem[Lon84]{MR739630}
R.~Longo.
\newblock Solution of the factorial {S}tone-{W}eierstrass conjecture. {A}n
  application of the theory of standard split {$W\sp{\ast} $}-inclusions.
\newblock {\em Invent. Math.}, 76(1):145--155, 1984.
\newblock \mathscinet{MR739630} \doi{10.1007/BF01388497}.

\bibitem[Lon89]{MR1027496}
Roberto Longo.
\newblock Index of subfactors and statistics of quantum fields. {I}.
\newblock {\em Comm. Math. Phys.}, 126(2):217--247, 1989.
\newblock \mathscinet{MR1027496}.

\bibitem[Lon94]{MR1257245}
Roberto Longo.
\newblock A duality for {H}opf algebras and for subfactors. {I}.
\newblock {\em Comm. Math. Phys.}, 159(1):133--150, 1994.
\newblock \mathscinet{MR1257245}.

\bibitem[LR97]{MR1444286}
R.~Longo and J.~E. Roberts.
\newblock A theory of dimension.
\newblock {\em $K$-Theory}, 11(2):103--159, 1997.
\newblock \mathscinet{MR1444286} \doi{10.1023/A:1007714415067}.

\bibitem[LR04]{MR2097363}
Roberto Longo and Karl-Henning Rehren.
\newblock Local fields in boundary conformal {QFT}.
\newblock {\em Rev. Math. Phys.}, 16(7):909--960, 2004.
\newblock \mathscinet{MR2097363} \doi{10.1142/S0129055X04002163}
  \arxiv{math-ph/0405067}.

\bibitem[M{\"u}g03]{MR1966524}
Michael M{\"u}ger.
\newblock From subfactors to categories and topology. {I}. {F}robenius algebras
  in and {M}orita equivalence of tensor categories.
\newblock {\em J. Pure Appl. Algebra}, 180(1-2):81--157, 2003.
\newblock \mathscinet{MR1966524} \doi{10.1016/S0022-4049(02)00247-5}
  \arXiv{math.CT/0111204}.

\bibitem[NT13]{MR3204665}
Sergey Neshveyev and Lars Tuset.
\newblock {\em Compact quantum groups and their representation categories},
  volume~20 of {\em Cours Sp\'ecialis\'es [Specialized Courses]}.
\newblock Soci\'et\'e Math\'ematique de France, Paris, 2013.
\newblock \mathscinet{MR3204665}.

\bibitem[NY16]{MR3509018}
Sergey Neshveyev and Makoto Yamashita.
\newblock Drinfeld center and representation theory for monoidal categories.
\newblock {\em Comm. Math. Phys.}, 345(1):385--434, 2016.
\newblock \mathscinet{MR3509018} \doi{10.1007/s00220-016-2642-7}
  \arxiv{1501.07390}.

\bibitem[NY18]{MR3933035}
Sergey Neshveyev and Makoto Yamashita.
\newblock Categorically {M}orita equivalent compact quantum groups.
\newblock {\em Doc. Math.}, 23:2165--2216, 2018.
\newblock \mathscinet{MR3933035} \arxiv{1704.04729}.

\bibitem[Ocn80]{MR596082}
Adrian Ocneanu.
\newblock Actions des groupes moyennables sur les alg\`ebres de von {N}eumann.
\newblock {\em C. R. Acad. Sci. Paris S\'er. A-B}, 291(6):A399--A401, 1980.
\newblock \mathscinet{MR596082}.

\bibitem[Ocn88]{MR996454}
Adrian Ocneanu.
\newblock Quantized groups, string algebras and {G}alois theory for algebras.
\newblock In {\em Operator algebras and applications, Vol.\ 2}, volume 136 of
  {\em London Math. Soc. Lecture Note Ser.}, pages 119--172. Cambridge Univ.
  Press, Cambridge, 1988.
\newblock \mathscinet{MR996454}.

\bibitem[Ost03a]{MR1976459}
Victor Ostrik.
\newblock Module categories, weak {H}opf algebras and modular invariants.
\newblock {\em Transform. Groups}, 8(2):177--206, 2003.
\newblock \mathscinet{MR1976459} \arXiv{math/0111139}.

\bibitem[Ost03b]{MR1976233}
Viktor Ostrik.
\newblock Module categories over the {D}rinfeld double of a finite group.
\newblock {\em Int. Math. Res. Not.}, (27):1507--1520, 2003.
\newblock \mathscinet{MR1976233},\arXiv{0202130}.

\bibitem[Pas73]{MR355613}
William~L. Paschke.
\newblock Inner product modules over {$B\sp{\ast} $}-algebras.
\newblock {\em Trans. Amer. Math. Soc.}, 182:443--468, 1973.
\newblock \mathscinet{MR355613} \doi{10.2307/1996542}.

\bibitem[Pen13]{MR3040370}
David Penneys.
\newblock A {P}lanar {C}alculus for {I}nfinite {I}ndex {S}ubfactors.
\newblock {\em Comm. Math. Phys.}, 319(3):595--648, 2013.
\newblock \mathscinet{MR3040370} \arXiv{1110.3504}
  \doi{10.1007/s00220-012-1627-4}.

\bibitem[Pen20]{MR4133163}
David Penneys.
\newblock Unitary dual functors for unitary multitensor categories.
\newblock {\em High. Struct.}, 4(2):22--56, 2020.
\newblock \mathscinet{MR4133163} \arxiv{1808.00323}.

\bibitem[Pop94]{MR1278111}
Sorin Popa.
\newblock Classification of amenable subfactors of type {II}.
\newblock {\em Acta Math.}, 172(2):163--255, 1994.
\newblock \mathscinet{MR1278111}, \doi{10.1007/BF02392646}.

\bibitem[Pop95a]{MR1334479}
Sorin Popa.
\newblock An axiomatization of the lattice of higher relative commutants of a
  subfactor.
\newblock {\em Invent. Math.}, 120(3):427--445, 1995.
\newblock \mathscinet{MR1334479} \doi{10.1007/BF01241137}.

\bibitem[Pop95b]{MR1339767}
Sorin Popa.
\newblock {\em Classification of subfactors and their endomorphisms}, volume~86
  of {\em CBMS Regional Conference Series in Mathematics}.
\newblock Published for the Conference Board of the Mathematical Sciences,
  Washington, DC, 1995.
\newblock \mathscinet{MR1339767}.

\bibitem[PR89]{MR1002543}
Judith~A. Packer and Iain Raeburn.
\newblock Twisted crossed products of {$C^*$}-algebras.
\newblock {\em Math. Proc. Cambridge Philos. Soc.}, 106(2):293--311, 1989.
\newblock \mathscinet{MR1002543} \doi{10.1017/S0305004100078129}.

\bibitem[PV08]{MR2370283}
Sorin Popa and Stefaan Vaes.
\newblock Strong rigidity of generalized {B}ernoulli actions and computations
  of their symmetry groups.
\newblock {\em Adv. Math.}, 217(2):833--872, 2008.
\newblock \mathscinet{MR2370283} \doi{10.1016/j.aim.2007.09.006}.

\bibitem[Rie74]{MR0367670}
Marc~A. Rieffel.
\newblock Morita equivalence for {$C^{\ast} $}-algebras and {$W^{\ast}
  $}-algebras.
\newblock {\em J. Pure Appl. Algebra}, 5:51--96, 1974.
\newblock \mathscinet{MR0367670}.

\bibitem[RR88]{MR920145}
Iain Raeburn and Jonathan Rosenberg.
\newblock Crossed products of continuous-trace {$C^\ast$}-algebras by smooth
  actions.
\newblock {\em Trans. Amer. Math. Soc.}, 305(1):1--45, 1988.
\newblock \mathscinet{MR920145} \doi{10.2307/2001039}.

\bibitem[RSW00]{MR1798596}
Iain Raeburn, Aidan Sims, and Dana~P. Williams.
\newblock Twisted actions and obstructions in group cohomology.
\newblock In {\em {$C^*$}-algebras ({M}\"{u}nster, 1999)}, pages 161--181.
  Springer, Berlin, 2000.
\newblock \mathscinet{MR1798596}.

\bibitem[RW98]{MR1634408}
Iain Raeburn and Dana~P. Williams.
\newblock {\em Morita equivalence and continuous-trace {$C^*$}-algebras},
  volume~60 of {\em Mathematical Surveys and Monographs}.
\newblock American Mathematical Society, Providence, RI, 1998.
\newblock \mathscinet{MR1634408} \doi{10.1090/surv/060}.

\bibitem[SY17]{1705.05600}
Yusuke Sawada and Shigeru Yamagami.
\newblock Notes on the bicategory of $\rm {W}^*$-bimodules, 2017.
\newblock \arxiv{1705.05600}.

\bibitem[Sza19]{MR4015345}
G\'{a}bor Szab\'{o}.
\newblock Actions of certain torsion-free elementary amenable groups on
  strongly self-absorbing {$\rm C^*$}-algebras.
\newblock {\em Comm. Math. Phys.}, 371(1):267--284, 2019.
\newblock \mathscinet{MR4015345} \doi{10.1007/s00220-019-03435-2}
  \arxiv{1807.03020}.

\bibitem[Tak02]{MR1873025}
Masamichi Takesaki.
\newblock {\em Theory of operator algebras. {I}}, volume 124 of {\em
  Encyclopaedia of Mathematical Sciences}.
\newblock Springer-Verlag, Berlin, 2002.
\newblock Reprint of the first (1979) edition, Operator Algebras and
  Non-commutative Geometry, 5, ISBN: 3-540-42248-X, \mathscinet{MR1873025}.

\bibitem[Tho11]{MR2771095}
Andreas Thom.
\newblock A remark about the {C}onnes fusion tensor product.
\newblock {\em Theory Appl. Categ.}, 25:No. 2, 38--50, 2011.
\newblock \mathscinet{MR2771095} \arxiv{math/0601045}.

\bibitem[Tom21]{MR4236062}
Reiji Tomatsu.
\newblock Centrally {F}ree {A}ctions of {A}menable {$\rm C^*$}-tensor
  categories on von {N}eumann algebras.
\newblock {\em Comm. Math. Phys.}, 383(1):71--152, 2021.
\newblock \mathscinet{MR4236062} \doi{10.1007/s00220-021-04037-7}
  \arxiv{1812.04222}.

\bibitem[Vae08]{MR2504433}
Stefaan Vaes.
\newblock Explicit computations of all finite index bimodules for a family of
  {${\rm II}_1$} factors.
\newblock {\em Ann. Sci. \'Ec. Norm. Sup\'er. (4)}, 41(5):743--788, 2008.
\newblock \mathscinet{MR2504433}.

\bibitem[Vae09]{MR2471930}
Stefaan Vaes.
\newblock Factors of type {II{$_1$}} without non-trivial finite index
  subfactors.
\newblock {\em Trans. Amer. Math. Soc.}, 361(5):2587--2606, 2009.
\newblock \mathscinet{MR2471930} \doi{10.1090/S0002-9947-08-04585-6}.

\bibitem[Vic11]{MR2794547}
Jamie Vicary.
\newblock Categorical formulation of finite-dimensional quantum algebras.
\newblock {\em Comm. Math. Phys.}, 304(3):765--796, 2011.
\newblock \mathscinet{MR2794547} \doi{10.1007/s00220-010-1138-0}
  \arxiv{0805.0432}.

\bibitem[Wat90]{MR996807}
Yasuo Watatani.
\newblock Index for ${C}^*$-subalgebras.
\newblock {\em Mem. Amer. Math. Soc.}, 83(424):vi+117 pp., 1990.
\newblock \mathscinet{MR996807}, \googlebooks{Bp2cmONVye0C}.

\bibitem[Yam04a]{MR2075605}
Shigeru Yamagami.
\newblock Frobenius algebras in tensor categories and bimodule extensions.
\newblock In {\em Galois theory, {H}opf algebras, and semiabelian categories},
  volume~43 of {\em Fields Inst. Commun.}, pages 551--570. Amer. Math. Soc.,
  Providence, RI, 2004.
\newblock \mathscinet{MR2075605}.

\bibitem[Yam04b]{MR2091457}
Shigeru Yamagami.
\newblock Frobenius duality in {$C^*$}-tensor categories.
\newblock {\em J. Operator Theory}, 52(1):3--20, 2004.
\newblock \mathscinet{MR2091457}.

\bibitem[Yam07]{MR2325696}
Shigeru Yamagami.
\newblock Notes on operator categories.
\newblock {\em J. Math. Soc. Japan}, 59(2):541--555, 2007.
\newblock \mathscinet{MR2325696} \arxiv{math/0212136}.

\bibitem[Zit07]{MR2298822}
Pasquale~A. Zito.
\newblock 2-{$C^*$}-categories with non-simple units.
\newblock {\em Adv. Math.}, 210(1):122--164, 2007.
\newblock \mathscinet{MR2298822} \doi{10.1016/j.aim.2006.05.017}
  \arxiv{math/0509266}.

\end{thebibliography}
}}

\end{document}